\documentclass[11pt,reqno]{amsart}

\usepackage{amssymb}
\usepackage{graphicx}
\usepackage{pb-diagram,pb-xy} 
\usepackage[cmtip,arrow]{xy} 
\usepackage[margin=1in]{geometry}  
\usepackage{hyperref} 

\numberwithin{equation}{section}

\theoremstyle{plain}
\newtheorem{theorem}[equation]{Theorem}

\newtheorem{lemma}[equation]{Lemma}

\newtheorem{cor}[equation]{Corollary}
\newtheorem{corollary}[equation]{Corollary}
\newtheorem{proposition}[equation]{Proposition}

\newtheorem{prop}[equation]{Proposition}

\newtheorem*{claim*}{Claim}
\newtheorem*{theorem*}{Theorem}
\newtheorem*{question*}{Question}

\newtheorem*{prop:cell-complex}{Proposition \ref{prop:cell-complex}}
\newtheorem*{thm:right}{Theorem \ref{thm:right}}
\newtheorem*{thm:left}{Theorem \ref{thm:left}}

\theoremstyle{definition}
\newtheorem{definition}[equation]{Definition}

\newtheorem{notation}[equation]{Notation}

\newtheorem{remark}[equation]{Remark}
\newtheorem{remarks}[equation]{Remarks}
\newtheorem{example}[equation]{Example}
\newtheorem{examples}[equation]{Examples}

\newcommand{\isom}{\cong}                       
\newcommand{\homeq}{\simeq}                     
\newcommand{\smsh}{\wedge}                      
\newcommand{\wdge}{\vee}                        
\newcommand{\union}{\cup}                       
\newcommand{\varcomp}{\mathbin{\hat{\circ}}}    

\newcommand{\Smsh}{\bigwedge}                   
\newcommand{\Wdge}{\bigvee}                     

\newcommand{\cat}[1]{\mathcal{#1}}              

\DeclareMathOperator*{\hocolim}{hocolim}
\DeclareMathOperator*{\colim}{colim}
\DeclareMathOperator*{\holim}{holim}
\newcommand{\Map}{\operatorname{Map} }
\newcommand{\Ext}{\operatorname{Ext} }
\newcommand{\Nat}{\operatorname{Nat} }
\newcommand{\creff}{\operatorname{cr} }
\newcommand{\Hom}{\operatorname{Hom} }
\newcommand{\Tot}{\operatorname{Tot} }
\DeclareMathOperator*{\hofib}{hofib}
\DeclareMathOperator*{\hocofib}{hocofib}
\DeclareMathOperator*{\thofib}{thofib}
\DeclareMathOperator*{\thocofib}{thocofib}
\DeclareMathOperator*{\Sing}{Sing}

\newcommand{\sset}{\mathsf{sSet}_*}                     
\newcommand{\finsset}{\mathsf{sSet}_*^{\mathsf{fin}}}   
\newcommand{\based}{\mathsf{Top}_*}                     
\newcommand{\spectra}{\mathsf{Spec}}                       
\newcommand{\finspec}{{\spectra^{\mathsf{fin}}}}        

\newcommand{\weq}{\; \tilde{\longrightarrow} \;}      
\newcommand{\lweq}{\; \tilde{\longleftarrow} \;}      
\newcommand{\fib}{\twoheadrightarrow}           
\newcommand{\epi}{\twoheadrightarrow}           
\newcommand{\cof}{\rightarrowtail}              
\newcommand{\into}{\hookrightarrow}

\newcommand{\dual}{\mathbb{D}}                  
\newcommand{\der}{\partial}                     

\newcommand{\ord}[1]{$#1$\textsuperscript{th}}

\begin{document}

\title[Operads and chain rules]{Operads and chain rules for the calculus of functors}
\author{Greg Arone}
\thanks{The first author was supported in part by NSF Research Grant DMS 0605073.}

\author{Michael Ching}

\subjclass[2000]{55P65}

\begin{abstract}
We study the structure possessed by the Goodwillie derivatives of a pointed homotopy functor of based topological spaces. These derivatives naturally form a bimodule over the operad consisting of the derivatives of the identity functor. We then use these bimodule structures to give a chain rule for higher derivatives in the calculus of functors, extending that of Klein and Rognes. This chain rule expresses the derivatives of $FG$ as a derived composition product of the derivatives of $F$ and $G$ over the derivatives of the identity.

There are two main ingredients in our proofs. Firstly, we construct new models for the Goodwillie derivatives of functors of spectra. These models allow for natural composition maps that yield operad and module structures. Then, we use a cosimplicial cobar construction to transfer this structure to functors of topological spaces. A form of Koszul duality for operads of spectra plays a key role in this.
\end{abstract}

\maketitle

In a landmark series of papers, \cite{goodwillie:1990}, \cite{goodwillie:1991} and \cite{goodwillie:2003}, Goodwillie outlines his `calculus of homotopy functors'. Let $F: \cat{C} \to \cat{D}$ (where $\cat{C}$ and $\cat{D}$ are each either $\based$, the category of pointed topological spaces, or $\spectra$, the category of spectra) be a pointed homotopy functor. One of the things that Goodwillie does is associate with $F$ a sequence of spectra, which are called the \emph{derivatives} of $F$. We denote these spectra by  $\der_1F, \der_2F, \ldots, \der_nF, \ldots$, or, collectively, by $\der_*F$. Importantly, for each $n$ the spectrum $\der_nF$ has a natural action of the symmetric group $\Sigma_n$. Thus, $\der_*F$ is a \emph{symmetric sequence} of spectra.

The importance of the derivatives of $F$ is that they contain substantial information about the homotopy type of $F$. Goodwillie's main construction in~\cite{goodwillie:2003} defines a sequence of `approximations' to $F$ together with natural transformations forming a so-called `Taylor tower'. This tower takes the form
\[ F \to \dots \to P_nF \to P_{n-1}F \to \dots \to P_0F \]
with $P_nF$ being the universal `$n$-excisive' approximation to $F$. (A functor is \emph{$n$-excisive} if it takes any $n+1$-dimensional cube with homotopy pushout squares for faces to a homotopy cartesian cube.) For `analytic' $F$, this tower \emph{converges} for sufficiently highly connected $X$, that is
\[ F(X) \homeq \holim_n P_nF(X). \]
In order to understand the functors $P_nF$ better, Goodwillie analyzes the fibre $D_nF$ of the map $P_nF \to P_{n-1}F$. This fibre is an `$n$-homogeneous' functor in an appropriate sense, and Goodwillie shows in \cite{goodwillie:2003} that $D_nF$ is determined by $\der_nF$, via the following formula. If $F$ takes values in $\spectra$ then
\[ D_nF(X)\simeq (\der_nF \smsh X^{\smsh n})_{h\Sigma_n} .\]
If $F$ takes values in $\based$ then one needs to prefix the right hand side with $\Omega^\infty$.

This paper investigates the question of what additional structure the collection $\der_*F$ naturally possesses, beyond the symmetric group actions. The first example of such structure was given by the second author in~\cite{ching:2005a}. There, he constructed an operad structure on the sequence $\der_*I_{\based}$, where $I_{\based}$ is the identity functor on $\based$. Our first main result says that if $F$ is a functor from $\based$ to $\based$, then $\der_*F$ has the structure of a bimodule over the operad $\der_*I_{\based}$. (For functors either only from or to $\based$, we get left or right module structures respectively.)

It turns out that these bimodule structures are exactly what is needed to write a `chain rule' for the calculus of functors. By a chain rule we mean a formula for describing the derivatives of a composite functor $FG$ in terms of $\der_*F$ and $\der_*G$. Such a chain rule was first studied by Klein and Rognes, in~\cite{klein/rognes:2002}, who provided a complete answer to this question for first derivatives. In this paper we extend some of their work to higher derivatives, although with some restrictions. In particular, we only consider reduced functors (those with $F(*) \homeq *$) and only derivatives based at the trivial object $*$. (Klein and Rognes consider derivatives at a general base object.) Our result expresses $\der_*(FG)$ as a derived `composition product' of the $\der_*I_{\based}$-bimodule structures on $\der_*F$ and $\der_*G$.

The proofs of our main theorems are rather roundabout, but give us additional interesting results along the way. We first treat the case of functors from $\spectra$ to $\spectra$, and construct new models for the derivatives of such functors. Then, to pass from $\spectra$ to $\based$, we rely heavily on the close connection between topological spaces and coalgebras over the cooperad $\Sigma^\infty\Omega^\infty$, and also on a form of `Koszul duality' for operads in $\spectra$. Koszul duality for operads was first introduced by Ginzburg and Kapranov, in~\cite{ginzburg/kapranov:1994}, in the context of operads of chain complexes. Some of their ideas were extended to operads of spectra by the second author in \cite{ching:2005a}. In particular, it was shown there that the operad $\der_*I_{\based}$ plays the role of the Koszul dual of the commutative cooperad, and hence is a spectrum-level version of the Lie operad. In this paper, we give a deeper topological reason behind this observation. The commutative cooperad appears because it is equivalent to the derivatives of the comonad $\Sigma^\infty \Omega^\infty$. One of our main results is that the derivatives of $I_{\based}$ and $\Sigma^\infty \Omega^\infty$ are related by this form of Koszul duality.

The module and bimodule structures on the derivatives of a general functor $F$ also arise via an extension of Koszul duality ideas to spectra. For example, we show that for $F: \based \to \based$, the derivatives of $F$ and $\Sigma^\infty F$ are related by a corresponding duality between left $\der_*I_{\based}$-modules and left $\der_*(\Sigma^\infty \Omega^\infty)$-comodules.

It seems to us that this paper gives a satisfactory answer to one of the open-ended questions proposed in the introduction to~\cite{ching:2005a}: is there a deeper connection between calculus of functors and the theory of operads? Yes, there is a deeper connection. It stems from two basic sources. The first is the fact that composition of functors is related to the composition product of symmetric sequences. The second is the relationship between $I_{\based}$ and $\Sigma^\infty\Omega^\infty$, which translates, on the level of derivatives, to Koszul duality of operads.

We now proceed with a more precise statement of our main results.

\subsection*{Our results}

As we mentioned already, our results are stated in the language of operads and modules over them. The collection of derivatives of a functor $F$ (of either based spaces or spectra) forms a symmetric sequence of spectra, that is, a sequence
\[ \der_*F = (\der_1F,\der_2F,\dots) \]
in which the \ord{n} term has an action of the symmetric group $\Sigma_n$. (In this paper, we only consider derivatives `at' the trivial object $*$. Derivatives based at other objects are more complicated entities.)

An \emph{operad} consist of a symmetric sequence together with various maps involving the smash products of the terms in the sequence. These maps can be succinctly described by the \emph{composition product}, a (non-symmetric) monoidal product on the category of symmetric sequences. An operad is precisely a monoid for this product. A \emph{module} over an operad $P$ is a symmetric sequence together with an action of the monoid $P$. Because the composition product is non-symmetric, right and left modules have very different flavours. We also have bimodules, either over a single operad, or two separate ones. We review all these notions in Section~\ref{sec:bar}.

We already mentioned that, in \cite{ching:2005a}, the second author constructed an operad structure on $\der_*I_{\based}$, where $I_{\based}$ denotes the identity functor on the category of based topological spaces. The derivatives $\der_*I_{\spectra}$ of the identity functor on spectra, also form an operad, albeit in a trivial way because $\der_*I_{\spectra}$ is equivalent to the unit object $\mathsf{1}$ for the composition product. (The identity functor on spectra is linear so all higher derivatives are trivial.) This observation allows us to state our first theorem in the following way.

\begin{theorem} \label{thm:intro-bimodules}
Let $F: \cat{C} \to \cat{D}$ be a homotopy functor with each of $\cat{C}$ and $\cat{D}$ either equal to $\based$ or $\spectra$. Then the derivatives of $F$ can be given the structure of a $(\der_*I_{\cat{D}},\der_*I_{\cat{C}})$-bimodule in a natural way.
\end{theorem}

We point out that a left or right module structure over $\der_*I_{\spectra} \homeq 1$ provides no more information than a symmetric sequence. Theorem \ref{thm:intro-bimodules} therefore only contains new content in the case that either $\cat{C}$ or $\cat{D}$ (or both) is equal to $\based$. The actual meaning of the theorem is that if $F$ is a functor \emph{from} $\based$, then $\der_*F$ has the structure of a \emph{right} $\der_*I_{\based}$-module; if $F$ is a functor \emph{to} $\based$, then $\der_*F$ has the structure of a \emph{left} $\der_*I_{\based}$-module; and if $F$ is a functor $\based \to \based$, then $\der_*F$ has the structure of a $\der_*I_{\based}$-\emph{bi}module. We have a couple of reasons for phrasing Theorem \ref{thm:intro-bimodules} in this general form. One is that we believe the corresponding statement should be true for categories other than $\based$ and $\spectra$, with module structures over operads formed from the derivatives of the relevant identity functors. Another reason is that it allows us to state a general version of our chain rule for Goodwillie derivatives, as follows.

\begin{theorem} \label{thm:intro-chainrule}
Let $F: \cat{D} \to \cat{E}$ and $G: \cat{C} \to \cat{D}$ be \emph{reduced} homotopy functors with $\cat{C}$, $\cat{D}$ and $\cat{E}$ each equal to either $\based$ or $\spectra$. Suppose further that the functor $F$ preserves filtered homotopy colimits. Then there is a natural equivalence of $(\der_*I_{\cat{E}},\der_*I_{\cat{C}})$-bimodules of the form
\[ \der_*(FG) \homeq \der_*(F) \circ_{\der_*I_{\cat{D}}} \der_*(G). \]
\end{theorem}

The right-hand side here is a \emph{derived} composition product `over' the derivatives of the identity functor on $\cat{D}$, using the right $\der_*I_{\cat{D}}$-module structure on $\der_*(F)$ and the left $\der_*I_{\cat{D}}$-module structure on $\der_*(G)$ of Theorem \ref{thm:intro-bimodules}. The derived composition product can be constructed explicitly as a two-sided bar construction and we make extensive use of such bar constructions in this paper. In the case that $\cat{D} = \spectra$, Theorem \ref{thm:intro-chainrule} reduces to the statement
\[ \der_*(FG) \homeq \der_*(F) \circ \der_*(G).\]
This identity may be viewed as a direct analogue of the classical {Fa\`{a}} di Bruno formula \cite{johnson:2002}. In the special case when $\cat{C}=\cat{D}=\cat{E}=\spectra$, Theorem~\ref{thm:intro-chainrule} was proved in \cite{ching:2007} by a different method.

Notice that the statement of Theorem \ref{thm:intro-chainrule} is restricted to \emph{reduced} functors, that is, those $G$ for which $G(*) \homeq *$. This is largely because we deal only with derivatives \emph{at $*$}. To state the analogous chain rule in the non-reduced case, we would have to consider derivatives based at other objects, which require some extra technology (such as parameterized spectra). We do believe that Theorem~\ref{thm:intro-chainrule}, with the notion of derivative suitably interpreted, should hold for derivatives based at arbitrary objects, and for non-reduced functors. This would generalize the full force of Klein and Rognes' chain rule \cite{klein/rognes:2002}. Indeed, their result originally inspired the form of Theorem~\ref{thm:intro-chainrule}.

The other restriction made in Theorem \ref{thm:intro-chainrule} is that the functor $F$ should preserve filtered homotopy colimits. This is an essential condition. A counterexample in the case this does not hold is given in Example \ref{ex:counterexample}. Intuitively, the reason for this hypothesis is that the derivatives of a functor depend only on the values of that functor on finite cell complexes. This condition ensures that the entire functor $F$ is determined by its values on such inputs.

The focus of this paper is on theory, but we also consider a few examples. In particular, we compute the right $\der_*(I_{\based})$-module structure on the derivatives of $A(-)$ (Waldhausen's algebraic K-theory of spaces functor, following Goodwillie's calculation of those derivatives in \cite[9.7]{goodwillie:2003}), and $\Sigma^\infty \Map(K,-)$ (the functor of stable mapping spaces out of a finite complex, following the first author's calculation of its derivatives in \cite{arone:1999}).

\newpage

\subsection*{Open questions, possible directions for future work}
\subsubsection*{Is there a more direct approach to the chain rule?} The following is a natural question to ask. A positive answer would point to a simpler and more direct way to prove many of our results. Suppose that $F$ and $G$ are homotopy functors such that the composition $FG$ is defined. For concreteness, let us assume that $F$ and $G$ are functors between categories that are either $\based$ or $\spectra$.

\begin{question*}
Is there a model of $\der_*$ that is endowed with natural maps
\[ \mu: \der_*F\circ \der_*G \longrightarrow \der_*(FG) \]
and
\[ \eta: \mathsf{1} \longrightarrow \der_*I \]
such that $\mu$ is associative, in the evident sense, $\eta$ induces an equivalence for $*=1$, and the two maps together are unital in the sense that the following composed maps equal the identity:
\[ \der_*F \isom \der_*F\circ \mathsf{1} \longrightarrow \der_*F \circ \der_*I \longrightarrow \der_*(FI)=\der_*F \]
and
\[ \der_*F \isom \mathsf{1}\circ \der_*F \longrightarrow \der_*I \circ\der_*F  \longrightarrow \der_*(IF)=\der_*F. \]
\end{question*}

It is relatively straightforward to construct a composition map that is associative and unital up to homotopy. However, to construct a strictly associative (or even $A_\infty$) model appears to be difficult.

A positive answer to this question would imply that if $F$ is a monad and $G$ is a module over $F$, then $\der_*F$ has a natural operad structure and $\der_*G$ has the structure of a module over $\der_*F$. Taking $F$ to be the identity functor would then imply Theorem~\ref{thm:intro-bimodules}. Furthermore, it would imply the existence of a natural map
\[ \der_*F \circ_{\der_*I_{\cat{D}}} \der_*G \longrightarrow \der_*(FG) \]
and the chain rule (Theorem~\ref{thm:intro-chainrule}) would then amount to the assertion that this map is an equivalence.

As far as we know, no-one has yet managed to construct such maps, although unpublished work in this direction has been done by Bill Richter and Andrew Mauer-Oats.

\subsubsection*{Does the chain rule hold for functors between categories other than $\based$ and $\spectra$?} Kuhn showed in \cite{kuhn:2007} that many of Goodwillie's constructions apply equally well to functors between other categories. In general, if $\cat{C}$ and $\cat{D}$ are (simplicial) model categories, one can define a Taylor tower for functors $F: \cat{C} \to \cat{D}$. Lurie \cite{lurie:2008a} has examined how some of these ideas can be developed in the context of $\infty$-categories. We suspect that suitable versions of Theorems \ref{thm:intro-bimodules} and \ref{thm:intro-chainrule} apply in these more general settings.

Suppose $\cat{C}$ is a category appropriate for `doing calculus' in one of the above senses. Then one can often define a stabilization of $\cat{C}$, denoted $\spectra(\cat{C})$, that plays the role of the category of spectra for topological spaces. Schwede and Shipley, in \cite{schwede/shipley:2003a}, have shown, in good cases, how to present $\spectra(\cat{C})$ as the category of modules over a kind of generalized ring spectrum, which we can denote $R_{\cat{C}}$. (In the classical case, we have $R_{\based} = S$, the sphere spectrum.) If $F: \cat{C} \to \cat{D}$ is a functor between categories for which this process works, one should be able to interpret the \ord{n} derivative of $F$ as a $(R_{\cat{D}},R_{\cat{C}}^{\smsh n})$-bimodule with an appropriate $\Sigma_n$-action. There are corresponding notions of operad and module for sequences of such objects, and we speculate that, in general, the derivatives of the identity functor on $\cat{C}$ have an operad structure. We conjecture that Theorems \ref{thm:intro-bimodules} and \ref{thm:intro-chainrule} carry over in essentially the same form.

Many of the ideas present in the proofs of these Theorems should be applicable in a more general setting. In particular, for $\cat{C}$ as above, we have an adjunction
\[ \Sigma^\infty_{\cat{C}}: \cat{C} \leftrightarrow \spectra(\cat{C}): \Omega^\infty_{\cat{C}} \]
and hence a comonad $\Sigma^\infty_{\cat{C}} \Omega^\infty_{\cat{C}}$. We propose that there exists a duality relationship between the derivatives of $I_{\cat{C}}$ and $\Sigma^\infty_{\cat{C}} \Omega^\infty_{\cat{C}}$ extending that described in this paper for $\cat{C} = \based$. In particular, this could provide interesting examples of Koszul dual operads.

A key example we have in mind for this would be to take $\cat{C}$ equal to the category of spaces over a fixed base space $X$. This is the correct context for describing the Taylor tower of a functor expanded at an arbitrary base object. As we mentioned already, we need this to complete the extension of the Klein-Rognes chain rule to higher derivatives. Their result fits into the conjectural framework described above.

Another family of examples is provided by categories of algebras over operads. This includes various categories of ring spectra (algebras over the associative or commutative operads in $\spectra$) and things like simplicial commutative algebras. If $\cat{C}$ is equal to the category of $P$-algebras, where $P$ is an operad of spectra, we conjecture that the operad $\der_*I_{\cat{C}}$ is equivalent to $P$. In this case, the suspension spectrum functor $\Sigma^\infty$ is related to topological Andr\'{e}-Quillen homology, and our results would tie in with work of Basterra and Mandell \cite{basterra/mandell:2005}.

A somewhat different class of examples arises from Weiss's orthogonal calculus~\cite{weiss:1995}. This is a calculus of functors $\cat{J} \to \cat{C}$, where $\cat{J}$ is the category of Euclidean vector spaces and $\cat{C}$ is either $\based$ or $\spectra$. Let $G: \cat{J} \to \cat{C}$ be such a functor. In this case, the derivatives $\der_*G$ do not form a symmetric sequence, but an \emph{orthogonal sequence} of spectra, that is, $\der_nG$ is a spectrum with an action of $O(n)$ for each $n$. The category of orthogonal sequences is \emph{left tensored} over the category of symmetric sequences. This means that if $M$ is a symmetric sequence and $Q$ is an orthogonal sequence, then there is a natural way to define a composition product $M \circ Q$, which is again an orthogonal sequence. It follows, in particular, that there is a well-defined notion for orthogonal sequences of a left module over an operad. The analogue of Theorem \ref{thm:intro-bimodules} would say that the orthogonal sequence $\der_*G$ has a left $\der_*I_{\based}$-module structure.

Now let $F: \cat{C} \longrightarrow \cat{D}$ be another functor where $\cat{D}$ is again either $\based$ or $\spectra$. Then the composite $FG$ is defined, and we believe that Theorem~\ref{thm:intro-chainrule} should basically hold verbatim, that is:
\[ \der_*(FG) \homeq \der_*F \circ_{\der_*I_{\cat{C}}} \der_*G \]
where now this is an equivalence of orthogonal sequences.

Finally, if $G: \cat{J} \to \based$, then we believe that the orthogonal sequences $\der_*(G)$ and $\der_*(\Sigma^\infty G)$ are related by the same sort of Koszul duality used in this paper. This often provides a practical way to calculate $\der_*(G)$ and this method has been used implicitly by the first author in some recent and current work \cite{arone:2002,arone:2009}.

Note that we do not know of a \emph{right} tensoring for the category of orthogonal sequence. This corresponds to the fact that we have a calculus for functors \emph{from} $\cat{J}$, but we do not know of any version of calculus for functors \emph{to} $\cat{J}$.

\subsubsection*{Can this approach be used to classify Taylor towers?} A basic problem in the calculus of functors is to extend Goodwillie's description of homogeneous functors to a classification of Taylor towers. One way we might answer this is to describe structure on the collection of derivatives of a functor that is sufficient for reconstructing the Taylor tower, and hence, in the analytic case, the functor itself. It is obvious that symmetric group actions described by Goodwillie are not sufficient, for these only determine the layers in the tower and not the (possibly nontrivial) extensions between those layers.

In this paper we show that the derivatives of a functor have more structure, namely that of a (bi)module over a certain operad. However, this structure is still not complete, and is not sufficient for recovering the Taylor tower from the derivatives. For example, for functors from $\spectra$ to $\spectra$, our results do not add any new structure to the derivatives - they still form just a symmetric sequence - but the Taylor tower of a functor from $\spectra$ to $\spectra$ cannot be recovered from its derivatives. McCarthy \cite{mccarthy:2001} has shown that obstructions to such a Taylor tower being the product of its layers live in the Tate cohomology of certain equivariant spectra. On the other hand, Kuhn shows in \cite{kuhn:2004} that, for functors of appropriately localized spectra (such as rational or $K(n)$-local), the relevant Tate cohomologies vanish and we do obtain split Taylor towers.

Now let $F$ be a functor of topological spaces. We show (see Remark \ref{rem:tate}) that the extent to which the bimodule $\der_*F$ fails to determine the Taylor tower of $F$ can also be measured using Tate cohomology in an analogous way to that of McCarthy. We might therefore expect that Kuhn's result can be generalized to functors of $K(n)$-local spaces to obtain a classification purely in terms of bimodule structures on derivatives.

We intend to come back to this in a future work.

\subsubsection*{How far can one develop Koszul duality for operads in $\spectra$?} In the next subsection of the introduction we review the notion of Koszul duality that we use in this paper, but for now let us just say that to an operad $P$, in $\spectra$, one can associate another operad, its `Koszul dual', which we denote by $\overline P$. This is a lift of Ginzburg and Kapranov's construction of the dg-dual \cite{ginzburg/kapranov:1994} for differential graded operads. Furthermore, if $M$ is a module (right, left or bi-) over $P$, one can construct a corresponding module $\overline M$ over $\overline P$. As part of the proof of Theorem \ref{thm:intro-chainrule}, we show that there is a natural equivalence
\[ M \circ_P N \homeq \overline M \circ_{\overline P} \overline N. \]
We believe that our work can also be used to construct an equivalence between certain homotopy categories of $P$-modules and of $\overline P$-modules.

One question, however, that we have failed to answer is whether or not the double Koszul dual of $P$ is equivalent, as an operad, to $P$. Ginzburg and Kapranov show this for operads of chain complexes, but for spectra we do not know if this is true. Without it, the duality picture is incomplete. We \emph{do} show that these operads are equivalent as symmetric sequences, but we do not have a map of operads relating them.

\subsection*{Outline of the paper}

This paper is long and rather technical and we feel that it would be useful to outline the central ideas of the proofs of Theorems \ref{thm:intro-bimodules} and \ref{thm:intro-chainrule} without cluttering those ideas with too much detail. Along the way, we collect some additional results that may be of independent interest. We hope the casual reader will be able to understand the philosophy behind the paper by looking at this section, without having to wade through the whole thing. As a result, we suppress some of the hypotheses needed to make all the statements in this section true. Some of these issues are addressed in the `Technical Remarks' section below.

Although the focus of this paper is functors to and/or from the category of based topological spaces, we start by considering functors from spectra to spectra. If $F,G: \spectra \to \spectra$ are such functors, then take
\[ \Nat(F,G) \]
to be the spectrum of natural transformations between $F$ and $G$. This can be defined as an enriched `end' based on the internal mapping objects in $\spectra$.

For any $F: \spectra \to \spectra$, we set
\[ \der^n(F) := \Nat(FX,X^{\smsh n}). \]
There is a natural $\Sigma_n$-action on $\der^n(F)$ coming from the permutation action on $X^{\smsh n}$ and so the collection $\der^*(F)$ becomes a symmetric sequence of spectra, contravariantly dependent on $F$. One of our central results (Section~\ref{sec:nat})  is that if $F$ is a cofibrant functor (in the usual model structure on the category of functors) then $\der^nF$ is naturally equivalent to the Spanier-Whitehead dual of $\der_nF$.

Thus, if $F$ is a cofibrant functor whose derivatives are homotopy-finite spectra, $\der^*F$ determines $\der_*F$. If $F$ does not have homotopy-finite derivatives, we can refine the definition of $\der^*F$ as follows. Approximate $F$ with an ind-functor $\{C_\alpha\}$, where each $C_\alpha$ does have finite derivatives. Then, define $\der^nF$ to be the \emph{pro-spectrum} $\{\der^nC_\alpha\}$. This pro-spectrum is then Spanier-Whitehead dual to $\der_nF$ in the sense of Christensen and Isaksen \cite{christensen/isaksen:2004}. We mostly suppress the pro-spectrum aspect in this outline. It suffices to say that, if properly defined, the symmetric sequence $\der^*F$ determines $\der_*F$.

The proof that $\der^*F$ is a model for the duals of $\der_*F$ makes use of the following auxiliary construction. Let $\Psi_nF : \spectra \longrightarrow \spectra$ be the functor defined by the formula
\[\Psi_nF(X)=\Map(\der^nF, X^{\smsh n})^{h\Sigma_n}\]
where $\Map(-,-)$ stands for the internal mapping object in $\spectra$. There is an evaluation map
\[\psi: F \longrightarrow \Psi_nF\]
and we prove that this map is a $D_n$-equivalence, that is, it becomes an equivalence after applying $D_n$. It follows that $\der_n(F)\simeq \der_n(\Psi_nF)$. We then prove that $\der_n(\Psi_nF)$ is the Spanier-Whitehead dual of $\der^nF$ to complete the proof.

One of the fundamental observations of this paper is that the construction $\der^*$ naturally relates the composition product of symmetric sequences with composition of functors. In Section~\ref{sec:comp-maps} we prove the existence of maps
\[ \tag{*} \der^*F \circ \der^*G \longrightarrow \der^*(FG) \]
and
\[\mathsf{1} \longrightarrow \der^*I_{\spectra}\]
that are associative and unital in the evident sense. In Section~\ref{sec:chainrule} we show that the maps (*) are weak equivalences of symmetric sequences. This proves Theorem \ref{thm:intro-chainrule} in the case that $F$ and $G$ are functors from $\spectra$ to $\spectra$.

With new models of derivatives, and the chain rule, for functors of spectra well understood, we turn to the method by which we transfer these results to functors of topological spaces. The key is that the maps (*) can be used to construct operad and module structures on the duals of the derivatives of various functors.

Recall that a functor $F: \spectra \longrightarrow \spectra$ is a called a \emph{comonad} if there exist natural transformations $F \longrightarrow FF$ and $F \longrightarrow I_{\spectra}$ that are coassociative and counital in an obvious way. A functor $G: \spectra \longrightarrow \cat{D}$ is said to be a \emph{right comodule} over $F$ if there is a natural transformation $G \longrightarrow GF$ that is coassociative and counital. We naturally also have left comodules and bi-comodules over a comonad.

Now suppose $F$ is a comonad in $\spectra$ and $G$ a comodule over $F$ (either right, left or bi-). Then the maps (*) endow $\der^*(F)$ with an operad structure, and $\der^*(G)$ with the structure of a module over $\der^*(F)$. Unfortunately, the shortcoming of this argument is that $\der^*(F)$ only has the correct homotopy type (i.e. that of the Spanier-Whitehead duals of the derivatives of $F$) when $F$ is a \emph{cofibrant} functor in the projective model structure on the category of functors. This means that we only really obtain structure on the derivatives of a comonad $F$ or comodule $G$ if these are also cofibrant functors. In order to transfer this structure to all comonads and comodules, we would need a cofibrant replacement that preserves the comonad and comodule structures. We do not know if such a replacement exists.

We use the operad structure described in the previous paragraph in the case that $F$ is the comonad $\Sigma^\infty \Omega^\infty$ which \emph{is} cofibrant in our model structure, so the concerns of the previous paragraph do not apply. A routine calculation with the Yoneda Lemma shows that
\[ \der^n(\Sigma^\infty \Omega^\infty) \homeq S, \]
where $S$ is the sphere spectrum, and that the resulting operad $\der^*(\Sigma^\infty \Omega^\infty)$ is equivalent to the commutative operad in the category of spectra. This is essentially Spanier-Whitehead dual to the statement that the derivatives $\der_*(\Sigma^\infty \Omega^\infty)$ form the commutative \emph{cooperad}.

Observe that for any functor $F: \based \to \spectra$, the composite $F \Omega^\infty$ is a right comodule over $\Sigma^\infty \Omega^\infty$, again by way of the $(\Sigma^\infty,\Omega^\infty)$-adjunction. We therefore deduce, using the composition maps (*), that $\der^*(F \Omega^\infty)$ is a right module over the operad $\der^*(\Sigma^\infty \Omega^\infty)$. Similarly, for any $G: \spectra \to \based$, $\Sigma^\infty G$ is a left comodule over $\Sigma^\infty \Omega^\infty$ and $\der^*(\Sigma^\infty G)$ is a left module over $\der^*(\Sigma^\infty \Omega^\infty)$.

The reason that the cooperad $\Sigma^\infty \Omega^\infty$ plays an important role in extending our results from functors on $\spectra$ to functors on $\based$ is that there is a close connection between $\based$ and the category of coalgebras over $\Sigma^\infty \Omega^\infty$. For example, one can show that the categories of $1$-connected spaces and of $1$-connected $\Sigma^\infty\Omega^\infty$-coalgebras are homotopy equivalent. More pertinent for us is the fact that as far as Taylor towers are concerned, functors to or from the category $\based$ are essentially the same as functors that are left or right comodules over $\Sigma^\infty\Omega^\infty$ respectively.

A more precise statement in this vein is given by Theorem~\ref{thm:intro-cosimplicial-resolution} below. This theorem describes the Taylor tower of any composite $FG$ where the `middle' category (i.e. the source of $F$ or target of $G$) is equal to $\based$. This result is fundamental to our paper. It describes how to build $P_n(FG)$ out of the Taylor towers of composite functors for which the middle category is $\spectra$. This allows us to transfer the chain rule for spectra to the unstable setting.

For any $F: \based \to \cat{D}$ and $G: \cat{C} \to \based$, the corresponding functors $F \Omega^\infty$ and $\Sigma^\infty G$ form right and left comodules, respectively, over the comonad $\Sigma^\infty \Omega^\infty$. There is then a natural map of the form
\[ \epsilon: FG \to \Tot(F \Omega^\infty (\Sigma^\infty \Omega^\infty)^{\bullet} \Sigma^\infty G). \]
The right-hand side here is the totalization of a cosimplicial object whose $k$-simplices are given by the composite
\[ F \Omega^\infty (\Sigma^\infty \Omega^\infty)^k \Sigma^\infty G \]
and whose coface and codegeneracy maps come from the comonad and comodule structures on $\Sigma^\infty \Omega^\infty$, $F \Omega^\infty$ and $\Sigma^\infty G$. The map $\epsilon$ is a natural coaugmentation of this cosimplicial object. We then prove the following result in Section~\ref{sec:cobar}.

\begin{theorem}\label{thm:intro-cosimplicial-resolution}
The coaugmentation map $\epsilon$ induces a weak equivalence
\[ P_n(FG) \weq \Tot(P_n(F\Omega^\infty(\Sigma^\infty\Omega^\infty)^{\bullet}\Sigma^\infty G)). \]
It follows that we also have a weak equivalence on the level of derivatives
\[ \der_n(FG) \homeq \Tot(\der_n(F\Omega^\infty(\Sigma^\infty\Omega^\infty)^{\bullet}\Sigma^\infty G)). \]
\end{theorem}

It is worth noting that the map of Theorem \ref{thm:intro-cosimplicial-resolution} factors as
\[ P_n(FG) \longrightarrow P_n(\Tot(F\Omega^\infty(\Sigma^\infty\Omega^\infty)^{\bullet}\Sigma^\infty G))\longrightarrow \Tot(P_n(F\Omega^\infty(\Sigma^\infty\Omega^\infty)^{\bullet}\Sigma^\infty G)). \]
These intermediate maps are not equivalences in general. We do know them to be equivalences if $F$ and $G$ are analytic functors, and in this case the theorem can be proved using connectivity estimates and the classical fact that the natural transformation
\[ I_{\based} \longrightarrow \Tot(\Omega^\infty (\Sigma^\infty\Omega^\infty)^{\bullet} \Sigma^\infty) \]
is an equivalence on simply connected spaces. (Recall that the right-hand side here is a model for the Bousfield-Kan $\mathbb{Z}$-completion functor of \cite{bousfield/kan:1972}.) We do not use these arguments. Instead, we give a `formal' proof of Theorem~\ref{thm:intro-cosimplicial-resolution} that does not assume analyticity or rely on connectivity estimates. This proof is by induction on the Taylor tower of $F$ using the fact that when $F$ is homogeneous, it factors as $F'\Sigma^\infty$ for some $F'$.

The ingredients we have described so far imply the following Proposition.

\begin{proposition} \label{prop:intro-indirect-chain-rule}
Consider functors $F: \based \to \cat{D}$ and $G: \cat{C} \to \based$ where $\cat{C}$ and $\cat{D}$ are themselves either $\based$ or $\spectra$. Then there is an equivalence of symmetric sequences of the form
\[ \der^*(FG) \homeq \der^*(F\Omega^\infty) \circ_{\der^*(\Sigma^\infty \Omega^\infty)} \der^*(\Sigma^\infty G). \]
where the right-hand side makes use the operad and modules $\der^*(\Sigma^\infty \Omega^\infty)$, $\der^*(F \Omega^\infty)$ and $\der^*(\Sigma^\infty G)$ that come from the comonad and comodule structures on these functors.
\end{proposition}

This is proved by applying the chain rule for functors of spectra to the equivalence of Theorem \ref{thm:intro-cosimplicial-resolution} and noting that the Spanier-Whitehead dual of the cosimplicial cobar construction is equivalent to a bar construction model for the derived composition product that appears on the right-hand side of Proposition \ref{prop:intro-indirect-chain-rule}. Strictly speaking, we have only described the proof of the proposition when $\cat{C}$ and $\cat{D}$ are both equal to $\spectra$. There are some additional problems that arise in the spaces case. We address these further in the section on technical remarks below.

We can think of Proposition \ref{prop:intro-indirect-chain-rule} as a kind of indirect chain rule, in that it describes the derivatives of $FG$ in terms of the derivatives of $F\Omega^\infty$ and $\Sigma^\infty G$. But note that taking either $F$ or $G$ to be the identity gives us a relationship between $\der^*(F)$ and $\der^*(F\Omega^\infty)$, on the one hand, and between $\der^*(G)$ and $\der^*(\Sigma^\infty G)$ on the other. In order to understand this relationship and write our chain rule in the form of Theorem \ref{thm:intro-chainrule}, we turn to Koszul duality.

Let $P$ be an operad in $\spectra$ with $P(1) = S$, the sphere spectrum, and consider the symmetric sequence $BP = \mathsf{1} \circ_{P} \mathsf{1}$ where the unit symmetric sequence $\mathsf{1}$ is given the obvious trivial left and right $P$-module structures. (Note that the composition product here is made in the \emph{derived} sense.) The main result of \cite{ching:2005a} was that $BP$ has a natural cooperad structure. The termwise Spanier-Whitehead dual of $BP$ is then an operad which we refer to as the \emph{Koszul dual} of $P$, and denote $\overline{P}$. (Strictly speaking, this generalizes Ginzburg and Kapranov's notion of the \emph{dg-dual} of an operad, but we find the term `Koszul dual' more appealing.) It was also shown in \cite{ching:2005a} that if $M$ is a right $P$-module then the symmetric sequence $M \circ_{P} \mathsf{1}$ is a right $BP$-comodule. The dual of this is then a right $\overline{P}$-module which we denote $\overline{M}$ and refer to as the \emph{Koszul dual} of the module $M$. There are analogous constructions for left $P$-modules and $P$-bimodules.

We now consider the implications of Proposition \ref{prop:intro-indirect-chain-rule} with these facts in mind. Firstly, if we take $F = G = I_{\based}$, then, using equivalences $\der^*(\Omega^\infty) \homeq \mathsf{1}$ and $\der^*(\Sigma^\infty) \homeq \mathsf{1}$, we get
\[ \der^*(I_{\based}) \homeq \mathsf{1} \circ_{\der^*(\Sigma^\infty \Omega^\infty)} \mathsf{1}. \]
Taking the Spanier-Whitehead dual of this, we see that $\der_*(I_{\based})$ is equivalent to the Koszul dual of the operad $\der^*(\Sigma^\infty \Omega^\infty)$, or equivalently, the Koszul dual of the commutative operad. This result was observed empirically in \cite{ching:2005a} but now we see a more solid reason for why this is true.

Next take just $F = I_{\based}$ in Proposition \ref{prop:intro-indirect-chain-rule} and $G$ to be any functor from either spaces or spectra to $\based$. Then we have
\[ \der^*(G) \homeq \mathsf{1} \circ_{\der^*(\Sigma^\infty \Omega^\infty)} \der^*(\Sigma^\infty G). \]
This implies that $\der_*(G)$ is equivalent to the Koszul dual of the module $\der^*(\Sigma^\infty G)$. Incidentally, this gives a practical method for finding the derivatives of such a functor $G$ because the derivatives of $\Sigma^\infty G$ are often easier to compute. It also tells us that $\der_*(G)$ has a left $\der_*(I_{\based})$-module structure. Similarly, if $F$ is a functor \emph{from} $\based$ to either spaces or spectra, then $\der_*(F)$ is equivalent to the Koszul dual of the right module $\der^*(F \Omega^\infty)$, and that $\der_*(F)$ is a left $\der_*(I_{\based})$-module. This completes the proof of Theorem \ref{thm:intro-bimodules}.

Lastly, we prove in Section~\ref{sec:koszul} the following result about Koszul dual operads and modules. If $P$ is an operad with right and left modules $R$ and $L$ respectively, and $\overline{P}$, $\overline{R}$ and $\overline{L}$ denote their Koszul duals, then we have an equivalence
\[ R \circ_{P} L \homeq \overline{R} \circ_{\overline{P}} \overline{L}. \]
Applying this to the statement of Proposition \ref{prop:intro-indirect-chain-rule}, we obtain exactly Theorem \ref{thm:intro-chainrule}.

The structure of our paper is as follows. The first part sets up the categories of functors that we are using and establishes their properties. Section 1 describes our categories of spaces and spectra and recalls basic results about realization of simplicial objects, and homotopy limits and colimits. Section 2 reviews Goodwillie's description of the Taylor tower and his definition of the derivatives of a functor. Section 3 proves an important result that says the Taylor tower of $FG$ is determined by the Taylor towers of $F$ and $G$. Sections 4 and 5 deal with the categories of functors, their structures and goes into some detail on the theory of cell functors and their subcomplexes. Section 6 recalls Christensen and Isaksen's results on pro-spectra and Spanier-Whitehead duality.

In the second part, we turn to operads and modules. Section 7 reviews the bar constructions for operads and modules that form a central part of this paper. Section 8 concerns the homotopy-invariance properties of the bar constructions. This motivates Section 9 which describes the cofibrant replacements for operads and modules needed to make the bar construction invariant. We get these cofibrant replacements by constructing model structures on the categories of operads and modules, the details of which are left to the Appendix. Section 10 is about $\Ext$-objects for modules which play a key role in the later proofs. Section 11 deals with `pro-' versions of symmetric sequences and modules and how Spanier-Whitehead duality works in this context.

The real substance of the paper starts in the third part. This deals with functors from spectra to spectra, their derivatives and the operad and module structures those derivatives possess. In Section 12, we construct our models for the derivatives of a functor from spectra to spectra. Section 13 describes the composition maps that relate these models, and Section 14 proves the chain rule for functors of spectra. In Section 15, we show that the (duals of the) derivatives of a comonad or comodule possess corresponding operad and module structures, and calculate what that structure is for the comonad $\Sigma^\infty \Omega^\infty$.

Finally, in the fourth section, we turn to functors of spaces. Section 16 gives the proof of Theorem \ref{thm:intro-cosimplicial-resolution}. Then in Section 17, we concentrate on functors from spaces to spectra and deduce Theorem \ref{thm:intro-bimodules} in that case. Sections 18 and 19 then apply the same ideas to functors from spectra to spaces, and from spaces to spaces. In these sections, we also discuss some basic examples of functors for which we can calculate the relevant module structures. Section 20 proves the key result on Koszul duality that allows us to deduce the final form of Theorem \ref{thm:intro-chainrule}, which we do in Section 21.

\subsection*{Technical remarks}

Here we make some comments about various technical issues we faced in implementing the ideas of the previous section, and how we solved them. In particular, we explain our choices of models for the homotopy theories of spaces and spectra.

\subsubsection*{Getting a category of functors}

The basic objects of study in this paper are functors $F: \cat{C} \to \cat{D}$ where $\cat{C}$ and $\cat{D}$ are categories of either based spaces or spectra. To put the homotopy theory of such functors on a solid footing, we would like to consider a model structure on the category $[\cat{C},\cat{D}]$ of all such functors. However, there are set-theoretic problems with defining $[\cat{C},\cat{D}]$ since $\cat{C}$ and $\cat{D}$ are not themselves small categories, so there is not in general a set of natural transformations between two such functors. We avoid this by considering only functors defined on the full subcategory $\cat{C}^{\mathsf{fin}}$ of finite cell complexes (or finite cell spectra) in $\cat{C}$. Since $\cat{C}^{\mathsf{fin}}$ is skeletally small, we obtain a well-defined functor category $[\cat{C}^{\mathsf{fin}},\cat{D}]$ on which we can put a projective model structure. From the point of view of finding Goodwillie derivatives, this restriction is not a problem because the derivatives of a functor $F$ are determined by its restriction to finite objects.

\subsubsection*{Use of EKMM S-modules}

Throughout this paper, the category $\spectra$ is taken to be the category of $S$-modules of EKMM with the standard model structure of \cite[VII]{elmendorf/kriz/mandell/may:1997}. The main reason for this is that every object in this model structure is fibrant. This has numerous technical advantages for us. In particular, it means that every functor with values in $\spectra$ is fibrant in the corresponding projective model structure. This is important in that it ensures that the natural transformation objects $\Nat(FX,X^{\smsh n})$ are homotopy-invariant (as long as $F$ is \emph{cofibrant}) without having to arrange separately for $X^{\smsh n}$ to be fibrant. The typical disadvantage of using EKMM $S$-modules for spectra is that the sphere spectrum $S$ is not cofibrant. This turns out not to be a significant problem for us. We make heavy use of a cofibrant replacement for $S$ which we denote $S_c$. It is conceivable that some of our work could be reproduced using, for example the category of orthogonal spectra \cite{mandell/may/schwede/shipley:2001} with the positive stable model structure, since this is known to have a symmetric monoidal fibrant replacement functor \cite{kro:2007}.

\subsubsection*{The functor $\Sigma^\infty \Omega^\infty$ and use of simplicial sets}

A key part of our argument relies on the functor $\Sigma^\infty \Omega^\infty$ from spectra to spectra. In particular, in order to construct the operad and module structures on which this paper is based, we need to use a model for this functor that has all of the following properties:
\begin{itemize}
  \item $\Sigma^\infty \Omega^\infty$ is cofibrant (in the projective model structure on $[\finspec,\spectra]$);
  \item $\Sigma^\infty \Omega^\infty$ is a homotopy functor (i.e. preserves \emph{all} weak equivalences);
  \item $\Sigma^\infty \Omega^\infty$ has a (strict) comonad structure.
\end{itemize}
The standard functors $\Sigma^\infty$ and $\Omega^\infty$ between the categories $\spectra$ and $\based$ (as used in \cite{elmendorf/kriz/mandell/may:1997}) do not have the necessary properties, in particular, they do not preserve all weak equivalences. To obtain a functor of the sort we need, we adopt the following definitions. Firstly, we use the category $\sset$ of pointed simplicial sets as our model for based topological spaces. Thus, results stated in this introduction for $\based$ are actually proved for $\sset$. The geometric realization and singular set functors preserve Taylor towers in an appropriate way which allows us to transfer these results back to functors of based topological spaces. Next, we define
\[ \Sigma^\infty(X) := S_c \smsh |X| \]
where $S_c$ is a cofibrant replacement for the sphere spectrum and $|X|$ denotes the geometric realization of the pointed simplicial set $X$. The functor $\Sigma^\infty: \sset \to \spectra$ now preserves all weak equivalences because every simplicial set is cofibrant. We also have an adjunction between $\Sigma^\infty$ and the functor $\Omega^\infty: \spectra \to \sset$ defined by
\[ \Omega^\infty(E) := \Sing_* \Map(S_c,X) \]
where $\Map(S_c,X)$ denotes the topological enrichment of $\spectra$, and $\Sing_*$ the pointed singular simplicial set functor. The functor $\Omega^\infty$ also preserves all weak equivalences and so $\Sigma^\infty \Omega^\infty$ is a homotopy functor as required. Moreover, the strict adjunction between $\Sigma^\infty$ and $\Omega^\infty$ ensures that $\Sigma^\infty \Omega^\infty$ has a strict comonad structure. Finally, we note that $\Sigma^\infty \Omega^\infty$ \emph{is} cofibrant in the projective model structure on the category of (simplicial) functors $\finspec \to \spectra$. In fact, it is a finite cell object in this category.

Recall that Lewis showed in \cite{lewis:1991} that no model for the stable homotopy category could have functors $\Sigma^\infty$ and $\Omega^\infty$ satisfying \emph{all} the properties one might want from such a pair. In our case, we do \emph{not} have an isomorphism
\[ \Sigma^\infty(X \smsh Y) \isom (\Sigma^\infty X) \smsh (\Sigma^\infty Y). \]
For example, this means that, for a simplicial set $X$, the diagonal on $X$ does not give $\Sigma^\infty X$ a strictly commutative coalgebra structure. We avoid this problem by using a different model for the commutative operad, namely the `coendomorphism operad' on the cofibrant sphere spectrum $S_c$. The spectrum $\Sigma^\infty X$ \emph{is} a coalgebra over this operad.

\subsubsection*{Pro-spectra and duality}

Spanier-Whitehead duality plays an important role in this paper. In order to make statements that hold for functors with derivatives that are possibly not finite spectra, we make heavy use of results of Christensen and Isaksen \cite{christensen/isaksen:2004} on the duality between pro-spectra and spectra. As we noted above, for a functor $F: \spectra \to \spectra$, the correct definition of $\der^n(F)$ is as a pro-spectrum formed by approximating $F$ with a filtered homotopy colimit of finite cell functors (that is, finite cell complexes in the cofibrantly-generated model structure on $[\finspec,\spectra]$). This makes the sequence of objects $\der^*(F)$ into a \emph{pro-symmetric sequence} rather than just a symmetric sequence, which requires us to develop a theory of modules over operads, and bar constructions in this context.

A fundamental part of this paper is the construction of composition maps of the form
\[ \der^*(F) \circ \der^*(G) \to \der^*(FG). \]
These form the basis of the operad and module structures we use to prove our main results. These composition maps must be done at the level of pro-spectra which requires a detailed understanding of the theory of cell functors (that is, functors that are cell complexes with respect to the cofibrantly-generated model structure on the category of functors). In particular, we point out the importance of Lemmas \ref{lem:key} and \ref{lem:factorization}. These in turn depend on basic properties of cell complexes in the category of spectra.

\subsubsection*{Dealing with the non-infinite-loop-space maps in the cobar construction}

The proof of Proposition \ref{prop:intro-indirect-chain-rule} described in the previous part of this introduction is, in reality, somewhat more delicate than indicated above. We can take the case $F = I_{\based}$ to illustrate why this is. The problem can be described in terms of the `end' coface maps in the cosimplicial cobar construction
\[ \Tot(\Omega^\infty (\Sigma^\infty \Omega^\infty)^{\bullet} \Sigma^\infty G). \]
At the lowest level, we have the map
\[ d^0: \Omega^\infty \Sigma^\infty G \to \Omega^\infty \Sigma^\infty \Omega^\infty \Sigma^\infty G \]
given by using the unit of the $\Sigma^\infty,\Omega^\infty$ adjunction to insert the first copy of $\Omega^\infty \Sigma^\infty$ on the right-hand side. The problem is that while this is a map between infinite-loop-spaces, it is not an infinite-loop-space map. It follows that we cannot easily analyze the map induced by $d^0$ on derivatives using our methods based solely on functors of spectra.

Because of this, we actually prove Theorem \ref{thm:intro-bimodules} first by a different method and then go back and deduce Proposition \ref{prop:intro-indirect-chain-rule} from this. This different method can be interpreted as finding an appropriate model for the cosimplicial cobar construction that builds in the non-infinite-loop-space maps in a way we can deal with.

The outline of this new proof of Theorem \ref{thm:intro-bimodules} is closely related to the way we showed that $\der^n(F)$ is the dual of the \ord{n} derivative of $F$ for $F: \spectra \to \spectra$. For example, given $F: \based \to \spectra$, we define, for each $n$, a functor
\[ \Phi_nF: \based \to \spectra \]
by
\[ \Phi_nF(X) := \Ext \left( \der^*(F\Omega^\infty), (\Sigma^\infty X)^{\leq n} \right). \]
This is an $\Ext$-object in the category of right $\der^*(\Sigma^\infty \Omega^\infty)$-modules. (We define the right $\der^*(\Sigma^\infty \Omega^\infty)$-module $(\Sigma^\infty X)^{\leq n}$ in \ref{def:sigmainfty-module}.) We also define a natural transformation
\[ \phi: F \to \Phi_nF \]
and show that $\phi$ is a $D_n$-equivalence, and hence induces an equivalence on \ord{n} derivatives. We then show that the \ord{n} derivative of $\Phi_nF$ is naturally equivalent to the \ord{n} term in the Koszul dual of the module $\der^*(F \Omega^\infty)$. This has a right $\der_*(I_{\based})$-module structure which proves Theorem \ref{thm:intro-bimodules} in this case. Analogous constructions do the same job for functors \emph{to} $\based$.

The functors $\Phi_nF$ form a kind of `fake' Taylor tower for the functor $F$. They are not, in general, equivalent to the Goodwillie's $P_nF$, but they do form a tower which in some sense captures the best approximation to the Taylor tower based on the right $\der_*(I_{\based})$-module structure on the derivatives of $F$. See Remark \ref{rem:tate} for more on this.

\subsection*{Acknowledgements}

The form of the chain rule given in Theorem \ref{thm:intro-chainrule} was originally suggested to the second author by Haynes Miller. The proof of Proposition \ref{prop:calculus} is partly due to Tom Goodwillie. We would like to thank these people and also Andrew Blumberg, Bill Dwyer, Jack Morava and Stefan Schwede for many useful conversations. Much of the work for this paper was done while the second author was in a postdoctoral position at Johns Hopkins University.

\newpage

\tableofcontents

\part{Basics}

\section{Categories of spaces and spectra}

This paper is about the homotopy theory of functors between categories of topological spaces and spectra. For various technical reasons, we use simplicial sets instead of topological spaces, and for spectra, we use EKMM's category of S-modules \cite{elmendorf/kriz/mandell/may:1997}. In this section, we recall various aspects of these categories.

\begin{definition}[Simplicial sets] \label{def:sset}
Let $\sset$ be the category of pointed simplicial sets. We denote the simplicial version of the standard $n$-simplex by $\Delta[n]$, and we write $\Delta[n]_+$ for this with a disjoint basepoint added. Then $\partial \Delta[n]_+$ denotes the simplicial $(n-1)$-sphere with a disjoint basepoint.  The category $\sset$ has the following properties:
\begin{itemize}
  \item $\sset$ is a closed symmetric monoidal category with respect to the smash product $\smsh$, with unit object $\Delta[0]_+$. We write $\sset(X,Y)$ for the internal mapping object;
  \item $\sset$ has a cofibrantly generated pointed proper simplicial monoidal model structure with generating cofibrations given by the inclusions
      \[ \partial \Delta[n]_+ \to \Delta[n]_+ \]
      for $n \geq 0$.
  \item Every object of $\sset$ is cofibrant in this model structure.
\end{itemize}
\end{definition}

\begin{definition}[Spectra] \label{def:spectra}
Let $\spectra$ be the category of $S$-modules of EKMM \cite{elmendorf/kriz/mandell/may:1997} (there this category is denoted $\cat{M}_S$). This is our model for the stable homotopy category. Unlike \cite{elmendorf/kriz/mandell/may:1997}, we refer to an object of $\spectra$ as a \emph{spectrum}, rather than an $S$-module. In \cite{elmendorf/kriz/mandell/may:1997}, the word `spectrum' is reserved for a more fundamental notion, with an $S$-module having additional structure. We choose this different terminology to avoid clashing with modules over operads which feature heavily in this paper. We stress that whenever we write `spectrum', we mean an object of $\spectra$, that is, an $S$-module in the sense of \cite{elmendorf/kriz/mandell/may:1997}.

The category $\spectra$ then has the following properties:
\begin{itemize}
  \item $\spectra$ is closed symmetric monoidal with respect to the smash product $\smsh_S$ which we write just as $\smsh$. The unit for the smash product is the \emph{sphere spectrum} $S$. For spectra $X,Y$, we write $\Map(X,Y)$ for the internal mapping object, that is spectrum of maps from $X$ to $Y$ (We avoid the usual notation $F_S(X,Y)$ because we often use $F$ to denote a arbitrary functor);
  \item $\spectra$ is enriched, tensored and cotensored over the category $\sset$. For spectra $X,Y$, we write $\spectra(X,Y)$ for the \emph{simplicial set} of maps from $X$ to $Y$, and for $K \in \sset, X \in \spectra$, we write $X \smsh K$ for the tensor object in $\spectra$. (In \cite{elmendorf/kriz/mandell/may:1997} only the enrichment over based topological spaces is considered. We obtain $\spectra(X,Y)$ by applying the singular simplicial set functor to the topological mapping object.)
  \item There is a cofibrantly generated pointed proper simplicial monoidal model structure on $\spectra$ in which the generating cofibrations are given by maps
      \[ S \smsh_{\cat{L}} \mathbb{L}\Sigma^\infty_q |\partial \Delta[n]_+| \to S \smsh_{\cat{L}} \mathbb{L}\Sigma^\infty_q |\Delta[n]_+| \]
      for $q,n \geq 0$. (See \cite[I.4-5]{elmendorf/kriz/mandell/may:1997} for the notation here.) This model structure is described in more detail in \cite[VII]{elmendorf/kriz/mandell/may:1997}.
  \item Every object in $\spectra$ is fibrant in this model structure.
\end{itemize}
\end{definition}

\begin{definition}[Cofibrant replacement for the sphere spectrum] \label{def:cofibrant-spectrum}
It is well known that the sphere spectrum $S$ is not cofibrant for the standard model structure on $\spectra$. We fix a cofibrant replacement given by
\[ S_c := S \smsh_{\cat{L}} \mathbb{L}S. \]
(See \cite[I.4-5]{elmendorf/kriz/mandell/may:1997} again for the notation.) An important fact for us is that the map
\[ * \to S_c \]
is precisely one of the generating cofibrations in $\spectra$ (that with $q,n = 0$ in Definition \ref{def:spectra}). In particular, this means that $S_c$ is a finite cell spectrum (see Definition \ref{def:cell-complex} below).
\end{definition}

\begin{definition}[Suspension spectrum and infinite loop-space functors] \label{def:suspec}
We warn the reader that the functors $\Sigma^\infty$ and $\Omega^\infty$ do not have the same meaning in this paper as in \cite{elmendorf/kriz/mandell/may:1997}. In particular, we define them to be the following functors between $\spectra$ and $\sset$:
\begin{itemize}
  \item Define $\Sigma^\infty: \sset \to \spectra$ by
  \[ \Sigma^\infty(K) := S_c \smsh K \]
  where here $\smsh$ denotes the tensoring of $\spectra$ over $\sset$, and $S_c$ is as in Definition \ref{def:cofibrant-spectrum}.
  \item Define $\Omega^\infty: \spectra \to \sset$ by
  \[ \Omega^\infty(E) := \spectra(S_c,E) \]
  where $\spectra(-,-)$ is the simplicial enrichment of the category $\spectra$.
\end{itemize}
The suspension spectrum and infinite loop-space functors are defined in \cite{elmendorf/kriz/mandell/may:1997} by $S \smsh K$ and $\spectra(S,E)$ respectively. We use $S_c$ because it gives these functors better homotopical properties for our purposes.
\end{definition}

\begin{lemma}
The functors $\Sigma^\infty$ and $\Omega^\infty$ of Definition \ref{def:suspec} have the following properties:
\begin{enumerate}
  \item $\Sigma^\infty$ is left adjoint to $\Omega^\infty$;
  \item $\Sigma^\infty$ preserves all weak equivalences and takes values in cofibrant spectra;
  \item $\Omega^\infty$ preserves all weak equivalences and takes values in fibrant pointed simplicial sets.
\end{enumerate}
\end{lemma}
\begin{proof}
The existence of an adjunction follows from the standard theory of enriched categories. Since $S_c$ is cofibrant, and every simplicial set is cofibrant, (2) follows from the pushout-product axiom in the simplicial model category $\spectra$. Part (3) also follows from the pushout-product axiom by way of the fact that every spectrum is fibrant.
\end{proof}

\begin{remark}
The functors that we are calling $\Sigma^\infty$ and $\Omega^\infty$ are naturally equivalent to the standard notions of the suspension spectrum of a space, and the infinite loop-space associated to a spectrum. For example, the composite $\Omega^\infty \Sigma^\infty$ is equivalent to the usual stable homotopy functor $Q$ (though defined in terms of simplicial sets).

Lewis showed in \cite{lewis:1991} that one cannot construct functors $\Sigma^\infty$ and $\Omega^\infty$ that have all the good properties one would hope for from such an adjunction. In our case, we are lacking an isomorphism between $\Sigma^\infty(X \smsh Y)$ and $(\Sigma^\infty X) \smsh (\Sigma^\infty Y)$. (This is essentially because $S_c \ncong S_c \smsh S_c$.)
\end{remark}

\begin{definition}[Cell complexes] \label{def:cell-complex}
A \emph{relative cell complex} in either $\sset$ or $\spectra$ is a map that can be expressed as the composite of a \emph{countable} sequence of pushouts along coproducts of the generating cofibrations. A \emph{cell complex} is an object $X$ such that $* \to X$ is a relative cell complex. We refer to Hirschhorn's book \cite[\S10]{hirschhorn:2003} for the general theory of cell complexes.

We refer to cell complexes in $\sset$ just as \emph{cell complexes}, and cell complexes in $\spectra$ as \emph{cell spectra}. This definition of cell spectra agrees with that given in \cite[III.2]{elmendorf/kriz/mandell/may:1997} (except that they call them `cell $S$-modules').

A cell complex or cell spectrum is \emph{finite} if it has a cell structure with finitely many cells. For each of the categories $\cat{C} = \sset, \spectra$, let $\cat{C}^{\mathsf{fin}}$ denote the full subcategory of $\cat{C}$ whose objects are the finite cell complexes. In each of these cases, $\cat{C}^{\mathsf{fin}}$ is a skeletally small category that inherits a simplicial enrichment from $\cat{C}$ with the same simplicial mapping objects $\cat{C}(X,Y)$.

We say that an object in $\cat{C}$ is \emph{homotopy-finite} if it is weakly equivalent (i.e. connected by a zigzag of weak equivalences) to an object in $\cat{C}^{\mathsf{fin}}$.
\end{definition}

\begin{remark} \label{rem:compactly-generated}
In a general cofibrantly generated model category, it is usual to allow cell complexes to be the composite of an arbitrarily long sequence of pushouts along coproducts of the generating cofibrations (i.e. a sequence indexed by any ordinal). This is necessary for the small object argument to produce the appropriate factorizations. In both the categories $\sset$ and $\spectra$, however, the domains of the generating cofibrations are $\omega$-small relative to these arbitrary cell complexes, where $\omega$ is the countable cardinal. It follows that only cell complexes given by countable sequences are necessary for the small object argument. We therefore restrict our notion of `cell complex' to those formed from such countable sequences. Another way to say this is that the model categories $\sset$ and $\spectra$ are \emph{compactly generated} in the sense of \cite[4.5]{may/sigurdsson:2006}.
\end{remark}

\begin{definition}[Presented cell complexes]
We emphasize that a \emph{cell complex} does not include any particular choice of cell structure, rather it is merely an object for which such a structure exists. A \emph{presented cell complex}, on the other hand, consists of the sequence of pushout squares that define a cell structure on a cell complex. In particular, the notion of a \emph{subcomplex} is well-defined only for presented cell complexes. We refer to \cite[10.6]{hirschhorn:2003} for an extended description of the theory of presented cell complexes and their subcomplexes. We describe this theory in more detail in \S\ref{sec:cell-functors} in the context of cell functors.
\end{definition}

Cell complexes in the categories $\sset$ and $\spectra$ satisfy some useful properties beyond those enjoyed in an arbitrary cofibrantly generated model category. The following proposition lists a couple of those properties that are important in this paper.

\begin{proposition}[Properties of cell complexes] \label{prop:cell-complex}
The following facts apply to cell complexes in either $\sset$ or $\spectra$:
\begin{enumerate}
  \item A relative cell complex is a monomorphism.
  \item Let $K$ be a finite cell complex and $X$ any cell complex. Then any map $K \to X$ factors via a finite subcomplex of $X$.
\end{enumerate}
\end{proposition}
\begin{proof}
These are familiar results in the cases of simplicial sets. A relative cell spectrum $f$ is a cofibration in the model structure on $\spectra$. By \cite[VII.4.14]{elmendorf/kriz/mandell/may:1997}, it is therefore a cofibration in the sense that it has the homotopy extension property. By \cite[I.8.1]{lewis/may/steinberger:1986}, $f$ is a \emph{spacewise closed inclusion}. Recall that a map $X \to Y$ of spectra (in the sense of EKMM \cite{elmendorf/kriz/mandell/may:1997}) ultimately consists of a map of based topological spaces $X(V) \to Y(V)$ for each finite-dimensional linear subspace $V \subset \mathbb{R}^\infty$. We say that $X \to Y$ is a \emph{spacewise closed inclusion} if each map $X(V) \to Y(V)$ is isomorphic to the inclusion of a closed subspace. In particular, note that a spacewise closed inclusion is a monomorphism which proves part (1). Part (2) is effectively \cite[III.2.3]{elmendorf/kriz/mandell/may:1997}.
\end{proof}

\begin{definition}[Simplicial and cosimplicial objects] \label{def:simplicial-objects}
Simplicial and cosimplicial objects in our categories $\sset$ and $\spectra$ play an important role in this paper. We write $\mathsf{\Delta}$ for the category whose objects are the totally ordered sets $\mathbf{n} := \{0,\dots,n\}$ for $n \geq 0$, and whose morphisms are the order-preserving functions. A \emph{simplicial object} in a category $\cat{C}$ is a functor $X_\bullet: \mathsf{\Delta}^{op} \to \cat{C}$, and a cosimplicial object in $\cat{C}$ is a functor $X^{\bullet}: \mathsf{\Delta} \to \cat{C}$. We refer to the object $X_k := X_{\bullet}(\mathbf{k})$ or $X^k := X^{\bullet}(\mathbf{k})$ as the \emph{object of $k$-simplices} in $X_{\bullet}$ or $X^{\bullet}$ respectively.

More explicitly, a simplicial object in $\cat{C}$ consists of a sequence of objects $X_k \in \cat{C}$ for $k \geq 0$ and:
\begin{itemize}
  \item \emph{face maps} $d_i: X_k \to X_{k-1}$ for $i = 0,\dots,k$;
  \item \emph{degeneracy maps} $s_j: X_k \to X_{k+1}$ for $j = 0,\dots,k$;
\end{itemize}
satisfying the `simplicial identities' (see \cite[I.1]{goerss/jardine:1999}). Dually, a cosimplicial object in $\cat{C}$ consists of objects $X^k \in \cat{C}$ for $k \geq 0$ and:
\begin{itemize}
  \item \emph{coface maps} $d^i: X^k \to X^{k+1}$ for $i = 0,\dots,k+1$;
  \item \emph{codegeneracy maps} $s^j: X^k \to X^{k-1}$ for $j = 0,\dots,k-1$;
\end{itemize}
satisfying corresponding `cosimplicial identities'.
\end{definition}

\begin{definition}[Geometric realization and totalization] \label{def:realization}
If $\cat{C}$ is either $\sset$ or $\spectra$, then a simplicial object $X_{\bullet}$ has a \emph{geometric realization}, written $|X_{\bullet}|$, and a cosimplicial object $X^{\bullet}$ has a \emph{totalization}, written $\Tot X^{\bullet}$. These can be defined as the following coend and end respectively:
\[ |X_{\bullet}| := \Delta[n]_+ \smsh_{\mathsf{\Delta}} X_n, \quad \Tot X^{\bullet} := \Map_{\mathsf{\Delta}}(\Delta[n]_+,X^n). \]
\end{definition}

\begin{definition}[Augmented simplicial objects] \label{def:augmented-simplicial}
Let $X_{\bullet}$ be a simplicial object in a category $\cat{C}$. An \emph{augmentation} of $X_{\bullet}$ is a map $\epsilon: X_0 \to Y$ in $\cat{C}$ such that $\epsilon d_0 = \epsilon d_1$. If $\cat{C}$ is either $\sset$ or $\spectra$, then such an augmentation determines a map
\[ \tilde{\epsilon}: |X_{\bullet}| \to Y. \]
Dually, if $X^{\bullet}$ is a cosimplicial object in $\cat{C}$, then an \emph{augmentation} of $X^{\bullet}$ is a map $\eta: Y \to X^0$ such that $d^0 \eta = d^1 \eta$. This determines a map
\[ \tilde{\eta}: Y \to \Tot X^{\bullet}. \]
\end{definition}

\begin{remark}
Another way to view an augmentation of the simplicial object $X_{\bullet}$ to $Y$ is as a map of simplicial objects from $X_{\bullet}$ to the constant simplicial object $Y_{\bullet}$ with value $Y$. If $f: X_{\bullet} \to Y_{\bullet}$ is a map of simplicial objects, then $f_0: X_0 \to Y_0 = Y$ is an augmentation of $X_{\bullet}$. Conversely, given an augmentation $\epsilon: X_0 \to Y$, we build such a map $f$ by taking $f_n : X_n \to Y_n = Y$ to be $\epsilon d_0 \dots d_0$.

Dually, an augmentation of the cosimplicial object $X^{\bullet}$ is equivalent to a map $Y^{\bullet} \to X^{\bullet}$ where $Y^{\bullet}$ is a constant cosimplicial object.
\end{remark}

An important idea for this paper is the notion of a simplicial or cosimplicial `contraction' as used, for example, by May \cite[\S9]{may:1972}. We recall this here.

\begin{lemma}[Simplicial contractions] \label{lem:simplicial-contraction}
Let $\cat{C}$ be either $\sset$ or $\spectra$. Let $X_{\bullet}$ be a simplicial object in $\cat{C}$ with augmentation $d_0: X_0 \to X_{-1}$. Suppose that for all $k \geq -1$ there exist maps
\[ s_{-1}: X_k \to X_{k+1} \]
that satisfy the \emph{extended simplicial identities}
\[ \begin{split} s_{-1}d_i &= d_{i+1}s_{-1} \text{ for $i \geq 0$} \\ s_{-1}s_j &= s_{j+1}s_{-1} \text{ for $j \geq -1$}. \end{split} \]
Then the induced map
\[ d_0: |X_{\bullet}| \to X_{-1} \]
is a homotopy equivalence in $\cat{C}$. The maps $s_{-1}$ are referred to as \emph{extra degeneracies} and provide a \emph{simplicial contraction} for the augmented simplicial object $X_{\bullet}$.
\end{lemma}
\begin{proof}
The map $s_{-1}: X_{-1} \to X_0$ determines
\[ \tilde{s}_{-1}: X_{-1} \to |X_{\bullet}| \]
such that $\tilde{d_0}\tilde{s}_{-1}$ is the identity on $X_{-1}$. The maps $s_{-1}: X_k \to X_{k+1}$ for $k \geq 0$ determine a homotopy between the identity map on $|X_{\bullet}|$ and $\tilde{s}_{-1}\tilde{d_0}$. Therefore $\tilde{d_0}$ is a homotopy equivalence.
\end{proof}

\begin{remark} \label{rem:simplicial-contraction}
Let $f:X_{\bullet} \to Y_{\bullet}$ be a map from $X_{\bullet}$ to the constant simplicial object $Y_{\bullet}$ with value $Y = X_{-1}$. A collection of extra degeneracies, as in Lemma \ref{lem:simplicial-contraction}, determines a map $i:Y_{\bullet} \to X_{\bullet}$ such that $fi$ is the identity on $Y_{\bullet}$, and a simplicial homotopy from $if$ to the identity on $X_{\bullet}$. In other words, it makes $Y_{\bullet}$ into a deformation retract of $X_{\bullet}$. Taking geometric realizations, we deduce that the $Y = |Y_{\bullet}|$ is a deformation retract of $|X_{\bullet}|$ which recovers Lemma \ref{lem:simplicial-contraction}.
\end{remark}

\begin{lemma} \label{lem:cosimplicial-contraction}
Let $\cat{C}$ be either $\sset$ or $\spectra$ and let $X^{\bullet}$ be a cosimplicial object in $\cat{C}$ with augmentation $d^0: X^{-1} \to X^0$. Suppose that for all $k \geq -1$, there exist maps
\[ s^{-1}: X^{k+1} \to X^k \]
that satisfy \emph{extended cosimplicial identities} dual to the extended simplicial identities of Lemma \ref{lem:simplicial-contraction}. Then the induced map
\[ d^0: X^{-1} \to \Tot X^{\bullet} \]
is a homotopy equivalence in $\cat{C}$. The maps $s^{-1}$ are referred to as \emph{extra codegeneracies} and provide a \emph{cosimplicial contraction} for the augmented cosimplicial object $X^{\bullet}$.
\end{lemma}
\begin{proof}
This is dual to the proof of Lemma \ref{lem:simplicial-contraction}.
\end{proof}

\begin{remark} \label{rem:cosimplicial-contraction}
Let $f: Y^{\bullet} \to X^{\bullet}$ be a map from the constant cosimplicial object $Y^{\bullet}$ with value $Y = X^{-1}$ to $X^{\bullet}$. A collection of extra codegeneracies, as in Lemma \ref{lem:cosimplicial-contraction}, determines a map $p: X^{\bullet} \to Y^{\bullet}$ such that $pf$ is the identity on $Y^{\bullet}$, and a cosimplicial homotopy between $fp$ and the identity on $X^{\bullet}$ making $Y^{\bullet}$ into a deformation retract of $X^{\bullet}$. Taking totalizations, we deduce that $Y = \Tot Y^{\bullet}$ is a deformation retract of $\Tot X^{\bullet}$, which recovers Lemma \ref{lem:cosimplicial-contraction}.
\end{remark}

\begin{remark}
There are similar results to Lemmas \ref{lem:simplicial-contraction} and \ref{lem:cosimplicial-contraction} for extra degeneracy and codegeneracy maps on `the other side', that is, $s_{k+1}: X_k \to X_{k+1}$ or $s^{k+1}: X^{k+1} \to X^k$ for $k \geq -1$ that satisfy corresponding extended simplicial or cosimplicial identities.
\end{remark}

We now briefly recall some of the homotopical properties of realization and totalization of simplicial and cosimplicial objects. Recall the existence of Reedy model structures on the categories of simplicial and cosimplicial objects in a model category (see \cite[Chapter 15]{hirschhorn:2003}). We use the following result several times.

\begin{prop} \label{prop:reedy}
\begin{enumerate}
  \item Let $f: X_{\bullet} \to Y_{\bullet}$ be a map between Reedy cofibrant simplicial objects in $\cat{C}$ such that each $f_n:X_n \to Y_n$ is a weak equivalence in $\cat{C}$. Then the induced map
  \[ |f|: |X_{\bullet}| \to |Y_{\bullet}| \]
  is a weak equivalence.
  \item Let $f: Y^{\bullet} \to X^{\bullet}$ be a map between Reedy fibrant cosimplicial objects in $\cat{C}$ such that each $f^n: Y^n \to X^n$ is a weak equivalence in $\cat{C}$. Then the induced map
      \[ \Tot f: \Tot Y^{\bullet} \to \Tot X^{\bullet} \]
      is a weak equivalence in $\cat{C}$.
\end{enumerate}
\end{prop}
\begin{proof}
This is \cite[18.6.6]{hirschhorn:2003}.
\end{proof}

\begin{definition}[Homotopy-invariant realization and totalization] \label{def:invariant-realization}
Let $\cat{C}$ be either $\sset$ or $\spectra$ and let $X_{\bullet}$ be a simplicial object in $\cat{C}$. The \emph{homotopy-invariant geometric realization} of $X_{\bullet}$ is given by taking the realization of a Reedy cofibrant replacement for $X_{\bullet}$. This is determined only up to weak equivalence. In practice, we can fix a Reedy cofibrant replacement functor and obtain a particular (functorial) choice of homotopy-invariant realization. We denote the homotopy-invariant realization of $X_{\bullet}$ by
\[ |\tilde{X}_{\bullet}|. \]

Dually, if $X^{\bullet}$ is a cosimplicial object in $\cat{C}$, then the \emph{homotopy-invariant totalization} of $X^{\bullet}$ is given by the totalization of a Reedy fibrant replacement for $X^{\bullet}$. We denote this by
\[ \widetilde{\Tot} X^{\bullet} = \Tot \tilde{X}^{\bullet}. \]

Proposition \ref{prop:reedy} implies that an objectwise weak equivalence between simplicial or cosimplicial objects induces a weak equivalence between their homotopy-invariant realizations or totalizations.
\end{definition}

\begin{lemma} \label{lem:contraction-homotopy}
Let $\cat{C}$ be either $\sset$ or $\spectra$ and let $X_{\bullet}$ be a simplicial object in $\cat{C}$. Let $X_0 \to Y$ be an augmentation of $X_{\bullet}$ that admits extra degeneracies as in Lemma \ref{lem:simplicial-contraction}. Then there is a weak equivalence
\[ |\tilde{X}_{\bullet}| \weq Y. \]
Let $X^{\bullet}$ be a cosimplicial object in $\cat{C}$ with augmentation $Y \to X^0$ that admits extra codegeneracies as in Lemma \ref{lem:cosimplicial-contraction}. Suppose that $Y$ is fibrant. Then there is a weak equivalence
\[ Y \to \widetilde{\Tot} X^{\bullet}. \]
\end{lemma}
\begin{proof}
Let $Y_{\bullet}$ denote the constant simplicial object with value $Y$. As in Remark \ref{rem:simplicial-contraction}, we have a deformation retract $Y_{\bullet} \to X_{\bullet}$. Applying a Reedy cofibrant replacement functor, we obtain a deformation retract
\[ \tilde{Y}_{\bullet} \to \tilde{X}_{\bullet} \]
which induces a deformation retract
\[ |\tilde{Y}_{\bullet}| \weq |\tilde{X}_{\bullet}| \]
whose homotopy inverse is a weak equivalence
\[ |\tilde{X}_{\bullet}| \to |\tilde{Y}_{\bullet}|. \]
Now a constant simplicial object with cofibrant value is Reedy cofibrant (because the latching maps are either the identity or map from the initial object). It follows that (whether or not $Y$ is cofibrant), there is a weak equivalence
\[ |\tilde{Y}_{\bullet}| \weq |Y_{\bullet}| = Y. \]
Combining this with the previous weak equivalence gives the first part of the Lemma. The second part is similar, using the fact that a constant cosimplicial object on fibrant object is Reedy fibrant.
\end{proof}

\begin{definition}[Homotopy limits and colimits] \label{def:holim}
We use homotopy-invariant versions of the homotopy limit and colimit of a diagram of pointed simplicial sets or spectra. The standard notions of homotopy limit and colimit (due to Bousfield and Kan \cite{bousfield/kan:1972}) preserve objectwise weak equivalences between only objectwise-fibrant, and objectwise-cofibrant diagrams respectively (see \cite[18.5.3]{hirschhorn:2003}). In our notation $\holim$ and $\hocolim$, we understand that appropriate fibrant or cofibrant replacements have been taken (if necessary) before applying the Bousfield-Kan construction.

Note that all objects in $\spectra$ are fibrant so that the standard homotopy limit is already the homotopically correct one. Similarly, all objects in $\sset$ are cofibrant and so the Bousfield-Kan homotopy colimit is already correct.

At many points in this paper, we take the homotopy limit or colimit of a diagram of \emph{functors} with values in either $\sset$ or $\spectra$. Unless noted otherwise, such homotopy limits and colimits are always formed objectwise.
\end{definition}

\begin{definition}[Functors] \label{def:functors}
We are interested in studying functors between the categories $\spectra$ and $\sset$. To make the technical constructions of this paper, we need some basic conditions on such functors. Let $F: \cat{C} \to \cat{D}$ be a functor where  $\cat{C}$ and $\cat{D}$ are each either $\spectra$ or $\sset$. Then we say
\begin{itemize}
  \item $F$ is \emph{pointed} if $F(*) = *$;
  \item $F$ is \emph{simplicial} if it induces maps of simplicial sets of the form
  \[ \cat{C}(X,Y) \to \cat{D}(FX,FY) \]
  where recall that $\cat{C}(X,Y)$ denotes the simplicial set of maps from $X$ to $Y$ in the category $\cat{C}$. Notice that a simplicial functor $F$ is pointed if and only if all these are maps of \emph{pointed} simplicial sets, that is, preserve the basepoint.
  \item $F$ is a \emph{homotopy functor} if it preserves weak equivalences, that is if $X \weq Y$ is a weak equivalence in $\cat{C}$, then $FX \weq FY$ is a weak equivalence in $\cat{D}$
  \item $F$ is \emph{finitary} if it preserves filtered homotopy colimits in the following sense: given a filtered diagram $X:\cat{I} \to \cat{C}$, the natural map
      \[ \hocolim_{i \in \cat{I}} F(X(i)) \to F \left( \hocolim_{i \in \cat{I}} X(i) \right), \]
      should be a weak equivalence.
\end{itemize}
Strictly speaking the map involved in the definition of when $F$ is finitary is a zigzag involving inverse weak equivalences. We therefore really mean that each of the forward maps involved in that zigzag should be also be a weak equivalence.
\end{definition}

\begin{remark} \label{rem:pointed}
The condition that a functor be \emph{pointed} is clearly somewhat limiting -- there are many interesting functors without the property that $F(*) = *$. We remark here that any \emph{reduced} functor $F: \cat{C} \to \cat{D}$ (i.e. with $F(*)$ weakly equivalent to $*$) is itself weak equivalent to a pointed functor. (By this we mean that there is a zigzag of natural transformations connecting them, each of which is an objectwise weak equivalence.)

If $\cat{D} = \sset$, then $F(*)$ is a retract of $F(X)$ for any $X \in \cat{C}$, because $*$ is a retract of $X$. Hence the map
\[ F(*) \to F(X) \]
is a monomorphism, and so a cofibration of simplicial sets. But $F(*)$ is a weakly contractible simplicial set since $F$ is reduced, and it follows that the map
\[ F(X) \to F(X)/F(*) \]
is a weak equivalence. We therefore obtain a model $\tilde{F}$ for the original functor $F$ by taking
\[ \tilde{F}(X) := F(X)/F(*) \]
and by construction this is pointed.

If $\cat{D} = \spectra$, there is a similar argument using the Quillen equivalence between $\spectra$ and the category $\mathsf{Sp}^{\Sigma}$ of symmetric spectra (based on simplicial sets) constructed by Schwede \cite{schwede:2001a}. In this case, we define the symmetric spectrum $\hat{F}(X)$ by
\[ \hat{F}(X)_n := \spectra((S^{-1}_c)^{\smsh n}, F(X)) \]
where $S^{-1}_c$ is a cofibrant replacement of the $-1$-sphere spectrum. The symmetric spectrum $\hat{F}(*)$ is now levelwise weakly equivalent to $*$ and so the map
\[ \hat{F}(X) \to \hat{F}(X)/\hat{F}(*) \]
is a level weak equivalence, and hence a stable weak equivalence. Returning to $\spectra$, by applying the other half of the Quillen adjunction to $\hat{F}(X)/\hat{F}(*)$, we obtain the required pointed model for the functor $F$.
\end{remark}

\section{Taylor tower and derivatives of functors of simplicial sets and spectra} \label{sec:Taylor}

This paper is about Goodwillie calculus applied to functors between the categories of simplicial sets and spectra, including all four combinations of source and target category. In \cite{goodwillie:2003}, Goodwillie describes the construction of the Taylor tower of a functor from topological spaces to spaces or spectra. Kuhn \cite{kuhn:2007} then shows that Goodwillie's work generalizes easily to functors between fairly arbitrary model categories including those we are interested in.

We do not need the details of the construction of the Taylor tower, so we recall only the key properties that it possesses. We do concentrate more on the \emph{derivatives} of the functors we are interested in because they are the focus of this paper. Note that we only consider Taylor towers expanded at the null object $*$, and only derivatives at $*$.

\begin{definition}[Cartesian and cocartesian cubes]
Recall that a cubical diagram of simplicial sets or spectra is \emph{cartesian} if the map from the initial vertex to the homotopy limit of the remaining diagram is an equivalence, and \emph{cocartesian} if the map to the terminal vertex from the homotopy colimit of the remaining diagram is an equivalence. Such a cubical diagram is \emph{strongly cocartesian} if every two-dimensional face is cocartesian.
\end{definition}

\begin{definition}[$n$-excisive functors] \label{def:excisive}
A homotopy functor $F: \cat{C} \to \cat{D}$ is \emph{$n$-excisive} if it takes strongly cocartesian $(n+1)$- dimensional cubes in $\cat{C}$ to cartesian cubes in $\cat{D}$. See \cite{goodwillie:1991} for more details on this definition.
\end{definition}

\begin{theorem}[Goodwillie \cite{goodwillie:2003}] \label{thm:Taylor}
Let $F: \cat{C} \to \cat{D}$ be a homotopy functor where $\cat{C}$ and $\cat{D}$ are either $\sset$ or $\spectra$. Then there exist homotopy functors
\[ P_nF: \cat{C} \to \cat{D} \]
for $n \geq 0$, and a diagram of natural transformations
\[ F \to \dots \to P_nF \to P_{n-1}F \to \dots \to P_0F \]
such that
\begin{itemize}
  \item $P_nF$ is $n$-excisive;
  \item the map $p_nF: F \to P_nF$ is initial, up to homotopy, among natural transformations from $F$ to an $n$-excisive functor;
  \item $P_nF$ is functorial in $F$, and an equivalence $F \to G$ induces an equivalence $P_nF \to P_nG$;
  \item $P_n$ preserves finite homotopy limits and filtered homotopy colimits; for spectrum-valued functors, $P_n$ preserves all homotopy colimits.
\end{itemize}
\end{theorem}

\begin{remark}
We need a couple of comments on the construction of $P_nF$:
\begin{itemize}
  \item if $F$ has any of the properties of being pointed, simplicial or finitary, then $P_nF$ has the same properties;
  \item the definition of $P_nF(X)$ depends on the value of $F$ on the joins of $X$ with finite sets. In particular, if $X$ is a finite cell complex, then $P_nF(X)$ depends only on the restriction of $F$ to finite cell complexes. Therefore, given a functor $F: \cat{C}^{\mathsf{fin}} \to \cat{D}$ (recall that $\cat{C}^{\mathsf{fin}}$ is the full subcategory of finite cell complexes), we can construct $P_nF: \cat{C}^{\mathsf{fin}} \to \cat{D}$.
\end{itemize}
\end{remark}

\begin{remark}
We give results in this paper for the calculus of functors to and/or from the category of pointed simplicial sets. Results for functors to and/or from based topological spaces can quickly be deduced via the Quillen equivalence between $\sset$ and $\based$. Explicitly, we have
\[ P_nF(|X|) \homeq P_n(F|-|)(X), \quad \text{for $F$ defined on $\based$} \]
and
\[ \Sing P_nF(X) \homeq P_n(\Sing F)(X), \quad \text{for $F$ taking values in $\based$}. \]
These equivalences follow from the construction of $P_n$ using the properties of the realization functor $|-|: \sset \to \based$ and the singular simplicial set functor $\Sing: \based \to \sset$.
\end{remark}

\begin{definition}[Layers of the Taylor tower] \label{def:layers}
The \emph{layers} in the Taylor tower of $F: \cat{C} \to \cat{D}$ are the functors $D_nF: \cat{C} \to \cat{D}$ given by
\[ D_nF:= \hofib(P_nF \to P_{n-1}F). \]
The functor $D_nF$ is \emph{$n$-homogeneous}, that is both $n$-excisive and \emph{$n$-reduced} (i.e. $P_{n-1}(D_nF) \homeq *$.) Goodwillie's classification of homogeneous functors then leads to the following proposition.
\end{definition}

\begin{proposition} \label{prop:D_nF}
In each case below, $F$ is a finitary homotopy functor between the given categories:
\begin{enumerate}
  \item for $F: \spectra \to \sset$, there is an $n$-homogeneous functor $\mathbb{D}_nF: \spectra \to \spectra$ such that
  \[ D_nF(X) \homeq \Omega^\infty (\mathbb{D}_nF)(X); \]
  \item for $F: \sset \to \spectra$, there is an $n$-homogeneous functor $\mathbb{D}_nF: \spectra \to \spectra$ such that
  \[ D_nF(X) \homeq (\mathbb{D}_nF)(\Sigma^\infty X); \]
  \item for $F: \sset \to \sset$, there is an $n$-homogeneous functor $\mathbb{D}_nF: \spectra \to \spectra$ such that
  \[ D_nF(X) \homeq \Omega^\infty (\mathbb{D}_nF)(\Sigma^\infty X). \]
\end{enumerate}
In each case, the construction of $\mathbb{D}_nF$ can be made functorial in $F$.
\end{proposition}
\begin{proof}
Goodwillie describes in \cite[2.1]{goodwillie:2003} a (functorial) infinite delooping for a homogeneous functor that takes values in based spaces. Using an appropriate Quillen equivalence to make sure this delooping lives in the category $\spectra$, we get in cases (1) and (3) functors
\[ B^{\infty}D_nF: \cat{C} \to \spectra \]
such that
\[ D_nF \homeq \Omega^\infty B^{\infty} D_nF. \]
This does not require the finitary condition and in case (1), we set $\mathbb{D}_nF := B^\infty D_nF$ and are done.

For (2), we use the classification of finitary homogeneous functors in terms of coefficient spectra given in \cite[\S5]{goodwillie:2003}. We have a natural equivalence
\[ D_nF(X) \homeq (E \smsh X^{\smsh n})_{h\Sigma_n} \homeq (E \smsh (\Sigma^\infty X)^{\smsh n})_{h\Sigma_n} \]
where $E$ is a spectrum with $\Sigma_n$-action. This gives (2) with $\mathbb{D}_nF(Y) := (E \smsh Y^{\smsh n})_{h\Sigma_n}$. This can be made functorial in $F$ since $E$ can be given by the formula
\[ E := \creff_n(D_nF)(S^0,\dots,S^0) \]
where $\creff_n$ is the \ord{n} cross-effect construction defined below (\ref{def:cross-effects}).

Finally, (3) is given by combining the constructions of (1) and (2).
\end{proof}

The spectrum $E$ that appears in the classification results in the proof of \ref{prop:D_nF} is the `\ord{n} derivative' of the functor $F$.

\begin{definition}[Cross-effects] \label{def:cross-effects}
Given a homotopy functor $F: \cat{C} \to \cat{D}$, the \emph{\ord{n} cross-effect of $F$} is the multivariable functor
\[ \creff_n(F) : \cat{C}^{n} \to \cat{D} \]
given by
\[ \creff_n(F)\left(X_1,\dots,X_n\right) := \thofib_{I \subset \{1,\dots,n\}} F\left(\Wdge_{i \in I} X_i\right). \]
This is the total homotopy fibre of the cube, indexed by subsets of $\{1,\dots,n\}$, consisting of $F$ applied to wedges of the corresponding subsets of the objects $X_i$. See \cite{goodwillie:1991} for a detailed description of total homotopy fibres of cubes, and \cite{goodwillie:2003} for more on cross-effects.

Note that a permutation $\sigma$ of $\{1,\dots,n\}$ induces a natural symmetry isomorphism
\[ \creff_n(F)\left(X_1,\dots,X_n\right) \isom \creff_n(F)\left(X_{\sigma(1)},\dots,X_{\sigma(n)}\right). \]
Strictly speaking, the functor $\creff_n(F)$ defined here only preserves weak equivalences between cofibrant inputs. We therefore tacitly compose with a cofibrant replacement (if necessary) to get an honest homotopy functor.
\end{definition}

\begin{definition}[Goodwillie derivatives] \label{def:derivative}
Let $F: \cat{C} \to \cat{D}$ be a homotopy functor and let $\mathbb{D}_nF$ be as in Proposition \ref{prop:D_nF}. The \emph{\ord{n} Goodwillie derivative of $F$} is given by evaluating the \ord{n} cross-effect of $\mathbb{D}_nF$ at the cofibrant sphere spectrum for each input.
\[ \der^G_n(F) := \creff_n(\mathbb{D}_nF)(S_c,\dots,S_c) \]
The symmetry isomorphisms give the spectrum $\der^G_n(F)$ an action of the symmetric group $\Sigma_n$. The construction of $\der^G_n(F)$ is natural in $F$.
\end{definition}

\begin{remark}
The superscript $G$ in $\der^G_n(F)$ is meant to indicate that these are the derivatives of $F$ as defined by Goodwillie. The main idea of this paper is to construct new models for these derivatives which we denote just by $\der_n(F)$. These are defined in \S\S\ref{sec:nat},\ref{sec:spaces-spectra}-\ref{sec:spaces-spaces}.
\end{remark}

\begin{remark}
Following through the definition of $\der^G_n(F)$, we can see that the derivatives of $F$ depend only on the restriction of $F$ to finite cell complexes. We therefore consider $\der^G_n(F)$ to be defined for any functor $F: \cat{C}^{\mathsf{fin}} \to \cat{D}$ where $\cat{C}$ and $\cat{D}$ are either $\sset$ or $\spectra$ and $\cat{C}^{\mathsf{fin}}$ denotes the full subcategory of finite cell complexes in $\cat{C}$.
\end{remark}

\begin{proposition} \label{prop:Dn-formula}
Let $F: \cat{C} \to \cat{D}$ be a finitary homotopy functor and let $\mathbb{D}_nF$ be as in Proposition \ref{prop:D_nF}. There is a zigzag of natural weak equivalences
\[ \mathbb{D}_nF(X) \homeq (\der^G_n(F) \smsh X^{\smsh n})_{h\Sigma_n}. \]
Combined with Proposition \ref{prop:D_nF}, this gives us formulas for layers in the Taylor tower of $F:\cat{C} \to \cat{D}$.
\end{proposition}
\begin{proof}
This is again Goodwillie's classification of finitary homogeneous functors.
\end{proof}

Finally, we record some of the properties of the process of taking derivatives.

\begin{prop} \label{prop:deriv-commute}
For homotopy functors $F: \cat{C} \to \cat{D}$ with $\cat{C},\cat{D} = \sset,\spectra$, the functor
\[ F \mapsto \der^G_n(F) \]
preserves finite homotopy limits and filtered homotopy colimits. If $\cat{D} = \spectra$, then $\der^G_n$ preserves all homotopy colimits.
\end{prop}
\begin{proof}
The $D_n$-construction has these properties (by \cite[1.18]{goodwillie:2003}) and taking cross-effects preserves these homotopy limits and colimits, so $\der^G_n$ also has them.
\end{proof}

\section{Basic results on Taylor towers of composite functors}

In this section, we prove some key results about the behaviour of Taylor towers of composite functors. The main result we need is the following.

\begin{proposition} \label{prop:calculus}
Let $F$ and $G$ be pointed simplicial homotopy functors (between any combination of the categories $\spectra$ and $\sset$). Suppose that $F$ and $G$ are composable so that $FG$ exists. Then:
\begin{enumerate}
  \item the natural map $P_n(FG) \to P_n((P_nF)G)$ is an equivalence;
  \item if $F$ is finitary, then the natural map $P_n(FG) \to P_n(F(P_nG))$ is an equivalence.
\end{enumerate}
\end{proposition}
\begin{proof}
In \cite[\S6]{ching:2007}, the second author proved these results in the case that $F$ and $G$ are functors from spectra to spectra. In remarks in that paper, it was indicated that the proofs largely carry over to the cases of functors to and/or from spaces. Here we fill in the details of these extensions.

For part (1), the first part of the proof for spectra given in \cite{ching:2007} applies directly to functors of simplicial sets also. This allows us to construct natural maps
\[ v_n(F,G): P_n((P_nF)G) \longrightarrow P_n(FG) \]
such that the composite
\[ {\dgTEXTARROWLENGTH=4em P_n(FG) \arrow{e,t}{p_nF} P_n((P_nF)G) \arrow{e,t}{v_n(F,G)} P_n(FG)} \]
is homotopic to the identity. (See \cite[6.9]{ching:2007}.)

The remainder of the proof given for spectra does not work for spaces, but we are grateful to Tom Goodwillie for providing the following more general argument which completes the proof of (1).

Consider the diagram
\[ \begin{diagram}
  \node{P_n(FG)} \arrow{e,t}{p_nF} \arrow{s} \node{P_n((P_nF)G)} \arrow{e,t}{v_n(F,G)} \arrow{s} \node{P_n(FG)} \arrow{s} \\
  \node{P_n((P_nF)G)} \arrow{e,t}{p_n(P_nF)} \node{P_n(P_n(P_nF)G)} \arrow{e,t}{v_n(P_nF,G)} \node{P_n((P_nF)G)}
\end{diagram} \]
where the vertical maps are all induced by $p_nF:F \to P_nF$. Each row here is homotopic to the identity by the previous constructions, but the middle vertical map is an equivalence since $P_nF \to P_n(P_nF)$ is an equivalence. Thus the left-hand vertical map is a retract of an equivalence, so is an equivalence. This completes part (1) of the Proposition.

For part (2), the argument given in \cite{ching:2007} works for any $F$ and $G$ where the `middle' category of the composition (i.e. the source category of $F$ and the target category of $G$) is $\spectra$. Here we describe the changes necessary to apply this method in the case that the middle category is $\sset$.

As in \cite{ching:2007}, we can reduce to the case that $F$ is $k$-homogeneous for some $k$, using part (1) of this proposition and using the fibre sequences $D_kF \to P_kF \to P_{k-1}F$. Now suppose that the source category of $F$ is $\sset$. Then, by Proposition \ref{prop:D_nF}, we can write
\[ F \homeq F' \Sigma^\infty \]
where $F'$ is a homogeneous functor with source category $\spectra$. The method of proof of \cite[6.11]{ching:2007} now applies if we can show that the map
\[ P_n(\Sigma^\infty G) \to P_n(\Sigma^\infty P_nG) \]
is an equivalence. In other words, the spectra to spectra result allows us to reduce to the case that $F = \Sigma^\infty$.

To see this case, it is sufficient to show that any map $\Sigma^\infty G \to H$ with $H$ $n$-excisive factors via the map
\[ \Sigma^\infty G \to \Sigma^\infty P_nG \to P_n(\Sigma^\infty P_nG). \]
The claimed equivalence then follows by the universal property of $P_n$ (see \ref{thm:Taylor}).

But $\Sigma^\infty G \to H$ is adjoint to a map $G \to \Omega^\infty H$ which factors via $P_nG$ since $\Omega^\infty H$ is $n$-excisive. This then gives us a factorization
\[ \Sigma^\infty G \to \Sigma^\infty P_nG \to H, \]
and the second map factors via $P_n(\Sigma^\infty P_nG)$ since $H$ is $n$-excisive. This completes the proof of (2).
\end{proof}

\begin{remark}
These are important results for our approach to the calculus of composite functors because they say that the terms of the Taylor tower of $FG$ depend only on the appropriate terms of the individual Taylor towers of $F$ and $G$.
\end{remark}

\begin{example} \label{ex:counterexample}
The following example of Kuhn \cite{kuhn:2007} shows that the finitary condition is necessary in part (2) of Proposition \ref{prop:calculus}. Let $F: \spectra \to \spectra$ be a non-smashing localization $L_E$ (e.g. with respect to mod $2$ K-theory), and let $G$ be the functor given by $G(X) = (X \smsh X)_{h\Sigma_2}$. Then
\[ P_1(FG)(S) \homeq \hocofib(L_E(S) \smsh \mathbb{R}P^\infty \to L_E(\mathbb{R}P^\infty)) \neq * \]
but $P_1G \homeq *$, so $P_1(F(P_1G)) \homeq *$.
\end{example}

\begin{remark}
Proposition \ref{prop:calculus} can be generalized to the following results. Recall that we say a functor $F$ is $m$-reduced if $P_{m-1}F \homeq *$, and that a map $F_1 \to F_2$ is $m$-reduced if it induces an equivalence $P_{m-1}F_1 \to P_{m-1}F_2$.
\begin{itemize}
  \item If $F_1 \to F_2$ is an $n$-reduced map between finitary homotopy functors, and $G$ is $m$-reduced (where $m,n \geq 1$) then the map $F_1G \to F_2G$ is $mn$-reduced.
  \item If $F$ is an $n$-reduced finitary homotopy functor (where $n \geq 1$), and $G_1 \to G_2$ is an $m$-reduced map between $d$-reduced homotopy functors (where $m \geq d \geq 1$), then the map $FG_1 \to FG_2$ is $m+d(n-1)$-reduced.
\end{itemize}
These facts can be proved by generalizations of the arguments used to prove Proposition \ref{prop:calculus}.
\end{remark}

\section{The category of functors} \label{sec:functors}

We are interested in studying functors $\cat{C} \to \cat{D}$, where the categories $\cat{C}$ and $\cat{D}$ are each either $\sset$ or $\spectra$. We cannot, however, define the category of all functors $\cat{C} \to \cat{D}$ without incurring set-theoretic problems. Fortunately, this is unnecessary for us since the derivatives of a functor depend only on its values on finite cell objects. We therefore make the following definition.

\begin{definition}[Functor categories]
Recall that if $\cat{C}$ is either $\sset$ or $\spectra$, then $\cat{C}^{\mathsf{fin}}$ is the full subcategory of finite cell complexes in $\cat{C}$. Let $[\cat{C}^{\mathsf{fin}},\cat{D}]$ be the category whose objects are the pointed simplicial (see \ref{def:functors}) functors $F:\cat{C}^{\mathsf{fin}} \to \cat{D}$ and whose morphisms $F \to G$ are the simplicial natural transformations. Since $\cat{C}^{\mathsf{fin}}$ is skeletally small, there is only a set of natural transformations between two such functors. Therefore $[\cat{C}^{\mathsf{fin}}, \cat{D}]$ is a (locally small) category in the usual sense (i.e. the morphisms between any two objects form a set).
\end{definition}

We have the following version of the Yoneda Lemma for the category $[\cat{C}^{\mathsf{fin}},\cat{D}]$:
\begin{lemma}[Enriched Yoneda Lemma] \label{lem:weak-yoneda}
Let $F:\cat{C}^{\mathsf{fin}} \to \cat{D}$ be a pointed simplicial functor, and take $K \in \cat{C}^{\mathsf{fin}}$ and $I \in \cat{D}$. Then there is a 1-1 correspondence between the set of simplicial natural transformations
\[ I \smsh \cat{C}(K,-) \to F \]
and the set of morphisms in $\cat{D}$ of the form
\[ I \to F(K). \]
\end{lemma}
\begin{proof}
This follows from \cite[1.9]{kelly:2005}.
\end{proof}

\begin{proposition} \label{prop:model-functors}
Let $\cat{C}$ and $\cat{D}$ be either $\sset$ or $\spectra$. Then there is a model structure on the functor category $[\cat{C}^{\mathsf{fin}},\cat{D}]$ with the following properties:
\begin{itemize}
  \item a natural transformation $F \to G$ is a \emph{weak equivalence} or \emph{fibration} if and only if $F(X) \to G(X)$ is a weak equivalence or fibration in$\cat{D}$, respectively, for all $X \in \cat{C}^{\mathsf{fin}}$;
  \item the model structure is cofibrantly generated with generating cofibrations of the form
  \[ I_0 \smsh \cat{C}(K,-) \to I_1 \smsh \cat{C}(K,-) \]
  where $I_0 \to I_1$ is one of the generating cofibrations in $\cat{D}$ and $K \in \cat{C}^{\mathsf{fin}}$. Similarly, the generating trivial cofibrations are of the form
  \[ J_0 \smsh \cat{C}(K,-) \to J_1 \smsh \cat{C}(K,-) \]
  where $J_0 \to J_1$ is one of the generating trivial cofibrations in $\cat{D}$ and $K \in \cat{C}^{\mathsf{fin}}$.
\end{itemize}
We refer to this as the \emph{projective model structure} on $[\cat{C}^{\mathsf{fin}},\cat{D}]$.
\end{proposition}
\begin{proof}
This is an enriched version of \cite[11.6.1]{hirschhorn:2003}. The proof of that result carries over to this case using the Yoneda Lemma, and the fact that limits and colimits in $[\cat{C}^{\mathsf{fin}},\cat{D}]$ can be calculated objectwise.
\end{proof}

\begin{remark} \label{rem:compactly-generated-functors}
Just as for $\sset$ and $\spectra$ (see Remark \ref{rem:compactly-generated}), the model structure we use on $[\cat{C}^{\mathsf{fin}},\cat{D}]$ is `compactly generated' (see \cite[4.5]{may/sigurdsson:2006}). This requires that the domains of the generating cofibrations be $\omega$-small relative to the $\mathbb{I}$-cell complexes (where $\mathbb{I}$ is the set of generating cofibrations and $\omega$ is the countable cardinal). This claim follows from the Yoneda Lemma and the corresponding claim in the model category $\cat{D}$. A consequence of this is that, as in $\sset$ and $\spectra$, we only require cell complexes formed from countable sequences of pushouts of coproducts of generating cofibrations in order to apply the small object argument. We therefore restrict our notion of `cell complex' in $[\cat{C}^{\mathsf{fin}},\cat{D}]$ to such cases. This gives us the following definition.
\end{remark}

\begin{definition}[Cell functors] \label{def:cell-functor}
A \emph{cell functor} in $[\cat{C}^{\mathsf{fin}},\cat{D}]$ is a cell complex with respect to the generating cofibrations of Proposition \ref{prop:model-functors}. We emphasize that a cell functor does not come with any specified cell structure, just the assertion that one exists.
A \emph{presented cell functor} is a cell functor $F$ together with a chosen cell structure. Explicitly, this consists of:
\begin{itemize}
  \item a sequence of functors
  \[ * = F_0 \to F_1 \to \dots \to F \]
  such that $F$ is the colimit of the $F_i$;
  \item a sequence of pushout squares of the form
  \[ \begin{diagram}
    \node{\Wdge_{\alpha \in A_i}^{\mathstrut} I_0^{\alpha} \smsh \cat{C}(K_{\alpha},-)} \arrow{e} \arrow{s} \node{F_i} \arrow{s} \\
    \node{\Wdge_{\alpha \in A_i}^{\mathstrut} I_1^{\alpha} \smsh \cat{C}(K_{\alpha},-)} \arrow{e} \node{F_{i+1}}
  \end{diagram} \]
  where $A_i$ is an indexing set (the set of \emph{cells of degree $i+1$}), each $I^{\alpha}_0 \to I^{\alpha}_1$ is a cofibration in $\cat{D}$, and each $K_{\alpha}$ is an object of $\cat{C}^{\mathsf{fin}}$.
\end{itemize}
\end{definition}

\begin{definition}[Cofibrant replacements for functors] \label{def:QF}
Given $F \in [\cat{C}^{\mathsf{fin}},\cat{D}]$, the small object argument determines a cofibrant replacement for $F$, which we write $QF$. The functor $QF$ comes with a canonical cell structure in which the cells of degree $i+1$ are in 1-1 correspondence with commutative diagrams of the form
\[ \begin{diagram}
  \node{I_0 \smsh \cat{C}(K,-)} \arrow{e} \arrow{s} \node{(QF)_i} \arrow{s} \\
  \node{I_1 \smsh \cat{C}(K,-)} \arrow{e} \node{F}
\end{diagram} \]
The approximation map $QF \to F$ is always a fibration and a weak equivalence in $[\cat{C}^{\mathsf{fin}},\cat{D}]$, that is, $QF(X) \to F(X)$ is a fibration and a weak equivalence for all $X \in \cat{C}^{\mathsf{fin}}$.
\end{definition}

\begin{definition}[Kan extension] \label{def:Kan-extension}
In order to describe the chain rule we have to be able to compose the functors we are working with. We therefore extend an object $F \in [\cat{C}^{\mathsf{fin}},\cat{D}]$ to a pointed simplicial functors $LF: \cat{C} \to \cat{D}$ by enriched left Kan extension. Explicitly, for $X \in \cat{C}$, we define
\[ LF(X) := \colim \left( \Wdge_{K,K' \in \cat{C}^{\mathsf{fin}}} F(K) \smsh \cat{C}(K,K') \smsh \cat{C}(K',X) \rightrightarrows \Wdge_{K \in \cat{C}^{\mathsf{fin}}} F(K) \smsh \cat{C}(K,X) \right). \]
One map in this coequalizer is given by the composition map
\[ \cat{C}(K,K') \smsh \cat{C}(K',X) \to \cat{C}(K,X) \]
and the other by
\[ FK \smsh \cat{C}(K,K') \to FK \smsh \cat{D}(FK,FK') \to FK'. \]
Equivalently, this is an (enriched) \emph{coend} (in the sense of MacLane \cite{maclane:1998}).

The Kan extension $LF$ satisfies the following universal property. If we have any pointed simplicial functor $H:\cat{C} \to \cat{D}$ then there is a 1-1 correspondence between natural transformations $F \to H |_{\cat{C}^{\mathsf{fin}}}$, i.e. from $F$ to the restriction of $H$ to $\cat{C}^{\mathsf{fin}}$, and natural transformations $LF \to H$.
\end{definition}

\begin{lemma} \label{lem:extend-basic}
Let $F \in [\cat{C}^{\mathsf{fin}},\cat{D}]$ be given by
\[ F(X) = I \smsh \cat{C}(K,X) \]
for some $I \in \cat{D}$ and $K \in \cat{C}^{\mathsf{fin}}$. Then the Kan extension of $F$ to all of $\cat{C}$ is given by the same formula, that is, there is a natural isomorphism
\[ LF(X) \isom I \smsh \cat{C}(K,X). \]
This isomorphism is also natural with respect to $F$, that is, with respect to $I$ and $K$.
\end{lemma}
\begin{proof}
The extension is given by
\[ LF(X) = \int^{L \in \cat{C}^{\mathsf{fin}}} I \smsh \cat{C}(K,L) \smsh \cat{C}(L,X) \]
which has a natural map to $I \smsh \cat{C}(K,X)$ given by the composition map
\[ \cat{C}(K,L) \smsh \cat{C}(L,X) \to \cat{C}(K,X). \]
But there is an inverse to this given by
\[ I \smsh \cat{C}(K,X) \to I \smsh \cat{C}(K,K) \smsh \cat{C}(K,X) \to \int^{L \in \cat{C}^{\mathsf{fin}}} I \smsh \cat{C}(K,L) \smsh \cat{C}(L,X) \]
using the unit map $\Delta[0]_+ \to \cat{C}(K,K)$, and the fact that $K \in \cat{C}^{\mathsf{fin}}$.
\end{proof}

\begin{remark} \label{rem:extend}
The claim of Lemma \ref{lem:extend-basic} applies to any presented cell functor in $[\cat{C}^{\mathsf{fin}},\cat{D}]$ in the following sense. The left Kan extension is a left adjoint and so preserves the pushout diagrams that describe the attaching of cells, and preserves the sequential colimit that describes the union of the cells. Therefore, if $F$ is the Kan extension of a presented cell functor, then we can still write it as the colimit of a sequence
\[ * = F_0 \to F_1 \to \dots \]
where $F_{i+1}$ is obtained from $F_i$ by a pushout diagram of the form given in Definition \ref{def:cell-functor}. In other words, it is still a cell complex with respect to the Kan extensions of the generating cofibrations in $[\cat{C}^{\mathsf{fin}},\cat{D}]$ (which by Lemma \ref{lem:extend-basic} are given by the same formulas).
\end{remark}

\begin{notation}
Partly as a result of Lemma \ref{lem:extend-basic} and Remark \ref{rem:extend} we now drop the extra notation for left Kan extension and write the extension of $F$ also as $F$. Context should determine the exact meaning. In particular, if $G \in [\cat{C}^{\mathsf{fin}},\cat{D}]$ and $F \in [\cat{D}^{\mathsf{fin}},\cat{E}]$ then we define $FG \in [\cat{C}^{\mathsf{fin}},\cat{E}]$ to be the composite of $G$ with the left Kan extension of $F$ to all of $\cat{D}$.
\end{notation}

\begin{lemma} \label{lem:cell-F(X)}
Let $F \in [\cat{C}^{\mathsf{fin}},\cat{D}]$ be a cell functor, left Kan extended to all of $\cat{C}$ as above. Let $X \in \cat{C}$ be any object. Then $F(X)$ is a cell complex in $\cat{D}$. (More generally, if $F \to F'$ is a relative cell functor, then $F(X) \to F'(X)$ is a relative cell complex in $\cat{D}$.)
\end{lemma}
\begin{proof}
Pick a presentation of $F$. Then $F(X)$ is the colimit of the sequence
\[ * = F_0(X) \to F_1(X) \to \dots, \]
and each $F_i(X) \to F_{i+1}(X)$ is the pushout of a map of the form
\[ \Wdge_{\alpha} I_0^{\alpha} \smsh \cat{C}(K_{\alpha},X) \to \Wdge_{\alpha} I_1^{\alpha} \smsh \cat{C}(K_{\alpha},X) \]
and so it is sufficient to show that each map
\[ I_0 \smsh \cat{C}(K,X) \to I_1 \smsh \cat{C}(K,X) \]
is a relative cell complex in $\cat{D}$. (A similar argument in the relative case also reduces to this claim.)

When $\cat{D} = \sset$, this map is of the form
\[ \tag{*} \partial \Delta[n]_+ \smsh \cat{C}(K,X) \to \Delta[n]_+ \smsh \cat{C}(K,X) \]
for some $n$. There is a relative cell structure on this map with an $n+k$-dimensional cell for each nondegenerate $k$-simplex in the simplicial set $\cat{C}(K,X)$. From this we obtain the required relative cell structure on $I_0 \smsh \cat{C}(K,X) \to I_1 \smsh \cat{C}(K,X)$.

When $\cat{D} = \spectra$, the map $I_0 \smsh \cat{C}(K,X) \to I_1 \smsh \cat{C}(K,X)$ is given by applying one of the functors $S \smsh_{\cat{L}} \mathbb{L}\Sigma^\infty_q$ to one of the maps (*) above. Since this functor takes the generating cofibrations in $\sset$ to generating cofibrations in $\spectra$, and preserves colimits, it preserves relative cell complexes.
\end{proof}

\begin{remark}[Cell structure on $F(X)$] \label{rem:cell-F(X)}
The proof of Lemma \ref{lem:cell-F(X)} determines an explicit cell structure on $F(X)$ where $F: \cat{C} \to \cat{D}$ is a presented cell functor and $X \in \cat{D}$ is any object. The cells in this structure are in 1-1 correspondence with pairs $(\alpha,\epsilon)$ where $\alpha$ is one of the cells in the presented cell structure on $F$, and $\epsilon$ is a nondegenerate simplex in the simplicial set $\cat{C}(K_{\alpha},X)$ where $K_{\alpha}$ is the object of $\cat{C}^{\mathsf{fin}}$ corresponding to the cell $\alpha$. Note that each such $\epsilon$ corresponds to a morphism
\[ K_{\alpha} \smsh \Delta[n]_+ \to X \]
in the category $\cat{C}$.
\end{remark}

\begin{lemma} \label{lem:homotopy-cell-functors}
Let $F \in [\finspec,\cat{D}]$ be a cell functor. Then the Kan extension of $F$ is a homotopy functor $\spectra \to \cat{D}$.
\end{lemma}
\begin{proof}
The basic cell functors $I \smsh \cat{C}(K,-)$ preserve weak equivalences when $I$ and $K$ are cofibrant because every object of $\based$ or $\spectra$ is fibrant and every simplicial set is cofibrant.

We now proceed by induction on a cell structure for $F$. Pick such a cell structure as in Definition \ref{def:cell-functor}. For any $X \in \cat{C}$, the map
\[ I_0 \smsh \cat{C}(K,X) \to I_1 \smsh \cat{C}(K,X) \]
is a cofibration in $\cat{D}$ and it follows that the pushout diagram determining $F_{i+1}(X)$ is a homotopy pushout. It also follows that $F_i(X) \to F_{i+1}(X)$ is a cofibration and so $F(X)$, as the colimit of the $F_i(X)$, is also their homotopy colimit. Now if $X \weq Y$ is a weak equivalence in $\cat{C}$, then it induces weak equivalences
\[ \Wdge_{\alpha} I_0^{\alpha} \smsh \cat{C}(K_{\alpha},X) \weq \Wdge_{\alpha} I_1^{\alpha} \smsh \cat{C}(K_{\alpha},X) \]
and so by induction on the homotopy pushout squares, it gives weak equivalences $F_i(X) \weq F_i(Y)$ for each $i$. Therefore, by the property of homotopy colimits, we get an equivalence $F(X) \weq F(Y)$. (See also Props. 13.5.4 and 17.9.1 of \cite{hirschhorn:2003} for statements of the invariance of homotopy colimits used here.)
\end{proof}

\begin{remark}
Lemma \ref{lem:homotopy-cell-functors} does not hold for cell functors $F:\sset \to \cat{D}$ because not all simplicial sets are fibrant. However, a similar argument shows that if $X \weq Y$ is a weak equivalence between fibrant simplicial sets, then $F(X) \weq F(Y)$ is a weak equivalence.
\end{remark}

\begin{lemma} \label{lem:finitary-cell-functors}
Let $F \in [\finspec,\cat{D}]$ be a cell functor. Then the Kan extension of $F$ to $\spectra$ is finitary.
\end{lemma}
\begin{proof}
We saw in the proof of Lemma \ref{lem:homotopy-cell-functors} that a cell functor is the homotopy colimit of a sequence of functors formed from taking homotopy pushouts along maps between coproducts of functors of the form $I \smsh \spectra(K,-)$ for $I \in \cat{D}$ the source/domain of one of the generating cofibrations, and $K \in \finspec$. Since all these homotopy colimits commute with filtered homotopy colimits, it is sufficient to show that $I \smsh \spectra(K,-)$ is finitary. Also $I \smsh - $ preserves filtered homotopy colimits, so it is enough that $\spectra(K,-)$ be finitary, which is well-known.
\end{proof}

\begin{remark} \label{rem:finitary}
Lemma \ref{lem:finitary-cell-functors} tells us that any pointed simplicial homotopy functor $F: \spectra \to \cat{D}$ for $\cat{D}$ either $\sset$ or $\spectra$ has a natural finitary approximation. If $QF$ denotes the cellular replacement in $[\finspec,\cat{D}]$ for the restriction of $F$ to $\finspec$, then $QF$ (Kan extended back to all of $\spectra$) is a finitary homotopy functor, and there is a natural transformation $QF \to F$ that is a weak equivalence of finite cell spectra. This map is a weak equivalence on all $X \in \spectra$ if and only if $F$ is finitary.
\end{remark}

\section{Subcomplexes of presented cell functors} \label{sec:cell-functors}

In this section we describe the theory of subcomplexes in the model category $[\cat{C}^{\mathsf{fin}},\cat{D}]$. We refer again to \cite[10.6]{hirschhorn:2003} for a general treatment. The situation is greatly simplified by the following lemma.

\begin{lemma} \label{lem:monomorphisms}
A relative cell functor in $[\cat{C}^{\mathsf{fin}},\cat{D}]$ (that is, a relative cell complex with respect to the generating cofibrations of Proposition \ref{prop:model-functors}) is a monomorphism.
\end{lemma}
\begin{proof}
Let $\iota: F \to G$ be a relative cell functor. To show that $\iota$ is a monomorphism, it is enough to show that $\iota_X: FX \to GX$ is a monomorphism for all $X \in \cat{C}^{\mathsf{fin}}$. But $FX \to GX$ is a relative cell complex by Lemma \ref{lem:cell-F(X)}, so a monomorphism by \ref{prop:cell-complex}(1).
\end{proof}

It follows from this lemma that a subcomplex of a presented cell functor is determined by its set of cells (see \cite[10.6.10]{hirschhorn:2003}). We therefore define a subcomplex as follows.

\begin{definition}[Subcomplexes of cell functors] \label{def:subcomplex}
Let $F$ be a presented cell functor (as in Definition \ref{def:cell-functor}). A \emph{subcomplex} $C$ of $F$ is a subset of the set of cells of $F$ that for each $i \geq 1$ satisfies the following inductive condition:
\begin{itemize}
  \item[$P_i$] Suppose that as a result of the condition $P_{i-1}$ we have constructed a functor $C_{i-1}$ and a monomorphism $C_{i-1} \to F_{i-1}$ (where $C_0 = *$). The condition $P_i$ is then that for each cell $\alpha$ of degree $i$ in $C$, the attaching map for $\alpha$ of the form
      \[ I^{\alpha}_0 \smsh \cat{C}(K_{\alpha},-) \to F_{i-1} \]
      factors via $C_{i-1} \to F_{i-1}$. (Such a factorization is unique because $C_{i-1} \to F_{i-1}$ is a monomorphism.)

      With this condition satisfied, we define $C_i$ by the pushout diagram
      \[ \begin{diagram}
        \node{\Wdge_{\alpha}^{\mathstrut} I^{\alpha}_0 \smsh \cat{C}(K_{\alpha},-)} \arrow{e} \arrow{s} \node{C_{i-1}} \arrow{s} \\
        \node{\Wdge_{\alpha}^{\mathstrut} I^{\alpha}_1 \smsh \cat{C}(K_{\alpha},-)} \arrow{e} \node{C_i}
      \end{diagram} \]
      with the map $C_i \to F_i$ then determined by the universal property of the pushout and the diagrams that define the cell structure on $F$. The map $C_i \to F_i$ constructed in this way is a relative cell functor and hence a monomorphism by Lemma \ref{lem:monomorphisms}.
\end{itemize}
The data associated to the subcomplex $C$ (that is, the sequence $* = C_0 \to C_1 \to \dots$ and the pushout diagrams above) form a presented cell functor in their own right. We abuse notation by writing $C$ for this functor, that is, for the colimit of the sequence of $C_i$. The colimit of the maps $C_i \to F_i$ is then a map $C \to F$ which we call the \emph{inclusion} of the subcomplex $C$ into $F$.
\end{definition}

\begin{definition}[Finite subcomplexes] \label{def:finite-subcomplex}
A subcomplex $C$ of a presented cell functor $F$ is \emph{finite} if it has finitely many cells. The finite subcomplexes of $F$ form a partially ordered set under inclusion, which we denote $\mathsf{Sub}(F)$. If $C,D \in \mathsf{Sub}(F)$ with $C \subset D$, then there is a unique map $C \to D$ between the corresponding functors that commutes with the inclusions into $F$. We call this the \emph{inclusion} of the subcomplex $C$ into the subcomplex $D$.
\end{definition}

\begin{lemma} \label{lem:lattice}
Let $F$ be a presented cell functor in $[\cat{C}^{\mathsf{fin}},\cat{D}]$ Then the category $\mathsf{Sub}(F)$ has the following properties:
\begin{enumerate}
  \item $\mathsf{Sub}(F)$ is a poset;
  \item each object in $\mathsf{Sub}(F)$ has finitely many predecessors;
  \item any finite set of objects in $\mathsf{Sub}(F)$ has a \emph{least} upper bound. In particular, $\mathsf{Sub}(F)$ is a filtered category.
\end{enumerate}
\end{lemma}
\begin{proof}
We already mentioned (1) in Definition \ref{def:finite-subcomplex}. A finite set has finitely many subsets which gives (2). For (3), it is sufficient to show that the (set-theoretic) union of a set of subcomplexes is a subcomplex. This is easily checked using the inductive definition (\ref{def:subcomplex}).
\end{proof}

\begin{proposition} \label{prop:subcomplexes}
Let $F: \cat{C}^{\mathsf{fin}} \to \cat{D}$ be a presented cell functor. Then:
\begin{enumerate}
  \item any cell of $F$ is contained in a finite subcomplex;
  \item any map $C \to F$, with $C$ a finite cell functor, factors via a finite subcomplex of $F$.
\end{enumerate}
\end{proposition}
\begin{proof}
We prove (1) by induction on the degree of the cell. A cell $\alpha$ of degree $i+1$ has an attaching map of the form
\[ \tag{*} I_0 \smsh \cat{C}(K,-) \to F_i \]
where $I_0$ is the domain of one of the generating cofibrations in $\cat{D}$ and $K \in \cat{C}^{\mathsf{fin}}$. By the Yoneda Lemma (\ref{lem:weak-yoneda}), (*) determines a map $f: I_0 \to F_i(K)$ in $\cat{D}$. Now $F_i(K)$ has a cell structure as in Remark \ref{rem:cell-F(X)} and so, by \ref{prop:cell-complex}(2), $f$ factors via a finite subcomplex $A \subset F_i(K)$. As noted in Remark \ref{rem:cell-F(X)}, each of the cells in $A$ corresponds to a cell in $F_i$ which of course has degree at most $i$. By the induction hypothesis, each of these is contained in a finite subcomplex of $F_i$. Taking the union of these finite subcomplexes gives a finite subcomplex $C$ of $F_i$ such that $A \subset C(K)$. But then the map $f$ factors as
\[ I_0 \to C(K) \to F_i(K) \]
and hence the original attaching map (*) factors as
\[ I_0 \smsh \cat{C}(K,-) \to C \to F_i. \]
Therefore, $C \cup \{\alpha\}$ is a finite cell complex containing $\alpha$.

For (2) pick a cell structure for $C$ with finitely many cells. We prove (2) by induction on the number of cells. With no cells, $C = *$ and so the map factors via the finite subcomplex $* \subset F$. Now suppose that $C$ is obtained by adding a cell to $C'$, and suppose that the restricted map $C' \to F$ factors via the finite subcomplex $D' \subset F$. Then we have the following diagram where the left-hand square is a pushout:
\[ \begin{diagram}
  \node{I_0 \smsh \cat{C}(K,-)} \arrow{e} \arrow{s} \node{C'} \arrow{s} \arrow{e} \node{D'} \arrow{s} \\
  \node{I_1 \smsh \cat{C}(K,-)} \arrow{e} \node{C} \arrow{e} \node{F}
\end{diagram} \]
By the Yoneda Lemma (\ref{lem:weak-yoneda}), the overall square corresponds to a square of the form
\[ \begin{diagram}
  \node{I_0} \arrow{e} \arrow{s} \node{D'(K)} \arrow{s} \\
  \node{I_1} \arrow{e} \node{F(K)}
\end{diagram} \]
The argument we used in part (1) now tells us that the map $I_1 \to F(K)$ factors via $D''(K)$ for some finite subcomplex $D'' \subset F$. Set $D = D' \cup D''$ so that $D$ is also a finite subcomplex of $F$. Then the map $I_1 \to D(K)$ corresponds by the Yoneda Lemma to a natural transformation
 \[ I_1 \smsh \cat{C}(K,-) \to D \]
which by construction factors the map $I_1 \smsh \cat{C}(K,-) \to F$. By the universal property of the pushout this gives us a factorization $C \to D \to F$. By induction, this is true for any finite cell functor $C$.
\end{proof}

\begin{corollary} \label{cor:subcomplexes}
Let $F:\cat{C}^{\mathsf{fin}} \to \cat{D}$ be a presented cell functor. Then the canonical map
\[ \colim_{C \in \mathsf{Sub}(F)} C \to F \]
is an isomorphism. The colimit here is taken over the poset of finite subcomplexes of $F$ (see Definition \ref{def:finite-subcomplex}).
\end{corollary}
\begin{proof}
We inductively construct maps from $F_i$ to the colimit such that the composite $F_i \to \colim C \to F$ is the usual inclusion. Suppose a map from $F_i$ has already been constructed. By Proposition \ref{prop:subcomplexes}(1), each cell $\alpha$ of degree $i+1$ is contained in some finite subcomplex $C$ of $F$ and so the map $I_1 \smsh \cat{C}(K,-) \to F$ associated to $\alpha$ factors via this colimit. But then the universal property of the pushout determines a map $F_{i+1} \to \colim C$ with the required property. Taking the colimit over $i$ we get a map
\[ F \to \colim_{C \in \mathsf{Sub}(F)} C \]
such that the composite $F \to \colim C \to F$ is the identity. The other composite is the identity on $\colim C$ and so we have an isomorphism.
\end{proof}

\begin{corollary} \label{cor:hocolim}
Every functor $F \in [\cat{C}^{\mathsf{fin}},\cat{D}]$ is weakly equivalent to a filtered homotopy colimit of finite cell functors.
\end{corollary}
\begin{proof}
Any $F$ is equivalent to its cellular replacement $QF$ (see Definition \ref{def:QF}) which is the strict colimit of its finite subcomplexes. It is therefore sufficient to show that the strict colimit of the canonical diagram of finite subcomplexes is equivalent to the homotopy colimit. To see this is it enough to show that the canonical diagram is a cofibrant object in the projective model structure on $\mathsf{Sub}(QF)$-indexed diagrams of functors. From this it follows that the strict and homotopy colimits are equivalent.

It is therefore enough to show that the canonical diagram of finite subcomplexes has the left-lifting property with respect to maps of $\mathsf{Sub}(QF)$-indexed diagrams in $[\cat{C}^{\mathsf{fin}},\cat{D}]$ that are objectwise trivial fibrations. Let $\mathcal{I}$ denote that canonical diagram, so that $\mathcal{I}(C) = C$ for $C \in \mathsf{Sub}(QF)$. Then, given a diagram:
\[ \begin{diagram}
  \node{*} \arrow{s} \arrow{e} \node{\mathcal{A}} \arrow{s,r,A}{\sim} \\
  \node{\mathcal{I}} \arrow{e} \node{\mathcal{B}}
\end{diagram}\]
we construct a lift by induction on the objects in $\mathsf{Sub}(QF)$. Suppose that we have constructed a lift for all proper subcomplexes of some $C \in \mathsf{Sub}(QF)$ (and suppose that those lifts are compatible with inclusions of subcomplexes). If $C'$ is the union of those proper subcomplexes, then together those lifts define a map $C' \to \mathcal{A}(C)$ that fits into a square
\[ \begin{diagram}
  \node{C'} \arrow{s} \arrow{e} \node{A(C)} \arrow{s,r,A}{\sim} \\
  \node{C} \arrow{e} \node{B(C)}
\end{diagram} \]
Now $C'$ is a subcomplex of $C$ (possibly not proper), so $C' \to C$ is a cofibration. This diagram therefore has a lift $C = \mathcal{I}(C) \to \mathcal{A}(C)$. Of course it might have many lifts, but anyone we pick is compatible with the morphisms in $\mathsf{Sub}(QF)$, namely the inclusions of subcomplexes. Since every object in $\mathsf{Sub}(QF)$ has finitely many subcomplexes, we can proceed with the induction, and thus obtain a lift $\mathcal{I} \to \mathcal{A}$ for the entire $\mathsf{Sub}(QF)$-indexed diagrams.
\end{proof}

\section{Pro-spectra} \label{sec:pro-spectra}

Our constructions of new models for the Goodwillie derivatives make substantial use of Spanier-Whitehead duality for spectra. It is well know that this is only really a duality theory for finite spectra. For general spectra, the appropriate extension of Spanier-Whitehead duality uses the category of pro-spectra. This theory was worked out by Christensen and Isaksen in \cite{christensen/isaksen:2004}. They constructed a model structure on the category of pro-spectra, and a zigzag of Quillen equivalences between that structure and the usual model structure on the opposite category of $\spectra$. In this section, we recall some of the main constructions and results of \cite{christensen/isaksen:2004}.

\begin{definition}[Pro-objects] \label{def:pro-objects}
Let $\cat{C}$ be any category. A \emph{pro-object} in $\cat{C}$ is a functor $X: \cat{J} \to \cat{C}$ where $\cat{J}$ is a small cofiltered category. If $X: \cat{J} \to \cat{C}$ is a pro-object and $j \in \cat{J}$, we write $X_j$ for the object of $\cat{C}$ given by evaluating $X$ at $j$.

Given two pro-objects $X: \cat{J} \to \cat{C}$ and $Y: \cat{K} \to \cat{C}$, a \emph{morphism of pro-objects} from $X$ to $Y$ is an element in the set
\[ \lim_{k \in \cat{K}} \colim_{j \in \cat{J}} {\textstyle\Hom_{\cat{C}}}(X_j,Y_k). \]
More explicitly, it consists of the following data:
\begin{itemize}
  \item a function $f$ from the set of objects of $\cat{K}$ to the set of objects of $\cat{J}$;
  \item for each $k \in \cat{K}$, a map $\phi_k: X_{f(k)} \to Y_k$;
\end{itemize}
subject to the following condition:
\begin{itemize}
  \item if $k \to k'$ is a map in $\cat{K}$, and $j$ is an object in $\cat{J}$ with maps $j \to f(k)$ and $j \to f(k')$ (such an object exists since $\cat{J}$ is cofiltered), then the following diagram commutes:
\[ \begin{diagram}
  \node[2]{X_{f(k)}} \arrow{e,t}{\phi_k} \node{Y_{k}} \arrow[2]{s} \\
  \node{X_j} \arrow{ne} \arrow{se} \\
  \node[2]{X_{f(k')}} \arrow{e,t}{\phi_{k'}} \node{Y_{k'}}
\end{diagram} \]
\end{itemize}
Let $(f,\phi_{\bullet})$ and $(f',\phi'_{\bullet})$ be two such sets of data. They determine the same morphism $X \to Y$ if for each $k \in \cat{K}$, there are maps $j \to f(k)$ and $j \to f(k')$ for some $j \in \cat{J}$ such that the following diagram commutes:
\[ \begin{diagram}
  \node{X_j} \arrow{s} \arrow{e} \node{X_{f(k)}} \arrow{s,r}{\phi_k} \\
  \node{X_{f'(k)}} \arrow{e,t}{\phi'_k} \node{Y_k}
\end{diagram} \]
For a fixed $\cat{C}$, the pro-objects in $\cat{C}$ and their morphisms form a category which we write $\mathsf{Pro}(\cat{C})$.
\end{definition}

\begin{remarks}\hfill
\begin{enumerate}
\item Recall that isomorphic pro-objects need not be indexed on the same cofiltered category. For example, if $X$ is a pro-object indexed on a category $\cat{J}$ that has an initial object $j_0$ then $X$ is isomorphic to the pro-object with value $X_{j_0}$ indexed on the trivial category with one object.
\item Any morphism of pro-objects has a \emph{level representation}, that is, by replacing the source and target with isomorphic pro-objects we can write it as a map $\phi: X \to Y$ where $X$ and $Y$ are indexed on the same cofiltered category $\cat{J}$ and $\phi$ is just a natural transformation between functors $\cat{J} \to \cat{C}$ (that is, we can take the function $f$ involved in $\phi$ to be the identity on the objects of $\cat{J}$).
\end{enumerate}
\end{remarks}

\begin{definition}[Properties of pro-objects]
A pro-object $X$ in the category $\cat{C}$ is said to have a property \emph{levelwise} if there is some pro-object $Y$, isomorphic to $X$, such that each $Y_j$ has that property.

A map $\phi:X \to Y$ of pro-objects has a property \emph{levelwise} if there exists a level representation $\phi'$ of $\phi$ such that each map $\phi'_j$ has that property.
\end{definition}

\begin{definition}[Ind-objects]
An \emph{ind-object} in the category $\cat{C}$ is a functor $X: \cat{J} \to \cat{C}$ where $\cat{J}$ is a small filtered category. We can identify an ind-object in $\cat{C}$ with a pro-object in $\cat{C}^{op}$ by identifying $X$ with the corresponding functor $\cat{J}^{op} \to \cat{C}^{op}$. A \emph{morphism of ind-objects} in $\cat{C}$ is a morphism of the corresponding pro-objects in $\cat{C}^{op}$. The ind-objects and their morphisms form a category which we write $\mathsf{Ind}(\cat{C})$.
\end{definition}

We now summarize the main results of \cite{christensen/isaksen:2004}.

\begin{theorem}[Christensen-Isaksen] \label{thm:pro-spectra} \hfill \\
(1) There is a model structure on the category $\mathsf{Pro}(\spectra)$ in which
\begin{itemize}
  \item a morphism $\phi: X \to Y$ is a \emph{weak equivalence} if for each $n \in \mathbb{Z}$ it induces an isomorphism
  \[ \colim_{k \in \cat{K}} [Y_k,S^n] \to \colim_{j \in \cat{J}} [X_j,S^n] \]
  where $[A,B]$ denotes the abelian group of weak homotopy classes of maps $A \to B$ in $\spectra$;
  \item a morphism of pro-spectra is a \emph{cofibration} if it has a level representation $\phi: X \to Y$ such that each map $\phi_j: X_j \to Y_j$ is a cofibration. In particular, if $X_j$ is cofibrant for all $j$, then $X$ is a cofibrant pro-spectrum.
\end{itemize}

(2) There is a model structure on the category $\mathsf{Ind}(\spectra)$ in which
\begin{itemize}
  \item a morphism $\phi: X \to Y$ is a \emph{weak equivalence} if for each $n \in \mathbb{Z}$ it induces an isomorphism
  \[ \colim_{j \in \cat{J}} \pi_n(X_j) \to \colim_{k \in \cat{K}} \pi_n(Y_k) \]
  where $\pi_n(A) = [S^n,A]$ denotes the \ord{n} homotopy group of the spectrum $A$;
  \item a morphism of ind-spectra is a \emph{fibration} if it has a level representation $\phi: X \to Y$ such that each map $\phi_j: X_j \to Y_j$ is a fibration in $\spectra$.
  \item an ind-spectrum is cofibrant if and only if it is levelwise homotopy-finite and has the left lifting property with respect to levelwise trivial fibrations (i.e. maps which have a level representation in which each map is a trivial fibration in $\spectra$).
\end{itemize}

(3) There is a Quillen equivalence between $\mathsf{Pro}(\spectra)$ and the opposite of $\mathsf{Ind}(\spectra)$ in which both sides of the equivalence are given by applying the functor $\Map(-,S)$ levelwise.

(4) There is a Quillen equivalence between $\mathsf{Ind}(\spectra)$ and $\spectra$ in which the left adjoint is the colimit functor
\[ \mathsf{Ind}(\spectra) \to \spectra; \quad X \mapsto \colim_{j \in \cat{J}} X_j \]
and the right adjoint is the functor $\spectra \to \mathsf{Ind}(\spectra)$ that sends a spectrum $X$ to the ind-spectrum with value $X$ indexed on the trivial category with one object.
\end{theorem}

\begin{definition}[Spanier-Whitehead duals] \label{def:SW-dual}
A \emph{Spanier-Whitehead dual} of a pro-spectrum $X$ is an object $\dual X$ in $\spectra$ that corresponds to $X \in \mathsf{Pro}(\spectra)$ under the Quillen equivalences of Theorem \ref{thm:pro-spectra}. Note that this dual is only defined up to weak equivalence, although a natural (but not canonical) choice can be made by fixing cofibrant replacement functors in the categories of pro- and ind-spectra.
\end{definition}

\begin{remark}
We are concerned mainly with pro-spectra whose indexing categories are the opposites of the categories $\mathsf{Sub}(F)$ (see Definition \ref{def:finite-subcomplex}) where $F$ is a presented cell functor. We saw in Lemma \ref{lem:lattice} that such indexing categories have various useful properties. The following lemma helps us give a homotopy colimit form for the Spanier-Whitehead dual of a pro-spectrum indexed by such a category.
\end{remark}

\begin{lemma} \label{lem:colim-ind}
Let $\cat{J}$ be a filtered poset in which each object has finitely many predecessors. (In particular, the categories $\mathsf{Sub}(F)$ of Definition \ref{def:finite-subcomplex} have these properties.) Let $X$ be a levelwise homotopy-finite ind-spectrum indexed on $\cat{J}$. Then
\[ (\mathbb{L}\colim)(X) \homeq \hocolim X \]
where the left-hand side is the left derived functor of the colimit functor from ind-spectra to spectra, and the right-hand side is the homotopy colimit of $X$ as a $\cat{J}$-indexed diagram of spectra.
\end{lemma}
\begin{proof}
It is sufficient to show that a cofibrant replacement for $X$ in the projective model structure on $\cat{J}$-indexed diagrams is also a cofibrant ind-spectrum. Applying the colimit functor to such a replacement is then a model for both the derived colimit of ind-spectra, and the homotopy colimit of $\cat{J}$-indexed diagrams.

So suppose that $X: \cat{J} \to \spectra$ is cofibrant in the projective model structure on such diagrams. We need to show that such an $X$ is a cofibrant ind-spectrum in the model structure of Theorem \ref{thm:pro-spectra}. Since we are assuming $X$ is levelwise homotopy-finite, we need only show that it is strictly cofibrant, that is, it has the left lifting property with respect to essentially levelwise trivial fibrations. So take a diagram of ind-spectra of the following form
\[ \begin{diagram}
  \node[2]{A} \arrow{s,r,A}{\sim} \\
  \node{X} \arrow{e} \node{B}
\end{diagram} \]
in which, for convenience, we choose a level representation (indexed, say, by $\cat{K}$) for $A \to B$ that is a levelwise trivial fibration, i.e. each map $A_{k} \to B_{k}$ is a trivial fibration in $\spectra$. Note that the morphism $X \to B$ of ind-spectra determines a collection of compatible maps $X_{j} \to B_{g(j)}$ where $g(j)$ is some object of $\cat{K}$ for each $j \in \cat{J}$.

We now inductively construct a \emph{functor} $f: \cat{J} \to \cat{K}$. (The condition that $f$ be a functor means that we have maps $f(j') \to f(j)$ when $j' < j$, and that these maps are compatible with composition.) The starting point for the induction is to define $f(j)$ for those $j \in \cat{J}$ with no predecessors. For such a $j$, we pick any object in $\cat{K}$ to be $f(j)$.

Now suppose that we have already constructed $f$ on the restriction of $\cat{J}$ to the (finitely many) $j'$ that map to some given $j$. Since $\cat{K}$ is filtered, there is some $k \in \cat{K}$ that accepts a map from all the $f(j')$. Also since $\cat{K}$ is filtered, we can pick these maps such that they are all compatible with the already chosen maps $f(j') \to f(j'')$. We may also assume that $k$ accepts a map $g(j) \to k$ (where $g(j)$ comes from the morphism $X \to B$ as above). Now set $f(j) := k$. Proceeding inductively, again using the fact that any object in $\cat{J}$ has finitely many predecessors, we obtain a functor $f: \cat{J} \to \cat{K}$. This also has the property that there is a map $g(j) \to f(j)$ in $\cat{K}$ for any $j \in \cat{J}$.

We now pullback the map $A \to B$ of $\cat{K}$-indexed diagrams to a map $A' \to B'$ of $\cat{J}$-indexed diagrams by
\[ A'_{j} := A_{f(j)}, \quad B'_{j} := B_{f(j)}. \]
Since $f$ is a functor, we do obtain a natural transformation $A' \to B'$. The morphism $X \to B$ of ind-spectra now determines a natural transformation $X \to B'$ by the maps
\[ X_{j} \to B_{g(j)} \to B_{f(j)} = B'_{j}. \]
This is a natural transformation by the definition of a morphism of ind-spectra.

Now the natural transformation $A' \to B'$ is still an objectwise trivial fibration, and so, since $X$ is assumed to be a cofibrant diagram, there is a lift $X \to A'$. But we have therefore constructed maps
\[ X_{j} \to A_{f(j)} \]
which together form a morphism of ind-spectra $X \to A$, which in turn lifts the morphism $X \to B$. We have therefore checked that $X$ is cofibrant as an ind-spectrum which completes the proof.
\end{proof}

\begin{remark} \label{rem:cofinite-directed}
The condition on the indexing category $\cat{J}$ required in Lemma \ref{lem:colim-ind} is dual to that used by Isaksen in \cite{isaksen:2002} to calculate limits and colimits of pro-objects. He defines a pro-object to be `cofinite directed' if its indexing category $\cat{J}$ satisfies that dual condition: i.e. that $\cat{J}$ is a cofiltered poset in which every element has finitely many successors. He also shows that any pro-object is isomorphic to one of this form, and dually any ind-object is isomorphic to one which satisfies the conditions of Lemma \ref{lem:colim-ind}.
\end{remark}

\begin{lemma} \label{lem:dual-hocolim}
Let $X: \cat{J} \to \spectra$ be a pro-spectrum. Suppose that the indexing category $\cat{J}$ is a cofiltered poset in which every element has finitely many successors, and suppose that each spectrum $X_j$ is both cofibrant and homotopy-finite. Then the Spanier-Whitehead dual of $X$ is given by
\[ \dual(X) \homeq \hocolim_{j \in \cat{J}} \Map(X_j,S). \]
\end{lemma}
\begin{proof}
Since each $X_j$ is cofibrant, the pro-spectrum $X$ is cofibrant (by Theorem \ref{thm:pro-spectra}(1)). Therefore, the ind-spectrum corresponding to $X$ under the Quillen equivalence of \ref{thm:pro-spectra}(3) is given by
\[ j \mapsto \Map(X_j,S). \]
The spectrum $\Map(X_j,S)$ is homotopy-finite since $X_j$ is, and the filtered category $\cat{J}^{op}$ satisfies the condition of Lemma \ref{lem:colim-ind}. Applying \ref{lem:colim-ind} we see that $\dual(X)$ is given by the indicated colimit.
\end{proof}

\begin{definition}[Directly dualizable pro-spectra] \label{def:dualizable}
The pro-spectra we deal with in this paper almost exclusively satisfy the conditions of Lemma \ref{lem:dual-hocolim}. It is helpful to have some terminology for this case. We say that a pro-spectrum $X$ is \emph{directly-dualizable} if
\begin{itemize}
  \item the indexing category for $X$ is a cofiltered poset in which every element has finitely many successors;
  \item each spectrum $X_j$ is cofibrant and homotopy-finite.
\end{itemize}
In this case we fix a model for the Spanier-Whitehead dual $\dual X$ using Lemma \ref{lem:dual-hocolim}. We define
\[ \dual X := \hocolim_{j \in \cat{J}} \Map(X_j,S). \]
The homotopy colimit here is formed in the category $\spectra$. Later on in this paper, we consider the Spanier-Whitehead duals of collections of pro-spectra that have extra structure. In those cases, we form the homotopy colimit in other categories in order to retain that additional information. (See \S\ref{sec:prosymseq}.)
\end{definition}

We now check that the definition of $\dual X$ in Definition \ref{def:dualizable} is functorial.

\begin{definition}[Spanier-Whitehead dual of a map] \label{def:pro-morphism}
Let $f:X \to Y$ be a morphism of pro-spectra with $X$ and $Y$ directly-dualizable as in Definition \ref{def:dualizable}. Then $f$ induces a morphism of spectra
\[ \dual f : \dual Y \to \dual X \]
in the following way. Recall that $f$ consists of maps $X_{f(j)} \to Y_j$ where $f:\cat{J} \to \cat{I}$ is a map from the indexing category of $Y$ to the indexing category of $X$. We then obtain dual maps
\[ \Map(Y_j,S) \to \Map(X_{f(j)},S) \]
which together make up the required map
\[ \dual f: \hocolim_{j \in \cat{J}} \Map(Y_j,S) \to \hocolim_{i \in \cat{I}} \Map(X_i,S). \]
This construction makes $\dual$ into a contravariant functor from the full subcategory of $\mathsf{Pro}(\spectra)$ consisting of the directly-dualizable spectra, to $\spectra$.
\end{definition}

The construction of $\dual f$ gives us the following useful way to decide if a morphism of pro-objects is a weak equivalence.

\begin{lemma} \label{lem:pro-equivalence}
Let $f:X \to Y$ be a morphism between directly-dualizable pro-spectra. Then $f$ is a weak equivalence (in the sense of Theorem \ref{thm:pro-spectra}) if and only if $\dual(f)$ is a weak equivalence in $\spectra$.
\end{lemma}
\begin{proof}
The functor $\dual$ is the derived functor of one side of a Quillen equivalence. It therefore preserves and reflects weak equivalences.
\end{proof}

\part{Operads and modules}

One of the main goals of this paper is to describe new structures on the Goodwillie derivatives of various sorts of functors. All of these new structures arise from the theory of operads, so we now turn to this. The aims of this part of the paper are as follows:
\begin{itemize}
  \item to recall the definitions of operads and modules over operads (as well as cooperads and comodules), and to describe the bar construction for operads, which is a crucial part of producing these new structures (this occupies \S\ref{sec:bar});
  \item to examine the homotopical properties of the bar construction (\S\ref{sec:bar-homotopy});
  \item to describe the homotopy theory of operads and modules in the category of spectra, in particular in order to produce appropriate cofibrant replacements (\S\ref{sec:cofibrant});
  \item to construct a theory of pro-modules and pro-comodules over operads and cooperads respectively, and to understand how Spanier-Whitehead duality relates pro-comodules to modules (\S\ref{sec:prosymseq}).
\end{itemize}

\section{Composition products, operads and bar constructions} \label{sec:bar}

We first recall the definition of operads and modules over them, and of the bar construction in this context. We start with symmetric sequences.

\begin{definition}[Symmetric sequences]
Let $\mathsf{\Sigma}$ denote the category whose objects are the finite sets $\{1,\dots,n\}$ for $n \geq 1$ and whose morphisms are bijections. Thus $\mathsf{\Sigma}(m,n)$ is empty unless $m = n$, in which case it is the group $\Sigma_n$.

Let $\cat{C}$ be any category. A \emph{symmetric sequence in $\cat{C}$} is a functor $A: \mathsf{\Sigma} \to \cat{C}$. Explicitly then, we can think of a symmetric sequence as consisting of a sequence $\{A(n)\}$ of objects of $\cat{C}$ together with a (left) $\Sigma_n$-action on $A(n)$ for each $n \geq 1$. A \emph{morphism of symmetric sequences} $f:A \to B$ is a natural transformation of functors or equivalently, a collection of $\Sigma_n$-equivariant maps $f_n: A(n) \to B(n)$.

The objects $A(n)$ involved in a symmetric sequence are called the \emph{terms} of the symmetric sequence, and a symmetric sequence $A$ is said to have some property \emph{termwise} if each $A(n)$ has that property.
\end{definition}

\begin{example}
For a functor $F: \cat{C} \to \cat{D}$ where $\cat{C}$ and $\cat{D}$ are either $\sset$ or $\spectra$, the Goodwillie derivatives $\der^G_*F$ of $F$ form a symmetric sequence in $\spectra$.
\end{example}

\begin{definition}[Composition product of symmetric sequences]
Suppose now that $\cat{C}$ is a cocomplete closed symmetric monoidal category with monoidal product denoted $\smsh$ and unit object $S$. (We have the category $\spectra$ in mind for $\cat{C}$.) If $A$ and $B$ are two symmetric sequences in $\cat{C}$, then the \emph{composition product} of $A$ and $B$ is the symmetric sequence $A \circ B$ given by
\[ (A \circ B)(n) := \Wdge_{\text{partitions of $\{1,\dots,n\}$}} A(k) \smsh B(n_1) \smsh \dots \smsh B(n_k). \]
The coproduct here is taken over all unordered partitions of the set $\{1,\dots,n\}$ into nonempty subsets. For each such partition, we fix an order on the pieces (hence a bijection between the set of pieces and $\{1,\dots,k\}$ where $k$ is the number of pieces), and an order on each piece (hence a bijection between each piece and $\{1,\dots,n_i\}$ where $n_i$ is the number of elements in the \ord{i} piece).

The $\Sigma_n$-action on $(A \circ B)(n)$ is given as follows. Let $\sigma$ be a permutation of $\{1,\dots,n\}$. For each partition $\lambda$ of $\{1,\dots,n\}$, $\sigma$ determines a new partition $\sigma(\lambda)$ where the pieces of $\sigma(\lambda)$ are obtained by applying $\sigma$ to the elements of the pieces of $\lambda$. Note that both $\lambda$ and $\sigma(\lambda)$ have the same number of pieces (say, $k$) and that with respect to our chosen orderings of those pieces, $\sigma$ determines an element $\sigma_* \in \Sigma_k$ and hence a map $A(k) \to A(k)$. Furthermore, $\sigma$ determines a bijection between the \ord{i} piece of $\lambda$ and the \ord{\sigma_*(i)} piece of $\sigma(\lambda)$ and hence a map $B(n_i) \to B(n_{\sigma_*(i)})$. Putting all these maps together for all partitions $\lambda$ determines the required map $(A \circ B)(n) \to (A \circ B)(n)$.
\end{definition}

\begin{remark}
The terminology `composition product' comes about for the following reason. To a symmetric sequence $A$ in $\cat{C}$, one can associate a functor from $\cat{C}$ to $\cat{C}$ given by
\[ F_A(X) := \Wdge_{n} \left(A(n) \smsh X^{\smsh n}\right)_{\Sigma_n}. \]
The composition product of symmetric sequences then mirrors the composition of functors in the sense that there is a natural isomorphism
\[ F_A F_B \isom F_{A \circ B}. \]
In fact one can view the taking of Goodwillie derivatives (at least for functors of spectra) as a partial inverse to this process. In particular, for an (appropriately cofibrant) symmetric sequence $A$ in $\spectra$, we have a natural equivalence of symmetric sequences
\[ \der_*(F_A) \homeq A. \]
Combining these observations, we obtain examples of the chain rule for spectra that we prove in \S\ref{sec:chainrule}. For symmetric sequences $A,B$, we have
\[ \der_*(F_A F_B) \homeq \der_*(F_A) \circ \der_*(F_B). \]
\end{remark}

\begin{definition}[Unit symmetric sequence] \label{def:unit-symmetric-sequence}
Let $\cat{C}$ be a cocomplete closed symmetric monoidal category with terminal object $*$ and unit object $S$. The \emph{unit symmetric sequence} in $\cat{C}$ is the symmetric sequence $\mathsf{1}$ given by
\[ \mathsf{1}(n) := \begin{cases} S & \text{if $n = 1$}; \\ * & \text{otherwise}. \end{cases} \]
\end{definition}

\begin{proposition} \label{prop:comprod-monoidal}
Let $\cat{C}$ be a pointed closed symmetric monoidal category. Then the composition product forms a (non-symmetric) monoidal product on the category of symmetric sequences in $\cat{C}$ with unit object given by the unit symmetric sequence $\mathsf{1}$.
\end{proposition}
\begin{proof}
This is standard (see \cite[1.68]{markl/shnider/stasheff:2002}).
\end{proof}

\begin{remark}
Proposition \ref{prop:comprod-monoidal} relies heavily on the hypothesis that $\cat{C}$ be \emph{closed} symmetric monoidal because we need the monoidal structure to commute with coproducts. This is necessary in order that the composition product be associative. By introducing higher-order versions of the composition product, we can make partial sense of this proposition in the non-closed case. See \cite{ching:2005c} for details.
\end{remark}

\begin{definition}[Operads] \label{def:operad}
Let $\cat{C}$ be a pointed closed symmetric monoidal category. An \emph{operad} in $\cat{C}$ is a monoid for the composition product. In other words, an operad consists of a symmetric sequence $P$ together with a \emph{composition map}
\[ P \circ P \to P \]
and a \emph{unit map}
\[ \mathsf{1} \to P \]
satisfying standard associativity and unit axioms. A \emph{morphism of operads} is a map of symmetric sequences that commutes with the operad structures.
\end{definition}

\begin{remark} \label{rem:operad}
From the definition of the composition product, we see that this definition of operad is equivalent to the traditional one (e.g. see \cite{may:1972}), that is, as a symmetric sequence $P$ together with a collection of composition maps
\[ P(k) \smsh P(n_1) \smsh \dots \smsh P(n_k) \to P(n_1+\dots+n_k) \]
and a unit map
\[ S \to P(1) \]
satisfying various equivariance, associativity and unit axioms. This traditional definition has the advantage that it does not require the underlying symmetric monoidal category $\cat{C}$ to be closed.
\end{remark}

\begin{definition}[Modules over operads] \label{def:modules}
Let $P$ be an operad in a closed symmetric monoidal category $\cat{C}$. A \emph{right $P$-module} consists of a symmetric sequence $R$ together with a right $P$-action map
\[ R \circ P \to R \]
satisfying the usual associativity and unit axioms. A \emph{left $P$-module} consists of a symmetric sequence $L$ and a left $P$-action map
\[ P \circ L \to L \]
again satisfying the usual axioms. A \emph{$P$-bimodule} consists of a symmetric sequence $M$ together with commuting right and left $P$-actions.

A \emph{morphism of right $P$-modules} is a map of symmetric sequences that commutes with the module structure maps. Similarly, we have \emph{morphisms of left $P$-modules} and \emph{morphisms of $P$-bimodules}. These notions then give us categories of right $P$-modules (denoted $\mathsf{Mod}_{\mathsf{right}}(P)$), left $P$-modules (denoted $\mathsf{Mod}_{\mathsf{left}}(P)$) and $P$-bimodules (denoted $\mathsf{Mod}_{\mathsf{bi}}(P)$).
\end{definition}

\begin{remark} \label{rem:algebras}
Left modules over an operad $P$ are related to $P$-algebras in the following way. If we allowed our symmetric sequences to include a \ord{0} term (i.e. if we included the empty set as an object in the category $\mathsf{\Sigma}$) and extended the definition of composition product in the usual way, then a $P$-algebra would be equivalent to a left $P$-module concentrated in the \ord{0} term (and equal to the terminal object of $\cat{C}$ in all other positions).
\end{remark}

\begin{definition}[Cooperads and comodules] \label{def:cooperads}
Let $\cat{C}$ be a symmetric monoidal category. Then the opposite category $\cat{C}^{op}$ has a natural symmetric monoidal structure given by that of $\cat{C}$. Also a symmetric sequence in $\cat{C}$ can be identified with a symmetric sequence in $\cat{C}^{op}$ via the isomorphism of categories $\mathsf{\Sigma} \isom \mathsf{\Sigma}^{op}$ that sends a bijection to its inverse.

We then define a \emph{cooperad} in $\cat{C}$ to be an operad in the symmetric monoidal category $\cat{C}^{op}$, but viewed as a symmetric sequence in $\cat{C}$ rather than $\cat{C}^{op}$. Thus a cooperad in $\cat{C}$ is a symmetric sequence together with structure maps of the form
\[ Q(n_1+\dots+n_k) \to Q(k) \smsh Q(n_1) \smsh \dots \smsh Q(n_k). \]
The opposite of a closed symmetric monoidal category is very rarely closed so we are relying on the traditional definition (Remark \ref{rem:operad}) to say what an operad in $\cat{C}^{op}$ is. If $Q$ is a cooperad in $\cat{C}$, we write $Q^{op}$ for the corresponding operad in $\cat{C}^{op}$.

It is useful to have notation for the dual of the composition product. We write $M \varcomp N$ for the composition product of the symmetric sequences $M$ and $N$ viewed as symmetric sequences in $\cat{C}^{op}$. The object $(M \varcomp N)(n)$ is then just the product (rather than the coproduct) of the same terms used to define $(M \circ N)(n)$. With this notation, a cooperad consists of a symmetric sequence $Q$ and a map $Q \to Q \varcomp Q$ satisfying certain axioms dual to those for operads. Note that the operation $\varcomp$ is unlikely to be associative, again since $\cat{C}^{op}$ is usually not \emph{closed} symmetric monoidal.

If $Q$ is a cooperad in $\cat{C}$, then a \emph{right $Q$-comodule} is a right $Q^{op}$-module considered as a symmetric sequence in $\cat{C}$, i.e. it consists of a symmetric sequence $R$ and a right $Q$-coaction $R \to R \varcomp Q$. A \emph{left $Q$-comodule} is a left $Q^{op}$-module, so consists of a symmetric sequence $L$ and a left $Q$-coaction $L \to Q \varcomp L$. A \emph{$Q$-bicomodule} is a $Q^{op}$-bimodule. As with operads and modules there are obvious notions of morphisms of cooperads and comodules and we obtain corresponding categories.
\end{definition}

\begin{definition}[Reduced operads] \label{def:reduced}
An operad $P$ is said to be \emph{reduced} if the unit map $S \to P(1)$ is an isomorphism. Note in particular that this means $P$ has a unique \emph{augmentation} $P(1) \to S$ given by the inverse of the unit map. If $P$ is reduced then the unit symmetric sequence $\mathsf{1}$ has both a right and left $P$-module structure. We denote by $\mathsf{Op}(\cat{C})$ the category consisting of the reduced operads in $\cat{C}$ and the morphisms of operads between them.
\end{definition}

We now recall the bar construction for operads. From here on, we take $\cat{C}$ to be the category $\spectra$ of spectra, although many of the remaining results of this section apply equally well in a more general setting. We leave the reader to extend to the general case as necessary.

\begin{definition}[Simplicial bar constructions] \label{def:bar}
Let $P$ be an operad in $\spectra$ with right module $R$ and left module $L$. Then the \emph{simplicial bar construction} on $P$ with coefficients in $R$ and $L$ is a simplicial object $B_{\bullet}(R,P,L)$ in the category of symmetric sequences of spectra. The $k$-simplices are given by
\[ B_k(R,P,L) := R \circ \underbrace{P \circ \dots \circ P}_k \circ L. \]
The face and degeneracy maps are given as follows:
\begin{itemize}
  \item $d_0: B_k(R,P,L) \to B_{k-1}(R,P,L)$ by the action map $R \circ P \to R$;
  \item $d_i: B_k(R,P,L) \to B_{k-1}(R,P,L)$ for $i = 1,\dots,k-1$ by the operad composition map $P \circ P \to P$ applied to the \ord{k} and \ord{k+1} factors of $P$;
  \item $d_k: B_k(R,P,L) \to B_{k-1}(R,P,L)$ by the action map $P \circ L \to L$;
  \item $s_j: B_k(R,P,L) \to B_{k+1}(R,P,L)$ for $j = 0,\dots,k$ by using the unit map $\mathsf{1} \to P$ to insert the \ord{j+1} copy of $P$.
\end{itemize}
This is a standard two-sided simplicial bar construction.

We now take the (termwise) geometric realization of this simplicial object to obtain what we call just the \emph{bar construction} on $P$ with coefficients in $R$ and $L$:
\[ B(R,P,L)(n) := |B_{\bullet}(R,P,L)(n)|. \]
This bar construction is our model for the derivative composition product $R \circ_{P} L$ of the right and left modules $R$ and $L$, \emph{over} $P$.
\end{definition}

The main result of \cite{ching:2005a} is the following:

\begin{proposition} \label{prop:bar}
Let $P$ be a \emph{reduced} operad in $\spectra$ with right module $R$ and left module $L$. Then there are natural maps
\[ \phi_{R,L}: B(R,P,L) \to B(R,P,\mathsf{1}) \varcomp B(\mathsf{1},P,L) \]
that are associative in the sense that the following diagram commutes
\[ \begin{diagram}
  \node{B(R,P,L)} \arrow{e,t}{\phi_{R,L}} \arrow{s,l}{\phi_{R,L}} \node{B(R,P,\mathsf{1}) \varcomp B(\mathsf{1},P,L)} \arrow{s,r}{\phi_{R,\mathsf{1}}} \\
  \node{B(R,P,\mathsf{1}) \varcomp B(\mathsf{1},P,L)} \arrow{e,t}{\phi_{\mathsf{1},L}} \node{B(R,P,\mathsf{1}) \varcomp B(\mathsf{1},P,\mathsf{1}) \varcomp B(\mathsf{1},P,L)}
\end{diagram} \]
(Strictly speaking, the bottom-right corner of this diagram does not make sense because the dual composition product $\varcomp$ is not associative. However, there is a natural way of forming iterated versions of $\varcomp$ by taking one large product. See \cite[Remark 2.20]{ching:2005a} or \cite{ching:2005c} for more details.)
\end{proposition}
\begin{proof}
This is in \cite[\S7.3]{ching:2005a}.
\end{proof}

\begin{corollary} \label{cor:bar}
Let $P$ be a reduced operad in $\spectra$ with right module $R$ and left module $L$. Then:
\begin{itemize}
  \item the \emph{reduced bar construction} $B(P) := B(\mathsf{1},P,\mathsf{1})$ forms a reduced cooperad in $\spectra$;
  \item the \emph{one-sided bar construction} $B(R,P,\mathsf{1})$ forms a right comodule over the cooperad $BP$;
  \item the \emph{one-sided bar construction} $B(\mathsf{1},P,L)$ forms a left comodule over the cooperad $BP$.
\end{itemize}
\end{corollary}
\begin{proof}
See \cite[Prop. 7.26]{ching:2005a}.
\end{proof}

\begin{definition}[Bisimplicial bar constructions] \label{def:bar-bi}
Let $P$ be an operad in $\spectra$ and let $M$ be a $P$-bimodule, $R$ a right $P$-module and $L$ a left $P$-module. Then we define the \emph{bisimplicial bar construction} on $M$ to be the bisimplicial object
\[ B_{\bullet,\bullet}(R,P,M,P,L) := R \circ P^{\bullet} \circ M \circ P^{\bullet} \circ L. \]
with face and degeneracy maps similar to those in the bar constructions above. The \emph{bimodule bar construction} on $M$ is then the realization
\[ B(R,P,M,P,L) := |B_{\bullet,\bullet}(R,P,M,P,L)|. \]
\end{definition}

\begin{remark}
Dual to the bar constructions considered above, we have cobar constructions for cooperads and their comodules. We do not make explicit use of these. Instead we use Spanier-Whitehead duality to write make most of our constructions in terms of operads and modules.
\end{remark}

We conclude this section by noting that if the right and left $P$-modules $R$ and $L$ involved in the bar construction $B(R,P,L)$ (or $B(R,P,M,P,L)$) are themselves bimodules, then this bar construction retains some of that additional structure. We start with the following construction.

\begin{prop} \label{prop:bar-comprod}
Let $P$ be a reduced operad in $\spectra$ and let $R$ and $L$ be right and left $P$-modules respectively. Let $A$ be any symmetric sequence. Then there are isomorphisms of symmetric sequences:
\[ \chi_r: A \circ B(R,P,L) \isom B(A \circ R,P,L) \]
and
\[ \chi_l: B(R,P,L) \circ A \isom B(R,P,L \circ A) \]
where, in the targets of these maps, we give $A \circ R$ the structure of a right $P$-module via
\[ (A \circ R) \circ P \isom A \circ (R \circ P) \to A \circ R \]
using the right $P$-module structure on $R$, and we give $L \circ A$ the structure of a left $P$-module similarly.
\end{prop}
\begin{proof}
To define $\chi_r$, we notice that
\[ [A \circ B(R,P,L)](n) := \Wdge_{\text{partitions of $\{1,\dots,n\}$}} A(k) \smsh B(R,P,L)(n_1) \smsh \dots \smsh B(R,P,L)(n_k). \]
Each term here is defined to be
\[ A(k) \smsh |B_{\bullet}(R,P,L)(n_1)| \smsh \dots \smsh |B_{\bullet}(R,P,L)(n_k)| \]
which by \cite[X.1.4]{elmendorf/kriz/mandell/may:1997} is isomorphic to
\[ |A(k) \smsh B_{\bullet}(R,P,L)(n_1) \smsh \dots \smsh B_{\bullet}(R,P,L)(n_k)| \]
which by definition is the same as
\[ |A(k) \smsh (R \circ P^\bullet \circ L)(n_1) \smsh \dots \smsh (R \circ P^\bullet \circ L)(n_k)| \]
Taking the coproduct over all $n_1+\dots+n_k = n$, (and because coproducts commute with realization), we get
\[ |(A \circ R \circ P^{\bullet} \circ L)(n)| \]
which is the definition of
\[ B(A \circ R,P,L)(n). \]
This sequence defines the isomorphism $\chi_r$, and $\chi_l$ is given similarly.
\end{proof}

\begin{definition}[Module structures on bar construction] \label{def:module-bar}
Let $P$ be a reduced operad in $\spectra$ and let $R$ and $L$ be right and left $P$-modules respectively. Suppose that the right module structure on $R$ is part of a $P$-bimodule structure. We then define a left $P$-module structure on $B(R,P,L)$ by
\[ P \circ B(R,P,L) \isom B(P \circ R,P,L) \to B(R,P,L). \]
The first map is the isomorphism $\chi_r$ of Proposition \ref{prop:bar-comprod}, and the second comes from the left module structure on the bimodule $R$.

Similarly, if the left module structure on $L$ is part of a $P$-bimodule structure then we define a right $P$-module structure on $B(R,P,L)$ by
\[ B(R,P,L) \circ P \isom B(R,P,L \circ P) \to B(R,P,L). \]
If $R$ and $L$ are both $P$-bimodules, then these constructions together give $B(R,P,L)$ a $P$-bimodule structure as well.
\end{definition}

\begin{remark} \label{rem:bibar}
We can think of the bimodule bar construction $B(R,P,M,P,L)$ as given by first forming $B(R,P,M)$ (i.e. taking the bar construction for the left module structure on $M$) and then forming $B(B(R,P,M),P,L)$ (i.e. taking the bar construction for the right module structure on $B(R,P,M)$ that comes via Definition \ref{def:module-bar} from the right module structure on $M$). Alternatively, we can do these constructions in the other order. In any case, we get
\[ B(R,P,M,P,L) \isom B(B(R,P,M),P,L) \isom B(R,P,B(M,P,L)). \]
These isomorphisms come from doing horizontal then vertical, or vertical then horizontal, realizations of the bisimplicial bar construction, instead of the diagonal realization. It also follows that if either $R$ or $L$ is a $P$-bimodule, then the bar construction $B(R,P,M,P,L)$ is a left or right $P$-module respectively.
\end{remark}

\begin{cor} \label{cor:bar-bi}
Let $P$ be a reduced operad in $\spectra$ and let $M$ be a $P$-bimodule. Then the bimodule bar construction
\[ B(\mathsf{1},P,M,P,\mathsf{1}) \]
forms a bicomodule over the cooperad $B(P)$.
\end{cor}
\begin{proof}
From Remark \ref{rem:bibar}, we can think of $B(\mathsf{1},P,M,P,\mathsf{1})$ as
\[ B(B(\mathsf{1},P,M),P,\mathsf{1}) \]
from which it follows by Corollary \ref{cor:bar} that this has a right $BP$-comodule structure. Alternatively, we can think of $B(\mathsf{1},P,M,P,\mathsf{1})$ as
\[ B(\mathsf{1},P,B(M,P,\mathsf{1})) \]
from which it follows that this has a left $BP$-comodule structure. These comodule structures commute and so we have a $BP$-bicomodule.
\end{proof}

\section{Homotopy invariance of the bar construction} \label{sec:bar-homotopy}

We now want to address the homotopical properties of the bar construction. Specifically, given weak equivalences between operads and modules, when do they induce weak equivalences between the corresponding bar constructions. For us, weak equivalences of operads and modules are always detected termwise.

\begin{definition}[Weak equivalences] \label{def:weq-operads}
Let $f:A \to B$ be a morphism of symmetric sequences of spectra. We say that $f$ is a \emph{weak equivalence} if the map
\[ f_n: A(n) \to B(n) \]
is a weak equivalence in the category $\spectra$ for all $n \geq 1$. A morphism of operads, modules, cooperads or comodules is said to be a \emph{weak equivalence} if the underlying map of symmetric sequences is a weak equivalence.
\end{definition}

In the most general case, we want to consider the effect on the bar construction of changing both the operad and the modules involved. Therefore we make the following definition.

\begin{definition}[Morphisms of modules]
Let $f:P \to P'$ be a morphism of operads of spectra. Let $R$ be a right $P$-module, and $R'$ a right $P'$-module. If $r:R \to R'$ is a morphism of symmetric sequences, then we say \emph{$r$ respects the module structures on $R$ and $R'$ via $f$} if the following diagram commutes:
\[ \begin{diagram}
  \node{R \circ P} \arrow{e} \arrow{s,l}{r \circ f} \node{R} \arrow{s,l}{r} \\
  \node{R' \circ P'} \arrow{e} \node{R'}
\end{diagram} \]
where the top and bottom maps are the module structures on $R$ and $R'$ respectively. If $l:L \to L'$ is a morphism of symmetric sequences from a left $P$-module to a left $P'$-module, then there is a corresponding definition of when \emph{$l$ respects the module structures on $L$ and $L'$ via $f$}, and similarly for bimodules.
\end{definition}

\begin{definition}[Induced maps on bar constructions] \label{def:bar-induced}
Now suppose that $f:P \to P'$ is a morphism of reduced operads, $r:R \to R'$ a morphism that respects right module structures on $R$ and $R'$ via $f$, and $l:L \to L'$ a morphism that respects left module structures on $L$ and $L'$ via $f$. Then the triple $(r,f,l)$ induces a morphism of symmetric sequences
\[ (r,f,l)_*: B(R,P,L) \to B(R',P',L') \]
via the induced maps $r \circ f^k \circ l: R \circ P^k \circ L$. In particular, $f$ induces a morphism
\[ f_*: B(P) \to B(P'). \]
If $P' = P$ and $f$ is the identity map, then $r:R \to R'$ is just a morphism of right $P$-modules, and induces a map
\[ r_*: B(R,P,\mathsf{1}) \to B(R',P,\mathsf{1}). \]
Similarly, $l:L \to L'$ is a morphism of left $P$-modules and induces a map
\[ l_*: B(\mathsf{1},P,L) \to B(\mathsf{1},P,L'). \]
\end{definition}

Now suppose that the maps $r,f,l$ in Definition \ref{def:bar-induced} are weak equivalences. It is not always true that the induced map $(r,f,l)_*$ is a weak equivalence. For example, $(r,f,l)_1$ is the map
\[ r_1 \smsh l_1 : R(1) \smsh L(1) \to R'(1) \smsh L'(1). \]
This is in general not a weak equivalence unless all these objects are cofibrant. In general we need some cofibrancy hypotheses in order that $(r,f,l)_*$ be a weak equivalence of symmetric sequences.

\begin{definition}[Termwise-cofibrant operads] \label{def:termwise-cofibrant}
Let $M$ be a symmetric sequence in $\spectra$. We say that $M$ is \emph{termwise-cofibrant} if:
\begin{enumerate}
  \item $M(n)$ is a cofibrant spectrum for $n \geq 2$; and
  \item either $M(1)$ is a cofibrant spectrum, or $M(1) \isom S$.
\end{enumerate}
Recall that the sphere spectrum $S$ is not cofibrant, so the alternatives in the second condition here are meaningful. It is important for us to include the case $M(1) \isom S$ to allow for reduced operads, and for the unit symmetric sequence $\mathsf{1}$.
\end{definition}

We can now state our main result on the homotopy invariance of the bar construction.

\begin{prop} \label{prop:bar-invariance}
Let $f:P \to P'$, $r:R \to R'$ and $l: L \to L'$ be as in Definition \ref{def:bar-induced}. Suppose that $f$, $r$ and $l$ are weak equivalences, and that the symmetric sequences $P,P',R,R',L,L'$ are all termwise-cofibrant. Then $B(R,P,L)$ and $B(R',P',L')$ are termwise-cofibrant, and the induced map
\[ \phi = (r,f,l)_*: B(R,P,L) \to B(R',P',L') \]
is a weak equivalence.
\end{prop}
\begin{proof}
We start by noting that, in a monoidal model category, a smash product of weak equivalences \emph{is} a weak equivalence if all the objects involved are cofibrant. This follows from the pushout-product axiom. Even if the unit object $S$ is not cofibrant, this claim is still true when the objects involved are either cofibrant or equal to $S$. This follows from the condition in a monoidal model structure that the map
\[ S_c \smsh X \to S \smsh X \isom X \]
is a weak equivalence (where $S_c$ is a cofibrant replacement for $S$).

Now note that the map $\phi_1: B(R,P,L)(1) \to B(R',P',L')(1)$ is isomorphic to
\[ r_1 \smsh l_1 : R(1) \smsh L(1) \to R'(1) \smsh L'(1).\]
The spectra involved here are either cofibrant, or isomorphic to $S$, and this map is a smash product of weak equivalences, so is itself a weak equivalence. Each of the terms $R(1) \smsh L(1)$ and $R'(1) \smsh L'(1)$ is either cofibrant or isomorphic to $S$, so the symmetric sequences $B(R,P,L)$ and $B(R',P',L')$ satisfy condition (2) of Definition \ref{def:termwise-cofibrant}.

Now consider $\phi_n: B(R,P,L)(n) \to B(R',P',L')(n)$ for some $n \geq 2$. This is the map on geometric realizations induced by
\[ \phi_{n,r}: \left[R \circ \overbrace{P \circ \dots \circ P}^r \circ L \right](n) \to \left[R' \circ \overbrace{P' \circ \dots \circ P'}^r \circ L' \right](n). \]
This is a coproduct of maps of the form
\[ R(i) \smsh \dots \smsh P(j) \smsh \dots \smsh L(k) \to R'(i) \smsh \dots \smsh P'(j) \smsh \dots \smsh L'(k) \]
which in turn is a smash product of weak equivalences in which all the objects involved are either cofibrant or isomorphic to $S$. So again it is itself a weak equivalence. Moreover, not all the indices $i,j,k,\dots$ can be equal to $1$, so this is a weak equivalence between cofibrant objects. The map $\phi_{n,r}$ is therefore a coproduct of weak equivalences between cofibrant spectra, so it too is a weak equivalence.

We have therefore shown that map $\phi_n$ is the realization of a levelwise weak equivalence of simplicial spectra. By Proposition \ref{prop:reedy}, it is now sufficient to show that each of these simplicial spectra is Reedy cofibrant (see \cite[15.3]{hirschhorn:2003}). It then follows that $\phi_n$ itself is a weak equivalence between cofibrant spectra, which completes the proof of the proposition.

To show that the simplicial bar construction $B_{\bullet}(R,P,L)(n)$ is Reedy cofibrant, we have to examine the \emph{latching maps}
\[ \lambda_t: \colim_{t \epi s} B_s(R,P,L)(n) \to B_t(R,P,L)(n) \]
where this colimit is taken over all surjections $t \epi s$ in the simplicial indexing category $\mathsf{\Delta}$ with $s < t$. We need to show that each $\lambda_r$ is a cofibration of spectra.

To see this, first note that surjections in the simplicial indexing category $\mathsf{\Delta}$ correspond to degeneracies in the simplicial object $B_{\bullet}(R,P,L)(n)$ which in turn come from the unit map $S \to P(1)$ of the operad $P$. Since $P$ is reduced, this unit map is an isomorphism and so the colimits in question take a particularly simple form.

Recall that $B_t(R,P,L)(n)$ is a coproduct of terms of the form
\[ R(i) \smsh (P(j_{1,1}) \smsh \dots \smsh P(j_{1,i})) \smsh \dots \smsh (P(j_{t,1}) \smsh \dots \smsh P(j_{t,m})) \smsh (L(k_1) \smsh \dots \smsh L(k_j)). \]
We have written it out like this to show that there are effectively $t$ copies of $P$ (coming from the composition product $R \circ P \circ \dots \circ P \circ L$), each contributing to one section of this smash product.

The colimit involved in the latching map can be described as the set of `degenerate' terms in this coproduct, or more precisely, those terms in which one of the copies of $P$ contributes only via $P(1)$ (that is, there is some $u$ such that $j_{u,v} = 1$ for all $v$).

The latching map $\lambda_t$ is then isomorphic to the inclusion of the coproduct of these degenerate terms into the full coproduct defining $B_t(R,P,L)(n)$. This map is a cofibration of spectra if each of the `nondegenerate' terms is cofibrant. Each such term is a smash product of cofibrant objects (and possibly some copies of $S$) so is indeed cofibrant.

Thus the latching maps are cofibrations, and $B(R,P,L)(n)$ is a Reedy cofibrant simplicial spectrum. Therefore $\phi_n$ is indeed a weak equivalence. This completes the proof that $\phi$ is a weak equivalence of symmetric sequences.
\end{proof}

\begin{remark}
Proposition \ref{prop:bar-invariance} is not special to operads in $\spectra$, but applies in any closed symmetric monoidal model category in which the relevant bar constructions can be formed.
\end{remark}

\begin{remark}
A version of Proposition \ref{prop:bar-invariance} holds for operads $P$ and $P'$ that are not reduced. The correct generalization of the cofibrant condition is to insist that the unit maps $S \to P(1)$ and $S \to P'(1)$ be cofibrations. The colimits involved in the latching maps then take on a more complicated form, but the latching maps can still be shown to be cofibrations.
\end{remark}

As a special case of Proposition \ref{prop:bar-invariance}, we get homotopy-invariance statements for the reduced, one-sided and bimodule bar constructions:
\begin{corollary}
With $f:P \to P'$, $r:R \to R'$ and $l: L \to L'$ as in Proposition \ref{prop:bar-invariance} and $m: M \to M'$ a morphism that respects $P$ and $P'$-bimodule structures on $M$ and $M'$ via $f$, each of the following maps is a weak equivalence of symmetric sequences (and hence of cooperads, or comodules as appropriate):
\begin{itemize}
  \item $f_*: B(P) \to B(P')$;
  \item $r_*: B(R,P,\mathsf{1}) \to B(R',P',\mathsf{1})$;
  \item $l_*: B(\mathsf{1},P,L) \to B(\mathsf{1},P',L')$;
  \item $m_*: B(\mathsf{1},P,M,P,\mathsf{1}) \to B(\mathsf{1},P,M,P,\mathsf{1})$.\qed
\end{itemize}
\end{corollary}

\section{Cofibrant replacements and model structures for operads and modules} \label{sec:cofibrant}

For the main part of this paper, we need homotopically-invariant versions of the various bar construction on an operad $P$. Proposition \ref{prop:bar-invariance} tells us that to do this, we first need to find termwise-cofibrant replacements for the operad $P$, and for the $P$-modules involved, and then take the bar construction. In order for this to be possible, we need to show that such termwise-cofibrant replacements actually exist.

We obtain the necessary termwise-cofibrant replacements by constructing projective model structures for our categories of operads and modules. Cofibrant replacements in these model structures then turn out to be termwise-cofibrant. The first part of this section concerns the existence of these model structures. For most of this we follow methods of EKMM \cite{elmendorf/kriz/mandell/may:1997} and the details are left to the Appendix.

\begin{remark} \label{rem:model}
Model structures on categories of operads have been extensively studied. Rezk \cite{rezk:1996} described model categories of operads of simplicial sets. Hinich \cite{hinich:1997} studied operads for chain complexes. Then Berger and Moerdijk \cite{berger/moerdijk:2003} proved a general result establishing the existence of projective model structures on operads in various contexts including topological spaces. Their examples do not include any models for stable homotopy theory, but Kro \cite{kro:2007} applied their methods to the category of orthogonal spectra of \cite{mandell/may/schwede/shipley:2001} with the positive stable model structure. Kro thus established the existence of a model structure on operads in this category. Spitzweck \cite{spitzweck:2001} analyzed the general case of operads in a cofibrantly generated monoidal model category and demonstrated the existence of a `$J$-semi model structure' (a notion slightly weaker than a model structure) on these.

These and other authors have studied model structures on categories of algebras and modules over operads. Again Rezk \cite{rezk:1996} gave the initial account of these in the context of simplicial sets and Berger-Moerdijk's work extended to categories of algebras. Schwede and Shipley \cite{schwede/shipley:2000} described general conditions for finding model structures on associative and commutative monoids (i.e. algebras over the associative and commutative operads). Most recently, Harper \cite{harper:2008} constructed model structures for the categories of algebras and left modules over an operad in symmetric spectra. Note that for some of these authors, operads are allowed to include zero terms. Essentially we are looking at the case where the zero terms are trivial which makes the homotopy theory simpler.

Model structures for the categories of associative and commutative algebras in $\spectra$ were studied in detail by EKMM \cite[VII]{elmendorf/kriz/mandell/may:1997} and extended to algebras over other operads in $\spectra$ by Basterra-Mandell \cite[8.6]{basterra/mandell:2005}. There is little new in our work here. We just verify that their approach applies to our case.
\end{remark}

In our model structures, fibrations (as well as weak equivalences) are always detected termwise.

\begin{definition}[Fibrations] \label{def:fib-operads}
Let $f:M \to N$ be a morphism of symmetric sequences of spectra. We say that $f$ is a \emph{fibration} if each map
\[ f_n: M(n) \to N(n) \]
is a fibration in $\spectra$. If $f$ is a morphism of operads, modules, cooperads or comodules, we say that $f$ is a \emph{fibration} if it is a fibration of the underlying symmetric sequences.
\end{definition}

To describe the generating cofibrations in our model categories, we define the free objects in each of these cases.

\begin{definition}[Free symmetric sequences] \label{def:free-symseq}
Let $X$ be a spectrum and fix an integer $n \geq 2$. The \emph{free symmetric sequence on $X$ in position $n$} is the symmetric sequence $A_n(X)$ given by
\[ A_n(X)(r) := \begin{cases} (\Sigma_n)_+ \smsh X & \text{if $r = n$}; \\ * & \text{otherwise}. \end{cases} \]
The functor $A_n$ from spectra to symmetric sequences is left adjoint to the functor that picks out the \ord{n} term of a symmetric sequence (and forgets the $\Sigma_n$-action).
\end{definition}

\begin{definition}[Free operads] \label{def:free-operads}
Say that a symmetric sequence $A$ is \emph{reduced} if $A(1) = *$ and write $\spectra^{\mathsf{\Sigma}}_{\mathsf{red}}$ for the full subcategory of $\spectra^{\mathsf{\Sigma}}$ consisting of the reduced symmetric sequences.

Now recall the definition of the free operad on $A$ using trees. For $n \geq 2$, let $\mathsf{T}_n$ be the set of (isomorphism classes of) rooted trees (where each internal vertex has at least two incoming edges) with leaves labelled $\{1,\dots,n\}$. For $T \in \mathsf{T}_n$, we set
\[ A(T) := \Smsh_{v \in T} A(i(v)) \]
where the smash product is taken over all internal vertices of $T$ and $i(v)$ is the number of incoming edges to the vertex $v$. The \emph{free reduced operad} on the reduced symmetric sequence $A$  is the operad $F(A)$ given by
\[ F(A)(n) := \begin{cases} \Wdge_{T \in \mathsf{T}_n} A(T) & \text{if $n > 1$}; \\ S & \text{if $n = 1$}. \end{cases} \]
with operad composition given by grafting trees. This construction defines a functor $F$ from reduced symmetric sequences (i.e. those concentrated in terms $2$ and above) to reduced operads. The functor $F$ is left adjoint to the forgetful functor. (See \cite[II.1.9]{markl/shnider/stasheff:2002} for more details on the free operad construction.)
\end{definition}

\begin{definition}[Free $P$-modules] \label{def:free-modules}
Next consider a fixed reduced operad $P$ in $\spectra$ and let $A$ be any symmetric sequence of spectra. The \emph{free right $P$-module on $A$} is the symmetric sequence
\[ R(A) := A \circ P \]
with right $P$-module structure given by
\[ (A \circ P) \circ P \isom A \circ (P \circ P) \to A \circ P. \]
Similarly, the \emph{free left $P$-module on $A$} is the symmetric sequence
\[ L(A) := P \circ A \]
with left $P$-module structure given by
\[ P \circ (P \circ A) \isom (P \circ P) \circ A \to P \circ A. \]
Finally, the \emph{free $P$-bimodule on $A$} is the symmetric sequence
\[ M(A) := P \circ A \circ P \]
with $P$-bimodule structure given similarly. Each of these constructions gives a functor from symmetric sequences to modules that is left adjoint to the appropriate forgetful functor.
\end{definition}

\begin{definition}[Free operads and modules on a spectrum] \label{def:more-free-operads}
Now let $X$ be a spectrum again. We define the \emph{free reduced operad on $X$ in position $n$} by:
\[ F_n(X) := F(A_n(X)) \]
for $n \geq 2$, the \emph{free right $P$-module on $X$ in position $n$} by:
\[ R_n(X) := R(A_n(X)), \]
the \emph{free left $P$-module on $X$ in position $n$} by:
\[ L_n(X) := L(A_n(X)) \]
and the \emph{free $P$-bimodule on $X$ in position $n$} by:
\[ M_n(X) := M(A_n(X)). \]
These constructions give functors from $\spectra$ to our categories of reduced operads and modules that are left adjoint to the functors that pick out the \ord{n} term (and forget the $\Sigma_n$-action).
\end{definition}

Next, we describe the generating cofibrations in each of our model categories of symmetric sequences, operads or modules.

\begin{definition}[Generating cofibrations for operads and modules] \label{def:gen-cofs}
Write $\mathbb{I}$ for the set of generating cofibrations in $\spectra$ (see Definition \ref{def:spectra}). Then we define the following sets of morphisms:
\begin{itemize}
  \item $\mathbb{I}_{\spectra^{\mathsf{\Sigma}}} := \left\{ A_n(I_0) \to A_n(I_1) \; | \; I_0 \to I_1 \in \mathbb{I}, \; n \geq 1 \right\}$;
  \item $\mathbb{I}_{\spectra^{\mathsf{\Sigma}}_{\mathsf{red}}} := \left\{ A_n(I_0) \to A_n(I_1) \; | \; I_0 \to I_1 \in \mathbb{I}, \; n \geq 2 \right\}$;
  \item $\mathbb{I}_{\mathsf{Op}(\spectra)} := \left\{ F_n(I_0) \to F_n(I_1) \; | \; I_0 \to I_1 \in \mathbb{I}, \; n \geq 2 \right\}$;
  \item $\mathbb{I}_{\mathsf{Mod}_{\mathsf{right}}(P)} := \left\{ R_n(I_0) \to R_n(I_1) \; | \; I_0 \to I_1 \in \mathbb{I}, \; n \geq 1 \right\}$;
  \item $\mathbb{I}_{\mathsf{Mod}_{\mathsf{left}}(P)} := \left\{ L_n(I_0) \to L_n(I_1) \; | \; I_0 \to I_1 \in \mathbb{I}, \; n \geq 1 \right\}$;
  \item $\mathbb{I}_{\mathsf{Mod}_{\mathsf{bi}}(P)} := \left\{ M_n(I_0) \to M_n(I_1) \; | \; I_0 \to I_1 \in \mathbb{I}, \; n \geq 1 \right\}$.
\end{itemize}
Similarly, if $\mathbb{J}$ is the set of generating trivial cofibrations in $\spectra$, we define corresponding sets $\mathbb{J}_{\cat{C}}$ of morphisms in each of the categories $\cat{C}$.
\end{definition}

\begin{theorem} \label{thm:projective-model}
Let $\cat{C}$ be one of the following categories:
\begin{itemize}
  \item $\spectra^{\mathsf{\Sigma}}$: the category of symmetric sequences in $\spectra$;
  \item $\spectra^{\mathsf{\Sigma}}_{\mathsf{red}}$: the category of reduced symmetric sequences in $\spectra$;
  \item $\mathsf{Op}(\spectra)$: the category of reduced operads in $\spectra$;
  \item $\mathsf{Mod}_{\mathsf{right}}(P)$: the category of right modules over a fixed reduced operad $P$ in $\spectra$;
  \item $\mathsf{Mod}_{\mathsf{left}}(P)$: the category of left modules over a fixed reduced operad $P$ in $\spectra$;
  \item $\mathsf{Mod}_{\mathsf{bi}}(P)$: the category of bimodules over a fixed reduced operad $P$ in $\spectra$.
\end{itemize}
Then there is a cofibrantly-generated simplicial model structure on $\cat{C}$ with weak equivalences and fibrations defined termwise (as in \ref{def:weq-operads} and \ref{def:fib-operads}), and with generating cofibrations given by the set $\mathbb{I}_{\cat{C}}$ of Definition \ref{def:gen-cofs}.
\end{theorem}
\begin{proof}
We leave the proof of this theorem to the appendix. See Proposition \ref{prop:right-mods}(4) and Corollary \ref{cor:left-mods-model}
\end{proof}

\begin{definition}[$\Sigma$-cofibrations] \label{def:sigma-cofibrant}
We say that a symmetric sequence is \emph{$\Sigma$-cofibrant} if it is cofibrant in the projective model structure on $\spectra^{\mathsf{\Sigma}}$, and a reduced symmetric sequence is \emph{$\Sigma$-cofibrant} if it is cofibrant in the projective model structure on $\spectra^{\mathsf{\Sigma}}_{\mathsf{red}}$.

A reduced operad in $\spectra$ is said to be \emph{$\Sigma$-cofibrant} if the underlying \emph{reduced} symmetric sequence is $\Sigma$-cofibrant. For an operad $P$ in $\spectra$, a $P$-module is said to be \emph{$\Sigma$-cofibrant} if the underlying symmetric sequence is $\Sigma$-cofibrant.

Similarly, we say that a map of modules or reduced operads is a $\Sigma$-cofibration if the corresponding map of symmetric sequences, or reduced symmetric sequences, is a cofibration in the relevant projective model structure.
\end{definition}

\begin{definition}[$\Sigma_n$-cofibrations] \label{def:sigma-n-cofibrant}
There is a projective model structure on the category $\spectra^{\Sigma_n}$ whose objects are spectra with $\Sigma_n$-actions, and whose morphisms are $\Sigma_n$-equivariant maps of spectra. If $E$ is an object in this category, we say that $E$ is \emph{$\Sigma_n$-cofibrant} if it is cofibrant in this projective model structure. Equivalently, this means that $E$ has the left-lifting property with respect to $\Sigma_n$-equivariant maps of spectra that are trivial fibrations in $\spectra$.
\end{definition}

\begin{remark} \label{rem:sigma-cofibrant}
The condition of being $\Sigma$-cofibrant can be verified termwise. A symmetric sequence, reduced operad, or module $A$ is $\Sigma$-cofibrant if and only if each $A(n)$ is $\Sigma_n$-cofibrant. (For reduced operads and reduced symmetric sequence, this needs to hold only $n \geq 2$.)
\end{remark}

\begin{lemma} \label{lem:termwise-cofibrant}
A $\Sigma$-cofibrant symmetric sequence, reduced operad or module is termwise-cofibrant.
\end{lemma}
\begin{proof}
Let $M$ be the $\Sigma$-cofibrant object. We know from Remark \ref{rem:sigma-cofibrant} that each $M(n)$ has the left-lifting property with respect to $\Sigma_n$-equivariant trivial fibrations. We have to show that $M(n)$ has the left-lifting property with respect to all trivial fibrations in $\spectra$. Take any trivial fibration $X \to Y$ and a map $f:M(n) \to Y$ in $\spectra$. Extending $f$ equivariantly to a map $M(n) \to \Map((\Sigma_n)_+,Y)$ we get a diagram of $\Sigma_n$-equivariant maps
\[ \begin{diagram}
  \node[2]{\Map((\Sigma_n)_+,X)} \arrow{s,r,A}{\sim} \\
  \node{M(n)} \arrow{e} \node{\Map((\Sigma_n)_+,Y)}
\end{diagram} \]
The vertical map here is still a trivial fibration since the discrete space $\Sigma_n$ is cofibrant. Therefore this diagram has a lift $M(n) \to \Map(\Sigma_n,X)$ which determines a lift $M(n) \to X$ of the original map $f:M(n) \to Y$. Thus, $M(n)$ has the necessary lifting property and is cofibrant in $\spectra$.
\end{proof}

\begin{definition}[Projective-cofibrations] \label{def:projectively-cofibrant}
Let $\cat{C}$ be one of the categories of Theorem \ref{thm:projective-model}. We use the term \emph{projective-cofibration} to describe the morphisms in $\cat{C}$ that are cofibrations in the model structure described in \ref{thm:projective-model}. The objects of $\cat{C}$ that are cofibrant in that model structure are then described as \emph{projectively-cofibrant}. We stress this so as not to confuse these with termwise-cofibrant objects.
\end{definition}

We now have three cofibrancy notions for operads and modules:
\begin{itemize}
  \item projectively-cofibrant (cofibrant in the relevant projective model structure);
  \item $\Sigma$-cofibrant (cofibrant as a symmetric sequence);
  \item termwise-cofibrant (individual spectra are cofibrant).
\end{itemize}
We have already shown that $\Sigma$-cofibrant objects are termwise-cofibrant. We now verify that, under suitable conditions, projective-cofibrant objects are $\Sigma$-cofibrant. This completes the construction of termwise-cofibrant replacements for operads and modules, and allows us to form homotopy-invariant versions of the bar constructions.

\begin{proposition} \label{prop:termwise-cofibrant}
A projectively-cofibrant reduced operad is $\Sigma$-cofibrant. If $P$ is a $\Sigma$-cofibrant reduced operad, then a projectively-cofibrant $P$-module (left-, right- or bi-) is $\Sigma$-cofibrant. If $P$ is a termwise-cofibrant reduced operad, then a projectively-cofibrant $P$-module is termwise-cofibrant.
\end{proposition}
\begin{proof}
Let $P$ be a projectively-cofibrant reduced operad. Then $P$ is a retract of a `cell operad' (that is, a cell complex formed from the generating cofibrations in $\mathsf{Op}(\spectra)$). If we can show that a cell operad is $\Sigma$-cofibrant, it follows that $P$ is too. So we can assume, without loss of generality, that $P$ actually is a cell operad.

We now use induction on a cell structure for $P$. The colimit of a sequence of cofibrations in a model category is a cofibration, so it is sufficient to show the following claim.

Suppose we have a pushout square
\[ \begin{diagram}
  \node{F(A)} \arrow{e} \arrow{s} \node{X} \arrow{s} \\
  \node{F(B)} \arrow{e} \node{X'}
\end{diagram} \]
in the category of reduced operads. Here $F(A) \to F(B)$ is the coproduct of some set of generating cofibrations in the model structure of Theorem \ref{thm:projective-model}. Note that any such coproduct is given by applying the free operad functor $F$ to a map $A \to B$ of symmetric sequences. It is sufficient then to show that, if $X$ is $\Sigma$-cofibrant, $X \to X'$ is a $\Sigma$-cofibration (i.e. a cofibration of the underlying symmetric sequences). This is related to the `Cofibration Hypothesis' used to establish the model structure on operads and we prove this claim in the appendix (Lemma \ref{lem:pushouts}).

It should be noted that in general it is not true that $X \to X'$ is always a $\Sigma$-cofibration. In particular, projective-cofibrations in $\mathsf{Op}(\spectra)$ are not always $\Sigma$-cofibrations. However, our proof demonstrates that a projective-cofibration with projective-cofibrant domain \emph{is} a $\Sigma$-cofibration.

The module cases are similar (with right modules being much easier to deal with since pushouts are then calculated on the underlying symmetric sequences).
\end{proof}

\begin{corollary} \label{cor:termwise-cofibrant}
If $P$ is a reduced operad in $\spectra$, then there is a functorial termwise-cofibrant replacement
\[ \tilde{P} \to P \]
such that:
\begin{itemize}
  \item $\tilde{P}$ is a termwise-cofibrant operad (in fact, we can take $\tilde{P}$ to be projectively-cofibrant);
  \item the map $\tilde{P} \to P$ is a trivial fibration, and so in particular $\tilde{P}(n) \to P(n)$ is a weak equivalence for all $n$.
\end{itemize}
Similarly, if $P$ is a termwise-cofibrant operad and $M$ a $P$-module (either right-, left- or bi-), then there is a functorial termwise-cofibrant replacement $\tilde{M} \weq M$.
\end{corollary}
\begin{proof}
Take $\tilde{P} \to P$ to be a functorial projectively-cofibrant replacement for $P$, as guaranteed by the model structure of Theorem \ref{thm:projective-model}, for example, using the small object argument. Then $\tilde{P}$ is termwise-cofibrant by Proposition \ref{prop:termwise-cofibrant} and Lemma \ref{lem:termwise-cofibrant}. The proof is similar for modules.
\end{proof}

\begin{remark} \label{rem:homotopy-invariant-bar}
Given a reduced operad $P$, left $P$-module $L$ and right $P$-module $R$, we can form a (functorial) homotopy-invariant version of the two-sided bar construction as follows. First let
\[ \tilde{P} \to P \]
be a (functorial) termwise-cofibrant replacement for $P$, as in Corollary \ref{cor:termwise-cofibrant}. Via the map $\tilde{P} \to P$, $L$ inherits a left $\tilde{P}$-module structure, and $R$ inherits a right $\tilde{P}$-module structure. We then take termwise-cofibrant replacements
\[ \tilde{L} \to L, \quad \tilde{R} \to R \]
of $L$ and $R$ respectively, as $\tilde{P}$-modules. The two-sided bar construction
\[ B(\tilde{R},\tilde{P},\tilde{L}) \]
is then a functorial homotopy invariant of the original data $P$, $R$ and $L$.
\end{remark}

Our projectively-cofibrant replacements are guaranteed by the small object argument in the relevant projective model category. It is, however, useful to have more explicit examples of projectively-cofibrant replacements. In the remainder of this section, we show that our bar constructions can be used to do this in some cases. The essential idea is that if $R$ is a $\Sigma$-cofibrant right $P$-module, then $B(R,P,P)$ is a projectively-cofibrant right $P$-module, that is weakly equivalent to $R$. Similar statements hold for left modules and bimodules.

\begin{definition}[Bar resolutions] \label{def:module-resolutions}
Let $P$ be a reduced operad in $\spectra$ and let $R$ be a right $P$-module. We define a map
\[ B(R,P,P) \to R \]
of symmetric sequences as follows. Treating $R$ as a constant simplicial object we define maps
\[ B_{\bullet}(R,P,P) \to R \]
by means of the iterated composition maps
\[ R \circ P^k \circ P \to R \circ P \to R. \]
These commute with the face and degeneracies and so taking realizations, we get the required map
\[ B(R,P,P) \to R. \]
Similarly, if $L$ is a left $P$-module, we obtain a map of symmetric sequences of the form
\[ B(P,P,L) \to L. \]
If $M$ is a $P$-bimodule, there is a corresponding map
\[ B(P,P,M,P,P) \to M. \]
\end{definition}

\begin{lemma} \label{lem:module-resolutions-eqs}
The maps $B(R,P,P) \to R$, $B(P,P,L) \to L$ and $B(P,P,M,P,P) \to M$ of Definition \ref{def:module-resolutions} are weak equivalences of symmetric sequences.
\end{lemma}
\begin{proof}
The maps module structure map $R \circ P \to R$ provides an augmentation of the simplicial bar construction $B_{\bullet}(R,P,P)$ (see Definition \ref{def:augmented-simplicial}). The unit maps $\mathsf{1} \to P$ applied on the right-hand end of the $k$-simplices objects $R \circ P^k \circ P$ provide a simplicial contraction and so by Lemma \ref{lem:simplicial-contraction}, the induced map $B(R,P,P) \to R$ is a homotopy equivalence and hence a weak equivalence in $\spectra$. Similarly for $B(P,P,L) \to L$ and $B(P,P,M,P,P) \to M$.
\end{proof}

\begin{lemma} \label{lem:module-resolutions}
The maps $B(R,P,P) \to R$, $B(P,P,L) \to L$ and $B(P,P,M,P,P) \to M$ of Definition \ref{def:module-resolutions} are morphisms of right, left and bi- $P$-modules respectively.
\end{lemma}
\begin{proof}
In the first case, this amounts to showing that the following diagram commutes:
\[ \begin{diagram}
  \node{B(R,P,P) \circ P} \arrow{s} \arrow{e} \node{B(R,P,P)} \arrow{s} \\
  \node{R \circ P} \arrow{e} \node{R}
\end{diagram} \]
This follows from the definition of the right $P$-module structure on $B(R,P,P)$. (See Definition \ref{def:module-bar}.) The other parts are similar.
\end{proof}

The previous two lemmas together say that $B(R,P,P)$ is some kind of resolution of $R$ in the category of right $P$-modules. Proposition \ref{prop:bar-resolution} below gives a condition that $B(R,P,P)$ be a projectively-cofibrant right $P$-module. To prove this, we need the following lemma.

\begin{lemma} \label{lem:sigma-cofibrant-comprod}
\hfill\begin{enumerate}
  \item Let $A$ and $B$ be $\Sigma$-cofibrant symmetric sequences in $\spectra$. Then $A \circ B$ is $\Sigma$-cofibrant.
  \item Let $P$ be a $\Sigma$-cofibrant reduced operad in $\spectra$ and let $R$ and $L$ be $\Sigma$-cofibrant right and left $P$-modules respectively. Then $B(R,P,L)$ is $\Sigma$-cofibrant.
\end{enumerate}
\end{lemma}
\begin{proof}
For (1), we can assume, without loss of generality, that $A$ is a cell complex with respect to the generating cofibrations in the model structure on $\spectra^{\mathsf{\Sigma}}$ (see Definition \ref{def:gen-cofs}). The composition product $A \circ B$ commutes with colimits in the $A$-variable. Therefore, by induction on a cell structure for $A$, it is sufficient to show that, for $k \geq 1$ and $I_0 \to I_1$ one of the generating cofibrations in $\spectra$, the map
\[ A_k(I_0) \circ B \to A_k(I_1) \circ B \]
is a projective-cofibration in $\spectra^{\mathsf{\Sigma}}$, where $A_k$ denotes the free symmetric sequence functor on an object in position $k$ (Definition \ref{def:free-symseq}).

It follows from the definition of $A_k$ that we can write
\[ [A_k(I_0) \circ B](n) \isom \Wdge_{n = n_1+\dots+n_k} I_0 \smsh B(n_1) \smsh \dots \smsh B(n_k) \]
where the coproduct is now taken over \emph{ordered} partitions of the set $\{1,\dots,n\}$ into $k$ pieces of sizes $n_1,\dots,n_k$. Since the pieces of the partition are now ordered, we can think of the map
\[ [A_k(I_0) \circ B](n) \to [A_k(I_1) \circ B](n) \]
as a special case of
\[ \tag{*} \Wdge_{n = n_1+\dots+n_k} I_0 \smsh B_1(n_1) \smsh \dots \smsh B_k(n_k) \to \Wdge_{n = n_1+\dots+n_k} I_1 \smsh B_1(n_1) \smsh \dots \smsh B_k(n_k) \]
where $B_1,\dots,B_k$ can now be different symmetric sequences. The map we are interested in then comes by taking each $B_i$ equal to $B$.

It is now sufficient to show that the map (*) is a $\Sigma_n$-cofibration for any $\Sigma$-cofibrant symmetric sequences $B_1,\dots,B_k$. We prove this by again assuming, without loss of generality, that the $B_i$ are cell complexes in $\spectra^{\mathsf{\Sigma}}$, and applying induction on cell structures. This works because each side of the map (*) preserves colimits in each $B_i$-variable.

We have now reduced to showing that the map (*) is a cofibration when $B_i = A_{n_i}(K_i)$ for some finite cell spectra $K_i \in \spectra$. In this case, we can rewrite (*) as
\[ \Wdge_{n = n_1+\dots+n_k} I_0 \smsh K_1 \smsh \dots \smsh K_k \smsh (\Sigma_{n_1} \times \dots \times \Sigma_{n_k})_+ \to \Wdge_{n = n_1+\dots+n_k} I_1 \smsh K_1 \smsh \dots \smsh K_k \smsh (\Sigma_{n_1} \times \dots \times \Sigma_{n_k})_+. \]
This can be rewritten in turn as
\[ (I_0 \smsh K_1 \smsh \dots \smsh K_k) \smsh (\Sigma_n)_+ \to (I_1 \smsh K_1 \smsh \dots \smsh K_k) \smsh (\Sigma_n)_+. \]
This is now the free $\Sigma_n$-spectrum functor applied to a cofibration in $\spectra$ and is therefore a $\Sigma_n$-cofibration as required.

For (2), we start by showing that $B_{\bullet}(R,P,L)$ is Reedy $\Sigma$-cofibrant (that is, Reedy cofibrant with respect to the Reedy model structure on simplicial symmetric sequences coming from the projective model structure on $\spectra^{\mathsf{\Sigma}}$). To see this, we consider the latching maps
\[ \colim_{m < n} B_m(R,P,L) \to B_n(R,P,L) = R \circ P^n \circ L. \]
Arguing as in the proof of Proposition \ref{prop:bar-invariance}, this map is the inclusion of a $\Sigma_n$-equivariant wedge summand. But then, since $R \circ P^n \circ L$ is $\Sigma$-cofibrant by (1), it follows that this map is a projective-cofibration. Thus the simplicial bar construction $B_{\bullet}(R,P,L)$ is $\Sigma$-cofibrant. By \cite[18.6.7]{hirschhorn:2003}, it follows that $B(R,P,L)$ is $\Sigma$-cofibrant.
\end{proof}

\begin{prop} \label{prop:bar-resolution}
Let $P$ be a $\Sigma$-cofibrant operad in $\spectra$ and let $R$ be a $\Sigma$-cofibrant right $P$-module. Then $B(R,P,P)$ is a projectively-cofibrant right $P$-module. Similarly, if $L$ is a $\Sigma$-cofibrant left $P$-module, $B(P,P,L)$ is a projectively-cofibrant left $P$-module, and if $M$ is a $\Sigma$-cofibrant $P$-bimodule, then $B(P,P,M,P,P)$ is a projectively-cofibrant $P$-bimodule.
\end{prop}
\begin{proof}
In the appendix (\ref{prop:right-mods}(5) and \ref{prop:left-mods-realization}), we show that the geometric realization of a simplicial $P$-module (right-, left- or bi-) inherits a $P$-module structure. The simplicial bar construction $B_{\bullet}(R,P,P)$ is a simplicial right $P$-module (via the regular right $P$-action on $R \circ P \circ \dots \circ P$) and our chosen right $P$-module structure on $B(R,P,P)$ comes from taking geometric realization of this. (See also \ref{prop:bar-comprod}.)

By \cite[18.6.7]{hirschhorn:2003} it is now sufficient to show that the simplicial bar construction $B_{\bullet}(R,P,P)$ is Reedy projectively-cofibrant. This is similar to the proof of Lemma \ref{lem:sigma-cofibrant-comprod}(2) except that we are working in the Reedy model structure on simplicial right $P$-modules, rather than on simplicial symmetric sequences.

The source of the \ord{n} latching map is
\[ \colim_{m < n} B_m(R,P,P) = \colim_{m < n} R \circ P^{m} \circ P. \]
This colimit is calculated in the category of right $P$-modules, but the free functor $- \circ P$ from symmetric sequences to right $P$-modules preserves colimits and so this is isomorphic to
\[ \left( \colim_{m < n} R \circ P^{m} \right) \circ P \]
where the colimit is now taken in the category of symmetric sequences. Since $- \circ P$ also preserves cofibrations (it takes generating cofibrations to generating cofibrations), it is now sufficient to show that the map
\[ \colim_{m < n} R \circ P^{m} \to R \circ P^{n} \]
is a projective-cofibration of symmetric sequences (i.e. a $\Sigma$-cofibration). But this map is equal to the corresponding latching map in the simplicial bar construction $B(R,P,\mathsf{1})$ which we showed in the proof of Lemma \ref{lem:sigma-cofibrant-comprod}(2) to be a $\Sigma$-cofibration.

Analogous arguments apply for the cases of $B(P,P,L)$ and $B(P,P,M,P,P)$.
\end{proof}

\begin{definition}[Bar resolutions] \label{def:bar-resolution}
Let $P$ be a $\Sigma$-cofibrant operad in $\spectra$. If $R$ is a $\Sigma$-cofibrant right $P$-module, then we call $B(R,P,P)$ the \emph{bar resolution of $R$ as a right $P$-module}. Similarly, if $L$ is a $\Sigma$-cofibrant left $P$-module, then $B(P,P,L)$ is the \emph{bar resolution of $L$ as a left $P$-module}, and if $M$ is a $\Sigma$-cofibrant $P$-bimodule, then $B(P,P,M,P,P)$ is the \emph{bar resolution of $M$ as a $P$-bimodule}.
\end{definition}

\section{Derived mapping objects for modules} \label{sec:ext}

As part of Theorem \ref{thm:projective-model} we constructed simplicial enrichments on the categories of symmetric sequences, and of right-, left- and bimodules over a fixed reduced operad $P$. These enrichments play an important role in the later sections of this paper so we describe them more explicitly here. In addition, the categories of symmetric sequences and \emph{right} $P$-modules have enrichments over $\spectra$, not just $\sset$ which are also important. The last part of this section is dedicated to constructing derived versions of the various mapping objects associated to these enrichments, based on the bar resolutions of Definition \ref{def:bar-resolution}.

\begin{definition}[Mapping objects] \label{def:enriched-symseq}
Let $M$ and $N$ be symmetric sequences of spectra. We define
\[ \Map_{\mathsf{\Sigma}}(M,N) := \prod_{r = 1}^{\infty} \Map(M(r),N(r))^{\Sigma_r} \]
where the $\Sigma_r$-fixed point object can be defined as the limit of the corresponding $\Sigma_r$-indexed diagram in $\spectra$. Equivalently, this is the `end' (in the sense of MacLane \cite[IX.5]{maclane:1998}) of the bifunctor
\[ \mathsf{\Sigma} \times \mathsf{\Sigma}^{op} \to \spectra; (r,s) \mapsto \Map(M(r),N(s)). \]
This definition produces an enrichment of the category of symmetric sequences over $\spectra$. There is a corresponding enrichment over $\sset$ which we write
\[ \Hom_{\mathsf{\Sigma}}(M,N) := \prod_{r = 1}^{\infty} \spectra(M(r),N(r))^{\Sigma_r}. \]
These are related by
\[ \Hom_{\mathsf{\Sigma}}(M,N) \isom \spectra(S,\Map_{\mathsf{\Sigma}}(M,N)). \]
This is the same as the simplicial enrichment considered in the proof of Theorem \ref{thm:projective-model} (see the appendix).
\end{definition}

\begin{definition} \label{def:enriched-symseq-maps}
Let $M,N,P$ be symmetric sequences in $\spectra$. Then there are natural maps
\[ \Map_{\mathsf{\Sigma}}(M,N) \to \Map_{\Sigma}(M \circ P,N \circ P) \]
constructed from the maps
\[ \Map(M(r),N(r)) \to \Map\left(M(r) \smsh P(n_1) \smsh \dots \smsh P(n_r), N(r) \smsh P(n_1) \smsh \dots \smsh P(n_r)\right). \]
There are also natural maps
\[ \Hom_{\mathsf{\Sigma}}(M,N) \to \Hom_{\Sigma}(P \circ M, P \circ N) \]
constructed from the composites
\[  \begin{split} \prod_{r = 1}^{\infty} \spectra(M(r),N(r))
    & \dgTEXTARROWLENGTH=4em\arrow{e,t}{\Delta} \left(\prod_{r = 1}^{\infty} \spectra(M(r),N(r))\right)^{\smsh k} \\
    & \dgTEXTARROWLENGTH=4em\arrow{e,t}{(\pi_{r_1},\dots,\pi_{r_k})} \spectra(M(r_1),N(r_1)) \smsh \dots \smsh \spectra(M(r_k),N(r_k)) \\
    & \dgTEXTARROWLENGTH=4em\arrow{e} \spectra\left(P(k) \smsh M(r_1) \smsh \dots \smsh M(r_k), P(k) \smsh N(r_1) \smsh \dots \smsh N(r_k)\right).
\end{split} \]
Note that in the latter case, we need to use the $k$-fold diagonal map for pointed simplicial sets to make the construction. The corresponding object cannot be defined using only the enrichment over $\spectra$ since we have no diagonal on the spectrum $\Map(M(r),N(r))$.
\end{definition}

\begin{definition}[Mapping objects for modules] \label{def:enriched-modules}
Let $P$ be an operad in $\spectra$ and let $R,R'$ be right $P$-modules. We then define a spectrum
\[ \Map^{\mathsf{right}}_{P}(R,R') := \lim \left(\Map_{\mathsf{\Sigma}}(R,R') \rightrightarrows \Map_{\mathsf{\Sigma}}(R \circ P,R') \right) \]
where one of the arrows on the right-hand side is induced by the module structure map $R \circ P \to R$, and the other is the composite
\[ \Map_{\mathsf{\Sigma}}(R,R') \to \Map_{\mathsf{\Sigma}}(R \circ P, R' \circ P) \to \Map_{\mathsf{\Sigma}}(R \circ P,R') \]
where the first map here comes from Definition \ref{def:enriched-symseq-maps}, and the second is induced by the module structure map $R' \circ P \to R'$. This gives us an enrichment of $\mathsf{Mod}_{\mathsf{right}}(P)$ over $\spectra$ and we get a corresponding simplicial enrichment by replacing $\Map_{\Sigma}(-,-)$ with $\Hom_{\Sigma}(-,-)$.

If $L$ and $L'$ are left $P$-modules, we define
\[ \Hom^{\mathsf{left}}_{P}(L,L') := \lim \left(\Hom_{\mathsf{\Sigma}}(L,L') \rightrightarrows \Hom_{\mathsf{\Sigma}}(P \circ L,L') \right) \]
in the corresponding way. Similarly, if $M$ and $M'$ are $P$-bimodules, we define
\[ \Hom^{\mathsf{bi}}_{P}(M,M') := \lim \left(\Hom_{\mathsf{\Sigma}}(M,M') \rightrightarrows \Hom_{\mathsf{\Sigma}}(P \circ M \circ P,M') \right). \]
Note that $\Hom_{P}(L,L')$ and $\Hom_{P}(M,M')$ are only simplicial sets. There are no natural enrichments for left and bi- $P$-modules over $\spectra$.
\end{definition}

\begin{remark}
The fact that we can enrich the category of right, but not left or bi-, modules over $\spectra$ is related to the fact that right modules (in $\spectra$) form a \emph{stable} model category. Right modules are really just functors from an appropriate ($\spectra$-enriched) category into $\spectra$ (see Appendix A) and categories of functors with values is a stable category are stable. On the other hand, left modules and bimodules are more closely related to algebras which generally do not form stable model categories.
\end{remark}

\begin{lemma} \label{lem:adjunction}
Let $P$ be a reduced operad in $\spectra$. The adjunctions between free and forgetful functors for $P$-modules extend to the following isomorphisms in the enriched setting.
\begin{itemize}
  \item $\Map^{\mathsf{right}}_{P}(A \circ P,R) \isom \Map_{\mathsf{\Sigma}}(A,R)$
  \item $\Hom^{\mathsf{left}}_{P}(P \circ A,L) \isom \Hom_{\mathsf{\Sigma}}(A,L)$
  \item $\Hom^{\mathsf{bi}}_{P}(P \circ A \circ P, M) \isom \Hom_{\mathsf{\Sigma}}(A,M)$
\end{itemize}
where $A$ is any symmetric sequence, $R$ is a right $P$-module, $L$ is a left $P$-module and $M$ is a $P$-bimodule.
\end{lemma}
\begin{proof}
These are standard.
\end{proof}

Our mapping objects have the usual homotopy invariance properties.

\begin{prop} \label{prop:map-homotopy}
Let $P$ be a reduced operad in $\spectra$.
\begin{enumerate}
  \item The construction $\Map^{\mathsf{right}}_{P}(R,R')$ for right $P$-modules $R,R'$ preserves weak equivalences between projectively-cofibrant objects in the $R$-variable. If $R$ is projectively-cofibrant, it preserves all weak equivalences in the $R'$-variable.
  \item The construction $\Hom^{\mathsf{left}}_{P}(L,L')$ for left $P$-modules $L,L'$ preserves weak equivalences between projectively-cofibrant objects in the $L$-variable. If $L$ is projectively-cofibrant, it preserves all weak equivalences in the $L'$-variable.
  \item The construction $\Hom^{\mathsf{bi}}_{P}(M,M')$ for $P$-bimodules $M,M'$ preserves weak equivalences between projectively-cofibrant objects in the $M$-variable. If $M$ is projectively-cofibrant, it preserves all weak equivalences in the $M'$-variable.
  \item The construction $\Map_{\mathsf{\Sigma}}(A,A')$ for symmetric sequences $A,A'$ preserves weak equivalences between projectively-cofibrant objects in the $A$-variable. If $A$ is projectively-cofibrant, it preserves all weak equivalences in the $A'$-variable.
\end{enumerate}
\end{prop}
\begin{proof}
In each case, we check that the given mapping object makes the module category into a enriched model category (in the sense of Hovey \cite{hovey:1999}) over either $\spectra$ or $\sset$. For this it is enough to check the claims for the generating cofibrations and trivial cofibrations. For example, in (1) we check that the map
\[ \Map^{\mathsf{right}}_{P}(R_n(I_1),R') \to \Map^{\mathsf{right}}_{P}(R_n(I_0),R') \]
is a fibration in $\spectra$ when $I_0 \to I_1$ is one of the generating cofibrations in $\spectra$ and $R_n$ is the free right $P$-module construction of Definition \ref{def:more-free-operads}. But by use of various enriched adjunction isomorphisms, this reduces to the fact that
\[ \Map(I_1,R'(n)) \to \Map(I_0,R'(n)) \]
is a fibration in $\spectra$.
\end{proof}

Proposition \ref{prop:map-homotopy} tells us that the homotopically correct form of the mapping object between two $P$-modules requires taking a projectively-cofibrant replacement for the first module. We are particularly interested in the case where such a projectively-cofibrant replacement is given by one of the bar resolutions of Definition \ref{def:bar-resolution}.

\begin{definition}[Ext-objects] \label{def:ext-modules}
Let $P$ be a reduced operad in $\spectra$ and let $R,R'$ be right $P$-modules and suppose that $R$ is $\Sigma$-cofibrant. We define
\[ \Ext^{\mathsf{right}}_{P}(R,R') := \Map^{\mathsf{right}}_{P}(B(R,P,P),R'). \]
Note that this $\Ext$-object is in $\spectra$. By Propositions \ref{prop:bar-resolution} and \ref{prop:map-homotopy}, this is the derived mapping object in the category of right $P$-modules (with the projective model structure).

Let $L,L'$ be left $P$-modules and suppose that $L$ is $\Sigma$-cofibrant. We define
\[ \Ext^{\mathsf{left}}_{P}(L,L') := \Hom^{\mathsf{left}}_{P}(B(P,P,L),L'). \]
This $\Ext$-object is a pointed simplicial set because left $P$-modules are enriched only over $\sset$, not $\spectra$.

Let $M,M'$ be $P$-bimodules and suppose that $M$ is $\Sigma$-cofibrant. We define
\[ \Ext^{\mathsf{bi}}_{P}(M,M') := \Hom^{\mathsf{bi}}_{P}(B(P,P,M,P,P),M'). \]
This again is a pointed simplicial set because $P$-bimodules are also enriched only over $\sset$.
\end{definition}

\begin{lemma} \label{lem:ext-homotopy}
Each of the $\Ext$-objects of Definition \ref{def:ext-modules} is homotopy invariant with respect to $\Sigma$-cofibrant modules in the first variable, and all modules in the second variable.
\end{lemma}
\begin{proof}
This follows from Propositions \ref{prop:bar-resolution} and \ref{prop:map-homotopy}.
\end{proof}

\begin{remark} \label{rem:ext-tot}
We can replace maps out of a geometric realization with the totalization of a corresponding cosimplicial object. For the $\Ext$-objects of Definition \ref{def:ext-modules} we get:
\[ \Ext^{\mathsf{right}}_{P}(R,R') \isom \Tot \left[ \Map^{\mathsf{right}}_{P}(R \circ P^{\bullet} \circ P, R')\right] \isom \Tot \left[ \Map_{\mathsf{\Sigma}}(R \circ P^{\bullet},R') \right] \]
The coface and codegeneracy maps are dual to those in the simplicial bar construction, except that in the second presentation, the final coface map is given by
\[ \delta^k: \Map_{\mathsf{\Sigma}}(R \circ P^k,R') \to \Map_{\mathsf{\Sigma}}(R \circ P^k \circ P,R' \circ P) \to \Map_{\mathsf{\Sigma}}(R \circ P^{k+1},R'). \]
Similarly, for left- and bimodules, we have
\[ \Ext^{\mathsf{left}}_{P}(L,L') \isom \Tot \left[ \Hom_{\mathsf{\Sigma}}(P^{\bullet} \circ L,L') \right] \]
and
\[ \Ext^{\mathsf{bi}}_{P}(M,M') \isom \Tot \left[ \Hom_{\mathsf{\Sigma}}(P^{\bullet} \circ M \circ P^{\bullet}, M') \right] \]
respectively.

If the operad $P$ and the module $R$, $L$ or $M$ are termwise-cofibrant (which we have assumed to be the case in defining the $\Ext$-objects), then the cosimplicial objects appearing above are Reedy fibrant (essentially by the same argument as in the proof of Proposition \ref{prop:bar-invariance}). The homotopy-invariance properties of Lemma \ref{lem:ext-homotopy} then follow by Lemma \ref{lem:cosimplicial-contraction}.
\end{remark}

\begin{definition}[Maps of Ext-objects] \label{def:ext-maps}
Let $P$ be a reduced operad in $\spectra$ and let $R,R'$ be right $P$-modules. Then the resolution map $B(R,P,P) \to R$ of Definition \ref{def:module-resolutions} gives us a natural map
\[ \Map^{\mathsf{right}}_{P}(R,R') \to \Ext^{\mathsf{right}}_{P}(R,R'). \]
Similarly, for left $P$-modules $L$ and $L'$, we have a map
\[ \Hom^{\mathsf{left}}_{P}(L,L') \to \Ext^{\mathsf{left}}_{P}(L,L') \]
and, for $P$-bimodules $M$ and $M'$, a map
\[ \Hom^{\mathsf{bi}}_{P}(M,M') \to \Ext^{\mathsf{bi}}_{P}(M,M'). \]
\end{definition}

\section{Pro-symmetric sequences and Spanier-Whitehead Duality} \label{sec:prosymseq}

We now consider Spanier-Whitehead duality for modules over operads. This requires us to extend the theory of pro-spectra described in \S\ref{sec:pro-spectra} to symmetric sequences, and to modules and comodules. The main aims of this section are: to define pro-symmetric sequences, pro-modules and pro-comodules; to construct specific Spanier-Whitehead duality functors relating these; and to establish the basic properties of Spanier-Whitehead duality, including how it interacts with the composition product and bar construction.

\begin{definition}[Pro-symmetric sequences]
A \emph{pro-symmetric sequence} in $\spectra$ is a pro-object in the category of symmetric sequences in $\spectra$. Note that this differs from a symmetric sequence in $\mathsf{Pro}(\spectra)$ although those two categories are equivalent. Explicitly then, a pro-symmetric sequence consists of a sequence $M(n)$ of pro-spectra, each indexed on the same cofiltered category, together with a $\Sigma_n$-action on $M(n)$ by levelwise maps.

Let $P$ be a reduced operad in $\spectra$. A \emph{pro-right-$P$-module} is a pro-object in the category of right $P$-modules. A \emph{pro-left-$P$-module} is a pro-object in the category of left $P$-modules.

Let $Q$ be a reduced cooperad in $\spectra$. A \emph{pro-right-$Q$-comodule} is a pro-object in the category of right $Q$-comodules. A \emph{pro-left-$Q$-comodule} is a pro-object in the category of left $Q$-comodules.

We emphasize that in each of these cases, the pro-object involves a single cofiltered category on which all the terms in the symmetric sequences are indexed.

If $M$ is a pro-symmetric sequence (or pro-module or -comodule) indexed on the cofiltered category $\cat{J}$, then we write $M_j$ for the symmetric sequence (or module or comodule) given by evaluating $M$ at $j \in \cat{J}$, and we write $M(n)$ for the pro-spectrum formed by the terms $M_j(n)$.
\end{definition}

\begin{remark}
We could also define pro-operads and pro-cooperads along similar lines, but we do not need these ideas in this paper. In particular, the pro-modules and pro-comodules we consider are all over individual (not pro-) operads and cooperads.
\end{remark}

\begin{definition}[Weak equivalence] \label{def:prosymseq-weq}
A morphism $f: A \to B$ of pro-symmetric sequences in $\spectra$ is a \emph{weak equivalence} if each map $f_n: A(n) \to B(n)$ is a weak equivalence of pro-spectra in the sense of Theorem \ref{thm:pro-spectra}. A morphism of pro-modules or pro-comodules is a \emph{weak equivalence} if the underlying morphism of pro-symmetric sequences is a weak equivalence.
\end{definition}

\begin{definition}[Directly-dualizable pro-symmetric sequences] \label{def:directly-dualizable}
Recall that a pro-spectrum $X$ is \emph{directly-dualizable} if the indexing category has a certain form and if each term $X_j$ is a cofibrant homotopy-finite spectrum (see Definition \ref{def:dualizable}). We extend this definition to pro-symmetric sequences. We say that a pro-symmetric sequence $M$ (or pro-module or pro-comodule) is \emph{directly-dualizable} if each $M(n)$ is a directly-dualizable pro-spectrum. We include the case where the indexing category for $M$ is trivial, so that a symmetric sequence (or operad, cooperad, module or comodule) $M$ is \emph{directly-dualizable} if each $M(n)$ is a cofibrant homotopy-finite spectrum. In particular, note that a directly-dualizable operad or module is termwise-cofibrant.
\end{definition}

\begin{remark}
The term `directly-dualizable' is meant to invoke the ease with which we can form the Spanier-Whitehead duals of these objects. In particular, assuming that each spectrum $E$ involved in one of our pro-symmetric sequences is cofibrant and homotopy-finite ensures that the naive dual $\Map(E,S)$ is a model for the homotopically-correct Spanier-Whitehead dual. Thus, no further cofibrant replacements are necessary to form that dual. For pro-objects, it also includes a condition on the indexing category that allows the dual spectrum to be formed as a homotopy colimit (see Lemma \ref{lem:colim-ind}).
\end{remark}

\begin{definition}[Spanier-Whitehead duality for pro-symmetric sequences] \label{def:dual-symseq}
Let $M$ be a directly-dualizable pro-symmetric sequence indexed on the cofiltered poset $\cat{J}$. We define the \emph{Spanier-Whitehead dual} of $M$ to be the symmetric sequence of spectra given by
\[ (\dual M)(n) := \dual M(n) \]
where $\dual M(n)$ is the Spanier-Whitehead dual of the pro-spectrum $M(n)$ as described in Definition \ref{def:dualizable}. Explicitly, then we have
\[ (\dual M)(n) := \hocolim_{j \in \cat{J}^{op}} \Map(M_j(n),S). \]
We also write this as
\[ \dual M := \hocolim_{j \in \cat{J}^{op}} \Map(M_j,S) \]
with the understanding that homotopy colimits of diagrams of symmetric sequences are formed termwise.
\end{definition}

We now look at how Spanier-Whitehead duality relates to the composition product. For this we need to extend the composition product to pro-symmetric sequences.

\begin{definition}[Composition products] \label{def:pro-comp}
Let $M$ and $N$ be pro-symmetric sequences indexed on the cofiltered categories $\cat{J}$ and $\cat{K}$ respectively. We then define $M \circ N$ to be the pro-symmetric sequence indexed on the cofiltered category $\cat{J} \times \cat{K}$ given by
\[ (M \circ N)_{(j,k)} := M_j \circ N_k \]
for $j \in \cat{J}, k \in \cat{K}$.
\end{definition}

We now show that Spanier-Whitehead duality commutes with composition products. For this we need the following lemma.

\begin{lemma} \label{lem:comprod-hocolim}
Let $M: \cat{J} \to \spectra^{\mathsf{\Sigma}}$ be a diagram of termwise-cofibrant symmetric sequences of spectra and let $N$ be any termwise-cofibrant symmetric sequence of spectra. Then we have a natural equivalence
\[ \hocolim_{j \in \cat{J}} (M_j \circ N) \homeq (\hocolim_{j \in \cat{J}} M_j) \circ N \]
If $\cat{J}$ is \emph{filtered}, then there is also a natural equivalence
\[ \hocolim_{j \in \cat{J}} (N \circ M_j) \homeq N \circ (\hocolim_{j \in \cat{J}} M_j). \]
In other words, composition product commutes with filtered homotopy colimits in both variables.
\end{lemma}
\begin{proof}
The first claim follows from the fact that homotopy colimits commute with smash products. For the second, we have
\[ N \circ (\hocolim M_j)(r) \homeq \Wdge \hocolim_{j_1,\dots,j_k \in \cat{J}} N(k) \smsh M_{j_1}(r_1) \smsh \dots \smsh M_{j_k}(r_k). \]
There is a natural map from
\[ \hocolim_{j \in \cat{J}} N(k) \smsh M_j(r_1) \smsh \dots \smsh M_j(r_k) \]
to this given by the diagonal on $\cat{J}$. But since $\cat{J}$ is filtered, the set of $k$-tuples of the form $(j,\dots,j)$ is cofinal in $\cat{J}^k$ and so this has an inverse (up to weak equivalence).
\end{proof}

\begin{lemma} \label{lem:dualcomprod}
Let $M$ and $N$ be directly-dualizable pro-symmetric sequences in $\spectra$. Then $M \circ N$ is also directly-dualizable, and there is a natural equivalence
\[ \widetilde{\dual(M)} \circ \widetilde{\dual(N)} \homeq \dual(M \circ N). \]
Here $\widetilde{\dual(-)}$ denotes a termwise-cofibrant replacement for the symmetric sequence $\dual(-)$.
\end{lemma}
\begin{proof}
Let $\cat{J}$ and $\cat{K}$ be the indexing categories for $M$ and $N$ respectively. Then $\cat{J}$ and $\cat{K}$ are cofiltered posets in which every element has finitely many successors. The same is therefore true of $\cat{J} \times \cat{K}$. The directly-dualizable hypothesis also says that each $M_j(r)$ and $N_k(r)$ is cofibrant and homotopy-finite. The same is then true of each $(M \circ N)_{(j,k)}(r)$. Therefore, $M \circ N$ is directly-dualizable.

Now for each $j \in \cat{J},k \in \cat{K}$, we have maps
\[ \widetilde{\Map(M_j(r),S)} \smsh \widetilde{\Map(N_k(n_1),S)} \smsh \dots \smsh \widetilde{\Map(N_k(n_r),S)} \to \Map(M_j(r) \smsh N_k(n_1) \smsh \dots \smsh N_k(n_r),S) \]
and these are weak equivalences because of the finiteness hypotheses. These make up weak equivalences of symmetric sequences
\[ \widetilde{\Map(M_j,S)} \circ \widetilde{\Map(N_k,S)} \weq \Map(M_j \circ N_k,S) \]
Then taking the homotopy colimit over $j \in \cat{J}$ and $k \in \cat{K}$, and combining these maps with Lemma \ref{lem:comprod-hocolim}, we get the required result.
\end{proof}

\begin{cor} \label{cor:pro-comprod-equivalence}
Let $A \weq B$ be a weak equivalence of directly-dualizable pro-symmetric sequences in $\spectra$, and let $C$ be another directly-dualizable symmetric sequence. Then the induced maps
\[ A \circ C \weq B \circ C \]
and
\[ C \circ A \weq C \circ B \]
are weak equivalences.
\end{cor}
\begin{proof}
This follows from Lemmas \ref{lem:pro-equivalence} and \ref{lem:dualcomprod} in the following way. By Lemma \ref{lem:pro-equivalence}, it is sufficient for the first part to check that
\[ \dual(B \circ C) \to \dual(A \circ C) \]
is a weak equivalence of symmetric sequences in $\spectra$. By Lemma \ref{lem:dualcomprod}, this is equivalent to checking that
\[ \widetilde{\dual B} \circ \widetilde{\dual C} \to \widetilde{\dual A} \circ \widetilde{\dual C} \]
is a weak equivalence. This is true because the composition product in spectra preserves weak equivalences of termwise-cofibrant symmetric sequences. The second part is similar.
\end{proof}

We now turn to Spanier-Whitehead duality for cooperads and pro-comodules. The main results are that the Spanier-Whitehead dual of a cooperad has a natural operad structure and that the Spanier-Whitehead dual of a pro-comodule has a natural module structure. In order to construct this module structure, we need to be careful with the homotopy colimits involved in our chosen model for Spanier-Whitehead duality. In particular, these homotopy colimits need to have the required model structures. We verify this by forming these homotopy colimits \emph{in} the category of modules and showing that these are equivalent to the homotopy colimits formed termwise in the category of spectra, and so correctly model the Spanier-Whitehead dual.

\begin{lemma} \label{lem:dual-cooperad}
Let $Q$ be a directly-dualizable reduced cooperad in $\spectra$. Then the symmetric sequence $\dual Q$ given (as in Definition \ref{def:dual-symseq}) by
\[ (\dual Q)(n) := \Map(Q(n),S) \]
has a natural reduced operad structure. If $M$ is a directly-dualizable (right- or left-) $Q$-comodule, then the symmetric sequence $\dual M$ (defined similarly) has a natural (right- or left- respectively) $\dual Q$-module structure.
\end{lemma}
\begin{proof}
The operad composition maps are given by
\[ \begin{split} \Map(Q(k),S) \smsh \Map(Q(n_1),S) &\smsh \dots \smsh \Map(Q(n_k),S) \\
        &\to \Map(Q(k) \smsh Q(n_1) \smsh \dots \smsh Q(n_k),S) \\
        &\to \Map(Q(n),S) \end{split} \]
where the first map comes from the fact that $\Map(-,-)$ is part of a closed symmetric monoidal structure on $\spectra$, and the second is the dual of the cooperad structure map
\[ Q(n) \to Q(k) \smsh Q(n_1) \smsh \dots \smsh Q(n_k). \]
The unit isomorphism (which makes this a \emph{reduced} operad) is given by
\[ S \isom \Map(S,S) \isom \Map(Q(1),S) \]
where the second map is induced by the isomorphism $Q(1) \isom S$. The module structures are defined similarly.
\end{proof}

\begin{definition}[Spanier-Whitehead duals of pro-comodules] \label{def:dual-pro-comodule}
Let $Q$ be a directly-dualizable cooperad in $\spectra$ and let $M$ be a directly-dualizable pro-$Q$-comodule (either right, left or bi-). If $M$ is indexed over $\cat{J}$ then this means each $M_j$ for $j \in \cat{J}$ has the structure of a $Q$-comodule. Then the symmetric sequences $\Map(M_j,S)$ are $\dual Q$-modules, by the same argument as in Lemma \ref{lem:dual-cooperad}. If $\widetilde{\dual Q}$ denotes a $\Sigma$-cofibrant replacement for $\dual Q$, then $\Map(M_j,S)$ become $\widetilde{\dual Q}$-modules.

We now define, as in \ref{def:dual-symseq}, the \emph{Spanier-Whitehead dual} of $M$ to be
\[ \dual M := \hocolim_{j \in \cat{J}} \Map(M_j,S). \]
We choose this homotopy colimit to be formed using the model structure on $\widetilde{\dual Q}$-modules of Theorem \ref{thm:projective-model}. In the Appendix we show (Proposition \ref{prop:filtered-hocolim}) that this is equivalent to the homotopy colimit formed in the category of symmetric sequences. It therefore does represent the termwise Spanier-Whitehead dual of $M$, but retains the additional structure of a $\widetilde{\dual Q}$-module.
\end{definition}

Next we extend the bar construction to pro-modules and check that the homotopical results of \S\ref{sec:bar-homotopy} carry over to this setting.

\begin{definition}[Bar constructions for pro-modules] \label{def:pro-bar}
Let $P$ be an operad in $\spectra$, let $R$ be a pro-right-$P$-module (indexed over $\cat{J}$) and $L$ a pro-left-$P$-module (indexed over $\cat{K}$). We can then form the \emph{two-sided bar construction} $B(R,P,L)$. This is the pro-symmetric sequence indexed over the cofiltered category $\cat{J} \times \cat{K}$ given by
\[ B(R,P,L)_{(j,k)} := B(R_j,P,L_k). \]
Note that this could also be defined as the realization of a simplicial bar construction based on the composition product for pro-symmetric sequences (see Definition \ref{def:pro-comp}).

We can treat the unit symmetric sequence as a pro-object indexed over the trivial category. We then have a one-sided bar construction $B(R,P,\mathsf{1})$ indexed over $\cat{J}$ and given by
\[ B(R,P,\mathsf{1})_j := B(R_j,P,\mathsf{1}). \]
This is clearly a pro-right-$BP$-comodule. We also have $B(\mathsf{1},P,L)$ indexed over $\cat{K}$ given by
\[ B(\mathsf{1},P,L)_k := B(\mathsf{1},P,L_k). \]
This is a pro-left-$BP$-comodule.

If in addition $M$ is a pro-$P$-bimodule indexed on the category $\cat{L}$, then we can form the \emph{bimodule bar construction} $B(R,P,M,P,L)$. This is a pro-symmetric sequence indexed on $\cat{J} \times \cat{L} \times \cat{K}$ given by
\[ B(R,P,M,P,L)_{(j,l,k)} := B(R_j,P,M_l,P,L_k). \]
In particular, we have a pro-$BP$-comodule $B(\mathsf{1},P,M,P,\mathsf{1})$ indexed on $\cat{L}$.
\end{definition}

In order to take the Spanier-Whitehead dual of the bar construction, we need to check that it preserves direct-dualizability.

\begin{lemma} \label{lem:bar-directly-dualizable}
Let $P$ be an operad in $\spectra$, $R$ a pro-right-$P$-module, $L$ a pro-left-$P$-module and $M$ a pro-$P$-bimodule. Suppose that $P$, $R$, $L$ and $M$ are directly-dualizable. Then $B(R,P,L)$ is also directly-dualizable and $B(R,P,M,P,L)$ are directly-dualizable.
\end{lemma}
\begin{proof}
If the indexing categories for $R$ and $L$ are denoted $\cat{J}$ and $\cat{K}$, then $\cat{J}$ and $\cat{K}$ are cofiltered posets in which every element has finitely many successors. It follows that $\cat{J} \times \cat{K}$, which is the indexing category for $B(R,P,L)$, is also a cofiltered poset in which every element has finitely many successors.

By assumption, all the spectra involved in the pro-symmetric sequences $R,P,L$ are cofibrant. In the proof of Proposition \ref{prop:bar-invariance} we showed that in this case the terms involved in $B(R,P,L)$ are the realizations of Reedy cofibrant simplicial objects, hence are themselves cofibrant spectra by \cite[18.6.7]{hirschhorn:2003}.

Finally, we must check that the spectra involved in $B(R,P,L)$ are homotopy-finite. It is sufficient to show this when $R$ and $L$ are ordinary (not pro-) $P$-modules. The spectrum $B(R,P,L)(n)$ is then the geometric realization of the simplicial spectrum $B_{\bullet}(R,P,L)(n)$. This simplicial spectrum only has nondegenerate simplices up to degree $n$ (since $P$ is a reduced operad), and each spectrum of simplices is homotopy-finite. Therefore we can write $B(R,P,L)(n)$ as a finite homotopy colimit of finite cell spectra. It follows that $B(R,P,L)(n)$ is homotopy-finite.

We have now checked that $B(R,P,L)$ is directly-dualizable, and a similar argument holds for $B(R,P,M,P,L)$.
\end{proof}

\begin{prop} \label{prop:bar-pro-invariance}
Let $P$ be an directly-dualizable operad in $\spectra$ and let $R \weq R'$ and $L \weq L'$ be weak equivalences between directly-dualizable right and left pro-$P$-modules respectively. Then the induced map
\[ B(R,P,L) \to B(R',P,L') \]
is a weak equivalence of pro-symmetric sequences.
\end{prop}
\begin{proof}
For any $k$, the map
\[ R \circ P^k \circ L(n) \to R' \circ P^k \circ L'(n) \]
is a weak equivalence of pro-spectra by Corollary \ref{cor:pro-comprod-equivalence}.
We therefore obtain a levelwise weak equivalence of simplicial pro-spectra
\[ B_{\bullet}(R,P,L)(n) \to B_{\bullet}(R',P,L')(n). \]
As in the proof of Proposition \ref{prop:bar-invariance}, it is now sufficient to show that each of these simplicial bar constructions is a cofibrant in the Reedy model structure on simplicial pro-spectra. Following the same argument as in \ref{prop:bar-invariance}, the latching maps for $B_{\bullet}(R,P,L)(n)$ are given by inclusions of wedge summands of pro-spectra. It is therefore sufficient to show that each of these wedge summands is a cofibrant pro-spectrum. This is true since, by hypothesis, all the spectra making up the pro-symmetric sequences $R$, $P$ and $L$ are cofibrant.
\end{proof}

We also record the following lemma.

\begin{lemma} \label{lem:bar-hocolim}
Let $P$ be a reduced operad in $\spectra$ and let $R: \cat{J} \to \mathsf{Mod}_{\mathsf{right}}(P)$ and $L: \cat{K} \to \mathsf{Mod}_{\mathsf{left}}(P)$ be filtered diagrams of right and left $P$-modules respectively. Then we have an equivalence
\[ \hocolim_{(j,k) \in \cat{J} \times \cat{K}} B(R_j,P,L_k) \homeq B\left(\hocolim_{j \in \cat{J}} R_j, P, \hocolim_{k \in \cat{K}} L_k \right). \]
In other words, the bar construction commutes with filtered homotopy colimits in both module variables.
\end{lemma}
\begin{proof}
This follows from Lemma \ref{lem:comprod-hocolim} and the fact that geometric realization of simplicial spectra commutes with filtered homotopy colimits.
\end{proof}

We also need to extend the $\Ext$-objects for modules of \S\ref{sec:ext} to pro-modules. For a pro-$P$-module $M$ and ordinary $P$-module $M'$, we define objects $\Ext(M,M')$. Applying our previous $\Ext$-construction levelwise, we obtain an ind-spectrum (or ind-simplicial set), which we can identify with an actual spectrum (or simplicial set) via the homotopy colimit (which in our cases represents the Quillen equivalence between ind-spectra and spectra of Theorem \ref{thm:pro-spectra}). These new $\Ext$-objects then have the homotopical properties one would hope for.

We start by describing mapping objects for individual pro-spectra.

\begin{definition}[Mapping objects for pro-spectra] \label{def:map-pro-spectra}
Let $X$ be a directly-dualizable pro-spectrum indexed on the cofiltered category $\cat{J}$, and let $Y$ be any spectrum. We then define
\[ \Map(X,Y) := \hocolim_{j \in \cat{J}^{op}} \Map(X_j,Y). \]
\end{definition}

\begin{definition}[Mapping objects for pro-modules] \label{def:pro-ext-modules}
Let $M$ be a pro-symmetric sequence, or pro-$P$-module (left-, right- or bi-) for some fixed reduced operad $P$ in $\spectra$. Let $N$ be a (non-pro-) symmetric sequence or $P$-module respectively. Suppose that $M$ is directly-dualizable and is indexed on the cofiltered poset $\cat{J}$. We then define
\[ \Map_{\mathsf{\Sigma}}(M,N) := \hocolim_{j \in \cat{J}^{op}} \Map_{\mathsf{\Sigma}}(M_j,N) \]
\[ \Map^{\mathsf{right}}_{P}(M,N) := \hocolim_{j \in \cat{J}^{op}} \Map^{\mathsf{right}}_{P}(M_j,N) \]
\[ \Hom^{\mathsf{left}}_{P}(M,N) := \hocolim_{j \in \cat{J}^{op}} \Hom^{\mathsf{left}}_{P}(M_j,N) \]
\[ \Hom^{\mathsf{bi}}_{P}(M,N) := \hocolim_{j \in \cat{J}^{op}} \Hom^{\mathsf{bi}}_{P}(M_j,N) \]
accordingly.

Similarly, in the module cases, we define $\Ext$-objects:
\[ \Ext^{\mathsf{right}}_{P}(M,N) := \hocolim_{j \in \cat{J}^{op}} \Ext^{\mathsf{right}}_{P}(M_j,N) \isom \Map^{\mathsf{right}}_{P}(B(M,P,P),N) \]
\[ \Ext^{\mathsf{left}}_{P}(M,N) := \hocolim_{j \in \cat{J}^{op}} \Ext^{\mathsf{left}}_{P}(M_j,N) \isom \Map^{\mathsf{left}}_{P}(B(P,P,M),N) \]
\[ \Ext^{\mathsf{bi}}_{P}(M,N) := \hocolim_{j \in \cat{J}^{op}} \Ext^{\mathsf{bi}}_{P}(M_j,N) \isom \Map^{\mathsf{bi}}_{P}(B(P,P,M,P,P),N) \]
\end{definition}

We conclude this section by describing the homotopy-invariance properties of these generalized mapping objects. In this case, we prove only the properties that we need later. In particular, we consider only mapping objects of the form $\Map_{P}(A,B)$ where the symmetric sequence $B$ has only finitely-many nontrivial terms.

\begin{lemma} \label{lem:pro-symseq-invariance}
Let $A \weq A'$ be a weak equivalence between levelwise $\Sigma$-cofibrant directly-dualizable pro-symmetric sequences, and let $B$ be a \emph{truncated} symmetric sequence (i.e. there exists $N$ such that $B(n) = *$ for $n > N$). Then the induced map
\[ \Map_{\mathsf{\Sigma}}(A',B) \to \Map_{\mathsf{\Sigma}}(A,B) \]
is a weak equivalence in $\spectra$.
\end{lemma}
\begin{proof}
Suppose that $A$ is indexed on the cofiltered category $\cat{J}$. We have
\[ \Map_{\mathsf{\Sigma}}(A,B) = \hocolim_{j \in \cat{J}^{op}} \prod_{n=1}^{N} \Map(A_j(n),B(n))^{\Sigma_n}. \]
Now $A_j(n)$ is $\Sigma_n$-cofibrant, so the fixed-point object $\Map(A_j(n),B(n))^{\Sigma_n}$ is equivalent to the homotopy fixed points $\Map(A_j(n),B(n))^{h\Sigma_n}$
Now $(E \Sigma_n)_+ \smsh A_j(n)$ is homotopy-finite (since $A$ is directly dualizable) so
\[ \Map(A_j(n),B(n)) \homeq \Map(A_j(n),S) \smsh B(n) \]
by \cite[III.7]{elmendorf/kriz/mandell/may:1997}. Now the filtered homotopy colimit commutes with the finite product and the homotopy fixed points, so we have
\[ \Map_{\mathsf{\Sigma}}(A,B) \homeq \prod_{n=1}^{N} \left[ \hocolim_{j \in \cat{J}^{op}} \Map(A_j(n),S) \smsh B(n) \right]^{h\Sigma_n} \]
which is the same as
\[ \prod_{n = 1}^{N} \left[ \dual A(n) \smsh B(n) \right]^{h\Sigma_n}. \]
Similarly we have
\[ \Map_{\mathsf{\Sigma}}(A',B) \homeq \prod_{n = 1}^{N} \left[ \dual A'(n) \smsh B(n) \right]^{h\Sigma_n}. \]
Now by Lemma \ref{lem:pro-equivalence}, the map
\[ \dual A'(n) \to \dual A(n) \]
is a weak equivalence in $\spectra$ and so the induced map
\[ \prod_{n = 1}^{N} \left[ \dual A'(n) \smsh B(n) \right]^{h\Sigma_n} \to \prod_{n = 1}^{N} \left[ \dual A(n) \smsh B(n) \right]^{h\Sigma_n} \]
is an equivalence as claimed.
\end{proof}

\begin{lemma} \label{lem:pro-map-invariance}
Let $P$ be a directly-dualizable reduced $\Sigma$-cofibrant operad in $\spectra$ and let $M \weq M'$ be a weak equivalence between levelwise $\Sigma$-cofibrant directly-dualizable pro-$P$-modules (either right-, left- or bi-), and let $N$ be a $P$-module that is truncated as a symmetric sequence (see Lemma \ref{lem:pro-symseq-invariance}). Then the induced map
\[ \Ext_{P}(M',N) \to \Ext_{P}(M,N) \]
is a weak equivalence in either $\spectra$ or $\sset$ as appropriate.
\end{lemma}
\begin{proof}
We prove the right module case, with the others being similar. As in Remark \ref{rem:ext-tot}, we can write $\Ext_{P}(M,N)$ as
\[ \Tot \Map_{\mathsf{\Sigma}}(M \circ P^{\bullet},N). \]
If $M \to M'$ is a weak equivalence, then so is each map $M \circ P^k \to M' \circ P^k$ (by Corollary \ref{cor:pro-comprod-equivalence}), and hence so is each
\[ \Map_{\mathsf{\Sigma}}(M' \circ P^k,N) \to \Map_{\mathsf{\Sigma}}(M \circ P^k,N). \]
The map
\[ \Ext_{P}(M',N) \to \Ext_{P}(M,N) \]
is the totalization of a levelwise weak equivalence of cosimplicial spectra (or simplicial sets in the left- and bimodule cases). We noted in Remark \ref{rem:ext-tot} that these cosimplicial objects are Reedy fibrant, so it follows by Proposition \ref{prop:reedy} that this map is a weak equivalence.
\end{proof}

\part{Functors of spectra}

We now turn to the main goal of this paper: to construct new models for Goodwillie derivatives that have appropriate operad and module structures. We start by concentrating on functors from spectra to spectra since these form the basis for all our constructions. Here is a brief summary of the next few sections:
\begin{itemize}
  \item we show that (at least for finite cell functors) the Goodwillie derivatives of $F: \spectra \to \spectra$ are the Spanier-Whitehead duals of certain `natural transformation objects' of the form $\Nat(FX,X^{\smsh n})$ (\S\ref{sec:nat});
  \item we construct `composition maps' that relate our models for the derivatives of $FG$ to the composition product of the corresponding models for $F$ and $G$, for functors $F,G: \spectra \to \spectra$ (\S\ref{sec:comp-maps}). These maps form the basis of all the operad and module structures that we later produce;
  \item we prove that the composition maps of the previous section are equivalences, thus establishing the chain rule for functors from spectra to spectra (\S\ref{sec:chainrule}). As well as being of interest in its own right, the chain rule for spectra is a key tool in our construction of module structure, and proof of the chain rule, for functors to or from spaces;
  \item we show that our composition maps give rise to operad structures on the duals of the derivatives of an appropriate comonad (\S\ref{sec:comonads}).
\end{itemize}

\section{Models for Goodwillie derivatives of functors of spectra} \label{sec:nat}

Our new models for Goodwillie derivatives are based on an enrichment of the functor category $[\finspec,\spectra]$ over $\spectra$, that is, the existence of a spectrum of natural transformations between two functors. We start by describing this.

\begin{definition}[Natural transformation objects] \label{def:nattrans}
Let $F$ and $G$ be pointed simplicial functors $\finspec \to \spectra$. We define the \emph{spectrum of natural transformations from $F$ to $G$} by the formula
\[ \Nat(F,G) := \lim \left( \prod_{K \in \finspec} \Map(FK,GK) \rightrightarrows \prod_{K,K' \in \finspec} \Map(\spectra(K,K') \smsh FK,GK') \right). \]
One of these maps is given by the simplicial structure on $F$ using:
\[ \spectra(K,K') \smsh FK \to FK' \]
and the other by the simplicial structure on $G$ using:
\[ GK \to \Map(\spectra(K,K'),GK'). \]
The objects $\Nat(F,G)$ then define an enrichment of $[\finspec,\spectra]$ over $\spectra$.
\end{definition}

\begin{lemma}[Strong Yoneda Lemma] \label{lem:yoneda}
There is a natural isomorphism
\[ \Nat(I \smsh \spectra(K,-),G) \isom \Map(I,GK) \]
for any $I \in \spectra$, $K \in \finspec$ and $G \in [\finspec,\spectra]$.
\end{lemma}
\begin{proof}
A map from left to right is given by
\[ \Nat(I \smsh \spectra(K,-),G) \to \Map(I \smsh \spectra(K,K),GK) \to \Map(I,GK) \]
using the unit map $\Delta[0]_+ \to \spectra(K,K)$ for the simplicial structure on $\spectra$. The inverse is the adjoint to
\[ \Map(I,GK) \smsh (I \smsh \spectra(K,X)) \to GK \smsh \spectra(K,X) \to GX. \]
\end{proof}

\begin{lemma} \label{lem:enriched-model}
The objects $\Nat(F,G)$ make the functor category $[\finspec,\spectra]$ into an enriched model category over the symmetric monoidal model category $\spectra$.
\end{lemma}
\begin{proof}
We need to show that if $F \cof F'$ is a cofibration and $G' \fib G$ a fibration then the map
\[ \Nat(F',G') \to \Nat(F,G') \times_{\Nat(F,G)} \Nat(F',G) \]
is a fibration in $\spectra$, and that it is a weak equivalence if either $F \to F'$ or $G' \to G$ is. Since $[\finspec,\spectra]$ is cofibrantly generated, it is sufficient to do this in the case that $F \to F'$ is either one of the generating cofibrations or generating trivial cofibrations in $[\finspec,\spectra]$.

In this case, by the Yoneda Lemma, the given map reduces to one of the form
\[ \tag{*} \Map(I_1,G'K) \to \Map(I_0,G'K) \times_{\Map(I_0,GK)} \Map(I_1,GK) \]
where $I_0 \to I_1$ is either one of the generating cofibrations, or generating trivial cofibrations in $\spectra$. This a fibration since $\spectra$ is a simplicial model category and $G'K \to GK$ is a fibration. If $F \to F'$ is a weak equivalence (i.e. one of the generating trivial cofibrations), then $I_0 \to I_1$ is also. If $G' \to G$ is a weak equivalence, then so is $G'K \to GK$. In either case, the map (*) is a weak equivalence.
\end{proof}

We now, finally, construct our models for the Goodwillie derivatives of a functor $F$ from spectra to spectra. Below we define a sequence of pro-objects associated to $F$. In the rest of this section we then show that these pro-objects are Spanier-Whitehead dual to the derivatives of $F$.

\begin{definition}[Models for the derivatives of functors of spectra] \label{def:derivatives-presented}
Let $F: \finspec \to \spectra$ be a presented cell functor and $n$ a positive integer. Define a pro-spectrum $\der^n(F)$, indexed on the cofiltered category $\mathsf{Sub}(F)^{op}$, by the formula
\[ \der^n(F) := \{ \Nat(CX,X^{\smsh n}) \}_{C \in \mathsf{Sub}(F)}. \]
The pro-object $\der^n(F)$ has a $\Sigma_n$-action that arises from the permutation action of $\Sigma_n$ on $X^{\smsh n}$. Together these pro-objects form a pro-symmetric sequence $\der^*(F)$ with \ord{n} term equal to $\der^n(F)$.
\end{definition}

\begin{definition} \label{def:derivatives}
Now let $F: \finspec \to \spectra$ be any pointed simplicial functor. Then there the cellular replacement $QF$ of $F$ (Definition \ref{def:QF}) comes with a canonical presentation and so we can form the pro-symmetric sequence $\der^*(QF)$ as in Definition \ref{def:derivatives-presented}. This pro-symmetric sequence comprises our models for the (Spanier-Whitehead duals of) the Goodwillie derivatives of $F$.

A natural transformation $\gamma: F \to G$ induces a morphism of pro-symmetric sequences $\der^*(\gamma): \der^*(QG) \to \der^*(QF)$ in the following way:
\begin{itemize}
  \item Recall that a degree $i$ cell $\alpha$ in the presented cell complex $QF$ corresponds to a commutative diagram of the form
  \[ \begin{diagram}
    \node{I_0 \smsh \spectra(K,-)} \arrow{e} \arrow{s} \node{(QF)_{i-1}} \arrow{s} \\
    \node{I_1 \smsh \spectra(K,-)} \arrow{e} \node{F}
  \end{diagram} \]
  Composing this with the diagram
  \[ \begin{diagram}
    \node{(QF)_{i-1}} \arrow{e,t}{Q\gamma} \arrow{s} \node{(QG)_{i-1}} \arrow{s} \\
    \node{F} \arrow{e,t}{\gamma} \node{G}
  \end{diagram} \]
  induced by $\gamma$, we obtain a corresponding cell $\gamma_{\alpha}$ of degree $i$ in $QG$.
  \item Now let $C$ be a finite subcomplex of $QF$. Then the cells $\gamma_{\alpha}$ for $\alpha \in C$ form a finite subcomplex $C'$ of $QG$ and $Q(\gamma)$ restricts to a map
      \[ \gamma_C : C \to C'. \]
      (Note that $C'$ can have fewer cells than $C$, for example if $G = *$.)
  \item We then define the map $\der^n(\gamma)_C: \Nat(C'(X),X^{\smsh n}) \to \Nat(CX,X^{\smsh n})$ induced by $\gamma: C \to C'$ coming from the previous construction. Together these make up a morphism of pro-objects $\der^n(QG) \to \der^n(QF)$. It is easy to see that these maps are $\Sigma_n$-equivariant and form a morphism of pro-symmetric sequences as claimed.
\end{itemize}
This construction makes $\der^*(Q-)$ into a contravariant functor from $[\finspec,\spectra]$ to the category of pro-symmetric sequences in $\spectra$.
\end{definition}

\begin{notation} \label{not:cofibrant}
In order to apply our results on Spanier-Whitehead duality to the pro-objects $\der^n(F)$, we need to establish that they are directly-dualizable in the sense of Definition \ref{def:directly-dualizable}. We already know from Lemma \ref{lem:lattice} that the indexing categories $\mathsf{Sub}(F)^{op}$ have the required properties. In Lemma \ref{lem:homotopy-finite} below, we check that the spectra $\Nat(CX,X^{\smsh n})$ are homotopy-finite. These spectra are, however, not in general cofibrant, so $\der^n(F)$ is usually not directly-dualizable.

We therefore introduce the following notation for levelwise cofibrant replacement of these pro-spectra. We write
\[ \tilde{\der}^n(F) \weq \der^n(F) \]
for such a levelwise cofibrant replacement. The pro-spectrum $\tilde{\der}^n(F)$ then \emph{is} directly-dualizable.

The precise meaning of $\tilde{\der}^n(F)$, however, depends on the context. The pro-symmetric sequence $\der^*(F)$ may have additional structure that we wish to preserve. For example, if $\der^*(F)$ is a module over an operad $P$, then we might use $\tilde{\der}^*(F)$ to denote a (levelwise) projectively-cofibrant replacement in the sense of Proposition \ref{sec:cofibrant}. We saw in Proposition \ref{prop:termwise-cofibrant} that these projectively-cofibrant replacements are always termwise-cofibrant, so this is consistent.

We also use tildes more generally to denote cofibrant replacement. In particular, we write $\widetilde{\Map}$ and $\widetilde{\Nat}$ as cofibrant replacements for $\Map$ and $\Nat$-objects respectively.
\end{notation}

\begin{lemma} \label{lem:homotopy-finite}
Let $C: \finspec \to \spectra$ be a finite cell functor. Then the spectrum $\Nat(CX,X^{\smsh n})$ is homotopy-finite.
\end{lemma}
\begin{proof}
We prove this by induction on the cell structure of $C$. First suppose that $C = I \smsh \spectra(K,-)$ where $I,K \in \finspec$. Then we have
\[ \Nat(CX,X^{\smsh n}) \isom \Map(I,K^{\smsh n}) \]
by the Yoneda Lemma \ref{lem:yoneda}. Now $\Map(I,K^{\smsh n}) \homeq \Map(I,S) \smsh K^{\smsh n}$ by \cite[III.7]{elmendorf/kriz/mandell/may:1997} and $\Map(I,S)$ and $K$ are homotopy-finite. Therefore, $\Map(I,K^{\smsh n})$ is homotopy-finite.

Now suppose that $C$ is obtained from $C'$ by adding a cell via the pushout diagram
\[ \begin{diagram}
  \node{I_0 \smsh \spectra(K,-)} \arrow{e} \arrow{s} \node{C'} \arrow{s} \\
  \node{I_1 \smsh \spectra(K,-)} \arrow{e} \node{C}
\end{diagram} \]
Applying $\Nat(-,X^{\smsh n})$ to this , and using the Yoneda Lemma \ref{lem:yoneda}, we get a pullback diagram
\[ \begin{diagram}
  \node{\Nat(CX,X^{\smsh n})} \arrow{e} \arrow{s} \node{\Map(I_1,K^{\smsh n})} \arrow{s} \\
  \node{\Nat(C'X,X^{\smsh n})} \arrow{e} \node{\Map(I_0,K^{\smsh n})}
\end{diagram} \]
The right-hand vertical map is a fibration by the pushout-product axiom in $\spectra$ because $I_0 \to I_1$ is a cofibration, and $K^{\smsh n}$ is fibrant. Therefore, the above square is a homotopy-pullback by \cite[13.3.8]{hirschhorn:2003}. The inductive hypothesis is that the bottom-left corner is homotopy-finite, and we saw above that the top- and bottom-right corners are homotopy-finite. Therefore, the top-left is homotopy-finite also. It follows by induction that $\Nat(CX,X^{\smsh n})$ is homotopy-finite for any finite cell functor $C$.
\end{proof}

\begin{definition} \label{def:d_*(F)}
Let $F: \finspec \to \spectra$ be a pointed simplicial functor. and let $\tilde{\der}^*(QF)$ be a termwise-cofibrant replacement for $\der^*(QF)$. It follows from Lemma \ref{lem:homotopy-finite} that the pro-symmetric sequence $\tilde{\der}^*(F)$ is directly-dualizable (see Definition \ref{def:directly-dualizable}) and so, following \ref{def:dual-symseq} we set
\[ \begin{split} \der_*(F) :&= \dual \tilde{\der}^*(QF) \\ &= \hocolim_{C \in \mathsf{Sub}(QF)} \Map(\widetilde{\Nat}(CX,X^{\smsh *}),S) \end{split} \]
where $\widetilde{\Nat}(CX,X^{\smsh *})$ denotes a cofibrant replacement of the spectrum $\Nat(CX,X^{\smsh *})$. Note that by Definition \ref{def:pro-morphism}, the construction of $\der_*(F)$ is functorial in $F$.
\end{definition}

The main aim of the rest of this section is to show that $\der_*(F)$ is a model for the symmetric sequence of Goodwillie derivatives of $F$. Our approach to this is to construct a natural transformation $\alpha$ from $F$ to a functor whose \ord{n} Goodwillie derivative is (equivariantly) equivalent to $\der_n(F)$. We then show that $\alpha$ is a $D_n$-equivalence (i.e. becomes an equivalence after applying $D_n$) and hence induces an equivalence of \ord{n} derivatives.

\begin{lemma} \label{lem:F_E}
Let $E$ be a $\Sigma_n$-cofibrant spectrum. Then the functor
\[ F_E(X) := \Map(E,X^{\smsh n})^{\Sigma_n} \]
is an $n$-excisive pointed simplicial homotopy functor in $[\finspec,\spectra]$ with \ord{n} Goodwillie derivative
\[ \der^G_n(F_E) \homeq \Map(E,S) \]
$\Sigma_n$-equivariantly.
\end{lemma}
\begin{proof}
The object $F_E(X)$ is equal to the natural transformation object $\Nat_{\Sigma_n}(E,X^{\smsh n})$ where we are considering a spectrum with $\Sigma_n$-action as a functor $\Sigma_n \to \spectra$. By the argument of the proof of \ref{lem:enriched-model}, these natural transformation objects provide $[\Sigma_n,\spectra]$ with an enrichment over $\spectra$ that makes it into a $\spectra$-model category. Therefore, if $E$ is $\Sigma_n$-cofibrant, it follows that $\Map(E,-)$ preserves all weak equivalences (since all objects in $[\Sigma_n,\spectra]$ are fibrant). But if $X \weq Y$ is a weak equivalence in $\finspec$, then $X^{\smsh n} \to Y^{\smsh n}$ is also a weak equivalence. This tells us that $F_E$ is a homotopy functor, and it is easy to check it is pointed and simplicial.

We now check that $F_E$ is $n$-excisive (see Definition \ref{def:excisive}). Firstly, note that when $E$ is $\Sigma_n$-cofibrant, $\Map(E,-)^{\Sigma_n}$ is equivalent to $\Map(E,-)^{h\Sigma_n}$. This therefore preserves homotopy cartesian squares. Therefore, since we know that $X \mapsto X^{\smsh n}$ is $n$-excisive, it follows that $F_E$ is $n$-excisive.

Finally, we calculate the \ord{n} Goodwillie derivative of $F_E$ using the cross-effect. The construction $\Map(E,-)^{\Sigma_n}$ commutes up to homotopy with cross-effects and the \ord{n} cross-effect of the functor $X \mapsto X^{\smsh n}$ is equivalent to $(X_1,\dots,X_n) \mapsto \prod_{\sigma \in \Sigma_n} X_{\sigma(1)} \smsh \dots \smsh X_{\sigma(n)}$. We therefore have
\[ \creff_n(F_E)(X_1,\dots,X_n) \homeq \Map \left(E, \prod_{\sigma \in \Sigma_n} X_{\sigma(1)} \smsh \dots \smsh X_{\sigma(n)} \right)^{\Sigma_n} \isom \Map(E,X_1 \smsh \dots \smsh X_n). \]
Following through this sequence of equivalences, we see that the symmetry isomorphisms come from the permutation action of $\Sigma_n$ on $X_1 \smsh \dots \smsh X_n$ together with the given action on $E$. We therefore get
\[ \der^G_n(F) = \creff_n(F_E)(S_c,\dots,S_c) \homeq \Map(E,S_c \smsh \dots \smsh S_c) \homeq \Map(E,S) \]
with $\Sigma_n$-action agreeing with that coming from $E$.
\end{proof}

\begin{corollary} \label{cor:derivatives}
Let $F: \finspec \to \spectra$ be a pointed simplicial functor and let $\tilde{\der}^n(QF)$ denote a levelwise $\Sigma_n$-cofibrant replacement for the pro-object $\der^n(QF)$. Then the functor
\[ \Psi_nF(X) := \Map(\tilde{\der}^n(QF), X^{\smsh n})^{\Sigma_n} \]
is an $n$-excisive pointed simplicial homotopy functor with \ord{n} Goodwillie derivative equivalent to $\der_n(F)$ (see Definition \ref{def:derivatives}).
\end{corollary}
\begin{proof}
The functor $\Psi_nF$ is the homotopy colimit of the functors $F_{\widetilde{\Nat}(CX,X^{\smsh n})}$, defined as in Lemma \ref{lem:F_E}. Since taking derivatives commutes with filtered homotopy colimits, the \ord{n} derivative of $\Psi_nF$ is equivalent to
\[ \hocolim_{C \in \mathsf{Sub}(QF)} \Map(\widetilde{\Nat}(CX,X^{\smsh n}),S) \]
which is the definition of $\der_n(F)$.
\end{proof}

We now construct a natural transformation from $F$ to $\Psi_nF$.

\begin{definition} \label{def:spectra-map}
For a finite cell functor $C: \finspec \to \spectra$, there is a $\Sigma_n$-equivariant evaluation map
\[ C(X) \smsh \Nat(CX,X^{\smsh n}) \to X^{\smsh n}. \]
This is adjoint to a map
\[ C(X) \to \Map(\Nat(CX,X^{\smsh n}),X^{\smsh n})^{\Sigma_n}. \]
Composing with a $\Sigma_n$-cofibrant replacement for $\Nat(CX,X^{\smsh n})$, we get
\[ \psi_C: C(X) \to \Map(\widetilde{\Nat}(CX,X^{\smsh n}), X^{\smsh n})^{\Sigma_n}. \]
Now let $F: \finspec \to \spectra$ be any presented cell functor. By Corollary \ref{cor:hocolim}, $F$ is equivalent to the homotopy colimit of the finite subcomplexes of $QF$. Therefore, taking the homotopy colimit over $C \in \mathsf{Sub}(QF)$ of the maps $\psi_C$, we get
\[ \psi: F(X) \to \Psi_nF(X). \]
Strictly speaking, the source of this map is (naturally) weakly equivalent to $F$, not equal to it, but for simplicity of notation, we write it just as $F$.
\end{definition}

\begin{prop} \label{prop:models-correct}
Let $F: \finspec \to \spectra$ be a pointed simplicial homotopy functor. Then the map
\[ \psi: F(X) \to \Map(\tilde{\der}^n(QF), X^{\smsh n})^{\Sigma_n} \]
of Definition \ref{def:spectra-map} is a $D_n$-equivalence (i.e. becomes an equivalence after taking $D_n$).
\end{prop}
\begin{proof}
Since $D_n$ commutes with filtered homotopy colimits, it is sufficient to prove that for any finite cell functor $C$, the map
\[ \psi: C(X) \to \Map(\widetilde{\Nat}(CX,X^{\smsh n}),X^{\smsh n})^{\Sigma_n} \]
is a $D_n$-equivalence. Here $\widetilde{\Nat}(CX,X^{\smsh n})$ denotes a $\Sigma_n$-cofibrant replacement for the natural transformation object $\Nat(CX,X^{\smsh n})$.

We do this by induction on the cell structure of $C$. Suppose that the following pushout diagram represents the attaching of a cells to $C'$ to get $C$:
\[ \tag{*} \begin{diagram}
  \node{I_0 \smsh \spectra(K,-)} \arrow{e} \arrow{s} \node{C'} \arrow{s} \\
  \node{I_1 \smsh \spectra(K,-)} \arrow{e} \node{C}
\end{diagram} \]
For any $X$, the map
\[ I_0 \smsh \spectra(K,X) \to I_1 \smsh \spectra(K,X) \]
is a cofibration in $\spectra$ (because $\spectra(K,X)$ is a simplicial set and hence cofibrant). The above diagram is therefore objectwise a homotopy pushout (and also a homotopy pushout in the functor category).

Applying $\widetilde{\Nat}(-,X^{\smsh n})$ to the above diagram, we get a homotopy pullback square in $\spectra$ that looks like:
\[ \begin{diagram}
  \node{\widetilde{\Nat}(CX,X^{\smsh n})} \arrow{e} \arrow{s} \node{\widetilde{\Map}(I_1,K^{\smsh n})} \arrow{s} \\
  \node{\widetilde{\Nat}(C'X,X^{\smsh n})} \arrow{e} \node{\widetilde{\Map}(I_0,K^{\smsh n})}
\end{diagram} \]
Now this is also a homotopy pushout square (since homotopy pushouts and pullbacks agree in $\spectra$), and so taking $\Map(-,X^{\smsh n})^{\Sigma_n}$, we get an objectwise homotopy pullback (and hence also homotopy pushout) square
\[ \tag{**} \begin{diagram}
  \node{\Map(\widetilde{\Map}(I_0,K^{\smsh n}),X^{\smsh n})^{\Sigma_n}} \arrow{e} \arrow{s}
    \node{\Map(\widetilde{\Nat}(C'X,X^{\smsh n}),X^{\smsh n})^{\Sigma_n}} \arrow{s} \\
  \node{\Map(\widetilde{\Map}(I_1,K^{\smsh n}),X^{\smsh n})^{\Sigma_n}} \arrow{e}
    \node{\Map(\widetilde{\Nat}(CX,X^{\smsh n}),X^{\smsh n})^{\Sigma_n}}
\end{diagram} \]
The map $\psi$ gives us a morphism from the original homotopy pushout square (*) to (**). By induction we may assume that $\psi$ is a $D_n$-equivalence on the top-left corner. Below we show that $\psi$ is a $D_n$-equivalence on the left-hand corners. Then, since $D_n$ commutes with homotopy pushouts (for spectrum-valued functors), it follows that $\psi$ is a $D_n$-equivalence on the bottom-right corner. This completes the induction and proves Proposition \ref{prop:models-correct} for all finite cell functors, and hence all presented cell functors.

We have now reduced the Proposition to the case where $F$ is of the form $I \smsh \spectra(K,-)$ for $I,K \in \finspec$. This means we need to show that the map
\[ \psi: I \smsh \spectra(K,X) \to \Map(\widetilde{\Map}(I,K^{\smsh n}),X^{\smsh n})^{\Sigma_n} \]
is a $D_n$-equivalence. We show that this map induces an equivalence on multilinearized \ord{n} cross-effects. This is sufficient by \cite[6.1]{goodwillie:2003}.

Recall that for functors from spectra to spectra, there is an equivalence between the \ord{n} cross-effect and the \ord{n} \emph{co-}cross-effect, where the latter is defined dually, as the total homotopy cofibre of the cube given by applying the functor to products. For the source of the map $\psi$, we therefore have
\[ \creff_n(I \smsh \spectra(K,-))(X_1,\dots,X_n) \homeq \thocofib\left( I \smsh \spectra(K,X_{i_1} \times \dots \times X_{i_r}) \right). \]
where the maps in the cube on the right-hand side are determined by the inclusions $* \to X_i$. To calculate this total cofibre we take cofibres in each direction in the cube successively. The maps is this cube are determined by inclusions of simplicial sets
\[ \spectra(K,X) \to \spectra(K,X \times X') \isom \spectra(K,X) \times \spectra(K,X') \]
and so are cofibrations. We may therefore take strict cofibres instead of homotopy cofibres. But then the total cofibre is equivalent to collapsing the subspace of
\[ \spectra(K,X_1) \times \dots \times \spectra(K,X_n) \]
in which any of the terms in the product is the basepoint, in other words, forming the smash product. This implies that
\[ \creff_n(I \smsh \spectra(K,-))(X_1,\dots,X_n) \homeq I \smsh \spectra(K,X_1) \smsh \dots \smsh \spectra(K,X_n). \]

For the target of $\psi$, we calculated the cross-effect in the proof of Lemma \ref{lem:F_E}. The map induced by $\psi$ on \ord{n} cross-effects is then of the form
\[ I \smsh \spectra(K,X_1) \smsh \dots \smsh \spectra(K,X_n) \to \Map(\widetilde{\Map}(I,K^{\smsh n}),X_1 \smsh \dots \smsh X_n). \]

The right-hand side here is already multilinear. For the left-hand side, notice that
\[ I \smsh \spectra(K,X_i) \homeq I \smsh \Sigma^\infty\Omega^\infty \Map(K,X_i) \]
which has multilinearization
\[ I \smsh \Sigma^\infty \Omega^\infty \Map(K,X_i) \to I \smsh \Map(K,X_i). \]
The map induced by $\psi$ on multilinearized \ord{n} cross-effects is therefore
\[ I \smsh \Map(K,X_1) \smsh \dots \smsh \Map(K,X_n) \to \Map(\widetilde{\Map}(I,K^{\smsh n}),X_1 \smsh \dots \smsh X_n). \]
But this is an equivalence whenever $I,K,X_1,\dots,X_n$ are finite cell spectra, which they are in our case. It follows that the original map $\psi$ is a $D_n$-equivalence as claimed.
\end{proof}

Putting this all together we deduce that the spectra $\der_nF$ of Definition \ref{def:derivatives} are naturally equivalent to the Goodwillie derivatives of $F$.
\begin{theorem} \label{thm:derivatives}
Let $F: \finspec \to \spectra$ be a pointed simplicial homotopy functor. Then there is a natural equivalence of symmetric sequences of spectra
\[ \der_*F \homeq \der^G_*(F). \]
\end{theorem}
\begin{proof}
The map $\psi: F \to \Psi_nF$ of Definition \ref{def:spectra-map} is a $D_n$-equivalence and so induces an equivalence of \ord{n} Goodwillie derivatives. Therefore, by Corollary \ref{cor:derivatives}, the \ord{n} derivative of $F$ is equivalent to $\der_nF$.
\end{proof}

\begin{example} \label{ex:SigOm}
The most important example of Theorem \ref{thm:derivatives} is when $F = \Sigma^\infty \Omega^\infty$ (see Definition \ref{def:suspec}). In that case we have
\[ \Sigma^\infty \Omega^\infty(X) = S_c \smsh \spectra(S_c,X) \]
which is a finite cell functor (with an obvious presentation consisting of a single cell). We therefore have
\[ \der^n(\Sigma^\infty \Omega^\infty) \isom \Nat(S_c \smsh \spectra(S_c,-),X^{\smsh n}) \isom \Map(S_c,S_c^{\smsh n}) \homeq S. \]
This recovers the now well-known fact that the \ord{n} derivative of $\Sigma^\infty \Omega^\infty$ is equivalent to the sphere spectrum $S$ with the trivial $\Sigma_n$-action.
\end{example}

\section{Composition maps for Goodwillie derivatives} \label{sec:comp-maps}

Our reason for introducing new models for Goodwillie derivatives is that these models admit `composition maps' that relate the derivatives of $F$ and $G$ to the derivatives of $FG$, where $F$ and $G$ are functors from spectra to spectra.

These composition maps take the following form. For presented cell functors $F,G: \spectra \to \spectra$, we construct a morphism of pro-symmetric sequences of the form
\[ \mu: \der^*(F) \circ \der^*(G) \to \der^*(Q(FG)). \]
Taking the Spanier-Whitehead dual of this map gives us, up to weak equivalence, a map
\[ \der_*(F) \circ \der_*(G) \to \der_*(FG). \]
In \S\ref{sec:chainrule}, we prove that this map is an equivalence which is the form of our `chain rule' for functors of spectra.

This section is the technical heart of this paper because we use these composition maps to obtain all our operad and module structures on Goodwillie derivatives. We start with some important factorization lemmas.

\begin{lemma} \label{lem:key}
Let $C:\cat{C}^{\mathsf{fin}} \to \cat{D}$ be a cell functor with $\cat{C},\cat{D}$ equal to either $\sset$ or $\spectra$. Let $X$ be a presented cell complex in $\cat{C}$ and $A$ a finite cell complex in $\cat{D}$. Suppose we have a map
\[ A \to C(X) \]
in $\cat{D}$. Then there is a finite subcomplex $L \subset X$ such that the above map factors as
\[ A \to C(L) \to C(X). \]
\end{lemma}
\begin{proof}
Fix a presentation of the cell functor $C$. Recall from Remark \ref{rem:cell-F(X)} that $C(X)$ has a natural cell structure in which the cells correspond to pairs $(\alpha,\epsilon)$ where $\alpha$ is a cell in $C$ and $\epsilon: K_{\alpha} \smsh \Delta[n]_+ \to X$ is a nondegenerate simplex in the simplicial set $\cat{C}(K_{\alpha},X)$. By \ref{prop:cell-complex}(2) the map $A \to C(X)$ factors via a finite subcomplex $B \subset C(X)$. Suppose that the cells in $B$ correspond to pairs $(\alpha_1,\epsilon_1),\dots,(\alpha_r,\epsilon_r)$.

Now the map $\epsilon_j:\Delta^k_+ \smsh K_{\alpha_j} \to X$ itself factors via a finite subcomplex $L_j \subset X$ by \ref{prop:cell-complex} again, since $\Delta^k_+ \smsh K_{\alpha_j}$ is a finite cell complex. Take $L$ to be a finite subcomplex of $X$ containing all of $L_1,\dots,L_r$.

We now claim that $C(L)$ is a subcomplex of $C(X)$ with respect to the cell structure of Remark \ref{rem:cell-F(X)}. In particular, we claim that $C(L)$ is the subcomplex whose cells are those for which the corresponding map $\epsilon: \Delta^k_+ \smsh K_{\alpha} \to X$ factors via $L$. This is true because $L \to X$ is a monomorphism by Proposition \ref{prop:cell-complex}(1), therefore any such map $\epsilon$ can factor via $L$ in at most one way. This produces a bijection between the cells of $C(L)$ and the cells in the claimed subcomplex of $C(X)$ which determines the claimed isomorphism.

Now by the construction of $L$, the subcomplex $B$ of $C(X)$ is contained in $C(L)$. Therefore the map $A \to C(X)$ factors via $C(L)$ as claimed.
\end{proof}

The key step in the construction of the map (*) is then the following factorization result which depends on Lemma \ref{lem:key}.

\begin{lemma} \label{lem:factorization}
Let $F,G,H$ be presented cell complexes (between the categories $\sset$ and $\spectra$) such that the composite $FG$ makes sense, and suppose we are given a natural transformation $\alpha: H \to FG$.

Let $E$ be a finite subcomplex of $H$. Then there are finite subcomplexes $C \subset F$ and $D \subset G$ such that $\alpha$ restricts to a map $E \to CD$. In other words we have a commutative diagram:
\[ \begin{diagram}
  \node{E} \arrow{e} \arrow{s,l}{\alpha|_E} \node{H} \arrow{s,r}{\alpha} \\
  \node{CD} \arrow{e} \node{FG}
\end{diagram} \]
where the top and bottom maps are given by inclusions of subcomplexes. The map $E \to CD$ is unique up to inclusions of $C$ and $D$ into other finite subcomplexes of $F$ and $G$.
\end{lemma}
\begin{proof}
We prove this by induction on the number of cells in $E$. For $E = *$, we can take $C = *$ and $D = *$, so it is true in the base case. Now suppose that $E$ is obtained by adding a cell to $E'$. By induction we can assume that there are finite subcomplexes $C'$ and $D'$ or $F$ and $G$ respectively, such that $\alpha$ restricts to a map $E' \to C'D'$. We then have the following diagram
\[ \begin{diagram}
  \node{I_0 \smsh \spectra(K,-)} \arrow{s} \arrow{e} \node{E'} \arrow{e} \arrow{s} \node{C'D'} \arrow{s} \\
  \node{I_1 \smsh \spectra(K,-)} \arrow{e} \node{E} \arrow{e} \node{FG}
\end{diagram} \]
The left-hand square is a pushout and so, to obtain $E \to CD$, it is sufficient to construct an appropriate factorization
\[ I_1 \smsh \spectra(K,-) \to CD \to FG \]
where $C$ is a finite subcomplex of $F$ containing $C'$ and $D$ is a finite subcomplex of $G$ containing $D'$. We now show how to do this.

The above diagram is adjoint, by the Yoneda Lemma \ref{lem:weak-yoneda}, to
\[ \tag{*} \begin{diagram}
  \node{I_0} \arrow{e} \arrow{s} \node{C'(D'(K))} \arrow{s} \\
  \node{I_1} \arrow{e} \node{F(G(K))}
\end{diagram} \]
Now recall that $D'(K)$ has a natural cell structure by Lemma \ref{lem:cell-F(X)} and so by Lemma \ref{lem:key}, the top map in this diagram factors as
\[ I_0 \to C'(L_0) \to C'(D'(K)) \]
for some finite subcomplex $L_0 \subset D'(K)$. Similarly, the bottom map factors as
\[ I_1 \to F(L_1) \to F(G(K)) \]
for a finite subcomplex $L_1 \subset G(K)$. Since $L_0$ is also a finite subcomplex of $G(K)$ (as $D'$ is a subcomplex of $G$), we can assume that $L_1$ contains $L_0$.

Altogether we get a commutative diagram
\[ \begin{diagram}
  \node{I_0} \arrow{s} \arrow{e} \node{C'(L_0)} \arrow{s} \arrow{e} \node{C'(D'(K))} \arrow{s} \\
  \node{I_1} \arrow{e} \node{F(L_1)} \arrow{e} \node{F(G(K))}
\end{diagram} \]
The first square is adjoint to
\[ \begin{diagram}
  \node{I_0 \smsh \spectra(L_1,-)} \arrow{s} \arrow{e} \node{C'} \arrow{s} \\
  \node{I_1 \smsh \spectra(L_1,-)} \arrow{e} \node{F}
\end{diagram} \]
By Proposition \ref{prop:subcomplexes}(2), the bottom map factors via a finite subcomplex $C$ of $F$ which we can assume contains $C'$. We also have a diagram
\[ \begin{diagram}
  \node{L_0 \smsh \spectra(K,-)} \arrow{s} \arrow{e} \node{D'} \arrow{s} \\
  \node{L_1 \smsh \spectra(K,-)} \arrow{e} \node{G}
\end{diagram} \]
The bottom map here factors via a finite subcomplex $D$ of $G$ which we can assume contains $D'$. But then we have the necessary factorization
\[ I_1 \smsh \spectra(K,-) \to CD \to FG \]
given by the composite
\[ I_1 \to C(L_1) \to C(D(K)). \]
This completes the induction step and hence the proof of the lemma.
\end{proof}

\begin{definition} \label{def:composition}
Now suppose that $F,G,H$ are presented cell functors in $[\finspec,\spectra]$ and let $\alpha: H \to FG$ be a natural transformation. We define maps of pro-symmetric sequences of the form
\[ \alpha^*: \der^*(F) \circ \der^*(G) \to \der^*(H) \]
as follows.

Definition \ref{def:pro-objects} tells us that, to define a morphism of pro-objects such as this, we have to choose, for each $H \in \mathsf{Sub}(H)$, objects $C \in \mathsf{Sub}(F)$ and $D \in \mathsf{Sub}(G)$, and maps
\[ \alpha^*_E: \Nat(CX,X^{\smsh *}) \circ \Nat(DX,X^{\smsh *}) \to \Nat(EX,X^{\smsh *}) \]
of symmetric sequences that satisfy appropriate compatibility conditions.

To do this, we use Lemma \ref{lem:factorization} to choose $C,D$ such that $\alpha$ restricts to $E \to CD$. We then construct the map $\alpha^*_E$ as composites
\[ \Nat(CX,X^{\smsh *}) \circ \Nat(DX,X^{\smsh *}) \to \Nat(CDX,X^{\smsh *}) \to \Nat(EX,X^{\smsh *}) \]
where the last map is induced by the restriction of $\alpha$, and the first map is given by the composites
\[ \begin{split} \Nat(CX,X^{\smsh k}) \smsh& \Nat(DX,X^{\smsh n_1}) \smsh \dots \smsh \Nat(DX,X^{\smsh n_k}) \\
    & \to \Nat(CX,X^{\smsh k}) \smsh \Nat((DX)^{\smsh k},X^{\smsh n}) \\
    & \to \Nat(C(DX),(DX)^{\smsh k}) \smsh \Nat((DX)^{\smsh k},X^{\smsh n}) \\
    & \to \Nat(C(DX),X^{\smsh n}). \end{split} \]
The first map here smashes together the natural transformations $DX \to X^{\smsh n_i}$. The second takes the natural transformation $CX \to X^{\smsh k}$, Kan extends $C$ to all of $\spectra$, and then applies the resulting map to $DX$ to get a natural transformation $C(DX) \to (DX)^{\smsh k}$. The third map is composition of natural transformations.
\end{definition}

\begin{proposition} \label{prop:composition}
let $F,G,H$ be presented cell functors in $[\finspec,\spectra]$ and let $\alpha: H \to FG$ be a natural transformation. The maps $\alpha^*_E$ of Definition \ref{def:composition} give a well-defined map of pro-symmetric sequences:
\[ \alpha^*: \der^*(F) \circ \der^*(G) \to \der^*(H). \]
\end{proposition}
\begin{proof}
To see that the $\alpha^*_E$ give us a map of pro-objects, we have to check that if $E' \in \mathsf{Sub}(H)$ has a factorization
\[ E' \to C'D' \to FG \]
and $E \subset E'$, then we have a commutative diagram
\[ \begin{diagram}
  \node[2]{\der^*(C') \circ \der^*(D')} \arrow{e,t}{\alpha^*_{E'}} \node{\der^*(E')} \arrow[2]{s} \\
  \node{\der^*(C'') \circ \der^*(D'')} \arrow{ne} \arrow{se} \\
  \node[2]{\der^*(C) \circ \der^*(D)} \arrow{e,t}{\alpha^*_E} \node{\der^*(E)}
\end{diagram} \]
for some $C''$ and $D''$ such that $C,C' \subset C''$ and $D,D' \subset D''$.

To see this, we take $C'' = C \union C'$ and $D'' = D \union D'$. We then have a commutative diagram
\[ \begin{diagram}
  \node{E} \arrow{e} \arrow[2]{s} \node{CD} \arrow{se} \\
  \node[3]{C''D''} \\
  \node{E'} \arrow{e} \node{C'D'} \arrow{ne}
\end{diagram} \]
from which the claim follows.
\end{proof}

\begin{definition} \label{def:mu}
Let $F$ and $G$ be presented cell functors in $[\finspec,\spectra]$. Then the map $\mu: Q(FG) \to FG$ determines a map of pro-symmetric sequences of the form
\[ \mu^*: \der^*(F) \circ \der^*(G) \to \der^*(Q(FG)). \]
Composing with appropriate cofibrant replacement for $\der^*(F)$ and $\der^*(G)$, we obtain a homotopy invariant version of the map $\mu^*$:
\[ \mu^*: \tilde{\der}^*(F) \circ \tilde{\der}^*(G) \to \der^*(Q(FG)). \]
Now for arbitrary pointed simplicial functors $F,G: \spectra \to \spectra$, we can apply the above construction to the presented cell functors $QF$ and $QG$ to obtain a map of pro-symmetric sequences of the form
\[ \mu^*: \tilde{\der}^*(QF) \circ \tilde{\der}^*(QG) \to \der^*(Q((QF)(QG))). \]
Recall that a natural transformation $F \to F'$ induces a morphism of pro-symmetric sequences
\[ \der^*(QF') \to \der^*(QF). \]
Then $\mu^*$ is natural with respect to these maps.
\end{definition}

\section{Chain rule for functors from spectra to spectra} \label{sec:chainrule}

We are now able to state and prove our chain rule for functors from spectra to spectra.

\begin{theorem}[Chain rule for spectra] \label{thm:chainrule}
Let $F,G: \spectra \to \spectra$ be pointed simplicial functors with $F$ a finitary homotopy functor. Then there is a natural equivalence of symmetric sequences of the form:
\[ \tilde{\der}_*(FG) \homeq \tilde{\der}_*(F) \circ \tilde{\der}_*(G) \]
where $\tilde{\der}_*(-)$ denotes a termwise-cofibrant replacement for the symmetric sequence $\der_*(F)$ of Definition \ref{def:d_*(F)}.
\end{theorem}

\begin{proof}
Using Lemma \ref{lem:dualcomprod}, the Spanier-Whitehead dual of the map
\[ \mu^*: \tilde{\der}^*(QF) \circ \tilde{\der}^*(QG) \to \der^*(Q((QF)(QG))) \]
is (up to inverse weak equivalences) a map
\[ \mu_*: \tilde{\der}_*((QF)(QG)) \to \tilde{\der}_*(F) \circ \tilde{\der}_*(G). \]
Since we are assuming that $F$ is finitary, we know that $QF(X) \weq F(X)$ for \emph{all} $X \in \spectra$ (not just the finite cell spectra). It follows that we have weak equivalences
\[ (QF)(QG) \weq F(QG) \weq FG, \]
the last because $F$ is a homotopy functor. Therefore, in the \emph{homotopy category} of symmetric sequences of spectra, we can view $\mu_*$ as a map
\[ \mu_*: \der_*(FG) \to \der_*(F) \circ \der_*(G) \]
(where $\circ$ here denotes the composition product induced on the homotopy category). To prove the Theorem, we show that $\mu_*$ is an isomorphism in the homotopy category. This is equivalent to showing that $\mu^*$ is a weak equivalence of pro-symmetric sequences, and we use both these points of view in this proof.

We start by using the naturality of the construction of $\mu_*$ to make a series of reductions. We also use the fact that we know $\der_*(F)$ is equivalent to the Goodwillie derivatives of $F$, and so enjoys the same properties as those derivatives (see Proposition \ref{prop:deriv-commute}).

First consider the following diagram
\[ \begin{diagram}
  \node{\der_n(FG)} \arrow{e,t}{\mu_n} \arrow{s,l}{\sim} \node{\left(\der_*(F) \circ \der_*(G)\right)(n)} \arrow{s,l}{\sim} \\
  \node{\der_n((P_nF)G)} \arrow{e,t}{\mu_n} \node{\left(\der_*(P_nF) \circ \der_(G)\right)(n)}
\end{diagram} \]
where the vertical maps are induced by $p_nF: F \to P_nF$. The left-hand vertical map is an equivalence by Proposition \ref{prop:calculus}(1). The right-hand vertical map is an equivalence because the \ord{n} term in the composition product $A \circ B$ only depends on the first $n$ terms in the symmetric sequence $A$. This reduces the Theorem to the case where $F$ is $n$-excisive for some $n$.

Now consider the diagrams
\[ \begin{diagram}
  \node{\der_*((D_kF)G)} \arrow{e,t}{\mu_*} \arrow{s} \node{\der_*(D_kF) \circ \der_*(G)} \arrow{s} \\
  \node{\der_*((P_kF)G)} \arrow{e,t}{\mu_*} \arrow{s} \node{\der_*(P_kF) \circ \der_*(G)} \arrow{s} \\
  \node{\der_*((P_{k-1}F)G)} \arrow{e,t}{\mu_*} \node{\der_*(P_{k-1}F) \circ \der_*(G)}
\end{diagram} \]
The left-hand column is a fibre sequence because $(D_kF)G \to (P_kF)G \to (P_{k-1}F)G$ is a fibre sequence, and derivatives preserve fibre sequences (\ref{prop:deriv-commute}). The right-hand column is a fibre sequence by Lemma \ref{lem:comprod-hocolim}. Therefore, by induction, we can reduce the Theorem to the case where $F$ is a homogeneous functor.

Now suppose that $F$ is $r$-homogeneous. Then by Proposition \ref{prop:Dn-formula}, we have
\[ F(X) \homeq (E \smsh X^{\smsh r})_{h\Sigma_r} \]
for some cofibrant spectrum $E$ with $\Sigma_r$-action (namely the \ord{r} derivative of $F$). Taking homotopy orbits in the $F$ variable commutes with both sides of $\mu_*$ (by \ref{prop:deriv-commute} and \ref{lem:comprod-hocolim}). Therefore, we reduce to proving that $\mu_*$ is an isomorphism in the case
\[ F(X) = E \smsh X^{\smsh r}. \]
Finally, since both sides of $\mu_*$ commute with filtered homotopy colimits (again by \ref{prop:deriv-commute} and \ref{lem:comprod-hocolim}), we may assume that $E$ is a finite cell spectrum.

To conclude then, we have reduced the proof of Theorem \ref{thm:chainrule} to the case where
\[ F(X) = E \smsh X^{\smsh r} \]
for some finite cell spectrum $E$ with $\Sigma_r$-action. For the remainder of this proof, we take $F$ to be given by this formula. We also assume without loss of generality that $G$ is a presented cell functor

We now analyze the map
\[ \mu^*: \tilde{\der}^*(QF) \circ \tilde{\der}^*(G) \to \der^*(Q((QF)G)) \]
in this case. First we calculate $\der^*(QF)$.

For each $\sigma \in \Sigma_r$ and each $C \in \mathsf{Sub}(QF)$, we define maps
\[ \Map(E,S) \arrow{e} \Nat(E \smsh X^{\smsh r},X^{\smsh r}) \arrow{e,t}{\sigma_*} \Nat(E \smsh X^{\smsh r},X^{\smsh r}) \arrow{e,t}{\iota^*} \Nat(CX,X^{\smsh r}). \]
The first map here is given by smashing a morphism $E \to S$ with $X^{\smsh r}$, the second by the permutation of $X^{\smsh r}$ determined by $\sigma$. The third map is induced by the composite
\[ \iota: CX \into QF(X) \to F(X) = E \smsh X^{\smsh r}. \]
These maps are clearly consistent with inclusions between the subcomplexes $C$ and so, taking cofibrant replacements, define a map of pro-spectra
\[ \beta: (\Sigma_r)_+ \smsh \widetilde{\Map}(E,S) \to \tilde{\der}^r(QF) \]
where the left-hand side is a pro-object indexed on the trivial category.

We now claim that $\beta$ is an equivalence of pro-spectra, thus giving us a calculation of $\der^r(QF)$. To see this, consider the following diagram
\[ \begin{diagram}
  \node{F(X)} \arrow{e,t}{\psi} \arrow{se} \node{\Map(\tilde{\der}^r(QF),X^{\smsh r})^{\Sigma_r}} \arrow{s,r}{\beta^*} \\
  \node[2]{\Map((\Sigma_r)_+ \smsh \widetilde{\Map(E,S)},X^{\smsh r})^{\Sigma_r}}
\end{diagram} \]
where $\psi$ is as in Definition \ref{def:spectra-map} and the right-hand map is induced by $\beta$. Following the definitions of $\psi$ and $\beta$, we see that the diagonal map is then given by
\[ E \smsh X^{\smsh r} \to \Map(\widetilde{\Map(E,S)},S) \smsh X^{\smsh r} \to \Map(\widetilde{\Map(E,S)},X^{\smsh r}). \]
The first map here is an equivalence since $E$ is a finite cell spectrum, and the second is an equivalence by \cite[III.7]{elmendorf/kriz/mandell/may:1997}. Since $\psi$ is a $D_n$-equivalence, it follows that the map $\beta^*$ in the above diagram is a $D_n$-equivalence. Taking $D_n$ gives us the dual of the map $\beta$ (by \ref{lem:F_E}) and so by \ref{lem:pro-equivalence}, we deduce that $\beta$ is an equivalence.

We now have a calculation of $\der^r(F)$ and, since $F$ is $r$-homogeneous, we know that $\der^k(F) \homeq *$ for $k \neq r$. Therefore, we can reinterpret the map $\mu^*$ as
\[ \Wdge_{n = n_1+\dots+n_k} \widetilde{\Map}(E,S) \smsh \tilde{\der}^{n_1}(G) \smsh \dots \smsh \tilde{\der}^{n_k}(G) \longrightarrow \der^{n}(Q(E \smsh G^{\smsh r})). \]
Here we have included the $\Sigma_r$ factor into the coproduct by making this a coproduct over all \emph{ordered} partitions of the set $\{1,\dots,n\}$ into $r$ pieces (whose sizes are then labelled $n_1,\dots,n_k$). (Recall that in the composition product, we take the coproduct over unordered partitions.)

We now think of the above map as a special case of
\[ \tag{*} \Wdge_{n = n_1+\dots+n_r} \widetilde{\Map}(E,S) \smsh \tilde{\der}^{n_1}(G_1) \smsh \dots \smsh \tilde{\der}^{n_r}(G_r) \to \der^n(Q(E \smsh G_1 \smsh \dots \smsh G_r)) \]
where $G_1,\dots,G_r$ are any presented cell functors. The map we are interested in is the case $G_1 = \dots = G_r = G$.

Let us be explicit about where this map comes from (which amounts to unpacking the definitions of $\mu^*$ and the map $\beta$ of the previous paragraph. Given a finite subcomplex $C \subset Q(E \smsh G_1 \smsh \dots \smsh G_r)$, we can see, by iterated applications of Lemma \ref{lem:factorization}, that the composite
\[ C \to E \smsh G_1 \smsh \dots \smsh G_r \]
factors as
\[ C \to E \smsh D_1 \smsh \dots \smsh D_r \]
for finite subcomplexes $D_i \subset G_i$. For each ordered partition $\lambda$ of the set $\{1,\dots,n\}$ into $r$ pieces, we then have a composition map
\[ \Map(E,S) \smsh \Nat(D_1X,X^{\smsh n_1}) \smsh \dots \smsh \Nat(D_rX,X^{\smsh n_r}) \to \Nat(CX,X^{\smsh n}) \]
given by the composite
\[ CX \to E \smsh D_1X \smsh \dots \smsh D_rX \to S \smsh X^{\smsh n_1} \smsh \dots \smsh X^{\smsh n_r} \isom X^{\smsh n}. \]
The final isomorphism in this sequence is based on the given partition $\lambda$ of $\{1,\dots,n\}$.

Our aim now is to prove that the maps above are equivalences for any $G_1,\dots,G_r$. We do this by looking at the cell structure of the $G_i$. Notice that each side of (*) preserves (after composing with a Spanier-Whitehead dual) filtered homotopy colimits in the variables $G_i$ (since taking derivatives and smash products commute with those). This allows us to reduce to the case that each $G_i$ is a finite cell functor. Now we induct on the cell structure. Notice that each side of (*) also preserves homotopy pushouts (for the same reasons). This allows us to reduce to the case where each $G_i$ is of the form
\[ G_i = J_i \smsh \spectra(K_i,-) \]
for some finite cell spectra $J_i,K_i$.

Now for these $G_i$, the source of the map (*) is isomorphic to
\[ \Wdge_{n = n_1+\dots+n_r} \Map(E,S) \smsh \Map(J_1,K_1^{\smsh n_1}) \smsh \dots \smsh \Map(J_r,K_r^{\smsh n_r}) \]
by  the Yoneda Lemma. By \cite[III.7]{elmendorf/kriz/mandell/may:1997}, this in turn is equivalent to
\[ \Wdge_{n = n_1+\dots+n_r} \Map(E \smsh J_1 \smsh \dots \smsh J_r, K_1^{\smsh n_1} \smsh \dots \smsh K_r^{\smsh n_r}) \]
since all these spectra are finite cell complexes. On the other hand, the target of (*) is now
\[ \der^n(Q(E \smsh J_1 \smsh \dots \smsh J_r \smsh \spectra(K_1,-) \smsh \dots \smsh \spectra(K_r,-))). \]
Writing $L = E \smsh J_1 \smsh \dots \smsh J_r$, we can write the map (*) as
\[ \tag{**} \Wdge_{n = n_1+\dots+n_r} \Map(L,K_1^{\smsh n_1} \smsh \dots \smsh K_r^{\smsh n_r}) \to \der^n(Q(L \smsh \spectra(K_1,-) \smsh \dots \smsh \spectra(K_r,-))). \]

To show that (**) is an equivalence, we show that the right-hand side can be identified with the \ord{r} cross-effect of the functor $X \mapsto \Map(L,X^{\smsh n})$ evaluated at $(K_1,\dots,K_r)$. To see this a similar argument to that used in the proof of Proposition \ref{prop:models-correct} shows that
\[ L \smsh \spectra(K_1,X) \smsh \dots \smsh \spectra(K_r,X) \]
is equivalent to the total homotopy cofibre of the cube whose terms are the finite cell functors
\[ L \smsh \spectra(K_{i_1} \wdge \dots \wdge K_{i_s},X) \]
for $\{i_1,\dots,i_s\} \subseteq \{1,\dots,r\}$.

We can now calculate the pro-object on the right-hand side of the map (**). Firstly, note that since $\der^n$ is dual to the \ord{n} Goodwillie derivative, it takes homotopy colimits to homotopy limits. Therefore, by the Yoneda Lemma:
\[ \begin{split} \der^n(Q(L \smsh \spectra(K_1,X) \smsh \dots \smsh \spectra(K_r,X)))
    &\homeq \thofib \left\{ \Map(L,(K_{i_1} \wdge \dots \wdge K_{i_r})^{n}) \right\} \\
    &\homeq \Map\left(L, \thofib((K_{i_1} \wdge \dots \wdge K_{i_r})^{n}) \right). \end{split} \]
Therefore, to show that (**) is an equivalence, it is sufficient to show that we have an equivalence
\[ \tag{***} \Wdge_{n = n_1+\dots+n_r} K_1^{n_1} \smsh \dots \smsh K_r^{\smsh n_r} \to \thofib \left\{ (K_{i_1} \wdge \dots \wdge K_{i_r})^{\smsh n} \right\} \]
where the left-hand side is the coproduct over all ordered partitions of a set of size $n$ into $r$ nonempty pieces. The map (***) is given by the inclusions
\[ K_1^{n_1} \smsh \dots \smsh K_r^{n_r} \to (K_1 \wdge \dots \wdge K_r)^{\smsh n} \]
This gives a map into the total homotopy fibre because the composites
\[ K_1^{n_1} \smsh \dots \smsh K_r^{n_r} \to (K_1 \wdge \dots \wdge K_r)^{\smsh n} \to (K_{i_1} \wdge \dots \wdge K_{i_s})^{\smsh n} \]
are trivial for any proper subset $\{i_1,\dots,i_s\}$ of $\{1,\dots,r\}$.

The right-hand side of the map (***) above is exactly the definition of the \ord{r} cross-effect of the functor $X \mapsto X^{\smsh n}$ evaluated at $(K_1,\dots,K_r)$. We can calculate this cross-effect by taking the iterated homotopy fibres and this gives precisely the left-hand side of (***) with the given map being an equivalence.

This concludes the proof that the map (*) above is an equivalence of pro-symmetric sequences for any presented cell functors $G_1,\dots,G_r$. Taking $G_1 = \dots = G_r = G$, we then deduce that the map
\[ \mu^*: \tilde{\der}^*(QF) \circ \tilde{\der}^*(QG) \to \tilde{\der}^*(Q((QF)(QG))) \]
is an equivalence as required. This completes the proof of the Theorem.
\end{proof}

For future reference, we note that we have proved that the map $\mu^*$ of Definition \ref{def:mu} is an equivalence of pro-symmetric sequences:

\begin{cor} \label{cor:chainrule}
Let $F,G: \spectra \to \spectra$ be presented cell functors (left Kan extended to all of $\spectra$ as usual). Then the map
\[ \mu^*: \tilde{\der}^*(F) \circ \tilde{\der}^*(G) \to \der^*(Q(FG)) \]
of Definition \ref{def:mu} is an equivalence of pro-symmetric sequences.
\end{cor}

\begin{remark}
Theorem \ref{thm:chainrule} was previously proved by the second author in \cite{ching:2007} by a different method. The added value of the result in this paper is that we have constructed an equivalence using explicit models for the derivatives of these functors. These constructions are used in the following sections to construct operad and module structures on such derivatives, and to prove chain rules for functors from or to the category of simplicial sets.
\end{remark}

\begin{remark}
Our proof of Theorem \ref{thm:chainrule} also reveals where the hypothesis that $F$ be finitary comes in. In general, we have proved that the composition product $\der_*(F) \circ \der_*(G)$ is equivalent to $\der_*((QF)G)$ where $QF$ is the finitary replacement for $F$ (see Remark \ref{rem:finitary}). When $F$ is itself finitary, this is of course equivalent to $\der_*(FG)$. Note that if $G$ takes values in homotopy-finite spectra, then $(QF)G \homeq FG$ even when $F$ is not finitary. Our chain rule therefore holds in this case without the finitary hypothesis.

Without any extra condition, a counterexample to Theorem \ref{thm:chainrule} is provided by Example \ref{ex:counterexample}. In that case, we have $\der_1(F) = L_E(S)$, $\der_1(G) = *$, but
\[ \der_1(FG) \homeq \hocofib(L_E(S) \smsh R\mathbb{P}^{\infty} \to L_E(R\mathbb{P}^{\infty})) \neq *. \]
\end{remark}

\section{Operad structures for comonads} \label{sec:comonads}

The philosophy behind much of this paper is that the composition maps $\mu^*$, of Definition \ref{def:mu}, allow us to construct operad and module structures on the (duals of) derivatives of functors that have appropriate extra structure. In general, however, there are rigidification problems with this. In particular, there does not seem to be a way to construct maps (on the point-set level) of the form
\[  \der^*(QF) \circ \der^*(QG) \to \der^*(Q(FG)) \]
for general functors $F,G: \spectra \to \spectra$. This means that we only get operad and module composition maps defined up to inverse weak equivalence, via the sequence
\[ \der^*(QF) \circ \der^*(QG) \to \der^*(Q((QF)(QG))) \lweq \der^*(Q(FG)) \]
where the second map is induced by $Q((QF)(QG)) \to Q(FG)$.

There are however circumstances in which this problem can be avoided, and these are sufficient for the purposes of this paper. For example, if we have a functor $F$ that has a comonad structure \emph{and} is also a presented cell functor, then the pro-symmetric sequence $\der^*(F)$ inherits a (rigid) operad structure. The comonad we are principally concerned with in this paper is $\Sigma^\infty \Omega^\infty$ which does have a cell structure, as we observed in Example \ref{ex:SigOm}. The main result of this section is the identification of the corresponding operad structure on $\der^*(\Sigma^\infty \Omega^\infty)$ with the commutative operad in $\spectra$.

\begin{definition} \label{def:comonad}
A \emph{comonad} on $\spectra$ is a functor $T: \spectra \to \spectra$ together with natural transformations
\[ m: T \to TT, \quad e: T \to I_{\spectra} \]
such that the following diagrams commute
\[ \begin{diagram}
  \node{T} \arrow{e,t}{m} \arrow{s,l}{m} \node{TT} \arrow{s,r}{Tm} \\
  \node{TT} \arrow{e,t}{mT} \node{TTT}
\end{diagram}, \quad
\begin{diagram}
  \node{T} \arrow{e,t}{m} \arrow{se,b}{1_T} \node{TT} \arrow{s,r}{Te} \\
  \node[2]{T}
\end{diagram}, \quad
\begin{diagram}
  \node{T} \arrow{e,t}{m} \arrow{se,b}{1_T} \node{TT} \arrow{s,r}{eT} \\
  \node[2]{T}
\end{diagram} \]
\end{definition}

\begin{example} \label{ex:SO-comonad}
Any adjunction gives rise to a comonad by composing the left and right adjoints. We are particularly interested in the case $T = \Sigma^\infty \Omega^\infty$ in which the comonad structure is given by the natural transformations
\[ m: \Sigma^\infty \Omega^\infty \to \Sigma^\infty \Omega^\infty \Sigma^\infty \Omega^\infty, \quad \epsilon: \Sigma^\infty \Omega^\infty \to I_{\spectra} \]
from the unit and counit maps of the $(\Sigma^\infty,\Omega^\infty)$ adjunction.
\end{example}

\begin{definition} \label{def:comonad-derivatives}
Let $T: \spectra \to \spectra$ be a comonad that is also the left Kan extension of a presented cell functor in $[\finspec,\spectra]$. The map $m: T \to TT$ then determines, by the construction of Definition \ref{def:composition}, a map of pro-symmetric sequences
\[ m^*: \der^*(T) \circ \der^*(T) \to \der^*(T). \]
We also define a map
\[ e^*: \mathsf{1} \to \der^*(T) \]
as follows. Recall that this consists of a map
\[ S \to \der^1(T) = \{\Nat(CX,X)\}_{C \in \mathsf{Sub}(T)}. \]
To define such a map of pro-objects, we have to give, for each $C \in \mathsf{Sub}(T)$, a map
\[ S \to \Nat(CX,X). \]
Such a map is given by the adjoint of the composite
\[ CX \into TX \arrow{e,t}{e_X} X. \]
These respect inclusions of subcomplexes of $T$ and so define the required map $e^*$ of pro-objects.
\end{definition}

\begin{proposition} \label{prop:operad}
Let $T: \spectra \to \spectra$ be a comonad that is also the left Kan extension of a presented cell functor in $[\finspec,\spectra]$. Then the maps $m^*$ and $e^*$ of Definition \ref{def:comonad-derivatives} make $\der^*(T)$ into an operad of pro-spectra (i.e. a monoid for the composition product of pro-symmetric sequences).
\end{proposition}
\begin{proof}
We have to check associativity, i.e. that the following diagrams commutes:
\[ \begin{diagram}
  \node{\der^*(T) \circ \der^*(T) \circ \der^*(T)} \arrow{e,t}{m^* \circ 1} \arrow{s,l}{1 \circ m^*} \node{\der^*(T) \circ \der^*(T)} \arrow{s,r}{m^*} \\
  \node{\der^*(T) \circ \der^*(T)} \arrow{e,t}{m^*} \node{\der^*(T)}
\end{diagram} \]
To see this, consider a finite subcomplex $E \subset T$. By Lemma \ref{lem:factorization}, the composite $E \to T \to TT$ factors as $E \to CD \to TT$ for some $C,D \subset T$. Similarly, the composite $C \to T \to TT$ factors as $C \to C'C'' \to TT$ for $C',C'' \subset T$, and $D \to T \to TT$ factors as $D \to D'D'' \to TT$ for $D',D'' \subset T$.

Now consider the composite $E \to T \to TTT$ where $T \to TTT$ is as in the coassociativity square in the definition of a comonad (\ref{def:comonad}). It follows that this map factors in either of the following ways:
\[ E \to C'C''D \to TTT, \quad E \to CD'D'' \to TTT. \]
Now set $A = C' \cup C$, $A' = C'' \cup D'$, $A'' = D \cup D''$ and note that $A,A',A''$ are all finite subcomplexes of $T$. We then have factorizations $E \to CA'' \to TT$, $E \to AD \to TT$, $C \to AA' \to TT$ and $D \to A'A'' \to TT$.

The composite of the top and right-hand maps in the associativity square above can now be described as built from the following sequences of maps:
\[ \begin{split} \Nat(AX,X^{\smsh *}) \circ  \Nat(A'X,X^{\smsh *}) \circ \Nat(A''X,X^{\smsh *})
                &\to \Nat(AA'X,X^{\smsh *}) \circ \Nat(A''X,X^{\smsh *}) \\
                &\to \Nat(CX,X^{\smsh *}) \circ \Nat(A''X,X^{\smsh *})\\
                &\to \Nat(CA'',X^{\smsh *}) \\
                &\to \Nat(E,X^{\smsh *}). \end{split} \]
Here we are making use of the general maps of the form
\[ \tag{*} \Nat(FX,X^{\smsh *}) \circ \Nat(F'X,X^{\smsh *}) \to \Nat(FF'X,X^{\smsh *}) \]
described in Definition \ref{def:composition} and used to construct the maps $m^*$. These maps are natural and so the above composite is equal to
\[ \begin{split} \Nat(AX,X^{\smsh *}) \circ  \Nat(A'X,X^{\smsh *}) \circ \Nat(A''X,X^{\smsh *})
                &\to \Nat(AA'X,X^{\smsh *}) \circ \Nat(A''X,X^{\smsh *}) \\
                &\to \Nat(AA'A''X,X^{\smsh *}) \\
                &\to \Nat(CA'',X^{\smsh *}) \quad \text{(using $C \to AA'$)} \\
                &\to \Nat(E,X^{\smsh *}) \quad \quad \text{(using $E \to CA''$)}. \end{split} \]
Similarly, the composite of the left-hand and bottom maps in the associativity square can be described by the following sequences
\[ \begin{split} \Nat(AX,X^{\smsh *}) \circ  \Nat(A'X,X^{\smsh *}) \circ \Nat(A''X,X^{\smsh *})
                &\to \Nat(AX,X^{\smsh *}) \circ \Nat(A'A''X,X^{\smsh *}) \\
                &\to \Nat(AA'A''X,X^{\smsh *}) \\
                &\to \Nat(AD,X^{\smsh *}) \quad \text{(using $D \to A'A''$)} \\
                &\to \Nat(E,X^{\smsh *}) \quad \quad \text{(using $E \to AD$)}. \end{split} \]
Firstly, the two composites
\[ E \to AD \to AA'A'' \quad \text{and} \quad E \to CA'' \to AA'A'' \]
are equal since they are each factorizations of the overall map $E \to T \to TTT$. Such factorizations are unique since the map $AA'A'' \to TTT$ is a monomorphism. Comparing the above sequences, it is now sufficient to show that the following square commutes:
\[ \begin{diagram}
  \node{\Nat(AX,X^{\smsh *}) \circ  \Nat(A'X,X^{\smsh *}) \circ \Nat(A''X,X^{\smsh *})} \arrow{e,t}{m^* \circ 1} \arrow{s,l}{1 \circ m^*} \node{\Nat(AX,X^{\smsh *}) \circ  \Nat(A'A''X,X^{\smsh *})} \arrow{s,r}{m^*} \\
  \node{\Nat(AA'X,X^{\smsh *}) \circ  \Nat(A''X,X^{\smsh *})} \arrow{e,t}{m^*} \node{\Nat(AA'A''X,X^{\smsh *})}
\end{diagram} \]
Following through the definitions, each way round this square is expressed by the composites
\[ AA'A''(X) \to (A'A''(X))^{\smsh k} \to (A''(X))^{\smsh n_1+\dots+n_k} \to X^{\smsh r_1+\dots+r_{n_k}}. \]
It therefore commutes and we have thus shown that the original associativity square commutes.

Now consider the unit diagram
\[ \begin{diagram}
  \node{\der^*(T)} \arrow{e,t}{1 \circ e^*} \arrow{se,b}{1} \node{\der^*(T) \circ \der^*(T)} \arrow{s,r}{m^*} \\
  \node[2]{\der^*(T)}
\end{diagram} \]
Again take a finite subcomplex $E \subset T$ and suppose that $E \to T \to TT$ factors as $E \to CD \to TT$ for finite subcomplexes $C,D \subset T$. We can assume that $C$ contains $E$ without loss of generality. The composite of the top and right-hand sides of the above diagram is then built from the maps
\[ \begin{split} \Nat(CX,X^{\smsh k})   &\to \Nat(CX,X^{\smsh k}) \smsh \Nat(DX,X) \smsh \dots \smsh \Nat(DX,X) \\
                                        &\to \Nat(CDX,X^{\smsh k}) \\
                                        &\to \Nat(EX,X^{\smsh k})
\end{split} \]
where we have used the map $S \to \Nat(DX,X)$ that comes from the composite $D \to T \arrow{e,t}{e} I_{\spectra}$. This composite can then be expressed by the sequence
\[ EX \to CDX \to (DX)^{\smsh k} \to (TX)^{\smsh k} \to X^{\smsh k} \]
based on some chosen natural transformation $CY \to Y^{\smsh k}$. On the other hand, the identity on $\der^*(T)$ can be thought of as expressed by the map
\[ \Nat(CX,X^{\smsh k}) \to \Nat(EX,X^{\smsh k}) \]
determined by the inclusion $E \to C$. It is therefore sufficient to show that the following diagram commutes
\[ \begin{diagram}
  \node{EX} \arrow{e} \arrow{sse} \node{CDX} \arrow{e} \arrow{s} \node{(DX)^{\smsh k}} \arrow{s} \\
  \node[2]{CTX} \arrow{e} \arrow{s} \node{(TX)^{\smsh k}} \arrow{s} \\
  \node[2]{CX} \arrow{e} \node{X^{\smsh k}}
\end{diagram} \]
The right-hand half commutes by the naturality of the chosen natural transformation $CY \to Y^{\smsh k}$. The composite
\[ E \to CD \to CT \to C \to T \]
is equal to
\[ E \to T \to TT \to T \]
which, by the unit axiom in the definition of a comonad, is equal to the inclusion
\[ E \to T. \]
But this of course factors via the inclusion of $E$ in $C$:
\[ E \to C \to T. \]
Finally, since $C \to T$ is a monomorphism, we obtain the commutativity of the left-hand half of the diagram above. This completes the check that the first unit diagram for $\der^*(T)$ commutes.

Commutativity of the second unit diagram for $\der^*(T)$ follows in a similar manner to the first. This completes the proof that $\der^*(T)$ forms an operad.
\end{proof}

\begin{corollary} \label{cor:operad}
Let $T: \spectra \to \spectra$ be a comonad that is the left Kan extension of a finite cell functor. Then the symmetric sequence
\[ \der^*(T) \isom \Nat(TX,X^{\smsh *}) \]
has the structure of an operad in $\spectra$.
\end{corollary}
\begin{proof}
In this case, the cofiltered category $\mathsf{Sub}(T)^{op}$ has an initial object ($T$ itself) and so the pro-symmetric sequence $\der^*(T)$ is isomorphic to the ordinary symmetric sequence $\Nat(TX,X^{\smsh *})$ (i.e. a pro-symmetric sequence indexed on the trivial category). The operad structure maps are then precisely the composites
\[ \Nat(TX,X^{\smsh *}) \circ \Nat(TX,X^{\smsh *}) \to \Nat(TTX, X^{\smsh *}) \to \Nat(TX,X^{\smsh *}). \]
\end{proof}

We now apply Corollary \ref{cor:operad} to the functor $\Sigma^\infty \Omega^\infty: \spectra \to \spectra$ (see Example \ref{ex:SO-comonad}). Recall from Example \ref{ex:SigOm} that $\Sigma^\infty \Omega^\infty$ is a finite cell functor with a presentation consisting of a single cell.

\begin{definition}[Commutative operad] \label{def:com}
The \emph{commutative operad} in $\spectra$ is the operad $\mathsf{Com}$ given by
\[ \mathsf{Com}(n) := S \]
for all $n$, with the trivial $\Sigma_n$-action, and with composition maps given by the unit isomorphisms
\[ S \smsh S \smsh \dots \smsh S \arrow{e,t}{\isom} S \]
for the symmetric monoidal structure on $\spectra$.
\end{definition}

\begin{proposition} \label{prop:com}
Corollary \ref{cor:operad} determines an operad structure on the symmetric sequence $\der^*(\Sigma^\infty \Omega^\infty)$ such that
\[ \der^*(\Sigma^\infty \Omega^\infty) \homeq \mathsf{Com} \]
in the category of operads in $\spectra$ (i.e. there is a zigzag of weak equivalences of operads between them).
\end{proposition}
\begin{proof}
We saw in Example \ref{ex:SigOm} that
\[ \der^n(\Sigma^\infty \Omega^\infty) \isom \Nat(S_c \smsh \spectra(S_c,X), X^{\smsh n}) \isom \Map(S_c,(S_c)^{\smsh n}) \]
using the Yoneda Lemma. Following through the definitions, we see that the Yoneda Isomorphisms identify the operad structure on $\der^*(\Sigma^\infty \Omega^\infty)$ with the `coendomorphism operad' structure on the objects $\Map(S_c,(S_c)^{\smsh *})$. The structure maps in this coendomorphism operad are given by
\[ \Map(S_c,(S_c)^{\smsh k}) \smsh \Map(S_c,(S_c)^{\smsh n_1}) \smsh \dots \smsh \Map(S_c,(S_c)^{\smsh n_k}) \to \Map(S_c,(S_c)^{\smsh(n_1+\dots+n_k)}) \]
given (informally) by
\[ (f,g_1,\dots,g_k) \mapsto f \circ (g_1 \smsh \dots \smsh g_k). \]
This coendomorphism operad is equivalent to the commutative operad via the following zigzag of weak equivalences of operads:
\[ \Map(S_c,S_c^{\smsh n}) \lweq \Map(S,S_c^{\smsh n}) \weq \Map(S,S^{\smsh n}) \isom S. \]
\end{proof}

\begin{definition} \label{def:tilde-der}
The operad $\der^*(\Sigma^\infty \Omega^\infty)$ is not reduced because
\[ \der^1(\Sigma^\infty \Omega^\infty) \isom \Map(S_c,S_c) \ncong S. \]
When we come to consider bar constructions over $\der^*(\Sigma^\infty \Omega^\infty)$ in the next few sections, it is important to work over a reduced operad. We also need to take a projectively-cofibrant replacement at some point. We therefore make the following definitions:
\begin{itemize}
  \item For any operad $P$ in $\spectra$, we can associate a reduced operad $P^{\mathsf{red}}$ in which
  \[ P^{\mathsf{red}}(n) := \begin{cases} S & \text{if $n = 1$}; \\ P(n) & \text{otherwise}. \end{cases} \]
  with composition maps given by
  \[ \begin{split} P^{\mathsf{red}}(k) \smsh P^{\mathsf{red}}(n_1) \smsh \dots \smsh P^{\mathsf{red}}(n_k)
                            &\to P(k) \smsh P(n_1) \smsh \dots \smsh P(n_k) \\
                            &\to P(n_1+\dots+n_k) \\
                            &= P^{\mathsf{red}}(n_1+\dots+n_k) \end{split} \]
  if $n_1+\dots+n_k > 1$ where $P^{\mathsf{red}}(1) \to P(1)$ is the unit map
  \[ S \to P(1) \]
  for the operad $P$. The only composition map that maps into $P^{\mathsf{red}}(1)$ is given by
  \[ P^{\mathsf{red}}(1) \smsh P^{\mathsf{red}}(1) = S \smsh S \to S = P^{\mathsf{red}}(1). \]
  In particular, we have a reduced operad $\der^*(\Sigma^\infty \Omega^\infty)^{\mathsf{red}}$.
  \item We also let $\tilde{\der}^*(\Sigma^\infty \Omega^\infty)$ denote a projectively-cofibrant replacement for the reduced operad $\der^*(\Sigma^\infty \Omega^\infty)^{\mathsf{red}}$.
\end{itemize}
The map $S \to \Map(S_c,S_c)$ is a weak equivalence and so we obtain a sequence of weak equivalences of operads
\[ \tilde{\der}^*(\Sigma^\infty \Omega^\infty) \weq \der^*(\Sigma^\infty \Omega^\infty)^{\mathsf{red}} \weq \der^*(\Sigma^\infty \Omega^\infty). \]
In particular, $\tilde{\der}^*(\Sigma^\infty \Omega^\infty)$ is an $E_\infty$-operad in $\spectra$, i.e. a cofibrant replacement for the commutative operad.
\end{definition}

\begin{definition} \label{def:d*I}
The main result of \cite{ching:2005a} is that the Spanier-Whitehead dual of the reduced bar construction on the commutative operad is, as a symmetric sequence, equivalent to the Goodwillie derivatives of the identity functor on based spaces, that is:
\[ \der^G_*(I_{\sset}) \homeq \dual B(\mathsf{Com}). \]
On the level of symmetric sequences, this result is due to the first author and Mark Mahowald in \cite{arone/mahowald:1999}. The operad structure is constructed in \cite{ching:2005a}.

Proposition \ref{prop:com}, together with the homotopy invariance of the bar construction (Proposition \ref{prop:bar-invariance}), then implies that those derivatives of the identity are equivalent to
\[ \dual B(\tilde{\der}^*(\Sigma^\infty \Omega^\infty)). \]
The cooperad $B(\tilde{\der}^*(\Sigma^\infty \Omega^\infty))$ is directly-dualizable by Lemma \ref{lem:bar-directly-dualizable} and so this dual has an operad structure by Lemma \ref{lem:dual-cooperad}.

It is convenient to have notation for the operad formed by this dual. Since we know that it is equivalent to the derivatives of the identity on based spaces, we write
\[ \der_*(I) := \widetilde{\dual B(\tilde{\der}^*(\Sigma^\infty \Omega^\infty))}. \]
(For convenience we include a $\Sigma$-cofibrant replacement in this definition.) In \S\ref{sec:spaces-spaces} below, we give another proof that this operad is equivalent to the derivatives of the identity on based spaces that does not directly use \cite{arone/mahowald:1999}. However, our proof (indeed the main ideas of this paper) is still based on the adjunction $(\Sigma^\infty,\Omega^\infty)$, which was also the basis for studying the derivatives of the identity in \cite{arone/kankaanrinta:1998(2)}, and continued in \cite{arone/mahowald:1999}.
\end{definition}

\part{Functors of spaces}

The remainder of this paper is concerned with models for the Goodwillie derivatives of functors to and/or from simplicial sets, rather than spectra, and the corresponding chain rules. All of our results are based on the cosimplicial cobar construction associated to the $(\Sigma^\infty, \Omega^\infty)$ adjunction. In particular, they are depend on a fundamental result expressing the Taylor tower of a composite of two functors in which the `middle' category of simplicial sets, in terms of the cobar construction. In \S\ref{sec:cobar} we state and prove this (surprisingly simple) result. We also point out that this result gives us a way to approach the calculation of the full Taylor tower of a composite functor, rather than just the derivatives which are the main focus of this paper.

Here is a summary of the rest of the paper:
\begin{itemize}
  \item we construct models for the derivatives of a pointed simplicial homotopy functor $F: \sset \to \spectra$ that have the structure of a right $\der_*(I)$-module (\S\ref{sec:spaces-spectra});
  \item we construct models for the derivatives of a pointed simplicial homotopy functor $F: \spectra \to \sset$ that have the structure of a left $\der_*(I)$-module (\S\ref{sec:spectra-spaces});
  \item combining the previous two sections, we construct models for the derivatives of a functor $F: \sset \to \sset$ that have the structure of a $\der_*(I)$-bimodule (\S\ref{sec:spaces-spaces});
  \item we then turn to proving chain rules involving functors to and/or from $\sset$. In preparation for this, we prove a result on bar constructions that is essentially a weak form of `Koszul duality' for operads and modules in $\spectra$ (\S\ref{sec:koszul});
  \item finally, using our Koszul duality result and previous constructions, we deduce the form of the chain rule for functors to and/or from pointed simplicial sets (\S\ref{sec:chainrule-spaces}).
\end{itemize}

\section{The cobar construction} \label{sec:cobar}

As mentioned above, all of the main results in the rest of this paper depend on the following fundamental result.

\begin{theorem} \label{thm:key}
Let $F: \sset \to \cat{D}$ be a pointed simplicial homotopy functor with $\cat{D}$ equal to either spaces or spectra. Let $G: \cat{C} \to \sset$ be a pointed simplicial homotopy functor with $\cat{C}$ equal to either spaces or spectra. Suppose also that $F$ is finitary. Then there are equivalences (natural in $F$ and $G$):
\[ \eta_n: P_n(FG) \weq \widetilde{\Tot} \; \left[ P_n(F \Omega^\infty (\Sigma^\infty \Omega^\infty)^{\bullet} \Sigma^\infty G) \right] \]
and
\[ \epsilon_n: D_n(FG) \weq \widetilde{\Tot} \; \left[ D_n(F \Omega^\infty (\Sigma^\infty \Omega^\infty)^{\bullet} \Sigma^\infty G) \right]. \]
Recall that $\widetilde{\Tot}$ denotes the homotopy-invariant totalization of a cosimplicial object in which a Reedy fibrant replacement is made before taking the totalization.
\end{theorem}
\begin{proof}
The right-hand side of $\eta_n$ is the totalization of the cosimplicial object with $k$-simplices
\[ P_n(F \Omega^\infty (\Sigma^\infty \Omega^\infty)^k \Sigma^\infty G), \]
(and similarly for $D_n$), and coface and codegeneracies given by the unit and counit of the $(\Sigma^\infty,\Omega^\infty)$ adjunction.

The maps $\eta_n$ and $\epsilon_n$ come from augmentations for these cosimplicial objects (in the sense of Definition \ref{def:augmented-simplicial}) given by the unit of the adjunction
\[ FG \to F \Omega^\infty \Sigma^\infty G. \]

We prove that $\eta_n$ is an equivalence via a number of steps:

(1) Now suppose first that $F \to F'$ is a $P_n$-equivalence. Then we have a commutative diagram
\[ \begin{diagram}
  \node{P_n(FG)} \arrow{e} \arrow{s,l}{\sim}
    \node{\widetilde{\Tot} \; \left[ P_n(F \Omega^\infty (\Sigma^\infty \Omega^\infty)^{\bullet} \Sigma^\infty G) \right]} \arrow{s,l}{\sim} \\
  \node{P_n(F'G)} \arrow{e}
    \node{\widetilde{\Tot} \; \left[ P_n(F' \Omega^\infty (\Sigma^\infty \Omega^\infty)^{\bullet} \Sigma^\infty G) \right]}
\end{diagram} \]
The vertical maps are equivalences by Proposition \ref{prop:calculus}(1), and since $\widetilde{\Tot}$ takes level equivalences of cosimplicial objects to equivalences. This diagram tells us that if $\eta_n$ is an equivalence for $F'$, then it is also an equivalence for $F$.

(2) Now suppose that $F \to F' \to F''$ is a fibre sequence of functors. Then we have a commutative diagram
\[ \begin{diagram}
  \node{P_n(FG)} \arrow{e} \arrow{s}
    \node{\widetilde{\Tot} \; \left[ P_n(F \Omega^\infty (\Sigma^\infty \Omega^\infty)^{\bullet} \Sigma^\infty G) \right]} \arrow{s} \\
  \node{P_n(F'G)} \arrow{e} \arrow{s}
    \node{\widetilde{\Tot} \; \left[ P_n(F' \Omega^\infty (\Sigma^\infty \Omega^\infty)^{\bullet} \Sigma^\infty G) \right]} \arrow{s} \\
  \node{P_n(F''G)} \arrow{e}
    \node{\widetilde{\Tot} \; \left[ P_n(F'' \Omega^\infty (\Sigma^\infty \Omega^\infty)^{\bullet} \Sigma^\infty G) \right]}
\end{diagram} \]
Since $P_n$ commutes with fibre sequences, and $\widetilde{\Tot}$ takes levelwise fibre sequences to fibre sequences, each of the columns here is a fibre sequence of functors. Therefore, if the top and bottom horizontal maps are equivalences, so too is the middle horizontal map. In other words, if $\eta_n$ is an equivalence for $F$ and $F''$, it is also an equivalence for $F'$.

(3) Finally suppose that the functor $F$ is equivalent to one of the form $H \Sigma^\infty$ for some $H: \spectra \to \cat{D}$. Then $\eta_n$ takes the form
\[ \eta_n: P_n(H \Sigma^\infty G) \to \widetilde{\Tot} \; \left[ P_n(H \Sigma^\infty \Omega^\infty \dots \Sigma^\infty G) \right]. \]
There are now extra codegeneracies in the cosimplicial object on the right-hand side of the form
\[ P_n(H \Sigma^\infty \Omega^\infty (\Sigma^\infty \Omega^\infty)^k \Sigma^\infty G) \to P_n(H \Sigma^\infty \Omega^\infty (\Sigma^\infty \Omega^\infty)^{k-1} \Sigma^\infty G) \]
given by the counit map $\Sigma^\infty \Omega^\infty \to I_{\spectra}$ applied to the first copy of $\Sigma^\infty \Omega^\infty$ on the left-hand side. These provide a cosimplicial contraction (in the sense of Lemma \ref{lem:cosimplicial-contraction}). By Lemma \ref{lem:contraction-homotopy}, it follows that $\eta_n$ is a weak equivalence.

Now we employ induction on the Taylor tower of $F$ to prove that $\eta_n$ is an equivalence in general. Firstly, by Proposition \ref{prop:D_nF}, the homogeneous layers $D_kF$ are equivalent to functors of the form $H \Sigma^\infty$. Therefore, by (3), $\eta_n$ is an equivalence for each $D_kF$. Then, by (2), applied to the fibre sequences $D_kF \to P_kF \to P_{k-1}F$, it follows that $\eta_n$ is an equivalence for each $P_kF$. Finally, by (1), since $\eta_n$ is an equivalence for $P_nF$, it is also an equivalence for $F$ itself.

The proof that $\epsilon_n$ is an equivalence is almost identical, using the fact that $D_n$ preserves fibre sequences, and that by taking $D_n$ of the result of Proposition \ref{prop:calculus}(1), we have equivalences
\[ D_n(FG) \weq D_n((P_nF)G). \]
\end{proof}

\begin{example}
Taking $F$ and $G$ both to be the identity functor $I$ on $\sset$, we see that
\[ P_n(I) \homeq \widetilde{\Tot}\;(P_n(\Omega^\infty \dots \Sigma^\infty)) \]
and hence
\[ \der_n(I) \homeq \widetilde{\Tot}\;(\der_n(\Omega^\infty \dots \Sigma^\infty)). \]
This is precisely the method used by Arone-Kankaanrinta \cite{arone/kankaanrinta:1998(2)} and Arone-Mahowald \cite{arone/mahowald:1999} to approach the calculation of the Taylor tower of the identity functor.

It is interesting to note that the totalization
\[ \widetilde{\Tot}(\Omega^\infty \dots \Sigma^\infty(X)) \]
is, in general, equal to the Bousfield-Kan $\mathbb{Z}$-completion of the space $X$, which, for simply-connected $X$, is equivalent to $X$. Another way to see this is to recall that the Taylor tower (at $*$) for the identity functor $I$ converges for simply-connected $X$. For such $X$, it then follows from Theorem \ref{thm:key} that
\[ \begin{split} X &\homeq \holim P_nI(X) \\
    &\homeq \holim \widetilde{\Tot} \; (P_n(\Omega^\infty \dots \Sigma^\infty)(X)) \\
    &\homeq \widetilde{\Tot} \; (\holim P_n(\Omega^\infty \dots \Sigma^\infty)(X)) \\
    &\homeq \widetilde{\Tot} \; (\Omega^\infty \dots \Sigma^\infty X). \end{split} \]
\end{example}

\begin{remark}
In fact, $P_n(FG)$ is actually equivalent to $\Tot^n$ of the cosimplicial object in Theorem \ref{thm:key}, and similarly for $D_n(FG)$. In other words this cosimplicial object is degenerate above degree $n$.
\end{remark}

\section{Functors from spaces to spectra} \label{sec:spaces-spectra}

We can now, finally, start producing the main results of this paper. We start with functors from spaces (i.e. simplicial sets) to spectra. In this section, we construct new models for the Goodwillie derivatives of such a functor. These new models come equipped with a natural right $\der_*(I)$-module structure.

For $F: \sset \to \spectra$, we consider the pro-symmetric sequence $\der^*(F \Omega^\infty)$ (which we know is the Spanier-Whitehead dual to the derivatives of the functor $F \Omega^\infty: \spectra \to \spectra$. By the methods of \S\ref{sec:comonads}, this pro-symmetric sequence is a pro-right-module over the operad $\der^*(\Sigma^\infty \Omega^\infty)$. Our models for the derivatives of $F$ are then the Spanier-Whitehead duals of the one-sided bar construction on the pro-module $\der^*(F \Omega^\infty)$ (see Definition \ref{def:right-full}).

To make this approach work, we have to understand how composing with $\Omega^\infty$ affects the process of taking finite subcomplexes of a presented cell functor. The key to this is the following lemma.

\begin{lemma} \label{lem:extend-omega}
Let $F \in [\finsset,\spectra]$ be a presented cell functor, and denote also by $F$ the (enriched) left Kan extension of $F$ to a functor $\sset \to \spectra$. Then the composite $F \Omega^\infty$ is a presented cell functor in $[\finspec,\spectra]$ in which the cells correspond 1-1 with the cells of $F$.
\end{lemma}
\begin{proof}
To define the cell structure on $F \Omega^\infty$, we set
\[ (F \Omega^\infty)_i := F_i \Omega^\infty. \]
Composing the attaching diagram for the cells of $F$ of degree $i+1$ with $\Omega^\infty$, we get a square:
\[ \begin{diagram}
  \node{\Wdge_{\alpha}^{\mathstrut} I^{\alpha}_0 \smsh \sset(K_{\alpha},\Omega^\infty(-))} \arrow{e} \arrow{s} \node{(F \Omega^\infty)_i} \arrow{s} \\
  \node{\Wdge_{\alpha}^{\mathstrut} I^{\alpha}_1 \smsh \sset(K_{\alpha},\Omega^\infty(-))} \arrow{e} \node{(F \Omega^\infty)_i}
\end{diagram} \]
This is a pushout square of functors $\finspec \to \spectra$ because the left Kan extension commutes with colimits. Here each $I^{\alpha}_0 \to I^{\alpha}_1$ is one of the generating cofibrations in $\spectra$ and $K_{\alpha} \in \finsset$. Now notice that there is an isomorphism of pointed simplicial sets
\[ \sset(K_{\alpha},\Omega^\infty X) \isom \spectra(\Sigma^\infty K_{\alpha}, X). \]
The above diagram therefore determines a pushout square
\[ \begin{diagram}
  \node{\Wdge_{\alpha}^{\mathstrut} I^{\alpha}_0 \smsh \spectra(\Sigma^\infty K_{\alpha},-)} \arrow{e} \arrow{s} \node{(F \Omega^\infty)_i} \arrow{s} \\
  \node{\Wdge_{\alpha}^{\mathstrut} I^{\alpha}_1 \smsh \spectra(\Sigma^\infty K_{\alpha},-)} \arrow{e} \node{(F \Omega^\infty)_i}
\end{diagram} \]
Now if $K_{\alpha} \in \finsset$, then $\Sigma^\infty K_{\alpha} \in \finspec$. This is because $\Sigma^\infty$ preserves colimits, and takes the generating cofibrations in $\sset$ to generating cofibrations in $\spectra$. Therefore the above square is the attaching diagram for a presented cell functor.

Finally, notice that $F \Omega^\infty$ is equal to the colimit of the $(F \Omega^\infty)_i = F_i \Omega^\infty$, and so this is a cell structure on $F \Omega^\infty$. It is clear from the construction that the cells of $F \Omega^\infty$ are in 1-1 correspondence with those of $F$.
\end{proof}

\begin{remark}
If $F: \finsset \to \spectra$ is a presented cell functor then each finite subcomplex $C' \in \mathsf{Sub}(F \Omega^\infty)$ (where $F \Omega^\infty$ has the cell structure of Lemma \ref{lem:extend-omega}) is isomorphic to $C \Omega^\infty$ for a finite subcomplex $C$ of $F$. Similarly, the pro-symmetric sequence
\[ \der^*(F \Omega^\infty) \]
is canonically isomorphic to the pro-symmetric sequence
\[ \{ \Nat(C\Omega^\infty X,X^{\smsh *}) \} \]
indexed on the finite subcomplexes $C \in \mathsf{Sub}(F)$. We do not distinguish between these two pro-objects.
\end{remark}

\begin{remark} \label{rem:FOE}
In what follows, we need to be clear about the meaning of expressions of the form $F \Omega^\infty E$, where $F:\finsset \to \spectra$ is a cell functor, and $E$ is any spectrum. We can think of $F \Omega^\infty E$ in two ways:
\begin{enumerate}
  \item as the left Kan extension of $F$ to a functor $\sset \to \spectra$ applied to the simplicial set $\Omega^\infty E$;
  \item as the left Kan extension of the cell functor $F \Omega^\infty$ (as in Lemma \ref{lem:extend-omega}) applied to the spectrum $E$.
\end{enumerate}
It follows from Remark \ref{rem:extend} that these two possibilities are naturally isomorphic.
\end{remark}

We now recall the notion of a `comodule' over a comonad, that is a functor with either a left or right action of the comonad. (This is also sometimes called a `coalgebra'.)

\begin{definition} \label{def:comodule}
Let $T: \spectra \to \spectra$ be a comonad (Definition \ref{def:comonad}) and let $F: \spectra \to \spectra$ be another functor. A \emph{right $T$-comodule structure} on $F$ is a natural transformation $r: F \to FT$ such that the following diagrams commute:
\[ \begin{diagram}
  \node{F} \arrow{e,t}{r} \arrow{s,l}{r} \node{FT} \arrow{s,r}{Fm} \\
  \node{FT} \arrow{e,t}{rT} \node{FTT}
\end{diagram}, \quad
\begin{diagram}
  \node{F} \arrow{e,t}{r} \arrow{se,b}{1_T} \node{FT} \arrow{s,r}{Te} \\
  \node[2]{F}
\end{diagram} \]
Dually, a \emph{left $T$-comodule structure} on $F$ is a natural transformation $l:F \to TF$ such that analogous diagrams commute.
\end{definition}

The argument of Proposition \ref{prop:operad} then generalizes as follows.

\begin{proposition} \label{prop:module}
Let $F,T: \spectra \to \spectra$ be presented cell functors and suppose that $T$ is a comonad and $F$ a right $T$-comodule. Then the map
\[ r^*: \der^*(F) \circ \der^*(T) \to \der^*(F) \]
induced by $r:F \to FT$ according to Definition \ref{def:composition} makes $\der^*(F)$ into a pro-right-module over the operad $\der^*(T)$. Similarly, if $F$ is a left $T$-comodule, then the map
\[ l^*: \der^*(T) \circ \der^*(F) \to \der^*(F) \]
induced by $l: F \to TF$ makes $\der^*(F)$ into a pro-left-module over $\der^*(T)$.
\end{proposition}
\begin{proof}
The proof of \ref{prop:operad} applies directly, replacing $T$ by $F$ as appropriate.
\end{proof}

We can now construct our models for the Goodwillie derivatives of a functor from simplicial sets to spectra.

\begin{definition} \label{def:right}
Let $F: \sset \to \spectra$ be a presented cell functor. Then we have a natural map
\[ r: F \Omega^\infty \to F \Omega^\infty \Sigma^\infty \Omega^\infty \]
given by the unit of the adjunction between $\Sigma^\infty$ and $\Omega^\infty$ that makes $F \Omega^\infty$ into a right comodule over the comonad $\Sigma^\infty \Omega^\infty$. Now $F \Omega^\infty$ is also a presented cell functor (by Lemma \ref{lem:extend-omega}). Therefore, by Proposition \ref{prop:module}, we have a map
\[ r^*: \der^*(F \Omega^\infty) \circ \der^*(\Sigma^\infty \Omega^\infty) \to \der^*(F \Omega^\infty) \]
that makes $\der^*(F \Omega^\infty)$ into a pro-right-$\der^*(\Sigma^\infty \Omega^\infty)$-module.

We now form the homotopically correct bar construction for the pro-module $\der^*(F \Omega^\infty)$:
\[ \der^*(F) := B(\tilde{\der}^*(F \Omega^\infty), \tilde{\der}^*(\Sigma^\infty \Omega^\infty), \mathsf{1}). \]
This is a pro-right-comodule over the cooperad $B(\tilde{\der}^*(\Sigma^\infty \Omega^\infty))$.
\end{definition}

\begin{definition} \label{def:right-full}
Now let $F: \finsset \to \spectra$ be any pointed simplicial functor, and let $QF$ be a cellular replacement for $F$ (see Definition \ref{def:QF}). Then we set
\[ \der_*(F) := \dual \der^*(QF) \]
where $\der^*(QF)$ is as in Definition \ref{def:right} using the standard cell structure on $QF$. This is the Spanier-Whitehead dual of a pro-right-comodule and so, according to Definition \ref{def:dual-pro-comodule}, can be given the structure of a right module over the dual operad. In this case that means $\der_*(F)$ is a right module over a $\Sigma$-cofibrant replacement of $\dual B(\tilde{\der}^*(\Sigma^\infty \Omega^\infty)$, that is, over $\der_*(I)$.

Explicitly, we can write
\[ \der_*(F) := \hocolim_{C \in \mathsf{Sub}(QF)} \Map \left(B(\tilde{\der}^*(C \Omega^\infty), \tilde{\der}^*(\Sigma^\infty \Omega^\infty), \mathsf{1}), S \right) \]
where the homotopy colimit is taken in the category of right $\der_*(I)$-modules.
\end{definition}

The following is the main result of this section.

\begin{theorem} \label{thm:right}
Let $F: \finsset \to \spectra$ be a pointed simplicial homotopy functor. Then there is a natural equivalence of symmetric sequences
\[ \der_*(F) \homeq \der^G_*(F). \]
That is, the right $\der_*(I)$-module $\der_*(F)$ consists of models for the Goodwillie derivatives of $F$.
\end{theorem}

Most of the remainder of this section deals with the proof of Theorem \ref{thm:right}. Our method of proof is similar to that of Theorem \ref{thm:derivatives} in that we construct a natural transformation $\phi$ between $F$ and another functor (which we call $\Phi_nF$) whose \ord{n} derivative is equivalent to $\der_nF$, and show that $\phi$ induces an equivalence after applying $D_n$. The construction of $\Phi_nF$ and the natural transformation $\phi$ is somewhat more involved than in \S\ref{sec:nat}.

\begin{definition} \label{def:sigmainfty-module}
Let $X$ be a pointed simplicial set. Define maps
\[ \Delta_r: \Sigma^\infty X \smsh \der^r(\Sigma^\infty \Omega^\infty) \to (\Sigma^\infty X)^{\smsh r} \]
by composing the unit map
\[ \Sigma^\infty X \to \Sigma^\infty \Omega^\infty \Sigma^\infty X \]
with the evaluation map
\[ \Sigma^\infty \Omega^\infty \Sigma^\infty X \smsh \Nat(\Sigma^\infty \Omega^\infty E, E^{\smsh k}) \to (\Sigma^\infty X)^{\smsh k}. \]
Smashing together these maps appropriately, we get
\[ \Delta_{r_1,\dots,r_k}: (\Sigma^\infty X)^{\smsh k} \smsh \der^{r_1}(\Sigma^\infty \Omega^\infty) \smsh \dots \smsh \der^{r^k}(\Sigma^\infty \Omega^\infty) \to (\Sigma^\infty X)^{\smsh(r_1+\dots+r_k)}. \]
\end{definition}

\begin{lemma} \label{lem:sigmainfty-module}
For any $X \in \sset$, the maps $\Delta_{r_1,\dots,r_k}$ make the symmetric sequence $(\Sigma^\infty X)^{\smsh *}$ into a right module over the operad $\der^*(\Sigma^\infty \Omega^\infty)$.
\end{lemma}
\begin{proof}
This follows from the properties of the unit map $\Sigma^\infty X \to \Sigma^\infty \Omega^\infty \Sigma^\infty X$ and the definition of the operad structure on $\der^*(\Sigma^\infty \Omega^\infty)$, using the fact that each of these comes from the $(\Sigma^\infty,\Omega^\infty)$ adjunction.
\end{proof}

\begin{remark} \label{rem:diagonal}
The right module structure on $(\Sigma^\infty X)^{\smsh *}$ of Lemma \ref{lem:sigmainfty-module} is the key to the definition of $\Phi_nF$. The following lemma gives us a way to interpret that module structure in terms of the diagonal map on the pointed simplicial set $X$.
\end{remark}

\begin{lemma} \label{lem:diagonal}
For any $X \in \sset$, the following diagram commutes
\[ \begin{diagram}
  \node{(\Sigma^\infty X) \smsh \der^r(\Sigma^\infty \Omega^\infty)} \arrow{e,t}{\isom} \arrow{s,l}{\Delta_r}
    \node{(S_c \smsh X) \smsh \Map(S_c,S_c^{\smsh r})} \arrow{s,r}{\Delta_X} \\
  \node{(\Sigma^\infty X)^{\smsh r}} \arrow{e,t}{\isom}
    \node{S_c^{\smsh k} \smsh X^{\smsh r}}
\end{diagram} \]
where the top horizontal map is given by the Yoneda isomorphism
\[ \Nat(\Sigma^\infty \Omega^\infty E, E^{\smsh r}) \isom \Map(S_c,S_c^{\smsh r}) \]
and the right-hand vertical map consists of the natural evaluation
\[ S_c \smsh \Map(S_c,S_c^{\smsh r}) \to S_c^{\smsh r} \]
and the diagonal map
\[ \Delta_X: X \mapsto X^{\smsh r}. \]
\end{lemma}
\begin{proof}
This comes from the fact that the evaluation
\[ \Sigma^\infty \Omega^\infty E \smsh \Nat(\Sigma^\infty \Omega^\infty E,E^{\smsh r}) \to E^{\smsh r} \]
corresponds under the Yoneda isomorphism $\Nat(\Sigma^\infty \Omega^\infty E,E^{\smsh r}) \isom \Map(S_c,S_c^{\smsh r})$ to the map
\[ \begin{split}
    S_c \smsh \spectra(S_c,E) \smsh \Map(S_c,S_c^{\smsh r})
        &\to S_c^{\smsh r} \smsh \spectra(S_c,E)^{\smsh r} \\
        &\to S_c^{\smsh r} \smsh \spectra(S_c^{\smsh r},E^{\smsh r}) \\
        &\to E^{\smsh r}
\end{split} \]
where the first map involves the diagonal on the pointed simplicial set $\spectra(S_c,E)$.
\end{proof}

\begin{definition} \label{def:truncation-right}
Let $(\Sigma^\infty X)^{\leq n}$ denote the symmetric sequence given by
\[ (\Sigma^\infty X)^{\leq n}(r) := \begin{cases} (\Sigma^\infty X)^{\smsh r} & \text{for $1 \leq r \leq n$}; \\ * & \text{for $r > n$}. \end{cases} \]
We call this the \emph{truncation of $(\Sigma^\infty X)^{\smsh *}$ at the \ord{n} term}. This truncation inherits a right module structure from $(\Sigma^\infty X)^{\smsh *}$ and there is a natural morphism of right modules
\[ (\Sigma^\infty X)^{\smsh *} \to (\Sigma^\infty X)^{\leq n}. \]
\end{definition}

\begin{definition} \label{def:Phi}
Now let $F: \sset \to \spectra$ be a presented cell functor. Recall from Definition \ref{def:right} that $\der^*(F \Omega^\infty)$ is a right module over the operad $\der^*(\Sigma^\infty \Omega^\infty)$. Also recall the definition of $\Ext$-objects for pro-right-$P$-modules from Definition \ref{def:pro-ext-modules}. We then make the following definition:
\[ \Phi_n(F)(X) := \Ext^{\mathsf{right}}_{\tilde{\der}^*(\Sigma^\infty \Omega^\infty)}\left(\tilde{\der}^*(F \Omega^\infty), (\Sigma^\infty X)^{\leq n}\right). \]
Recall that the $\Ext$-objects for right $P$-modules are formed in $\spectra$. Hence $\Phi_n(F)$ is a functor from $\sset$ to $\spectra$.
\end{definition}

We now show that $\Phi_n(F)$ has \ord{n} Goodwillie derivative equivalent to the spectrum $\der_n(F)$ of Definition \ref{def:right-full}.

\begin{prop} \label{prop:phi-derivative}
Let $R$ be a levelwise-$\Sigma$-cofibrant directly-dualizable pro-right-module over the operad $\tilde{\der}^*(\Sigma^\infty \Omega^\infty)$. Then the functor $\sset \to \spectra$ given by
\[ X \mapsto \Ext^{\mathsf{right}}_{\tilde{\der}^*(\Sigma^\infty \Omega^\infty)}\left(R, (\Sigma^\infty X)^{\leq n}\right) \]
is a pointed simplicial homotopy functor and has \ord{n} Goodwillie derivative naturally (and $\Sigma_n$-equivariantly) equivalent to
\[ \dual B(R, \tilde{\der}^*(\Sigma^\infty \Omega^\infty), \mathsf{1})(n). \]
\end{prop}
\begin{proof}
For brevity, we write $P = \tilde{\der}^*(\Sigma^\infty \Omega^\infty)$ in this proof. We start with the case where the pro-module $R$ is indexed over the trivial category, so that $R$ is a just a right $P$-module in the usual sense.

Recall from Remark \ref{rem:ext-tot} that
\[ \Ext_P(R,(\Sigma^\infty X)^{\leq n}) \isom \Tot_{k \in \Delta} \left[ \Map_{\mathsf{\Sigma}}(R \circ P^{\circ k}, (\Sigma^\infty X)^{\leq n}) \right]. \]
The $k$-simplices in this cosimplicial object are isomorphic to
\[ \tag{*} \prod_{r = 1}^{n} \Map\left(R \circ P^{\circ k}(r), (\Sigma^\infty X)^{\smsh r}\right)^{\Sigma_r}. \]
Since $R$ and $P$ are $\Sigma$-cofibrant, so is $R \circ P^k$, by Lemma \ref{lem:sigma-cofibrant-comprod}. Therefore, by Lemma \ref{lem:F_E}, each of the terms in this product is a pointed simplicial homotopy functor. It follows that the $\Ext$-object under consideration is also a pointed simplicial homotopy functor.

Now let $(\Sigma^\infty X)^{=n}$ denote the symmetric sequence
\[ (\Sigma^\infty X)^{=n}(r) := \begin{cases} (\Sigma^\infty X)^{\smsh n} & \text{if $r = n$}; \\ * & \text{otherwise}. \end{cases} \]
Then $(\Sigma^\infty X)^{=n}$ has a trivial right $P$-module structure in which the only non-trivial structure map is the isomorphism
\[ \begin{split} (\Sigma^\infty X)^{\smsh n} P(1) \smsh \dots \smsh P(1)
    &\isom (\Sigma^\infty X)^{\smsh n} \smsh S \smsh \dots \smsh S \\
    &\isom (\Sigma^\infty X)^{\smsh n} \end{split} \]W
Relative to this module structure, the obvious inclusion defines a morphism of right $P$-modules
\[ \iota: (\Sigma^\infty X)^{=n} \to (\Sigma^\infty X)^{\leq n} \]
and we therefore have an induced map of spectra
\[ \iota_*: \Ext^{\mathsf{right}}_{P}(R,(\Sigma^\infty X)^{=n}) \to \Ext^{\mathsf{right}}_{P}(R,(\Sigma^\infty X)^{\leq n}). \]
Now we can write the source of $\iota_*$ as a totalization in the same way we did for the target. The map $\iota_*$ is then given by taking the totalization of a map of cosimplicial objects that on $k$-simplices is given by the inclusion
\[ \iota_*^k: \Map\left(R \circ P^{\circ k}(n), (\Sigma^\infty X)^{\smsh n}\right)^{\Sigma_n} \to \prod_{r = 1}^{n} \Map\left(R \circ P^{\circ k}(r), (\Sigma^\infty X)^{\smsh r}\right)^{\Sigma_r}. \]
The key step is now that the \ord{r} term in this product is $r$-excisive by Lemma \ref{lem:F_E} (and since precomposing with $\Sigma^\infty$ preserves $n$-excisive functors). This tells us that the map $\iota_*^k$ is a $D_n$-equivalence, and moreover, is an equivalence on \ord{n} cross-effects (because the \ord{n} cross-effect of an $(n-1)$-excisive functor is trivial). But taking totalization commutes with cross-effects (both are types of homotopy limit) and so it follows that the map $\iota_*$ is also an equivalence on \ord{n} cross-effects, and hence on Goodwillie derivatives.

We have therefore established that $\iota_*$ induces an equivalence of \ord{n} Goodwillie derivatives. We complete the proof of the proposition by calculating the \ord{n} derivative of $\Ext^{\mathsf{right}}_{P}(R,(\Sigma^\infty X)^{=n})$. Looking at the cosimplicial form for this discussed above (i.e. the source of the map $\iota_*^k$), we see that it is isomorphic to the cosimplicial object
\[ \Map \left(B_{\bullet}(R,P,\mathsf{1})(n),(\Sigma^\infty X)^{\smsh n}\right)^{\Sigma_n}. \]
The totalization of this is isomorphic to
\[ \Map \left( B(R,P,\mathsf{1})(n),(\Sigma^\infty X)^{\smsh n}\right)^{\Sigma_n} \]
which by Lemma \ref{lem:F_E} has \ord{n} Goodwillie derivative given by
\[ \dual B(R,P,\mathsf{1})(n). \]
This completes the case where $R$ is indexed over the trivial category. For a general pro-module $R$, the proposition follows by taking the relevant filtered homotopy colimit.
\end{proof}

\begin{cor} \label{cor:phi-derivative}
The functor $\Phi_n(F)$ (Definition \ref{def:Phi}) is a pointed simplicial homotopy functor whose \ord{n} Goodwillie derivative is equivalent to $\der_n(F)$ (Definition \ref{def:right-full}).
\end{cor}

\begin{remark} \label{rem:phi-n-excisive}
The proof of Proposition \ref{prop:phi-derivative} does not imply that $\Phi_n(F)$ is $n$-excisive. We see there that $\Phi_n(F)$ is the totalization of a cosimplicial object which is levelwise $n$-excisive, but totalization does not commute with $P_n$ so we cannot conclude that $\Phi_n(F)$ is $n$-excisive.

However, it is true that $\Phi_n(F)$ is $n$-excisive. This does not play any role in the rest of this paper, so we only provide a sketch of a proof. The key idea is to see that the totalization defining $\Phi_n(F)$ can be calculated at the $\Tot^n$ term in the totalization tower. The reason behind this is that \ord{r} term in the simplicial bar construction $B_{\bullet}(R,P,P)$ is degenerate above the $r$-simplices. Explicitly, we can see that every term in the composition product $(R \circ P^k \circ P)(r)$ is degenerate (i.e. comes from something in the $(R \circ P^{k-1} \circ P)(r)$ by applying a degeneracy) if $k > r$.
\end{remark}

We now turn to the second part of the proof of Theorem \ref{thm:right} which is to construct a natural transformation relating $F$ and $\Phi_n(F)$ that induces an equivalence on \ord{n} derivatives.

\begin{definition} \label{def:phi'-F}
First let $F: \sset \to \spectra$ be a finite cell functor. We define maps
\[ \phi'_F(r) := F(X) \smsh \der^r(F \Omega^\infty) \to (\Sigma^\infty X)^{\smsh r}. \]
These are made by combining the map
\[ F(X) \to F \Omega^\infty \Sigma^\infty (X) \]
that comes from the $(\Sigma^\infty, \Omega^\infty)$-adjunction with the evaluation map
\[ F \Omega^\infty \Sigma^\infty (X) \smsh \Nat(F \Omega^\infty E, E^{\smsh r}) \to (\Sigma^\infty X)^{\smsh r}. \]
The maps $\phi'_F(r)$ are $\Sigma_r$-equivariant and so together define a map
\[ \phi'_F: F(X) \to \Map_{\mathsf{\Sigma}}\left(\der^*(F \Omega^\infty), (\Sigma^\infty X)^{\smsh *}\right). \]
\end{definition}

\begin{lemma} \label{lem:phi'-F}
The map $\phi'_F$ of Definition \ref{def:phi'-F} factors via the corresponding mapping object for right $\tilde{\der}^*(\Sigma^\infty \Omega^\infty)$-modules instead of symmetric sequences. In other words, we have a commutative diagram:
\[ \begin{diagram}
  \node{F(X)} \arrow{e,t}{\phi''_F} \arrow{se,b}{\phi'_F}
    \node{\Map^{\mathsf{right}}_{\tilde{\der}^*(\Sigma^\infty \Omega^\infty)}\left(\der^*(F \Omega^\infty), (\Sigma^\infty X)^{\smsh *}\right)} \arrow{s} \\
  \node[2]{\Map_{\mathsf{\Sigma}}\left(\der^*(F \Omega^\infty), (\Sigma^\infty X)^{\smsh *}\right)}
\end{diagram} \]
\end{lemma}
\begin{proof}
This is equivalent to saying that the following diagram commutes:
\[ \begin{diagram}
  \node{F(X) \smsh \der^k(F \Omega^\infty) \smsh \der^{r_1}(\Sigma^\infty \Omega^\infty) \smsh \dots \smsh \der^{r_k}(\Sigma^\infty \Omega^\infty)} \arrow{e} \arrow{s,r}{\phi'_F(k)} \node{C(X) \smsh \der^{r_1+\dots+r_k}(C \Omega^\infty)} \arrow{s,r}{\phi'_F(r_1+\dots+r_k)} \\
  \node{(\Sigma^\infty X)^{\smsh k} \smsh \der^{r_1}(\Sigma^\infty \Omega^\infty) \smsh \dots \smsh \der^{r_k}(\Sigma^\infty \Omega^\infty)} \arrow{e} \node{(\Sigma^\infty X)^{\smsh (r_1+\dots+r_k)}}
\end{diagram} \]
where the top and bottom maps are given by the module structures on $\der^*(C\Omega^\infty)$ and $(\Sigma^\infty X)^{\smsh *}$ respectively.

This in turn boils down to the commutativity of the following diagrams
\[ \begin{diagram}
  \node{C \Omega^\infty (\Sigma^\infty X) \smsh \Nat(C \Omega^\infty E, E^{\smsh r})} \arrow{s} \arrow{e}
    \node{(\Sigma^\infty X)^{\smsh r}} \arrow{s} \\
  \node{C \Omega^\infty (\Sigma^\infty \Omega^\infty \Sigma^\infty X) \smsh \Nat(C \Omega^\infty E, E^{\smsh r})} \arrow{e}
    \node{(\Sigma^\infty \Omega^\infty \Sigma^\infty X)^{\smsh r}}
\end{diagram} \]
which follows from the naturality of the horizontal evaluation maps.
\end{proof}

\begin{definition} \label{def:phi}
With $F: \sset \to \spectra$ still a finite cell functor, we can now construct a natural transformation $\phi: F \to \Phi_n(F)$ based on the maps
\[ \phi''_F:  F(X) \to \Map^{\mathsf{right}}_{\tilde{\der}^*(\Sigma^\infty \Omega^\infty)}\left(\der^*(F \Omega^\infty), (\Sigma^\infty X)^{\smsh *}\right) \]
of Lemma \ref{lem:phi'-F}. We combine these with the cofibrant replacement map
\[ \tilde{\der}^*(F \Omega^\infty) \to \der^*(F \Omega^\infty), \]
the truncation map
\[ (\Sigma^\infty X)^{\smsh *} \to (\Sigma^\infty X)^{\leq n} \]
and the map from the strict mapping object for right modules to the derived mapping object
\[ \Map^{\mathsf{right}}_P(R,R') \to \Ext^{\mathsf{right}}_P(R,R') \]
to get a map of the form
\[ \phi_F: F(X) \to \Ext_{\tilde{\der}^*(\Sigma^\infty \Omega^\infty)}\left(\tilde{\der}^*(F \Omega^\infty), (\Sigma^\infty X)^{\leq n}\right) = \Phi_n(F). \]
Finally, given any pointed simplicial functor $F: \sset \to \spectra$, we define $\phi_F: F \to \Phi_n(F)$ by taking the homotopy colimit over finite subcomplexes $C \subset QF$ of the maps $\phi_C$ above.
\end{definition}

\begin{remark} \label{rem:phi}
Strictly speaking, the source of the natural transformation $\phi: F \to \Phi_n(F)$ should be written as
\[ \hocolim_{C \in \mathsf{Sub}(QF)} C \]
but this is equivalent to $F$ by Corollary \ref{cor:hocolim}. We abuse the notation slightly and just write $\phi$ as a map from $F$ to $\Phi_n(F)$.
\end{remark}

To complete the proof of Theorem \ref{thm:right}, it is now sufficient to show that $\phi$ induces an equivalence of \ord{n} Goodwillie derivatives between $F$ and $\Phi_n(F)$. We start by constructing an alternative description of $\phi$.

\begin{definition} \label{def:phi-FOS}
Let $F: \sset \to \spectra$ be a finite cell functor. The construction of Definition \ref{def:phi'-F} gives us a map
\[ F \Omega^\infty \dots \Sigma^\infty X \to \Map_{\mathsf{\Sigma}} \left( \Nat(F \Omega^\infty \dots \Sigma^\infty \Omega^\infty E, E^{\smsh *}), (\Sigma^\infty X)^{\smsh *} \right). \]
The construction of Definition \ref{def:composition} gives a map
\[ \der^*(F \Omega^\infty) \circ \dots \circ \der^*(\Sigma^\infty \Omega^\infty) \to \Nat(F \Omega^\infty \dots \Sigma^\infty \Omega^\infty E, E^{\smsh *}). \]
Together with the previous map, this gives
\[ F \Omega^\infty \dots \Sigma^\infty X \to \Map_{\mathsf{\Sigma}} \left( \der^*(F \Omega^\infty) \circ \dots \circ \der^*(\Sigma^\infty), (\Sigma^\infty X)^{\smsh *} \right). \]
Using the argument of Lemma \ref{lem:phi'-F}, we can see that this factors via the corresponding mapping object for right $\der^*(\Sigma^\infty \Omega^\infty)$-modules, where the right module structure on
\[ \der^*(F \Omega^\infty) \circ \dots \circ \der^*(\Sigma^\infty \Omega^\infty) \]
is given by the regular action on the rightmost term. Composing also with appropriate cofibrant replacements, we get maps
\[ F \Omega^\infty \dots \Sigma^\infty X \to \Map^{\mathsf{right}}_{\tilde{\der}^*(\Sigma^\infty \Omega^\infty)} \left( \tilde{\der}^*(F \Omega^\infty) \circ \dots \circ \tilde{\der}^*(\Sigma^\infty \Omega^\infty), (\Sigma^\infty X)^{\smsh *} \right). \]
Finally, suppose that $F: \sset \to \spectra$ is any pointed simplicial functor. Taking the homotopy colimit of the above maps over finite subcomplexes $C \subset QF$, we get
\[ F \Omega^\infty \dots \Sigma^\infty X \to \Map^{\mathsf{right}}_{\tilde{\der}^*(\Sigma^\infty \Omega^\infty)} \left( \tilde{\der}^*(F \Omega^\infty) \circ \dots \circ \tilde{\der}^*(\Sigma^\infty \Omega^\infty), (\Sigma^\infty X)^{\smsh *} \right). \]
and we denote this map $\phi''_{F \Omega^\infty \dots \Sigma^\infty}$.
\end{definition}

\begin{lemma} \label{lem:phi-tot}
Let $F: \sset \to \spectra$ be a pointed simplicial functor. Then we have a commutative diagram
\[ \begin{diagram} \dgARROWLENGTH=3.5em
  \node{F(X)} \arrow{e,t}{\phi''_F} \arrow{s}
    \node{\Map^{\mathsf{right}}_{P}\left(\tilde{\der}^*(F \Omega^\infty), (\Sigma^\infty X)^{\leq n}\right)} \arrow{s} \\
  \node{\widetilde{\Tot} \left[ F \Omega^\infty \dots \Sigma^\infty (X) \right]} \arrow{e,t}{\phi''_{F \Omega^\infty \dots \Sigma^\infty}}
    \node{\widetilde{\Tot} \left[\Map^{\mathsf{right}}_{P}\left(\tilde{\der}^*(F \Omega^\infty) \circ \dots \circ \tilde{\der}^*(\Sigma^\infty \Omega^\infty), (\Sigma^\infty X)^{\leq n}\right)\right]}
\end{diagram} \]
where $P = \tilde{\der}^*(\Sigma^\infty \Omega^\infty)$. The bottom-left term is the (homotopy-invariant) totalization of the cosimplicial object formed from the $(\Sigma^\infty,\Omega^\infty)$ adjunction (as in Theorem \ref{thm:key}). The bottom-right is the totalization of the cosimplicial object formed from the operad structure on $\der^*(\Sigma^\infty \Omega^\infty)$ and the pro-right-module structure on $\der^*(F \Omega^\infty)$. The vertical maps come from standard coaugmentations of those cosimplicial objects, and the horizontal maps are as constructed in Lemma \ref{lem:phi'-F} and Definition \ref{def:phi-FOS}.
\end{lemma}
\begin{proof}
This all follows from the similarity between (and naturality of) the constructions of the maps $\phi''_F$ and $\phi''_{F \Omega^\infty \dots \Sigma^\infty}$.
\end{proof}

\begin{remark}
The cosimplicial object involved in the bottom-right corner of the diagram in Lemma \ref{lem:phi-tot} is exactly that whose (strict) totalization is
\[ \Ext^{\mathsf{right}}_{P}(\tilde{\der}^*(F \Omega^\infty), (\Sigma^\infty X)^{\leq n}) = \Phi_n(F). \]
We noted in Remark \ref{rem:ext-tot} that this cosimplicial object is Reedy fibrant. It follows that the homotopy-invariant $\Tot$ in that diagram is weakly equivalent to the strict $\Tot$. The composite of the top and right-hand maps in this diagram is, up to that weak equivalence, the definition of $\phi: F \to \Phi_n(F)$. The diagram in Lemma \ref{lem:phi-tot} therefore provides, up to weak equivalence, a factorization of $\phi$.
\end{remark}

\begin{lemma} \label{lem:phi-FOS}
The map
\[ \phi''_{F \Omega^\infty \dots \Sigma^\infty}: F \Omega^\infty \dots \Sigma^\infty(X) \to \Map^{\mathsf{right}}_{P}\left(\tilde{\der}^*(F \Omega^\infty) \circ \dots \circ \tilde{\der}^*(\Sigma^\infty \Omega^\infty), (\Sigma^\infty X)^{\leq n}\right). \]
of Definition \ref{def:phi-FOS} is a $D_n$-equivalence.
\end{lemma}
\begin{proof}
Consider the following diagram:
\[ \begin{diagram}
  \node{Q(F \Omega^\infty \dots)(\Sigma^\infty X)} \arrow{e,t}{(1)} \arrow{s,l}{\sim}
    \node{\Map_{\mathsf{\Sigma}}\left( \tilde{\der}^*(Q(F \Omega^\infty \dots)), (\Sigma^\infty X)^{\leq n} \right)} \arrow{s,r}{(2)} \\
  \node{(F \Omega^\infty \dots)(\Sigma^\infty X)} \arrow{e,t}{(3)} \arrow{se,b}{\phi''_{F \Omega^\infty \dots \Sigma^\infty}}
    \node{\Map_{\mathsf{\Sigma}}\left( \tilde{\der}^*(F \Omega^\infty) \circ \dots, (\Sigma^\infty X)^{\leq n} \right)} \arrow{s,r}{(4)} \\
  \node[2]{\Map^{\mathsf{right}}_{P}\left( \tilde{\der}^*(F \Omega^\infty) \circ \dots \circ \tilde{\der}^*(\Sigma^\infty \Omega^\infty), (\Sigma^\infty X)^{\leq n} \right)}
\end{diagram} \]
where we describe the numbered maps as follows:
\begin{enumerate}
  \item The target of this map is isomorphic to
  \[ \prod_{r = 1}^{n} \Map(\tilde{\der}^*(Q(F \Omega^\infty \dots)), (\Sigma^\infty X)^{\smsh r})^{\Sigma_r}. \]
  and so we take (1) to be given by the relevant maps $\psi$ from Proposition \ref{prop:models-correct}.

  We also claim that (1) is a $D_n$-equivalence. The \ord{r} term in the product is $r$-excisive by Lemma \ref{lem:F_E} and so taking $D_n$ we only need to consider the $r = n$ term. The map (1) is then a $D_n$-equivalence by Proposition \ref{prop:models-correct}.
  \item This is given by an iterated version of the map $\mu$ of Definition \ref{def:mu}. By Theorem \ref{thm:chainrule} and Lemma \ref{lem:pro-map-invariance}, this is an equivalence.
  \item This is similar to Definition \ref{def:phi-FOS} but without introducing an extra copy of $\Omega^\infty \Sigma^\infty$ as was done in Definition \ref{def:phi'-F}.
  \item This is the adjunction isomorphism of Lemma \ref{lem:adjunction}.
\end{enumerate}
Following through all of these definitions, we see that this square commutes. Since (1) is a $D_n$-equivalence, (2) is an equivalence, and (4) is an isomorphism, we deduce that $\phi''_{F \Omega^\infty \dots \Sigma^\infty}$ is a $D_n$-equivalence as claimed.
\end{proof}

\begin{prop} \label{prop:phi}
Let $F: \sset \to \spectra$ be a pointed simplicial homotopy functor. The map
\[ \phi: F \to \Phi_n(F) \]
of Definition \ref{def:phi} is a $D_n$-equivalence.
\end{prop}
\begin{proof}
By Theorem \ref{thm:key}, the following map is an equivalence
\[ D_nF \to \widetilde{\Tot} \; D_n(F \Omega^\infty \dots \Sigma^\infty). \]
Set
\[ \Phi^{\bullet}_n(F) := \Map^{\mathsf{right}}_{P}\left(B_{\bullet}(\tilde{\der}^*(F \Omega^\infty),P,P), (\Sigma^\infty X)^{\leq n}\right), \]
so that by definition
\[ \Phi_n(F) := \widetilde{\Tot} \; \Phi^{\bullet}_n(F). \]
Then Theorem \ref{thm:key} and Lemma \ref{lem:phi-FOS} together give us an equivalence
\[ D_nF \weq \widetilde{\Tot} \; D_n(F \Omega^\infty \dots \Sigma^\infty) \weq \widetilde{\Tot} \; D_n(\Phi^{\bullet}_n(F)) \]
and hence an equivalence
\[ D_n(D_nF) \to D_n(\widetilde{\Tot} \; D_n(\Phi^{\bullet}_n(F))). \]
Now consider the commutative diagram
\[ \begin{diagram}
  \node{D_n(D_nF)} \arrow{s,l}{\sim} \arrow{e,t}{\sim} \node{D_n(\widetilde{\Tot} \; D_n(\Phi^{\bullet}_n(F)))} \arrow{s} \\
  \node{D_n(P_nF)} \arrow{e} \node{D_n(\widetilde{\Tot} \; P_n(\Phi^{\bullet}_n(F)))} \\
  \node{D_n(F)} \arrow{e,t}{D_n(\phi)} \arrow{n,l}{\sim} \node{D_n(\widetilde{\Tot} \; \Phi^{\bullet}_n(F))} \arrow{n,r}{\sim}
\end{diagram} \]
We just saw that the top map is an equivalence. The left-hand vertical maps are equivalence by standard results of calculus. The bottom-right vertical map is an equivalence because each $\Phi^k_n(F)$ is $n$-excisive by Lemma \ref{lem:F_E} (see also the proof of Proposition \ref{prop:phi-derivative}). Finally, to see that the top-right vertical map is an equivalence, note that it is sufficient to show it is an equivalence on \ord{n} cross-effects instead of $D_n$. But taking cross-effects commutes with $\widetilde{\Tot}$ and for any functor $G$, $D_nG \to P_nG$ induces an equivalence on \ord{n} cross-effects, so we are done. The conclusion therefore is that the map $\phi$ is a $D_n$-equivalence.
\end{proof}

We can now deduce Theorem \ref{thm:right} which we restate here.

\begin{thm:right}
Let $F: \sset \to \spectra$ be a pointed simplicial homotopy functor. Then there is a natural equivalence of symmetric sequences
\[ \der_*(F) \homeq \der^G_*(F). \]
That is, the right $\der_*(I)$-module $\der_*(F)$ consists of models for the Goodwillie derivatives of $F$.
\end{thm:right}
\begin{proof}[Proof of Theorem \ref{thm:right}]
This follows from Corollary \ref{cor:phi-derivative} and Proposition \ref{prop:phi}.
\end{proof}

\begin{remark} \label{rem:tate}
In the course of proving Theorem \ref{thm:right} we constructed natural transformations $F \to \Phi_nF$ with $\Phi_nF$ an $n$-excisive functor (see Remark \ref{rem:phi-n-excisive}). The natural truncation maps
\[ (\Sigma^\infty X)^{\leq n} \to (\Sigma^\infty X)^{\leq (n-1)} \]
determine natural transformations
\[ \Phi_nF \to \Phi_{n-1}F \]
and we therefore obtain a tower of functors
\[ F \to \dots \to \Phi_nF \to \Phi_{n-1}F \to \dots \to \Phi_0(F) = *. \]
This, however, in general is not equivalent to the Taylor tower of $F$. To see this, we can calculate the fibre $\Delta_nF$ of the map $\Phi_nF \to \Phi_{n-1}F$. This turns out to be
\[ \Delta_nF := \Map\left(B(\tilde{\der}^*(F\Omega^\infty), \tilde{\der}^*(\Sigma^\infty \Omega^\infty), \mathsf{1})(n), (\Sigma^\infty X)^{\smsh n}\right)^{\Sigma_n}. \]
(Note that this object appeared in the proof of Proposition \ref{prop:phi-derivative}, where we showed that it is $D_n$-equivalent to $\Phi_nF$.) This in turn is equivalent to
\[ \left[\dual B(\tilde{\der}^*(F\Omega^\infty), \tilde{\der}^*(\Sigma^\infty \Omega^\infty), \mathsf{1})(n) \smsh (\Sigma^\infty X)^{\smsh n}\right]^{h\Sigma_n} \]
or, by Theorem \ref{thm:right}, equivalent to
\[ \left[ \der_nF \smsh (\Sigma^\infty X)^{\smsh n}\right]^{h\Sigma_n}. \]
Therefore $\Delta_nF$ is not equivalent to $D_nF$ (which it would be if the $\Phi_nF$ formed the Taylor tower), but instead fits into a fibre sequence
\[ \tag{*} D_nF \to \Delta_nF \to \operatorname{Tate}_{\Sigma_n}(\der_nF \smsh (\Sigma^\infty X)^{\smsh n}) \]
coming from the norm map relating homotopy orbits and homotopy fixed points. The last term is the Tate spectrum for the action of $\Sigma_n$ on $\der_nF \smsh (\Sigma^\infty X)^{\smsh n}$ (see \cite{greenlees/may:1995} for more details).

Another way to obtain this map is to note that by the universal property of the Taylor tower, the map $\phi: F \to \Phi_nF$ factors as
\[ F \to P_nF \to \Phi_nF. \]
The resulting maps $P_nF \to \Phi_nF$ are compatible with the tower maps, and so we get an induced map on fibres. Altogether we have a map of fibre sequences
\[ \begin{diagram}
  \node{D_nF} \arrow{s} \arrow{e} \node{\Delta_nF} \arrow{s} \\
  \node{P_nF} \arrow{s} \arrow{e} \node{\Phi_nF} \arrow{s} \\
  \node{P_{n-1}F} \arrow{e} \node{\Phi_{n-1}F}
\end{diagram} \]
We now notice that the Tate spectra appearing in the sequence (*) are precisely the obstructions to the tower $\Phi_*F$ being exactly the Taylor tower of $F$. In particular, if all these Tate spectra vanish, then $P_nF \homeq \Phi_nF$ for all $n$.

Notice that the sequence of functors $\Phi_*F$ is determined completely by the right $\der^*(\Sigma^\infty \Omega^\infty)$-module $\der^*(F \Omega^\infty)$. We can think of $\Phi_*F$ as the best approximation to the Taylor tower of $F$ based on the information contained in that module. In particular, if those Tate spectra all vanish, then the Taylor tower of $F$ is determined by $\der^*(F \Omega^\infty)$, and hence also, in fact, by the right $\der_*(I)$-module $\der_*(F)$. (This last claim is true because the chain rules we prove in \S\ref{sec:chainrule-spaces} below, allow us to recover the derivatives (and module structure) of $F \Omega^\infty$ from those of $F$. See Example \ref{ex:chainrule}.)

These comments should be viewed as analogous to the work of McCarthy \cite{mccarthy:2001} showing that the Taylor tower of a functors from spectra to spectra splits (i.e. is determined by the derivatives) when the corresponding Tate objects vanish. This work was used extensively by Kuhn \cite{kuhn:2004} to study functors of spectra localized at the Morava K-theories.
\end{remark}

\begin{example}[Stable mapping spaces] \label{ex:right}
Let $K$ be a finite cell complex in $\sset$ and consider the functor
\[ \sset \to \spectra; \quad X \mapsto \Sigma^\infty \sset(K,X). \]
The first author showed in \cite{arone:1999} that
\[ \der_n(\Sigma^\infty \sset(K,-)) \homeq \dual K^{\smsh n}/\Delta^nK \]
where here $\Delta^nK$ denotes the `fat' diagonal:
\[ \Delta^nK = \{(k_1,\dots,k_n) \in K^{\smsh n} \; | \; k_i = k_j \text{ for some $i \neq j$} \}. \]
In our case, the functors $\Sigma^\infty \sset(K,-)$ are finite cell functors in $[\finsset,\spectra]$ so our models for their derivatives are particularly easy to calculate. We have
\[ \Sigma^\infty \sset(K,\Omega^\infty X) \isom S_c \smsh \spectra(\Sigma^\infty K, X) \]
and so
\[ \der^n(\Sigma^\infty \sset(K,\Omega^\infty(-)) = \Nat(S_c \smsh \spectra(\Sigma^\infty K, X), X^{\smsh n}) \isom \Map(S_c,(\Sigma^\infty K)^{\smsh n}) \homeq (\Sigma^\infty K)^{\smsh n} \]
by the Yoneda Lemma. The right $\der^*(\Sigma^\infty \Omega^\infty)$-module structure is then given by Definition \ref{def:sigmainfty-module}.

We now obtain $\der_*(\Sigma^\infty \sset(K,-))$ by taking the dual of the bar construction. In this case, we can factor the suspension spectrum out of the bar construction, which can be done entirely on the space level. We then obtain
\[ \der_*(\Sigma^\infty \sset(K,-)) \homeq \dual B(K^{\smsh *}, \mathsf{Com}, \mathsf{1}) \]
where this bar construction is over the commutative operad in $\sset$, with right module $K^{\smsh *}$ given in much the same way as Definition \ref{def:sigmainfty-module}. Now using the description of this bar construction in terms of trees (see \cite{ching:2005a}) we can show that
\[ B(K^{\smsh *}, \mathsf{Com}, \mathsf{1})(n) \homeq K^{\smsh n}/\Delta^nK \]
which recovers the first author's previous calculation. The right $\der_*(I)$-module structure on the symmetric sequence $\der_*(\Sigma^\infty \sset(K,-))$ is then dual to a right comodule structure on this bar construction over the cooperad formed by the partition poset complexes (whose duals are equivalent to $\der_*(I)$ -- see \cite{arone/mahowald:1999}).
\end{example}

\begin{example}[Waldhausen's algebraic K-theory of spaces] \label{ex:A(X)}
For a pointed simplicial set $X$, let $A(X)$ denote Waldhausen's algebraic K-theory of spaces functor applied to the geometric realization of $X$. The functor $A(X)$ is not reduced ($A(*)$ is equal to the algebraic K-theory of the sphere spectrum), so we consider instead the relative version:
\[ \tilde{A}(X) := \hofib(A(X) \to A(*)). \]
Note that for pointed $X$, there is a splitting
\[ A(X) \homeq \tilde{A}(X) \wdge A(*) \]
and so we can instead write
\[ \tilde{A}(X) \homeq \hocofib(A(*) \to A(X)). \]
In \cite[\S9]{goodwillie:2003}, Goodwillie calculates the derivatives of $A$ (and hence of $\tilde{A}$ since they are the same). He does this via a \emph{trace map}
\[ \tau: A(X) \to L(X) \]
where $L(X) = \Sigma^\infty (X^{S^1})_+$, and $X^{S^1}$ denotes the \emph{free loop-space}, i.e. the space of \emph{unbased} maps $S^1 \to X$. Relating the free loop-space to our notation, where everything is pointed, we have a cofibre sequence (of simplicial sets):
\[ (*^{S^1})_+ \to (X^{S^1})_+ \to X^{S^1} \isom \sset((S^1)_+,X). \]
It follows that for pointed $X \in \sset$, the relative version of $L(X)$ gives:
\[ \tilde{L}(X) := \hocofib(L(*) \to L(X)) \homeq \Sigma^\infty \sset((S^1)_+,X). \]
In other words, $\tilde{L}(X)$ is one of the stable mapping spaces considered in Example \ref{ex:right}, namely $F_{(S^1)_+}$.

The trace map $\tau$ induces a map $\tilde{A}(X) \to \tilde{L}(X)$ and hence a map on derivatives
\[ \tau_*: \der_*(\tilde{A}) \to \der_*(\tilde{L}). \]
Now $\tau$ is $S^1$-equivariant with respect to the regular $S^1$-action on $L(X)$, and the trivial action on $A(X)$. This carries over to the derivatives, and so $\tau_*$ factors via
\[ \der_*(\tilde{A}) \to [\der_*(\tilde{L})]^{hS^1} \]
Goodwillie's result is then that this map is an equivalence. (He actually proves a more general version involving the derivatives of $\tilde{A}$ at \emph{any} base space, not just for derivatives at $*$.)

We now use our knowledge of the right $\der_*(I)$-module structure on the derivatives of $\tilde{L}$ (from Example \ref{ex:right}) to calculate the module structure on the derivatives of $\tilde{A}$. This is very simple because the naturality tells us that $\tau_*$ is a morphism of right $\der_*(I)$-modules. The equivariance of $\tau_*$ then gives us a factorization
\[ \der_*(\tilde{A}) \to [\der_*(\tilde{L})]^{hS^1} \]
in the category of right modules. However, homotopy fixed points of right modules are calculated termwise, and so this map is the equivalence considered by Goodwillie. Therefore, as right $\der_*(I)$-modules, we have
\[ \der_*(\tilde{A}) \homeq [\der_*(\tilde{L})]^{hS^1}. \]
By Example \ref{ex:right}, we have
\[ \der_*(\tilde{L}) \homeq \dual B(((S^1)_+)^{\smsh *}, \mathsf{Com}, \mathsf{1}) = \dual B((S^1)^*_+, \mathsf{Com}, \mathsf{1}). \]
Therefore, since the $S^1$-action is free,
\[ \der_*(\tilde{A}) \homeq \dual \left[ B((S^1)^*_+, \mathsf{Com},\mathsf{1})/S^1 \right]. \]
The right-hand side here is equivalent to
\[ \dual \left[((S^1)^n/\Delta^n(S^1))/(S^1) \right] \]
which is Goodwillie's original description of these derivatives. The bar construction version (which can be understood explicitly in terms of spaces of trees, see \cite{ching:2005a}) allows the module structure to be seen.

As we noted in Remark \ref{rem:tate}, the right $\der_*(I)$-module $\der_*(\tilde{A})$ contains `more' of the information of the full Taylor tower of $\tilde{A}$ than just the derivatives. In particular, we can build the `fake' Taylor tower $\Phi_*(\tilde{A})$ that differs from the actual Taylor tower by certain Tate spectra.
\end{example}

\section{Functors from spectra to spaces} \label{sec:spectra-spaces}

We now consider functors from spectra to spaces. In this case the derivatives have the property that they form a left module over the operad $\der_*(I)$. Most of the material here is dual in a sense to that of section \ref{sec:spaces-spectra} (with right modules replaced by left modules), and so we abbreviate some of the exposition.

\begin{lemma} \label{lem:extend-sigma}
Let $F: \finspec \to \sset$ be a presented cell functor. Then $\Sigma^\infty F$ has a natural structure of a presented cell functor in $[\finspec,\spectra]$ with cells in 1-1 correspondence to those in $F$.
\end{lemma}
\begin{proof}
The generating cofibrations in $[\finspec,\sset]$ are of the form
\[ I_0 \smsh \spectra(K,-) \to I_1 \smsh \spectra(K,-) \]
where $K \in \finspec$ and $I_0 \to I_1$ is one of the generating cofibrations in $\sset$. Recall that then $\Sigma^\infty I_0 \to \Sigma^\infty I_1$ is one of the generating cofibrations in $\spectra$ and so the induced map
\[ \Sigma^\infty I_0 \smsh \spectra(K,-) \to \Sigma^\infty I_1 \smsh \spectra(K,-) \]
is one of the generating cofibrations in $[\finspec,\spectra]$. Since $\Sigma^\infty$ preserves colimits, it follows that attaching diagrams for cells in $F$ determine attaching diagrams for a cell structure on $\Sigma^\infty F$ with
\[ (\Sigma^\infty F)_i = \Sigma^\infty F_i. \]
\end{proof}

\begin{definition} \label{def:left}
Let $F: \finspec \to \sset$ be a presented cell functor. We then have a left comodule structure map
\[ l: \Sigma^\infty F \to \Sigma^\infty \Omega^\infty \Sigma^\infty F \]
for $\Sigma^\infty F$ over the comonad $\Sigma^\infty \Omega^\infty$. According to Proposition \ref{prop:module}, this determines structure maps
\[ \der^*(\Sigma^\infty \Omega^\infty) \circ \der^*(\Sigma^\infty F) \to \der^*(\Sigma^\infty F) \]
that make the symmetric sequence $\der^*(\Sigma^\infty F)$ into a pro-left-module over the operad $\der^*(\Sigma^\infty \Omega^\infty)$. We then define
\[ \der^*(F) := B(\mathsf{1},\tilde{\der}^*(\Sigma^\infty \Omega^\infty),\tilde{\der}^*(\Sigma^\infty F)). \]
This is a pro-right-comodule over $B(\tilde{\der}^*(\Sigma^\infty \Omega^\infty))$.

Now for a general pointed simplicial functor $F:\spectra \to \sset$, we set
\[ \der_*(F) := \dual \der^*(QF) \]
where $QF$ denotes the cellular replacement for $F$ in $[\finspec,\sset]$. The symmetric sequence $\der_*(F)$ is then a left $\der_*(I)$-module. Explicitly, we have
\[ \der_*(F) := \hocolim_{C \in \mathsf{Sub}(QF)} \Map(B(\mathsf{1},\tilde{\der}^*(\Sigma^\infty \Omega^\infty), \tilde{\der}^*(\Sigma^\infty C)), S) \]
with the homotopy colimit formed in the category of left $\der_*I$-modules.
\end{definition}

\begin{theorem} \label{thm:left}
Let $F: \spectra \to \sset$ be a pointed simplicial homotopy functor. Then there is a natural equivalence of symmetric sequences
\[ \der_*(F) \homeq \der^G_*(F). \]
In other words, the left $\der_*(I)$-module constructed in Definition \ref{def:left} provides models for the Goodwillie derivatives of $F$.
\end{theorem}

Our proof of Theorem \ref{thm:left} is similar to that of \ref{thm:right} with a few variations in the constructions. We start with a construction that plays the role (for left modules) of the right module $(\Sigma^\infty X)^{\smsh *}$ of Definition \ref{def:sigmainfty-module}.

\begin{definition} \label{def:sigmainfty-module-left}
For any $X \in \spectra$, we define a left $\der^*(\Sigma^\infty \Omega^\infty)$-module structure on the symmetric sequence
\[ \Map(S_c,X^{\smsh *}). \]
Recall that $\der^n(\Sigma^\infty \Omega^\infty) \isom \Map(S_c,S_c^{\smsh n})$. The left module structure maps are then of the form
\[ \Map(S_c,S_c^{\smsh k}) \smsh \Map(S_c,X^{\smsh n_1}) \smsh \dots \smsh \Map(S_c,X^{\smsh n_k}) \to \Map(S_c,X^{\smsh (n_1+\dots+n_k)}) \]
and are given by smashing together terms of the form $\Map(S_c,X^{\smsh n_i})$ to get
\[ \Map(S_c ,S_c^{\smsh k}) \smsh \Map(S_c^{\smsh k},X^{\smsh (n_1+\dots+n_k)}) \]
and then using the usual composition of mapping objects.

The truncation $\Map(S_c,X^{\leq n})$ of this symmetric sequence (see Definition \ref{def:truncation-right}) then inherits a left module structure and there is a natural map of left modules:
\[ \Map(S_c,X^{\smsh *}) \to \Map(S_c,X^{\leq n}). \]
\end{definition}

\begin{remark}
Note that we have a natural equivalence $E \weq \Map(S_c,E)$ for all spectra $E$, so the symmetric sequence $\Map(S_c,X^{\smsh *})$ is equivalent to $X^{\smsh *}$. The left module structure of Definition \ref{def:sigmainfty-module-left} is equivalent to the left $\mathsf{Com}$-module structure on $X^{\smsh *}$ given by the isomorphisms
\[ S \smsh X^{\smsh n_1} \smsh \dots \smsh X^{\smsh n_k} \isom X^{\smsh (n_1+\dots+n_k)}. \]
\end{remark}

\begin{definition} \label{def:Phi-left}
Let $F: \spectra \to \sset$ be a presented cell functor. We define $\Phi_n(F): \spectra \to \sset$ by
\[ \Phi_n(F) := \Ext^{\mathsf{left}}_{\tilde{\der}^*(\Sigma^\infty \Omega^\infty)}\left(\tilde{\der}^*(\Sigma^\infty F), \Map(S_c,X^{\leq n})\right) \]
where the $\Ext$-objects are here calculated in the category of left $P$-modules (with $P = \tilde{\der}^*(\Sigma^\infty \Omega^\infty)$). Recall from Definition \ref{def:ext-modules} that these $\Ext$-objects are pointed simplicial sets, not spectra. Therefore, $\Phi_n(F)$ is a functor from $\spectra$ to $\sset$.
\end{definition}

We now show that $\Phi_n(F)$ has \ord{n} Goodwillie derivative equivalent to the object $\der_n(F)$ of Definition \ref{def:left}.

\begin{prop} \label{prop:Phi-left-derivative}
Let $P$ denote the operad $\tilde{\der}^*(\Sigma^\infty \Omega^\infty)$ and let $L$ be a levelwise $\Sigma$-cofibrant directly-dualizable pro-left-$P$-module. Then the functor $\spectra \to \sset$ given by:
\[ X \mapsto \Ext^{\mathsf{left}}_{P}\left(L,\Map(S_c,X^{\leq n})\right) \]
is a pointed simplicial homotopy functor with \ord{n} Goodwillie derivative naturally equivalent to
\[ \dual B(\mathsf{1},P,L)(n). \]
\end{prop}
\begin{proof}
This is proved in a similar fashion to Proposition \ref{prop:phi-derivative}. We can identify the $k$-simplices in the cosimplicial object defining this $\Ext$ as
\[ \prod_{r = 1}^{n} \Map((P^k \circ L)(r),\Map(S_c,X^{\smsh r}))^{\Sigma_r} \]
which shows that $\Phi_nF$ is a pointed simplicial homotopy functor (by Lemma \ref{lem:F_E}). We then consider the map
\[ \iota_*: \Ext^{\mathsf{left}}_{P}\left(L, \Map(S_c,X^{=n})\right) \to \Ext^{\mathsf{left}}_{P}\left(L, \Map(S_c,X^{\leq n})\right) \]
induced by the inclusion of left modules
\[ \iota: \Map(S_c,X^{=n}) \to \Map(S_c,X^{\leq n}). \]
By Lemma \ref{lem:F_E} again we see that $\iota_*^k$, the map on $k$-simplices, is a $D_n$-equivalence between $n$-excisive functors, hence an equivalence on \ord{n} cross-effects. Therefore, taking totalizations, which commute with cross-effects, we see that $\iota_*$ is an equivalence on \ord{n} cross-effects, and so is a $D_n$-equivalence. Finally, we identify the source of $\iota_*$, up to equivalence, with
\[ \Omega^\infty \Map(B(\mathsf{1},P,L)(n),X^{\smsh n})^{\Sigma_n} \]
which, by Lemma \ref{lem:F_E} once more, has \ord{n} derivative $\dual B(\mathsf{1},P,L)(n)$.
\end{proof}

\begin{cor} \label{cor:Phi-left-derivative}
Let $F: \spectra \to \sset$ be a presented cell functor. Then $\Phi_n(F)$ is also a pointed simplicial homotopy functor and has \ord{n} Goodwillie derivative equivalent to $\der_n(F)$ (of Definition \ref{def:left}).
\end{cor}

\begin{remark} \label{rem:Phi-left-n-excisive}
As discussed in Remark \ref{rem:phi-n-excisive} for functors from spaces to spectra, the functor $\Phi_n(F)$ of Definition \ref{def:Phi-left} is, in fact, $n$-excisive.
\end{remark}

\begin{definition} \label{def:phi-left}
For a presented cell functor $F: \spectra \to \sset$, we now construct a map $\phi: F \to \Phi_n(F)$ along similar lines to the map $\phi$ of Definition \ref{def:phi}.If $F$ is a finite cell functor, we define
\[ \phi'_F(n): F(X) \smsh \der^n(\Sigma^\infty F) \to \Map(S_c,X^{\smsh n}) \]
that are adjoint to the natural evaluation
\[ \Sigma^\infty F(X) \smsh \Nat(\Sigma^\infty F(X),X^{\smsh n}) \to X^{\smsh n}. \]

The maps $\phi'_F(n)$ are $\Sigma_n$-equivariant and so together they determine a map
\[ \phi'_F: F(X) \to \Hom_{\mathsf{\Sigma}}(\der^*(\Sigma^\infty F),\Map(S_c,X^{\smsh *})). \]
Now the symmetric sequences on the right-hand side here are left $\tilde{\der}^*(\Sigma^\infty \Omega^\infty)$-modules, and we claim that the above map factors via the mapping object for left modules on the right-hand side in place of the mapping object for symmetric sequences. This is equivalent to showing that the following diagram commutes:
\[ \begin{diagram}
  \node{F(X) \smsh \der^{k}(\Sigma^\infty \Omega^\infty) \smsh \der^{n_1}(\Sigma^\infty F) \smsh \dots \smsh \der^{n_k}(\Sigma^\infty F)} \arrow{e} \arrow{s,r}{(\phi'_F(n_1) \smsh \dots \smsh \phi_F(n_r)) \circ \Delta_{F(X)}} \node{F(X) \smsh \der^{n_1+\dots+n_k}(\Sigma^\infty F)} \arrow{s,r}{\phi'_F(n_1+\dots+n_k)} \\
  \node{\der^{k}(\Sigma^\infty \Omega^\infty) \smsh \Map(S_c,X^{\smsh n_1}) \smsh \dots \smsh \Map(S_c,X^{\smsh n_k})} \arrow{e} \node{\Map(S_c,X^{\smsh(n_1+\dots+n_k)})}
\end{diagram} \]
The left vertical map comes from combining the diagonal on the pointed simplicial set $C(X)$ with the maps $\beta_C(n_1),\dots,\beta_C(n_k)$. The top horizontal map is the module structure on $\der^*(\Sigma^\infty C)$, and the bottom horizontal map is the module structure on $\Map(S_c,X^{\smsh *})$.

Making use of various adjunctions, the commutativity of the above diagram is implied by that of
\[ \begin{diagram}
  \node{\Sigma^\infty F(X) \smsh \Map(S_c,S_c^{\smsh k})} \arrow{e} \arrow{s,l}{\Delta_{F(X)}} \node{\Sigma^\infty \Omega^\infty \Sigma^\infty F(X) \smsh \Map(S_c,S_c^{\smsh k})} \arrow{s} \\
  \node{\Sigma^\infty(F(X)^{\smsh k}) \smsh \Map(S_c,S_c^{\smsh k})} \arrow{e} \node{(\Sigma^\infty F(X))^{\smsh k}}
\end{diagram} \]
Here the top map comes from the $(\Sigma^\infty,\Omega^\infty)$-adjunction, the left vertical map from the diagonal on $F(X)$, the bottom map from identifying $\Sigma^\infty F(X)$ with $S_c \smsh F(X)$, and the right vertical map from the Yoneda isomorphism
\[ \Map(S_c,S_c^{\smsh k}) \isom \Nat(\Sigma^\infty \Omega^\infty E, E^{\smsh k}). \]
The commutativity of this diagram is given by Lemma \ref{lem:diagonal}.

We therefore obtain a natural (in both $X$ and $F$) transformation
\[ \phi''_F: F(X) \to \Hom^{\mathsf{left}}_{\tilde{\der}^*(\Sigma^\infty \Omega^\infty)}\left(\der^*(\Sigma^\infty F), \Map(S_c,X^{\smsh *})\right). \]
Composing with the cofibrant replacement map
\[ \tilde{\der}^*(\Sigma^\infty F) \to \der^*(\Sigma^\infty F), \]
the truncation map
\[ \Map(S_c,X^{\smsh *}) \to \Map(S_c,X^{\leq n}), \]
and the map from $\Hom^{\mathsf{left}}_{P}$ to $\Ext^{\mathsf{left}}_{P}$ of Definition \ref{def:ext-maps}, we get a natural transformation
\[ \phi_F: F(X) \to \Ext^{\mathsf{left}}_{\tilde{\der}^*(\Sigma^\infty \Omega^\infty)}\left(\tilde{\der}^*(\Sigma^\infty F), \Map(S_c,X^{\leq n})\right). \]
Now for general $F: \spectra \to \sset$, we take the homotopy colimit of $\phi_C$ over $C \in \mathsf{Sub}(QF)$ and get
\[ \phi: F \to \Phi_n(F). \]
\end{definition}

The remainder of the proof of Theorem \ref{thm:left} then consists of showing that the map $\phi$ of Definition \ref{def:phi-left} is a $D_n$-equivalence. We start by interpreting $\phi$ in the following way, analogous to Lemma \ref{lem:phi-tot}.

\begin{lemma} \label{lem:phi-left-tot}
Let $F: \spectra \to \sset$ be a pointed simplicial functor. Then there is a commutative diagram of the form
\[ \begin{diagram} \dgARROWLENGTH=3.5em
  \node{F(X)} \arrow{e,t}{\phi''_F} \arrow{s}
    \node{\Hom^{\mathsf{left}}_{P}\left(\tilde{\der}^*(\Sigma^\infty F), \Map(S_c,X^{\leq n})\right)} \arrow{s} \\
  \node{\widetilde{\Tot} \; \left[ \Omega^\infty \dots \Sigma^\infty F(X)\right]} \arrow{e,t}{\phi''_{\Omega^\infty \dots \Sigma^\infty F}}
    \node{\widetilde{\Tot} \; \left[\Hom^{\mathsf{left}}_{P}\left(\tilde{\der}^*(\Sigma^\infty \Omega^\infty) \circ \dots \circ \tilde{\der}^*(\Sigma^\infty F), \Map(S_c,X^{\leq n})\right)\right]}
\end{diagram} \]
\end{lemma}
\begin{proof}
The bottom horizontal map is the totalization of the map of cosimplicial objects formed from maps $\phi''_{\Omega^\infty \dots \Sigma^\infty F}$ defined in a manner analogous to those of Definition \ref{def:phi-FOS}.
The commutativity of this diagram then comes from the naturality of the constructions of the maps $\phi''_F$ and $\phi''_{\Omega^\infty \dots \Sigma^\infty F}$.
\end{proof}

\begin{lemma} \label{lem:phi-left-OSF}
Let $F: \spectra \to \sset$ be a finite cell functor. The map
\[ \phi''_{\Omega^\infty \dots \Sigma^\infty F}: \Omega^\infty \dots \Sigma^\infty F(X) \to \Hom^{\mathsf{left}}_{P}\left(\tilde{\der}^*(\Sigma^\infty \Omega^\infty) \circ \dots \circ \tilde{\der}^*(\Sigma^\infty F), \Map(S_c,X^{\leq n})\right) \]
is a $D_n$-equivalence.
\end{lemma}
\begin{proof}
We have the following diagram, corresponding to that of Lemma \ref{lem:phi-FOS}, in which all the other maps are $D_n$-equivalences, showing that $\phi''_{\Omega^\infty \dots \Sigma^\infty F}$ is also a $D_n$-equivalence.
\[ \begin{diagram}
  \node{\Omega^\infty Q(\dots \Sigma^\infty F)} \arrow{e} \arrow{s}
    \node{\Hom_{\mathsf{\Sigma}}\left(\tilde{\der}^*(Q(\dots \Sigma^\infty F)), \Map(S_c,X^{\leq n})\right)} \arrow{s} \\
  \node{\Omega^\infty \dots \Sigma^\infty F} \arrow{e} \arrow{se,b}{\phi''_{\Omega^\infty \dots \Sigma^\infty F}}
    \node{\Hom_{\mathsf{\Sigma}}\left(\dots \circ \tilde{\der}^*(\Sigma^\infty F), \Map(S_c,X^{\leq n})\right)} \arrow{s} \\
  \node[2]{\Hom^{\mathsf{left}}_{P}\left(\tilde{\der}^*(\Sigma^\infty \Omega^\infty) \circ \dots \circ \tilde{\der}^*(\Sigma^\infty F), \Map(S_c,X^{\leq n})\right)}
\end{diagram} \]
\end{proof}

\begin{prop} \label{prop:phi-left}
Let $F: \spectra \to \sset$ be a pointed simplicial homotopy functor. Then the map
\[ \phi: F \to \Phi_n(F) \]
of Definition \ref{def:phi-left} is a $D_n$-equivalence.
\end{prop}
\begin{proof}
Arguing as in the proof of Proposition \ref{prop:phi}, this follows from Lemmas \ref{lem:phi-left-tot} and \ref{lem:phi-left-OSF}.
\end{proof}

This gives us Theorem \ref{thm:left} which we restate.

\begin{thm:left}
Let $F: \spectra \to \sset$ be a pointed simplicial homotopy functor. Then there is a natural equivalence of symmetric sequences
\[ \der_*(F) \homeq \der^G_*(F). \]
In other words, the left $\der_*(I)$-module constructed in Definition \ref{def:left} provides models for the Goodwillie derivatives of $F$.
\end{thm:left}
\begin{proof}
This follows from Corollary \ref{cor:Phi-left-derivative} and Proposition \ref{prop:phi-left}.
\end{proof}

\begin{remark}
The ideas of Remark \ref{rem:tate} apply to the case of functors from spectra to spaces as well. We have a tower of functors
\[ F \to \dots \to \Phi_nF \to \Phi_{n-1}F \to \dots \to \Phi_0F = * \]
with $\Phi_nF$ being $n$-excisive. The fibre
\[ \Delta_nF := \hofib(\Phi_nF \to \Phi_{n-1}F) \]
is equivalent to
\[ \Delta_nF(X) \homeq \Omega^\infty(\der_nF \smsh X^{\smsh n})^{h\Sigma_n} \]
and there is a fibre sequence
\[ D_nF(X) \to \Delta_n(F) \to \Omega^\infty \operatorname{Tate}_{\Sigma_n}(\der_nF \smsh X^{\smsh n}). \]
When all these Tate spectra vanish, we have $P_nF \homeq \Phi_nF$ and the Taylor tower of $F$ can be reassembled from the left $\der_*(I)$-module $\der_*F$.
\end{remark}

\begin{example} \label{ex:left}
Let $K$ be a finite cell complex in $\sset$ and consider the functor
\[ \spectra \to \sset; \quad X \mapsto K \smsh \Omega^\infty X. \]
This example is in some sense dual to the stable mapping space example considered in Example \ref{ex:right}. We calculate its derivatives and their left $\der_*(I)$-module structure.

We start with
\[ \der^*(\Sigma^\infty K \smsh \Omega^\infty(-)) = \Nat(\Sigma^\infty K \smsh \spectra(S_c,X),X^{\smsh *}) \isom \Map(\Sigma^\infty K, (S_c)^{\smsh n}) \homeq \dual K. \]
The left $\der^*(\Sigma^\infty \Omega^\infty)$-module structure on this symmetric sequence is then essentially given by the diagonal on $K$ by the maps
\[ S \smsh \Map(K,S) \smsh \dots \smsh \Map(K,S) \to \Map(K^{\smsh n},S) \arrow{e,t}{\Delta_K} \Map(K,S). \]
The derivatives of $K \smsh \Omega^\infty(-)$ are now given by
\[ \der_*(K \smsh \Omega^\infty(-)) = \dual B(\mathsf{1},\tilde{\der}^*(\Sigma^\infty \Omega^\infty), \tilde{\der}^*(\Sigma^\infty K \smsh \Omega^\infty(-))) \]
which, in this case, is equivalent to
\[ \dual B(\mathsf{1},\mathsf{Com},\underline{\dual K}) \]
where $\underline{\dual K}$ denotes the left $\mathsf{Com}$-module with $\dual_K$ in every term and structure maps as above.

We remark that the left $\der_*(I)$-module $\der_*(K \smsh \Omega^\infty(-))$ appeared in \cite[8.10]{ching:2005a} (there written $M_K$) as a natural way to associate a left $\der_*(I)$-module to a space $K$. Here we have shown that this module arises as the derivatives of the functor $K \smsh \Omega^\infty(-): \spectra \to \sset$. Unlike in Example \ref{ex:right}, we do not have a more explicit description of the individual derivatives of this functor, besides the bar construction above.
\end{example}

\section{Functors from spaces to spaces} \label{sec:spaces-spaces}

To deal with functors from spaces to spaces, we combine the constructions of sections \ref{sec:spaces-spectra} and \ref{sec:spectra-spaces} above. We thus obtain a $\der_*(I)$-bimodule structure on the derivatives of such a functor.

\begin{definition} \label{def:bi}
Let $F: \sset \to \sset$ be a presented cell functor. Then $\Sigma^\infty F \Omega^\infty$ is a presented cell functor (by \ref{lem:extend-omega} and \ref{lem:extend-sigma}) that has a bi-comodule structure over the comonad $\Sigma^\infty \Omega^\infty$. The symmetric sequence $\der^*(\Sigma^\infty F \Omega^\infty)$ then has the structure of a pro-bimodule over the operad $\der^*(\Sigma^\infty \Omega^\infty)$.

We then set
\[ \der^*(F) := B(\mathsf{1},P,\tilde{\der}^*(\Sigma^\infty F \Omega^\infty), P, \mathsf{1}). \]
where $P = \tilde{\der}^*(\Sigma^\infty \Omega^\infty)$. This is the bimodule bar construction of Definition \ref{def:pro-bar} and forms a pro-bicomodule over $B(\tilde{\der}^*(\Sigma^\infty \Omega^\infty))$. Then, for a general pointed simplicial functor $F: \sset \to \sset$, we set
\[ \der_*(F) := \dual \der^*(QF) \]
This is then a $\der_*(I)$-bimodule.
\end{definition}

\begin{theorem} \label{thm:bi}
Let $F: \sset \to \sset$ be a pointed simplicial homotopy functor. Then there is a natural equivalence of symmetric sequences
\[ \der_*(F) \homeq \der^G_*(F). \]
\end{theorem}
\begin{proof}
This follows the same pattern as Theorems \ref{thm:right} and \ref{thm:left}. Firstly, we can use the constructions of \ref{def:sigmainfty-module} and \ref{def:sigmainfty-module-left} to make the symmetric sequence
\[ \Map(S_c,(\Sigma^\infty X)^{\smsh *}) \]
into a $\der^*(\Sigma^\infty \Omega^\infty)$-bimodule for any $X \in \sset$. We then define
\[ \Phi_n(F) := \Ext^{\mathsf{bi}}_{\tilde{\der}^*(\Sigma^\infty \Omega^\infty)}\left(\tilde{\der}^*(\Sigma^\infty F \Omega^\infty), \Map(S_c,(\Sigma^\infty X)^{\leq n})\right) \]
where this is the $\Ext$-object for $\tilde{\der}^*(\Sigma^\infty \Omega^\infty)$-bimodules (Definition \ref{def:ext-modules}).

Combining the arguments of \ref{prop:phi-derivative} and \ref{prop:Phi-left-derivative} we see that $\Phi_n(F)$ is a pointed simplicial homotopy functor with \ord{n} Goodwillie derivative equivalent to $\der_n(F)$.

We then construct a comparison map
\[ \phi: F \to \Phi_n(F) \]
is a similar way to the maps $\phi$ of Definitions \ref{def:phi} and \ref{def:phi-left}. When $F$ is a finite cell functor, $\phi$ comes from the evaluation maps
\[ \Sigma^\infty F \Omega^\infty(E) \smsh \der^r(\Sigma^\infty F \Omega^\infty) \to E^{\smsh r}. \]
Taking the adjoint to this, setting $E = \Sigma^\infty X$ and combining with the unit map from $F(X)$ to $F \Omega^\infty \Sigma^\infty(X)$, we get maps
\[ F(X) \to \spectra\left(\der^r(\Sigma^\infty F \Omega^\infty), \Map(S_c,(\Sigma^\infty X)^{\smsh r})\right) \]
which together form the basis of $\phi$. For general pointed simplicial $F$, we take the homotopy colimit over the finite subcomplexes of $QF$.

We then see that $\phi$ factors via
\[ \Tot \left[ \Omega^\infty \dots \Sigma^\infty F \Omega^\infty \dots \Sigma^\infty \right]. \]
This is the totalization of the bicosimplicial object formed using the unit and counit maps of the $(\Sigma^\infty,\Omega^\infty)$-adjunction. We then see that $\phi$ is a $D_n$-equivalence using Theorem \ref{thm:key} (twice, once on each side) and a compilation of Lemmas \ref{lem:phi-FOS} and \ref{lem:phi-left-OSF}. The theorem then follows.
\end{proof}

\begin{remark}
The ideas of Remark \ref{rem:tate} again apply in the spaces to spaces case. The $\der_*(I)$-bimodule $\der_*F$ determines a tower of functors $\Phi_*F$ with $\Phi_nF$ being $n$-excisive. The layers of this tower are
\[ \Delta_nF(X) \homeq \Omega^\infty(\der_nF \smsh (\Sigma^\infty X)^{\smsh n})^{h\Sigma_n} \]
and fit into fibre sequences
\[ D_nF \to \Delta_nF \to \Omega^\infty \operatorname{Tate}_{\Sigma_n}(\der_nF \smsh (\Sigma^\infty X)^{\smsh n}). \]
Thus when the Tate spectra vanish, the $\der_*(I)$-bimodule $\der_*F$ determines the Taylor tower of $F$.
\end{remark}

\begin{examples} \label{ex:spaces-spaces}
Notice that the functors $\Sigma^\infty \sset(K,-)$ and $K \smsh \Omega^\infty(-)$, whose derivatives were calculated in Examples \ref{ex:right} and \ref{ex:left} respectively, really come from functors from spaces to spaces of the form
\[ \sset(K,-) \quad \text{and} \quad K \smsh -. \]
In Example \ref{ex:right} we showed that the derivatives of $\Sigma^\infty \sset(K,-)$ are given by
\[ \dual B(K^{\smsh *}, \mathsf{Com}, \mathsf{1}). \]
To recover the derivatives of $\sset(K,-)$ we know now that we should now take a bar construction on the left with respect to the left $\der^*(\Sigma^\infty \Omega^\infty)$-action on the dual of this. The symmetric sequence $K^{\smsh *}$ has a natural left $\mathsf{Com}$-module structure (essentially this is the same structure described in Definition \ref{def:sigmainfty-module-left}). The derivatives of $\sset(K,-)$ are then given by the bimodule bar construction
\[ \der_*(\sset(K,-)) \homeq \dual B(\mathsf{1},\mathsf{Com},K^{\smsh *},\mathsf{Com},\mathsf{1}). \]
It can be shown that this is equivalent, as a $\der_*(I)$-bimodule, to
\[ \Map(K,\der_*(I)) \]
where the symmetric sequence $\{\Map(K,\der_nI)\}$ gets a $\der_*(I)$-bimodule structure from the bimodule structure on $\der_*(I)$ itself (that comes from the operad structure), and the diagonal map on $K$.

More generally, our methods can be used to produce, for any pointed simplicial homotopy functor $F: \sset \to \sset$ and any finite cell spectrum $K$, an equivalence of $\der_*(I)$-bimodules
\[ \der_*(\sset(K,F)) \to \Map(K,\der_*(I)). \]

The derivatives of $K \smsh - $ are given by the bimodule bar construction
\[ \der_*(K \smsh -) \homeq \dual B(\mathsf{1},\mathsf{Com}, \underline{\dual K}, \mathsf{Com}, \mathsf{1}) \]
where $\underline{\dual K}$ is as in Example \ref{ex:left}, and has the left $\mathsf{Com}$-module structure described there, and the right $\mathsf{Com}$-module structure coming from
\[ \dual K \smsh S \smsh \dots \smsh S \isom \dual K. \]
These derivatives are equivalent, as a $\der_*(I)$-bimodule to the symmetric sequence $\{K^n \smsh \der_nI\}$ with bimodule structure again coming from that on $\der_*(I)$ and the diagonal on $K$.
\end{examples}

\section{A Koszul duality result for operads of spectra} \label{sec:koszul}

Having produced the claimed module structures on the derivatives of functors to and/or from spaces, we now turn to the chain rules for such derivatives. This essentially follows from Theorem \ref{thm:key}, but to put it in the form we are looking for, we need a further result about bar constructions for operads in $\spectra$ and their modules. This result is essentially a weak form of the `Koszul duality' for operads of Ginzburg-Kapranov \cite{ginzburg/kapranov:1994} transferred to the context of spectra.

\begin{definition} \label{def:koszul-map}
Let $P$ be a reduced operad in $\spectra$ with right and left $P$-modules $R$ and $L$ respectively. Ignoring the homotopy-theoretic consequences for the moment, let us write $\dual X$ for $\Map(X,S)$. Then $\dual B(\mathsf{1},P,\mathsf{1})$ is an operad with right and left modules $\dual B(R,P,\mathsf{1})$ and $\dual B(\mathsf{1},P,L)$ respectively. Furthermore, recall (by dualizing the maps of Proposition \ref{prop:bar}) that we have composition maps
\[ \dual B(R,P,\mathsf{1}) \circ \dual B(\mathsf{1},P,L) \to \dual B(R,P,L) \]
and hence
\[ \dual B(R,P,\mathsf{1}) \circ \dual B(\mathsf{1},P,\mathsf{1})^{k} \circ \dual B(\mathsf{1},P,L) \to \dual B(R,P,L) \]
for each $k$. Together these make up a map of symmetric sequences
\[ B(\dual B(R,P,\mathsf{1}), \dual B(\mathsf{1},P,\mathsf{1}), \dual B(\mathsf{1},P,L)) \to \dual B(R,P,L). \]
To make this have better homotopical properties, we can compose with termwise-cofibrant replacements for the operad and modules on the left-hand side to get
\[ \Gamma: B\left(\widetilde{\dual B(R,P,\mathsf{1})}, \widetilde{\dual B(\mathsf{1},P,\mathsf{1})}, \widetilde{\dual B(\mathsf{1},P,L)}\right) \to \dual B(R,P,L). \]
\end{definition}

\begin{theorem}[Weak Koszul duality for operads in $\spectra$] \label{thm:koszul}
Let $P$ be a reduced operad in $\spectra$ with right and left pro-$P$-modules $R$ and $L$ respectively. Suppose that $P$, $R$, $L$ are directly-dualizable. Then the map
\[ \Gamma: B\left(\widetilde{\dual B(R,P,\mathsf{1})}, \widetilde{\dual B(\mathsf{1},P,\mathsf{1})}, \widetilde{\dual B(\mathsf{1},P,L)}\right) \to \dual B(R,P,L) \]
of Definition \ref{def:koszul-map} is a weak equivalence of symmetric sequences.

If, moreover, $R$ is a $P$-bimodule, then $\Gamma$ is an equivalence of left $P$-modules. If $L$ is a $P$-bimodule, $\Gamma$ is an equivalence of right $P$-modules. If $R$ and $L$ and both $P$-bimodules, then $\Gamma$ is an equivalence of $P$-bimodules.
\end{theorem}

\begin{corollary} \label{cor:koszul-operad}
Let $P$ be a directly-dualizable reduced operad in $\spectra$. Then there is an equivalence of \emph{symmetric sequences}
\[ \dual B(\widetilde{\dual BP}) \homeq P. \]
In other words, if we define the \emph{Koszul dual} of $P$ to be the operad
\[ K(P) := \widetilde{\dual BP}, \]
then $K(K(P)) \homeq P$.
\end{corollary}
\begin{proof}
This follows from Theorem \ref{thm:koszul} by taking $R = L = P$, using equivalences of $BP$-comodules of the form $\mathsf{1} \weq B(P,P,\mathsf{1})$ and $\mathsf{1} \weq B(\mathsf{1},P,P)$, as well as the equivalence $B(P,P,P) \weq P$ of Lemma \ref{lem:module-resolutions-eqs}.
\end{proof}

\begin{remark}
We call this a \emph{weak} form of Koszul duality because it only establishes an equivalence of symmetric sequences $K(K(P)) \homeq P$. Both these objects are operads and we are unable to show that there is an equivalence of \emph{operads} connecting them. This obviously would be a significant improvement on Corollary \ref{cor:koszul-operad}.
\end{remark}

Before we can prove Theorem \ref{thm:koszul}, we need one more fact about bar constructions.

\begin{lemma} \label{lem:bar-right-fibre}
Let $P$ be a reduced operad in $\spectra$, and suppose that
\[ R \to R' \to R'' \]
is a sequence of right $P$-modules such that
\[ R(n) \to R'(n) \to R''(n) \]
is a cofibre sequence in $\spectra$ for each $n$, i.e. $R \to R' \to R''$ is a \emph{termwise cofibre sequence}. Suppose that all the symmetric sequences $R,R',R'',P,L$ are termwise-cofibrant. Then the corresponding sequence
\[ B(R,P,L) \to B(R',P,L) \to B(R'',P,L) \]
is also a termwise cofibre sequence.
\end{lemma}
\begin{proof}
Taking smash products with a fixed spectrum preserves cofibre sequences, so we get termwise cofibre sequences of the form
\[ R \circ P^k \circ L \to R' \circ P^k \circ L \to R'' \circ P^k \circ L. \]
But then geometric realization (of Reedy-cofibrant objects) takes levelwise cofibre sequences of simplicial spectra to cofibre sequences of spectra. Therefore we get a termwise cofibre sequence
\[ B(R,P,L) \to B(R',P,L) \to B(R'',P,L) \]
as claimed.
\end{proof}

\begin{proof}[Proof of Theorem \ref{thm:koszul}]
We start by reducing to the case where $R$ and $L$ are ordinary $P$-modules, rather than pro-modules, so assume we already know that case. Suppose that $R$ is indexed on the cofiltered category $\cat{J}$ and $L$ on the cofiltered category $\cat{K}$. Then the source of the map $\Gamma$ is
\[ B\left(\hocolim_{j \in \cat{J}^{op}} \widetilde{\dual B(R_j,P,\mathsf{1})}, \widetilde{\dual B(\mathsf{1},P,\mathsf{1})}, \hocolim_{k \in \cat{K}^{op}} \widetilde{\dual B(\mathsf{1},P,L_k)}\right). \]
By Lemma \ref{lem:bar-hocolim}, this is equivalent to
\[ \hocolim_{(j,k) \in \cat{J}^{op} \times \cat{K}^{op}} B\left(\widetilde{\dual B(R_j,P,\mathsf{1})}, \widetilde{\dual B(\mathsf{1},P,\mathsf{1})},\widetilde{\dual B(\mathsf{1},P,L_k)}\right) \]
which, by our assumption, is equivalent via $\Gamma$ to
\[ \hocolim_{(j,k) \in \cat{J}^{op} \times \cat{K}^{op}} \dual B(R_j,P,L_k). \]
This is the definition of
\[ \dual B(R,P,L) \]
and so we have reduced to the case of ordinary $P$-modules. This case occupies the rest of the proof.s

Recall that the \emph{truncation} of a right $P$-module $R$ is the right $P$-module $R^{\leq n}$ given by
\[ R^{\leq n}(k) := \begin{cases} * & \text{if $k > n$}; \\ R(k) & \text{if $k \leq n$}, \end{cases} \]
and that there is a morphism of right $P$-modules
\[ R \to R^{\leq n}. \]
For any left $P$-module $L$, the induced map
\[ B(R,P,L) \to B(R^{\leq n},P,L) \]
is an isomorphism on terms up to and including $n$, that is
\[ B(R,P,L)(k) \isom B(R^{\leq n},P,L)(k) \]
for $k \leq n$.

Now consider the commutative diagram
\[ \begin{diagram}
  \node{B\left(\widetilde{\dual B(R^{\leq n},P,\mathsf{1})}, \widetilde{\dual B(\mathsf{1},P,\mathsf{1})}, \widetilde{\dual B(\mathsf{1},P,L)}\right)(k)} \arrow{s} \arrow{e,t}{\Gamma} \node{\dual B(R^{\leq n},P,L)(k)} \arrow{s} \\
  \node{B\left(\widetilde{\dual B(R,P,\mathsf{1})}, \widetilde{\dual B(\mathsf{1},P,\mathsf{1})}, \widetilde{\dual B(\mathsf{1},P,L)}\right)(k)} \arrow{e,t}{\Gamma} \node{\dual B(R,P,L)(k)}
\end{diagram} \]
where $k \leq n$. The remarks of the previous paragraph tell us that the vertical maps are equivalences and so we see that it is enough to prove the Theorem where the right $P$-module $R$ is \emph{bounded} (i.e. equal to $R^{\leq n}$ for some $n$).

Now define a right $P$-module $R^{=n}$ by
\[ R^{=n}(k) := \begin{cases} * & \text{if $k \neq n$}; \\ R(k) & \text{if $k = n$}. \end{cases} \]
There are then morphisms of right $P$-modules
\[ R^{=n} \to R^{\leq n}, \quad R^{\leq n} \to R^{\leq (n-1)} \]
and the resulting sequence
\[ R^{=n} \to R^{\leq n} \to R^{\leq (n-1)} \]
is a termwise cofibre sequence. Consider the commutative diagram
\[ \begin{diagram}
  \node{B\left(\widetilde{\dual B(R^{\leq (n-1)},P,\mathsf{1})}, \widetilde{\dual B(\mathsf{1},P,\mathsf{1})}, \widetilde{\dual B(\mathsf{1},P,L)}\right)(k)} \arrow{s} \arrow{e,t}{\Gamma} \node{\dual B(R^{\leq (n-1)},P,L)(k)} \arrow{s} \\
  \node{B\left(\widetilde{\dual B(R^{\leq n},P,\mathsf{1})}, \widetilde{\dual B(\mathsf{1},P,\mathsf{1})}, \widetilde{\dual B(\mathsf{1},P,L)}\right)(k)} \arrow{s} \arrow{e,t}{\Gamma} \node{\dual B(R^{\leq n},P,L)(k)} \arrow{s} \\
  \node{B\left(\widetilde{\dual B(R^{= n},P,\mathsf{1})}, \widetilde{\dual B(\mathsf{1},P,\mathsf{1})}, \widetilde{\dual B(\mathsf{1},P,L)}\right)(k)} \arrow{e,t}{\Gamma} \node{\dual B(R^{= n},P,L)(k)}
\end{diagram} \]
Then Lemma \ref{lem:bar-right-fibre}, together with the fact that Spanier-Whitehead duality takes cofibre sequences (of cofibrant and homotopy-finite spectra) to cofibre sequences (since cofibre and fibre sequences are equivalent in $\spectra$), implies that the vertical sequences here are cofibre sequences. By induction then, it is sufficient to prove the Theorem when the right $P$-module $R$ is concentrated in one position (i.e. equal to $R^{=n}$ for some $n$).

For a right $P$-module $R$ concentrated in a single term, the module structure is trivial (except for composition with the unit of the operad). Equivalently, there is an isomorphism of right $P$-modules
\[ R \isom R \circ \mathsf{1} \]
where the right $P$-module structure on $R \circ \mathsf{1}$ comes only from that on $\mathsf{1}$, i.e. via the augmentation of the reduced operad $P$.

The isomorphism of Proposition \ref{prop:bar-comprod} now tells us that
\[ B(R,P,L) \isom R \circ B(\mathsf{1},P,L). \]
Taking duals we get
\[ \dual B(R,P,L) \isom \dual(R \circ B(\mathsf{1},P,L)) \homeq \widetilde{\dual R} \circ \widetilde{\dual B(\mathsf{1},P,L)} \]
by Lemma \ref{lem:dualcomprod}. We therefore have also
\[ \widetilde{\dual B(R,P,L)} \homeq \widetilde{\dual R} \circ \widetilde{\dual B(\mathsf{1},P,L)}. \]
Using these equivalences on both source and target of the map $\Gamma$, we see that in this case $\Gamma$ is equivalent to the map
\[ B\left(\widetilde{\dual R} \circ \widetilde{\dual B(\mathsf{1},P,\mathsf{1})}, \widetilde{\dual B(\mathsf{1},P,\mathsf{1})}, \widetilde{\dual B(\mathsf{1},P,L)}\right) \to \widetilde{\dual R} \circ \widetilde{\dual B(\mathsf{1},P,L)}. \]
But this map is the composite of the isomorphism $\chi_r$ of Proposition \ref{prop:bar-comprod} and an equivalence of the form
\[ B(P',P',L') \to L' \]
from Lemma \ref{lem:module-resolutions-eqs}. Therefore, $\Gamma$ is indeed an equivalence for such $R$. This completes the proof.
\end{proof}

\section{Chain rules for functors of spaces and spectra} \label{sec:chainrule-spaces}

In this section, we prove a collection of chain rules for the Goodwillie derivatives of functors between spaces and spectra. In the spaces to spaces case, this generalizes the result of Klein-Rognes \cite{klein/rognes:2002} to higher derivatives.

\begin{theorem}[Chain rules in which the middle category is spectra] \label{thm:chainspectra}
Let $F: \spectra \to \cat{D}$ and $G: \cat{C} \to \spectra$ be pointed simplicial homotopy functors with $F$ finitary, and the categories $\cat{C}$ and $\cat{D}$ each either spaces or spectra. Then we have a natural equivalence (of symmetric sequences, left $\der_*(I)$-modules, right $\der_*(I)$-modules, or $\der_*(I)$-bimodules as appropriate):
\[ \der_*(FG) \homeq \der_*(F) \circ \der_*(G). \]
\end{theorem}
\begin{proof}
The case $\cat{C} = \cat{D} = \spectra$ was Theorem \ref{thm:chainrule}. If $\cat{C} = \sset, \cat{D} = \spectra$, then $\der_*(FG)$ and $\der_*(G)$ are as defined in Definition \ref{def:right}. Now Theorem \ref{thm:chainrule} tells us (here we use the hypothesis that $F$ is finitary) that we have an equivalence of pro-right-$\der^*(\Sigma^\infty \Omega^\infty)$-modules:
\[ \tilde{\der}^*(FG \Omega^\infty) \homeq \tilde{\der}^*(F) \circ \tilde{\der}^*(G \Omega^\infty). \]
Taking cofibrant replacements and applying Proposition \ref{prop:bar-pro-invariance}, we get an equivalence of pro-right-comodules
\[ B(\tilde{\der}^*(F G \Omega^\infty), \tilde{\der}^*(\Sigma^\infty \Omega^\infty), \mathsf{1}) \homeq B(\tilde{\der}^*(F) \circ \tilde{\der}^*(G \Omega^\infty), \tilde{\der}^*(\Sigma^\infty \Omega^\infty), \mathsf{1}). \]
Using the isomorphism of Proposition \ref{prop:bar-comprod} and taking Spanier-Whitehead duals, we get, by Lemma \ref{lem:dualcomprod}, the required equivalence
\[ \der_*(FG) \homeq \der_*(F) \circ \der_*(G). \]
The cases with $\cat{D} = \sset$ are similar.
\end{proof}

\begin{theorem}[Chain rules in which the middle category is spaces] \label{thm:chainspaces}
Let $F: \sset \to \cat{D}$ and $G: \cat{C} \to \sset$ be pointed simplicial homotopy functors with $F$ finitary, and the categories $\cat{C}$ and $\cat{D}$ each either spaces or spectra. Then we have a natural equivalence (of symmetric sequences, left modules, right modules, or bimodules as appropriate):
\[ \der_*(FG) \homeq B \left(\der_*(F), \der_*(I), \der_*(G)\right). \]
\end{theorem}
\begin{proof}
By replacing $F$ and $G$ with $QF$ and $G$ if necessary, we can, without loss of generality suppose that $F$ and $G$ are presented cell functors. (This uses the condition that $F$ be finitary since then $QF(X) \homeq F(X)$ for all $X \in \sset$ and so we have $FG \homeq (QF)(QG)$.)

Iterating Theorem \ref{thm:chainspectra}, we obtain equivalences
\[ \der^*(F \Omega^\infty) \circ \dots \circ \der^*(\Sigma^\infty G) \weq \der^*(Q(F\Omega^\infty \dots \Sigma^\infty G)). \]
These equivalences preserve the simplicial structures on the two sides and so taking realizations, we get a weak equivalence
\[ B \left(\tilde{\der}^*(F \Omega^\infty), \tilde{\der}^*(\Sigma^\infty), \tilde{\der}^*(\Sigma^\infty G) \right) \weq \left|\der^*(Q(F \Omega^\infty \dots \Sigma^\infty G))\right|. \]
We also have a map
\[ \left|\der^*(Q(F \Omega^\infty \dots \Sigma^\infty G)) \right| \to \der^*(Q(FG)) \]
built from the unit maps $FG \to F \Omega^\infty \dots \Sigma^\infty G$. Composing these and taking the Spanier-Whitehead duals (of these pro-objects) we get maps
\[ \der_*(FG) \to \widetilde{\Tot} \; \der_*(F \Omega^\infty \dots \Sigma^\infty G) \weq \dual B \left(\tilde{\der}^*(F \Omega^\infty), \tilde{\der}^*(\Sigma^\infty), \tilde{\der}^*(\Sigma^\infty G) \right). \]
However, the first map here is a weak equivalence by Theorem \ref{thm:key}, and so we get an equivalence
\[ \der_*(FG) \weq \dual B \left(\tilde{\der}^*(F \Omega^\infty), \tilde{\der}^*(\Sigma^\infty), \tilde{\der}^*(\Sigma^\infty G) \right). \]
Now we apply our Koszul duality result (Theorem \ref{thm:koszul}) to the right-hand side here which gives us a weak equivalence
\[ B \left(\der_*(F), \der_*(I), \der_*(G)\right) \weq \dual B \left(\tilde{\der}^*(F \Omega^\infty), \tilde{\der}^*(\Sigma^\infty \Omega^\infty), \tilde{\der}^*(\Sigma^\infty G) \right). \]
Combining these we have the necessary equivalence
\[ \der_*(FG) \homeq B \left(\der_*(F), \der_*(I), \der_*(G)\right). \]
\end{proof}

\begin{remark}
Theorem \ref{thm:chainspectra} can be written in the same form as Theorem \ref{thm:chainspaces} since the derivatives of the identity on spectra form the unit symmetric sequence $\mathsf{1}$ and we have
\[ B\left(\der_*(F), \mathsf{1}, \der_*(G)\right) \isom \der_*(F) \circ \der_*(G). \]
Thus in general, if $F:\cat{E} \to \cat{D}$ and $G: \cat{C} \to \cat{E}$ are pointed simplicial homotopy functors with $F$ finitary, then we have
\[ \der_*(FG) \homeq B \left(\der_*(F), \der_*(I_{\cat{E}}), \der_*(G)\right). \]
Recall that the bar construction is a model for the derived composition product, so this formula can also be rewritten as
\[ \der_*(FG) \homeq \der_*F \circ_{\der_*I} \der_*G. \]
\end{remark}

\begin{remark}
Inclusion of the zero-simplices of the simplicial bar construction gives us a map
\[ \der_*(F) \circ \der_*(G) \to B\left(\der_*(F), \der_*(I), \der_*(G)\right). \]
Putting this together with the equivalence of Theorem \ref{thm:chainspaces}, we get a \emph{weak} (i.e. only defined up to inverse weak equivalences) natural map
\[ \der_*(F) \circ \der_*(G) \to \der_*(FG). \]
The inverse weak equivalences prevent us from using this map to construct directly further operad and module structures. However, we do have explicit descriptions of the inverse weak equivalences involved and so by keeping track of these and making sure they are coherent, we might hope to obtain operad and module structures in any case.

As far as we know, no-one has yet constructed models for the derivatives of functors from spaces to spaces that allow for (suitably associative) point-set level maps
\[ \der_*(F) \circ \der_*(G) \to \der_*(FG). \]
Unpublished work has been done in this direction by Bill Richter and Andrew Mauer-Oats. It also seems that Lurie's framework of $\infty$-categories (see \cite{lurie:2008} and \cite{lurie:2008a}) could be useful for solving these rigidification problems.
\end{remark}

\begin{example} \label{ex:chainrule}
Let $F: \sset \to \spectra$ be any pointed simplicial homotopy functor. We saw in \S\ref{sec:spaces-spectra} that we can recover the derivatives of $F$ from the (dual) derivatives of $F \Omega^\infty$ along with the right $\der^*(\Sigma^\infty \Omega^\infty)$-module structure on those derivatives. We can now see how to reverse that process.

Theorem \ref{thm:chainspaces} implies that we have an equivalence
\[ \der_*(F \Omega^\infty) \homeq B(\tilde{\der}_*F, \tilde{\der}_*I, \tilde{\der}_*(\Omega^\infty)). \]
The derivatives of $\Omega^\infty$ are, as a left $\der_*(I)$-module given by the unit symmetric sequence with the trivial module structure. Taking duals, we can therefore write
\[ \der^*(F\Omega^\infty) \homeq \dual B(\tilde{\der}_*F,\tilde{\der}_*I, \mathsf{1}). \]
In order to recover the right $\der^*(\Sigma^\infty \Omega^\infty)$-module structure, we have to examine things more closely. Recall that this module structure is determined by the right-coaction map
\[ F \Omega^\infty \to F \Omega^\infty \Sigma^\infty \Omega^\infty. \]
According to the naturality of our chain rule, the corresponding map on derivatives is given by
\[ B(\tilde{\der}_*F, \tilde{\der}_*I, \tilde{\der}_*(\Omega^\infty)) \to B(\tilde{\der}_*F, \tilde{\der}_*I, \tilde{\der}_*(\Omega^\infty \Sigma^\infty \Omega^\infty)). \]
In order to understand this, we need to decipher the map
\[ \tag{*} \der_*(\Omega^\infty) \to \der_*(\Omega^\infty \Sigma^\infty \Omega^\infty) \]
of left $\der_*(I)$-modules. Note that this is equivalent to a map $\mathsf{1} \to \der_*(\Sigma^\infty \Omega^\infty)$ but that this does not preserve the left module structures. Instead, Theorem \ref{thm:left} tells us that (*) is given by
\[ \dual B\left(\mathsf{1},\der^*(\Sigma^\infty \Omega^\infty), \der^*(\Sigma^\infty \Omega^\infty)\right) \to \dual B\left(\mathsf{1},\der^*(\Sigma^\infty \Omega^\infty),\der^*(\Sigma^\infty \Omega^\infty \Sigma^\infty \Omega^\infty)\right). \]
We now have an equivalence
\[ \der^*(F \Omega^\infty) \homeq \dual B\left(\tilde{\der}_*F, \tilde{\der}_*(I), \dual B(\mathsf{1}, \der^*(\Sigma^\infty \Omega^\infty), \der^*(\Sigma^\infty \Omega^\infty))\right) \]
and the right $\der^*(\Sigma^\infty \Omega^\infty)$-module structure is now given by the induced action on the right of $B(\mathsf{1},\der^*(\Sigma^\infty \Omega^\infty),\der^*(\Sigma^\infty \Omega^\infty))$.

A similar argument applies for functors from spectra to simplicial sets and allows us to recover the (dual) derivatives of $\Sigma^\infty F$ and the left $\der^*(\Sigma^\infty \Omega^\infty)$-action from the left $\der_*I$-module $\der_*F$. Combining these two cases, we can do the same for functors from spaces to spaces and bimodules.
\end{example}

\begin{remark}
The relationship between $\der^*(F \Omega^\infty)$ and $\der_*F$ is an example of a `Koszul duality' for modules. For any (let us say termwise homotopy-finite) reduced operad $P$ and right $P$-module, we can form the right $\dual BP$-module $\dual B(R,P,\mathsf{1})$. Dually, from a right $\dual BP$-module $M$, we can form
\[ \dual B(M, \dual BP, \dual B(\mathsf{1},P,P)). \]
This inherits a right $P$-module structure from that on $B(\mathsf{1},P,P)$. Theorem \ref{thm:koszul} implies that composing these two constructions, we recover the original right $P$-module $R$. Subject to some finiteness conditions, these constructions set up a contravariant equivalence of homotopy categories between right $P$-modules and right $\dual BP$-modules. Again, analogous results hold for left modules and bimodules.
\end{remark}

\addcontentsline{toc}{part}{Appendix}

\appendix

\section{Categories of operads, modules and bimodules}

Here we describe in more detail some of the structure of the categories of operads, modules and bimodules that we use in this paper. In particular, we address the following related topics:
\begin{itemize}
  \item simplicial enrichment and tensoring;
  \item simplicial model structures;
  \item homotopy colimits;
  \item geometric realization of simplicial objects.
\end{itemize}
This material is used extensively in \S\ref{sec:cofibrant} to produce cofibrant replacements for operads and modules, and then in \S\S\ref{sec:spaces-spectra}-\ref{sec:spaces-spaces}, where we use filtered homotopy colimits of modules to form our models for Goodwillie derivatives.

Recall our notation for the following categories:
\begin{itemize}
  \item $\spectra^{\mathsf{\Sigma}}$: the category of symmetric sequences in $\spectra$
  \item $\spectra^{\mathsf{\Sigma}}_{\mathsf{red}}$: the category of reduced symmetric sequences in $\spectra$ (i.e. concentrated in terms $2$ and above)
  \item $\mathsf{Op}(\spectra)$: the category of reduced operads in $\spectra$
  \item $\mathsf{Mod}_{\mathsf{right}}(P)$: for a fixed reduced operad $P$ in $\spectra$, the category of right $P$-modules
  \item $\mathsf{Mod}_{\mathsf{left}}(P)$: for a fixed reduced operad $P$ in $\spectra$, the category of left $P$-modules
  \item $\mathsf{Mod}_{\mathsf{bi}}(P)$: for a fixed reduced operad $P$ in $\spectra$, the category of $P$-bimodules
\end{itemize}

We start by considering the categories $\spectra^{\mathsf{\Sigma}}$, $\spectra^{\mathsf{\Sigma}}_{\mathsf{red}}$ and $\mathsf{Mod}_{\mathsf{right}}(P)$. These are simpler than the others because they can all be realized as categories of functors $\cat{C} \to \spectra$ for appropriate categories $\cat{C}$:
\begin{itemize}
  \item for $\spectra^{\mathsf{\Sigma}}$, we have $\cat{C} = \mathsf{\Sigma}$;
  \item for $\spectra^{\mathsf{\Sigma}}_{\mathsf{red}}$, we have $\cat{C} = \mathsf{\Sigma} - \{1\}$
  \item for $\mathsf{Mod}_{\mathsf{right}}(P)$, we take $\cat{C}$ to be the following category enriched over $\spectra$: $\cat{C}$ has objects $\mathbb{N}$ and enrichment given by
      \[ \cat{C}(k,n) := \Wdge_{\text{partitions of $\{1,\dots,n\}$}} P(n_1) \smsh \dots \smsh P(n_k). \]
      The category of right $P$-modules is then equivalent to the category of $\spectra$-enriched functors $\cat{C} \to \spectra$.
\end{itemize}
We can give $\Sigma$ an enrichment over $\spectra$ by
\[ \Sigma(n,n) := S \smsh (\Sigma_n)_+, \quad \Sigma(m,n) := * \text{ for $m \neq n$}. \]
The categories of symmetric sequences are then also equivalent to categories of $\spectra$-enriched functors.

\begin{prop} \label{prop:right-mods}
Let $\cat{C}$ be a small $\spectra$-enriched category and let $\spectra^{\cat{C}}$ denote the category of $\spectra$-enriched functors from $\cat{C}$ to $\spectra$. Then
\begin{enumerate}
  \item $\spectra^{\cat{C}}$ is enriched over $\sset$ with
      \[ \spectra^{\cat{C}}(F,G) := \lim \left( \prod_{C \in \cat{C}} \spectra(FC,GC) \rightrightarrows \prod_{C,C' \in \cat{C}} \spectra(FC \smsh \cat{C}(C,C'),GC') \right) \]
  \item $\spectra^{\cat{C}}$ is tensored over $\sset$ with
      \[ (K \otimes F)(C) := K \smsh F(C); \]
  \item $\spectra^{\cat{C}}$ has all limits and colimits, and these are all calculated objectwise;
  \item $\spectra^{\cat{C}}$ has a simplicial model structure with generating cofibrations of the form
      \[ I_0 \smsh \cat{C}(C,-) \to I_1 \smsh \cat{C}(C,-) \]
      for $C \in \cat{C}$ and $I_0 \to I_1$ one of the generating cofibrations in $\spectra$;
  \item the intrinsic geometric realization of a simplicial object $F_{\bullet}$ in $\spectra^{\cat{C}}$ (see \cite[18.6.2]{hirschhorn:2003}) is isomorphic to that calculated termwise in $\spectra$, that is, there is a natural isomorphism
      \[ |F_{\bullet}|(C) \isom |F_{\bullet}(C)|; \]
  \item homotopy colimits in $\spectra^{\cat{C}}$ are calculated objectwise, that is, for a diagram $\mathcal{F}: \cat{J} \to \spectra^{\cat{C}}$, there is a natural equivalence
      \[ \left[\hocolim_{j \in \cat{J}} \mathcal{F}_j\right](C) \homeq \hocolim_{j \in \cat{J}} \left[\mathcal{F}_j(C)\right] \]
      where the homotopy colimit on the left is calculated in the simplicial model category $\spectra^{\cat{C}}$ and that on the right in $\spectra$.
\end{enumerate}
\end{prop}
\begin{proof}
Parts (1)-(3) are standard results of enriched category theory (see \cite{kelly:2005}). The proof of (4) is the same as that of Proposition \ref{prop:model-functors}. Homotopy colimits and geometric realization are based on the tensoring over $\sset$ and taking colimits. Therefore parts (2) and (3) imply (5) and (6).
\end{proof}

\begin{remark}
The simplicial enrichments of part (1) of Proposition \ref{prop:right-mods} are isomorphic to those given in Definitions \ref{def:enriched-symseq} and \ref{def:enriched-modules}. The generating cofibrations described in part (4) can easily be rewritten in the form given in Definition \ref{def:gen-cofs}.
\end{remark}

We now turn to the categories $\mathsf{Op}(\spectra)$, $\mathsf{Mod}_{\mathsf{left}}(P)$ and $\mathsf{Mod}_{\mathsf{bi}}(P)$. The main goal of the rest of this appendix is to show that certain analogues of the parts of Proposition \ref{prop:right-mods} hold also in these categories. These analogues are easy generalizations of results of EKMM \cite[VII]{elmendorf/kriz/mandell/may:1997} for categories of S-algebras and we follow their approach closely. Much of their analysis applies directly and we only fill in the details needed to transfer their arguments to our setting.

The main idea is that each of our categories of interest is equivalent to the category of algebras for the monad given by the corresponding `free object' functor on the category of symmetric sequences. The structures we are interested in are transferred via this monad from the corresponding structures for symmetric sequences.

\begin{definition} \label{def:free-functors}
We recall the definitions of the relevant free functors:
\begin{itemize}
  \item $F: \spectra^{\mathsf{\Sigma}}_{\mathsf{red}} \to \spectra^{\mathsf{\Sigma}}_{\mathsf{red}}$ given in Definition \ref{def:free-operads} as
      \[ F(A)(n) := \Wdge_{T \in \mathsf{T}_n} A(T) \]
      with the wedge product taken over all rooted trees with leaves labelled $1,\dots,n$;
  \item for a fixed reduced operad $P$, the functor $L: \spectra^{\mathsf{\Sigma}} \to \spectra^{\mathsf{\Sigma}}$ given in Definition \ref{def:free-modules} as
      \[ L(A) := P \circ A \]
  \item for a fixed reduced operad $P$, the functor $M: \spectra^{\mathsf{\Sigma}} \to \spectra^{\mathsf{\Sigma}}$ given by
      \[ M(A) := P \circ A \circ P. \]
\end{itemize}
\end{definition}

\begin{lemma}
Each of the functors $F,L,M$ of Definition \ref{def:free-functors} has the structure of a monad on the appropriate category of symmetric sequences. The categories $\mathsf{Op}(\spectra)$, $\mathsf{Mod}_{\mathsf{left}}(P)$ and $\mathsf{Mod}_{\mathsf{bi}}(P)$ are equivalent to the categories of algebras over the monads $F$, $L$ and $M$ respectively.
\end{lemma}
\begin{proof}
For $F$, the monad structure is given by grafting trees. For $L$ and $M$, the monad structure comes from the operad structure on $P$. The identification of operads and modules with algebras over these monads is standard.
\end{proof}

\begin{lemma}
Each of the functors $F$, $L$ and $M$ is simplicial with respect to the enrichments of the categories of symmetric sequences given by Proposition \ref{prop:right-mods}(1).
\end{lemma}
\begin{proof}
The simplicial enrichments of $L$ and $M$ are constructed in Definition \ref{def:enriched-symseq-maps}. Let $A$ and $B$ be reduced symmetric sequences. Then there is a projection map
\[ \pi_n: \Hom_{\Sigma}(A,B) \to \Hom_{\Sigma_n}(A(n),B(n)) \]
for each $n$. For each tree $T \in \mathsf{T}_n$, we then get
\[ \begin{split} \pi_T: \Hom_{\Sigma}(A,B)
    &\dgTEXTARROWLENGTH=4em\arrow{e,t}{\Smsh_v \pi_{i(v)} \circ \Delta} \Smsh_{v \in T} \Hom_{\Sigma_{i(v)}}(A(i(v),B(i(v)))) \\
    &\dgTEXTARROWLENGTH=4em\arrow{e} \Hom_{\prod_v \Sigma_{i(v)}}(A(T),B(T))
\end{split} \]
where $\Delta$ denotes the reduced diagonal on the pointed simplicial set $\Hom_{\Sigma}(A,B)$. Wedging together over $T \in \mathsf{T}_n$, we then get
\[ \Hom_{\Sigma}(A,B) \to \Hom_{\Sigma_n}(F(A)(n),F(B)(n)) \]
and together these form the required map
\[ \Hom_{\Sigma}(A,B) \to \Hom_{\Sigma}(F(A),F(B)). \]
This gives the simplicial enrichment of the free operad functor $F$.
\end{proof}

One of the key technical results we need is the fact that our free functors preserve `reflexive coequalizers', for which we now recall the definition.

\begin{definition}
A diagram of the form
\[ A \rightrightarrows B \to C \]
is a \emph{reflexive coequalizer} if it is a coequalizer diagram and the two maps from $A$ to $B$ have a common right inverse.
\end{definition}

\begin{lemma}
Each of the functors $F$, $L$ and $M$ preserves reflexive coequalizers.
\end{lemma}
\begin{proof}
Let $A \rightrightarrows B \to C$ be a reflexive coequalizer in $\spectra^{\mathsf{\Sigma}}_{\mathsf{red}}$. Since colimits of symmetric sequences are determined termwise, each of the diagrams
\[ A(n) \rightrightarrows B(n) \to C(n) \]
is a reflexive coequalizer in $\spectra$. As in the proof of \cite[II.7.2]{elmendorf/kriz/mandell/may:1997}, it follows that for any tree $T \in \mathsf{T}_n$, the diagram
\[ A(T) \rightrightarrows B(T) \to C(T) \]
is a coequalizer. Taking coproducts preserve coequalizers and so
\[ F(A)(n) \rightrightarrows F(B)(n) \to F(C)(n) \]
is a coequalizer. Finally, since again colimits in $\spectra^{\mathsf{\Sigma}}_{\mathsf{red}}$ are computed termwise, it follows that
\[ F(A) \rightrightarrows F(B) \to F(C) \]
is a coequalizer. Thus $F$ preserves reflexive coequalizers.

The proofs for $L$ and $M$ are similar based again on \cite[II.7.2]{elmendorf/kriz/mandell/may:1997} and the fact that the terms in $P \circ A$ and $P \circ A \circ P$ are given by taking appropriate smash products of the terms in $A$ (together with terms in $P$).
\end{proof}

\begin{corollary} \label{cor:left-mods-model}
The categories $\mathsf{Op}(\spectra)$, $\mathsf{Mod}_{\mathsf{left}}(P)$ and $\mathsf{Mod}_{\mathsf{bi}}(P)$ have all limits and colimits, and are enriched and tensored over $\sset$.
\end{corollary}
\begin{proof}
Using the preceding lemmas, this is given by \cite[II.7.4 and VII.2.10]{elmendorf/kriz/mandell/may:1997}.
\end{proof}

The tensors in these categories now allow us to define geometric realization for simplicial objects in the usual way (see \cite[18.6]{hirschhorn:2003}).

\begin{definition}
Let $\cat{C}$ be any category enriched and tensored over $\sset$, and with all colimits, and let $X_{\bullet}$ be a simplicial object in $\cat{C}$. The \emph{geometric realization} of $X_{\bullet}$, denoted $|X_{\bullet}|_{\cat{C}}$, is the object of $\cat{C}$ given by the coend (see \cite[IX.6]{maclane:1998})
\[ |X_{\bullet}|_{\cat{C}} := \Delta[n]_+ \otimes_{\mathsf{\Delta}} X_n \]
where $\otimes$ denotes the tensoring of $\cat{C}$ over $\sset$.
\end{definition}

\begin{proposition} \label{prop:left-mods-realization}
Let $\cat{C}$ be one of the categories $\mathsf{Op}(\spectra)$, $\mathsf{Mod}_{\mathsf{left}}(P)$ or $\mathsf{Mod}_{\mathsf{bi}}(P)$ and let $X_{\bullet}$ be a simplicial object in $\cat{C}$. Then we have a natural isomorphism of symmetric sequences
\[ \tag{*} |X_{\bullet}|_{\cat{C}} \isom |X_{\bullet}|_{\spectra^{\mathsf{\Sigma}}}. \]
\end{proposition}
\begin{proof}
The corresponding result for $S$-algebras is proved in \cite[VII.3.3]{elmendorf/kriz/mandell/may:1997} and we follow that argument. We do the operad case, with the left module and bimodule cases similar.

Firstly, \cite[X.1.4]{elmendorf/kriz/mandell/may:1997} implies that, for a simplicial symmetric sequence $A_{\bullet}$:
\[ F|A_{\bullet}|_{\spectra^{\mathsf{\Sigma}}} \isom |F(A_{\bullet})|_{\spectra^{\mathsf{\Sigma}}}. \]
This is because $F(A)$ is built from coproducts and smash products each of which is preserved by geometric realization. This implies that if $X_{\bullet}$ is a simplicial object in $\mathsf{Op}(\spectra)$, then $|X_{\bullet}|_{\spectra^{\mathsf{\Sigma}}}$ inherits the structure of an $F$-algebra (i.e. an operad).

Now the geometric realization $|-|_{\cat{C}}$ is left adjoint to the functor from $\cat{C}$ to $s\cat{C}$ (i.e. the category of simplicial objects in $\cat{C}$ given by
\[ Y \mapsto \Map(\Delta[-]_+,Y). \]
Here $\Map(-,-)$ denotes the cotensoring of $\cat{C}$ over $\sset$. Note that since cotensors in $\cat{C}$ are calculated termwise, this is the same as the cotensoring of the category of symmetric sequences over $\sset$. We now prove the Proposition by showing that $|-|_{\spectra^{\mathsf{\Sigma}}}$, with the $F$-algebra structure of the previous paragraph, is also left adjoint to this same functor.

Suppose given a map
\[ f: |X_{\bullet}|_{\spectra^{\mathsf{\Sigma}}} \to Y \]
of symmetric sequences. This is adjoint to a map
\[ \bar{f}: X_{\bullet} \to \Map(\Delta[-]_+,Y) \]
of simplicial symmetric sequences. It is now sufficient to show that $f$ is a map in $\cat{C}$ if and only if $\bar{f}$ is a map in $s\cat{C}$.

Suppose first that $\bar{f}$ is a map in $s\cat{C}$. Then we have the following commutative diagram
\[ \tag{**} \begin{diagram}
  \node{FX_{\bullet}} \arrow{s} \arrow{e,t}{F\bar{f}} \node{F\Map(\Delta[-]_+,Y)} \arrow{e} \node{\Map(\Delta[-]_+,FY)} \arrow{s} \\
  \node{X_{\bullet}} \arrow[2]{e,t}{\bar{f}} \node[2]{\Map(\Delta[-]_+,Y)}
\end{diagram} \]
Taking the geometric realization and using the counit map $|\Map(\Delta[-]_+,Y)| \to Y$, and the isomorphism $F|X_{\bullet}| \isom |FX_{\bullet}|$ we get a commutative diagram
\[ \tag{***} \begin{diagram}
  \node{F|X_{\bullet}|} \arrow{e,t}{Ff} \arrow{s} \node{FY} \arrow{s} \\
  \node{|X_{\bullet}|} \arrow{e,t}{f} \node{Y}
\end{diagram} \]
which tells us that $f$ is a map in $\cat{C}$.

Conversely, if $f$ is a map in $\cat{C}$, we have a commutative diagram (***). Applying $\Map(\Delta[-]_+,-)$ to this, and using the unit map $X_{\bullet} \to \Map(\Delta[-]_+,|X_{\bullet}|)$, we get the previous diagram (**). Therefore $\bar{f}$ is a map in $s\cat{C}$.
\end{proof}

We now turn to the existence of simplicial model structures on the categories $\mathsf{Op}(\spectra)$, $\mathsf{Mod}_{\mathsf{left}}(P)$ or $\mathsf{Mod}_{\mathsf{bi}}(P)$. These follow essentially by Theorem VII.4.7 of \cite{elmendorf/kriz/mandell/may:1997}. To apply this we have to check that the free object monads for these categories satisfy the `Cofibration Hypothesis' (see \cite[VII.4.12]{elmendorf/kriz/mandell/may:1997}). This involves an analysis of certain colimits.

\begin{lemma} \label{lem:pushouts}
Let $\cat{C}$ be one of the categories $\mathsf{Op}(\spectra)$, $\mathsf{Mod}_{\mathsf{left}}(P)$ or $\mathsf{Mod}_{\mathsf{bi}}(P)$. Consider a pushout diagram in $\cat{C}$ of the form
\[ \begin{diagram}
  \node{F(A)} \arrow{e} \arrow{s} \node{X'} \arrow{s} \\
  \node{F(B)} \arrow{e} \node{X}
\end{diagram} \]
where $F$ denotes the free object functor in $\cat{C}$ and the left-hand vertical map is induced by a morphism $A \to B$ of symmetric sequences. Suppose that $A \to B$ is a coproduct of a set of generating cofibrations in $\spectra^{\mathsf{\Sigma}}$. Then each of the maps $X'(n) \to X(n)$ is a monomorphism in $\spectra$. Moreover, if $X'$ is $\Sigma$-cofibrant (and $P$ is $\Sigma$-cofibrant when $\cat{C}$ is one of the module categories) then the map $X' \to X$ is a $\Sigma$-cofibration.
\end{lemma}
\begin{proof}
We follow the approach of \cite[VII.3.5-3.9]{elmendorf/kriz/mandell/may:1997}. This involves constructing a simplicial model for the pushout $X$. This model has $k$-simplices
\[ F(B) \amalg \underbrace{F(A) \amalg \dots \amalg F(A)}_k \amalg X' \]
with face maps given by the maps $F(A) \to X'$ and $F(A) \to F(B)$, and the appropriate codiagonals, and degeneracies given by the inclusions. Note that here $\amalg$ means the coproduct in $\cat{C}$. The map $X' \to X$ then factors as
\[ \tag{*} X' \to F(B) \amalg X' \to |F(B) \amalg F(A)^{\bullet} \amalg X'| \isom X. \]
The key step is now to show that for any $Y \in \cat{C}$ and $M \in \spectra^{\mathsf{\Sigma}}$, the map
\[ Y \to F(M) \amalg Y \]
is the inclusion of a wedge summand in the category of symmetric sequences. The description of colimits in $\cat{C}$ given in \cite[II.7.4]{elmendorf/kriz/mandell/may:1997} tells us that, as a symmetric sequence, $F(M) \amalg Y$ is given by the coequalizer of a diagram
\[ \tag{**} F(FFM \wdge FY) \rightrightarrows F(FM \wdge Y) \]
For each of our categories $\cat{C}$, we can write $FY = \bar{F}Y \wdge Y$ (as symmetric sequences). Then (**) becomes
\[ Y \wdge \bar{F}Y \wdge FFM \wdge \bar{F}(FFM \wdge FY) \rightrightarrows Y \wdge FM \wdge \bar{F}(FM \wdge Y) \]
and it can be checked that this diagram has the trivial diagram $Y \rightrightarrows Y$ (with both maps the identity) as a wedge summand. Passing to coequalizers, we see that $Y$ is a wedge summand of $F(M) \amalg Y$.

The degeneracy maps in the simplicial object on the right-hand side of (*) are now all inclusions of wedge summands on the level of symmetric sequences. It follows from this that the inclusion of the zero simplices (i.e. the second map in (*)) is a spacewise inclusion of spectra, hence a monomorphism. The first map in (*) is itself an inclusion of a wedge summand, hence a monomorphism, so we deduce the first part of the lemma.

For the second part, we first show that $F(M) \amalg Y$ is $\Sigma$-cofibrant when $Y$ is a $\Sigma$-cofibrant object of $\cat{C}$ and $M$ is a $\Sigma$-cofibrant symmetric sequence. This follows by an argument analogous to that of \cite[VII.6.1]{elmendorf/kriz/mandell/may:1997}. (See also \cite[4.6]{harper:2008} and \cite[Lemma 3]{spitzweck:2001}.) The map $X' \to F(B) \amalg X'$ in (*) is a retract of a $\Sigma$-cofibrant object, so is a $\Sigma$-cofibration. Since $X'$ is $\Sigma$-cofibrant, this map is isomorphic to the inclusion of a subcomplex (in the category of symmetric sequences). Similarly, all the degeneracy maps in the simplicial model for the pushout $X$ are inclusions of subcomplexes of symmetric sequences. It follows, by \cite[X.2.7]{elmendorf/kriz/mandell/may:1997} that the second map in (*) is also the inclusion of a subcomplex, and hence the composite $X' \to X$ is a $\Sigma$-cofibration.
\end{proof}

\begin{lemma} \label{lem:free-filtered}
The monads $F,L,M$ of Definition \ref{def:free-functors} (considered as functors from symmetric sequences to symmetric sequences) preserve filtered colimits.
\end{lemma}
\begin{proof}
This is related to Lemma \ref{lem:comprod-hocolim}. The main point is that if $A,B: \cat{J} \to \spectra$ are filtered diagrams of spectra, then
\[ \colim_{j \in \cat{J}} A_j \smsh B_j \isom [\colim_{j \in \cat{J}} A_j] \smsh [\colim_{j' \in \cat{J}} B_{j'}]. \]
Since the functors $F$, $L$ and $M$ are all built from smash products (and coproducts), they also preserve filtered colimits.
\end{proof}

\begin{lemma} \label{lem:colimits}
Let $\cat{C}$ be one of the categories $\mathsf{Op}(\spectra)$, $\mathsf{Mod}_{\mathsf{left}}(P)$ or $\mathsf{Mod}_{\mathsf{bi}}(P)$. Let $\mathcal{X}: \cat{J} \to \cat{C}$ be a filtered diagram in $\cat{C}$. Then the colimit of this sequence calculated in $\cat{C}$ is naturally isomorphic to the colimit calculated in the underlying category of symmetric sequences.
\end{lemma}
\begin{proof}
According to \cite[II.7.4]{elmendorf/kriz/mandell/may:1997}, the colimit of $\mathcal{X}$ in $\cat{C}$ is given by the coequalizer (in $\spectra^{\mathsf{\Sigma}}$) of
\[ F(\colim F(X_j)) \rightrightarrows F(\colim X_j) \]
where here $\colim$ denotes the colimit as symmetric sequences, and $F$ denotes the free object functor for $\cat{C}$. By Lemma \ref{lem:free-filtered} this is isomorphic to the coequalizer of
\[ \colim FF(X_i) \rightrightarrows \colim F(X_i) \]
which, since colimits commute, is just
\[ \colim X_i. \]
Thus the colimit in $\cat{C}$ is isomorphic to that in the underlying category of symmetric sequences.
\end{proof}

\begin{corollary}
The categories $\mathsf{Op}(\spectra)$, $\mathsf{Mod}_{\mathsf{left}}(P)$ and $\mathsf{Mod}_{\mathsf{bi}}(P)$ have simplicial cofibrantly-generated model structures with generating cofibrations given by applying the appropriate free object functor to the generating cofibrations in either $\spectra^{\mathsf{\Sigma}}_{\mathsf{red}}$ or $\spectra^{\mathsf{\Sigma}}$.
\end{corollary}
\begin{proof}
Lemmas \ref{lem:pushouts} and \ref{lem:colimits} together form the `Cofibration Hypothesis' and the corollary then follows essentially by \cite[VII.4.7]{elmendorf/kriz/mandell/may:1997} (or, to be precise, by a symmetric sequence version of this result).
\end{proof}

It now remains only to address the question of filtered homotopy colimits on the categories $\mathsf{Op}(\spectra)$, $\mathsf{Mod}_{\mathsf{left}}(P)$ and $\mathsf{Mod}_{\mathsf{bi}}(P)$.

\begin{prop} \label{prop:filtered-hocolim}
Let $\cat{C}$ be one of the categories $\mathsf{Op}(\spectra)$, $\mathsf{Mod}_{\mathsf{left}}(P)$ or $\mathsf{Mod}_{\mathsf{bi}}(P)$, with $P$ $\Sigma$-cofibrant. Let $\mathcal{X}: \cat{J} \to \cat{C}$ be a filtered diagram in $\cat{C}$. Then there is a natural equivalence between the homotopy colimit of $\mathcal{X}$ as calculated in $\cat{C}$, or in the underlying category of symmetric sequences.
\end{prop}
\begin{proof}
We use the fact that the homotopy colimit of $\mathcal{X}$ can be calculated by taking the strict colimit of a cofibrant approximation to $\mathcal{X}$ in the projective model structure on the relevant category of diagrams. By Lemma \ref{lem:colimits}, the strict colimit is the same whether calculated in the category $\cat{C}$ or in symmetric sequences. It is therefore enough to prove the following claim: suppose that $\mathcal{X}$ is cofibrant in the projective model structure on diagrams $\cat{J} \to \cat{C}$; then $\mathcal{X}$ is also cofibrant in the projective model structure on diagrams of symmetric sequences. Note that if $\cat{J}$ is the trivial category, this reduces to the fact that projectively-cofibrant operads and modules are $\Sigma$-cofibrant (i.e. termwise-cofibrant). We prove this statement in the same way: using a diagrammatic version of Lemma \ref{lem:pushouts}.

We can easily reduce to the case that $\mathcal{X}$ is a cell object in the diagram category $[\cat{J},\cat{C}]$. This means that $\mathcal{X}$ is the colimit of a sequence
\[ * = \mathcal{X}^{(0)} \to \mathcal{X}^{(1)} \to \dots \]
in which each map comes from a pushout in $[\cat{J},\cat{C}]$ of the form
\[ \begin{diagram}
  \node{\Wdge_{\alpha}^{\mathstrut} F(A_{\alpha}) \smsh \cat{J}(j_{\alpha},-)} \arrow{e} \arrow{s} \node{\mathcal{X}^{(i)}} \arrow{s} \\
  \node{\Wdge_{\alpha}^{\mathstrut} F(B_{\alpha}) \smsh \cat{J}(j_{\alpha},-)} \arrow{e} \node{\mathcal{X}^{(i+1)}}
\end{diagram} \]
where the $A_{\alpha} \to B_{\alpha}$ are generating cofibrations in the relevant category of symmetric sequences. In order to show that the colimit $\mathcal{X}$ is cofibrant as a diagram of symmetric sequences, it is sufficient to show for each $i$, that if $\mathcal{X}^{(i)}$ is cofibrant, then the map $\mathcal{X}^{(i)} \to \mathcal{X}^{(i+1)}$ is a cofibration (of diagrams of symmetric sequences). This follows from the same argument as the second part of Lemma \ref{lem:pushouts}.
\end{proof}

\bibliographystyle{amsplain}
\bibliography{mcching}

\end{document}